\documentclass[11pt, reqno]{article}
\usepackage[utf8]{inputenc}
\usepackage{amsmath}

\usepackage{amsfonts, mathrsfs}
\usepackage{mathtools}
\usepackage{amsthm}
\usepackage{amssymb}
\usepackage{caption}
\usepackage{subcaption}
\usepackage{graphicx}
\usepackage[affil-it]{authblk}
\usepackage[sc]{mathpazo}
\usepackage{dsfont}
\usepackage{enumerate}
\usepackage{wrapfig}
\usepackage[linesnumbered,ruled,vlined]{algorithm2e}
\DeclarePairedDelimiter\autobracket{(}{)}
\newcommand{\br}[1]{\autobracket*{#1}}
\usepackage[dvipsnames]{xcolor}
\newtheorem{theorem}{Theorem}
\newtheorem{lemma}[theorem]{Lemma}
\newtheorem{prop}[theorem]{Proposition}
\newtheorem{corollary}[theorem]{Corollary}

\newtheorem{remark}[theorem]{Remark}

\newtheorem*{claim*}{Claim}
\newtheorem{claim}[theorem]{Claim}

\numberwithin{equation}{section}
\numberwithin{theorem}{section}

\let\plainqed\qedsymbol
\newcommand{\claimqed}{$\lrcorner$}
\newenvironment{claimproof}{\begin{proof}\renewcommand{\qedsymbol}{\claimqed}}{\end{proof}\renewcommand{\qedsymbol}{\plainqed}}

\usepackage[letterpaper, margin=1in]{geometry}
\usepackage{hyperref}
\usepackage[sort&compress,numbers]{natbib}

\newcommand\blue[1]{#1}

\usepackage{array}   
\newcolumntype{L}[1]{>{\raggedright\let\newline\\\arraybackslash\hspace{0pt}}m{#1}}
\newcolumntype{C}[1]{>{\centering\let\newline\\\arraybackslash\hspace{0pt}}m{#1}}
\newcolumntype{R}[1]{>{\raggedleft\let\newline\\\arraybackslash\hspace{0pt}}m{#1}}

\def\sss{\scriptscriptstyle}
\newcommand{\PR}{\ensuremath{\mathbb{P}}}
\newcommand{\oP}{o_{\sss \PR}}
\newcommand{\OP}{O_{\sss \PR}}
\newcommand{\ThetaP}{\Theta_{\sss \PR}}
\newcommand{\newGamma}{\mathrm{Gamma}}  

\newcommand\disteq{\stackrel{\mathclap{\normalfont\mbox{d}}}{=\joinrel=}}
\newcommand{\bX}{\mathbf{X}} 

\newtheorem{innercustomgeneric}{\customgenericname}
\providecommand{\customgenericname}{}
\newcommand{\newcustomtheorem}[2]{%
  \newenvironment{#1}[1]
  {%
   \renewcommand\customgenericname{#2}%
   \renewcommand\theinnercustomgeneric{##1}%
   \innercustomgeneric
  }
  {\endinnercustomgeneric}
}

\newcustomtheorem{customthm}{Theorem}
\newcustomtheorem{customlemma}{Lemma}
\newcustomtheorem{customproposition}{Proposition}

\def\sss{\scriptscriptstyle}

\newcommand{\pto}{\ensuremath{\xrightarrow{\ \mathbb{P}\ }}}  
\newcommand{\dto}{\ensuremath{\ \xrightarrow{d} \ }}  
\newcommand{\R}{\mathbb{R}}                 
\newcommand{\N}{\mathbb{N}}                 
\newcommand{\PP}{\mathbb{P}}  				
\newcommand{\E}{\mathbb{E}}                 

\newcommand\blfootnote[1]{%
  \begingroup
  \renewcommand\thefootnote{}\footnote{#1}%
  \addtocounter{footnote}{-1}%
  \endgroup
}  

\title{Many-server asymptotics for Join-the-Shortest Queue in the Super-Halfin-Whitt Scaling Window}

\begin{document}
\author{
  Zhisheng Zhao$^1$, Sayan Banerjee$^2$, Debankur Mukherjee$^3$
}
\renewcommand\Authands{, }
\maketitle
\begin{abstract}
The Join-the-Shortest Queue (JSQ) is a classical benchmark for the performance of parallel-server queueing systems due to its strong optimality properties. Recently, there has been a significant progress in understanding its large-system asymptotic behavior since the work of Eschenfeldt and Gamarnik (Math.~Oper.~Res.~43 (2018) 867-886). In this paper, we analyze the JSQ policy in the \emph{super-Halfin-Whitt} scaling window when load per server $\lambda_N$ scales with the system size $N$ as $\lim_{N\rightarrow\infty}N^{\alpha}(1-\lambda_N)=\beta$ for $\alpha\in(1/2,1)$ and $\beta>0$. We establish that the centered and scaled total queue-length process converges to a certain Bessel process with negative drift and the associated (centered and scaled) steady-state total queue length, indexed by $N$, converges to a $\mathrm{Gamma}(2,\beta)$ distribution. The limit laws are universal in the sense that they do not depend on the value of $\alpha$, and exhibit fundamentally different behavior from both the Halfin-Whitt regime ($\alpha=1/2$) and the Non-degenerate Slowdown (NDS) regime ($\alpha=1$).

   \blfootnote{$^1$Georgia Institute of Technology, \emph{Email:} \href{mailto:zhisheng@gatech.edu}{zhisheng@gatech.edu}}
\blfootnote{$^2$University of North Carolina at Chapel Hill, \emph{Email:} \href{mailto:sayan@email.unc.edu}{sayan@email.unc.edu}}
\blfootnote{$^3$Georgia Institute of Technology, \emph{Email:}  \href{mailto:debankur.mukherjee@isye.gatech.edu}{debankur.mukherjee@isye.gatech.edu}}
\blfootnote{\emph{Keywords and phrases}. Join-the-Shortest Queue, diffusion limit, steady state analysis, regenerative process }
\end{abstract}

\section{Introduction}\label{sec:INTRODUCTION}
\subsection{Background and motivation.}
A canonical setup for parallel-server systems consists of $N$ identical servers, each with a dedicated queue.
Tasks arrive into the system as a Poisson process of rate $\lambda(N)$ and must be assigned to one of the queues instantaneously upon arrival, where they wait until executed. 
Tasks are assumed to have unit-mean exponentially distributed service
times, and the service discipline at each server is oblivious to the actual service requirements (viz., FCFS).
The Join-the-Shortest Queue (JSQ) policy for many-server systems has been a classical quantity of interest and has served as a benchmark for the quality of performance of task assignment policies. 
In the above setup, JSQ exhibits several strong optimality properties among the class of all non-anticipative task-assignment policies~\cite{EVW80, Winston77}.
In particular, it minimizes the joint queue length vector (in a stochastic majorization sense) and stochastically minimizes the total number of tasks in the system, and hence the mean overall delay.

While the exact analysis of the JSQ policy is intractable, the research community has made significant progress in understanding its behavior in various asymptotic regimes, primarily when the system is close to the boundary of its capacity region. 
The capacity region of the JSQ policy for the above homogeneous system of $N$ servers consists of the arrival rates $\lambda(N) < N$.
\blue{In the \emph{conventional heavy-traffic regime}, for a \emph{fixed} $N$, the behavior of the queue lengths is characterized as $\lambda(N)\to N$.
We note that this is not an asymptotic in $N$ since $N$ is fixed. Instead, in this regime, the system behavior is analyzed as the total arrival rate approaches $N$.}
There is a huge body of literature on this heavy-traffic analysis, which we do not attempt to review here.
Interested readers may look at~\cite{Foschini77,FS78,Reiman84,ZHW95} and the references therein for some of the related works.
More recently, motivated by the applications in large-scale service systems, such as data centers and cloud networks, there has been a growing interest in understanding the behavior the JSQ policy as the number of servers $N\to\infty$.
In that case, if the load per server, \blue{defined as $\lambda(N)/N$,} is fixed, \blue{that is, if $\lambda(N) = \lambda N$ for some fixed $\lambda\in (0,1)$}, then asymptotically, the fluid-scaled steady-state occupancy process becomes degenerate.
Specifically, as $N\to\infty$, a $\lambda$ proportion of servers have queue length 1 and the number of servers with queue length 2 or more vanishes~\cite{MBLW16-3}.
The behavior becomes intricate when $\lambda(N)$ scales with $N$ in a way that $\lambda(N)/N\to 1$ as $N\to\infty$. 
This is known as the \emph{many-server heavy-traffic regime}.
In a breakthrough work, Eschenfeldt and Gamarnik~\cite{EG15} characterized the transient limit of the occupancy process in the so-called Halfin-Whitt regime when $\lambda(N) = N - \beta \sqrt{N}$ for some $\beta>0$.
Since then, over the last few years, several works have been published investigating the many-server heavy-traffic limit of the JSQ policy, more of which we mention in Section~\ref{ssec:lit-rev} below.

The situation becomes more challenging 
when the system load is heavier than the Halfin-Whitt regime, that is, when $N-\lambda(N) = O(N^{\frac{1}{2} - \varepsilon})$ for some $\varepsilon\in (0, 0.5)$.
This is known as the \emph{super-Halfin-Whitt regime}\blue{, which was first introduced by Liu and Ying~\cite{LY21}}.
Note that due to ergodicity of the system and the fact that the service times are exponentially distributed with mean~1, the expected steady-state number of busy servers equals $\lambda(N)$, and thus, the idleness process (the process denoting the total number of idle servers) scales as $N^{\frac{1}{2} - \varepsilon}$.
However, comparing the JSQ system with the corresponding M/M/$N$ system \blue{(with $N$ parallel servers and a centralized queue for waiting tasks)}, one expects that the total number of waiting tasks centered at $N$, scales as  $N^{\frac{1}{2} + \varepsilon}$.
In other words, 
the idle-server process vanishes on the scale of the centered total number of tasks. 
The basic difference between the JSQ (parallel server) system and \blue{the M/M/$N$ (centralized queue) system} lies in the non-idling nature of the latter.
In the JSQ dynamics, there can be idle servers in the system, while having waiting tasks. 
The above observation poses an important question:
\emph{does the total number of tasks exhibit same asymptotic behavior as the M/M/$N$ system in the super-Halfin-Whitt regime? 
If not, then what is the cost of maintaining parallel queues instead of a centralized one?}
In this paper, we characterize the many-server asymptotics of the JSQ policy in the super-Halfin-Whitt regime, and answer the above questions. 
We discover that  even though the idle-server process is negligible in magnitude on the $N^{\frac{1}{2} + \varepsilon}$ scale, it evolves on a faster time scale and its local time accumulated at the reflection boundary provides a non-trivial positive drift to the total number of tasks.
For this reason, asymptotically, the centered and scaled total number of tasks in steady state is distributed as the sum of two independent exponential random variables for the JSQ policy, as opposed to a single exponential random variable in the M/M/$N$ case.
Moreover, both the steady state and process-level limiting behavior are universal in the sense that they do not depend on $\varepsilon \in (0, 0.5)$ and are fundamentally different from what have been observed in the Halfin-Whitt regime ($\varepsilon = 0$) and the Non-degenerate Slowdown (NDS) regime ($\varepsilon=0.5$).

\subsection{Literature review.}\label{ssec:lit-rev}
The literature on the performance analysis of the JSQ policy can be broadly categorized into two groups: 
(1) \emph{Many-server heavy-traffic regime:} When the number of servers, $N$, tends to infinity and the relative load per server for the $N$-th system, $\lambda(N)/N$, tends to 1 as $N\to\infty$. 
(2) \emph{Conventional heavy-traffic regime:} When the number of servers is fixed and the load per server $\lambda$ tends to 1.
Although the focus of the current work lies in the former regime, we will discuss that the two regimes share some commonality in performance if $\lambda(N)/N$, approaches 1 at a sufficiently fast rate.

In the subcritical regime, when $\lambda(N) = \lambda N$ for some $\lambda\in (0,1)$, Mukherjee et al.~\cite{MBLW16-3} characterized the transient and stationary behavior of the fluid limit for the JSQ policy using the time-scale separation technique by Hunt and Kurtz~\cite{HK94}.
The study of the many-server heavy traffic regime gained momentum since the seminal work by Eschenfeldt and Gamarnik~\cite{EG15}.
Here, the authors considered the limit of the system occupancy process $(Q_1^{(N)}(t), Q_2^{(N)}(t),\ldots)$, where $Q_i^{(N)}(t)$ is the number of servers with queue length $i$ or larger in the $N$-th system at time~$t$. 
Specifically, if $Q_3^{(N)}(0) = 0$, then uniformly on any finite time interval, $((N-Q_1^{(N)})/\sqrt{N}, Q_2^{(N)}/\sqrt{N})$ converges weakly to a certain two-dimensional reflected Ornstein-Uhlenbeck (OU) process with singular noise.
This convergence has been extended to steady state by Braverman~\cite{Braverman18} using a sophisticated generator expansion framework via the Stein's method, enabling the interchange of $N\to\infty$ and $t\to\infty$ limits.
Subsequently, the tail and bulk behavior of the stationary distribution of the limiting diffusion have been studied by Banerjee and Mukherjee~\cite{BM19a, BM19b}, although an explicit characterization of the stationary distribution is yet unknown.
Convergence to the above OU process has been extended for a class of power-of-$d$ policies in~\cite{MBLW16-3} and for the Join-Idle-Queue policy in~\cite{MBLW16-1}.

In the sub-Halfin-Whitt regime, i.e., when $N-\lambda(N) = O(N^{\frac{1}{2} + \varepsilon})$ for some $\varepsilon\in (0, 0.5)$, Liu and Ying~\cite{LY19} considered a general class of policies, including the JSQ policy, under the assumption that each server has a buffer size $b=o(\log N)$. 
They showed that in the steady state, the expected waiting time per job is $O\br{\frac{\log N}{\sqrt{N}}}$. 
As observed in~\cite{LY19}, the results in~\cite{EG15, Braverman18, LY19} imply that a phase transition occurs at $\varepsilon = 0$ where the limit of the quantity 
$\log\big[\E(\sum_{i=2}^\infty Q_i^{(N)})\big]/\log N$
jumps from 0 for $\varepsilon<0$, to $1/2$ for $\varepsilon = 0$.
Under the finite buffer assumption, Liu and Ying~\cite{LY21} further considered the JSQ policy in the super-Halfin-Whitt regime, i.e., when $N-\lambda(N) = O(N^{\frac{1}{2} - \varepsilon})$ for some $\varepsilon\in (0, 0.5)$. 
Here, the authors showed that in steady state, the expected number of servers with queue length 2 is $O(N^{\frac{1}{2}+\varepsilon}\log N)$ and with queue length~3 or larger is $o(1)$ and
conjectured that the true order of $\E\big(Q_2^{(N)}\big)$ should be $N^{\frac{1}{2}+\varepsilon}$. 
Our results confirm this conjecture as a corollary, without having any restriction on the buffer capacity.
\blue{To the best of our knowledge, the terms `sub/super-Halfin-Whitt' regimes were first introduced in~\cite{LY19} and~\cite{LY21} respectively. 
Since then the community adopted these terminologies in various papers (including ours) to describe the scaling of the load per server relative to the popular Halfin-Whitt regime. }

When $N-\lambda(N) = O(1)$, the system load is heavier \blue{than} the super-Halfin-Whitt regime. This is known as the Non-Degenerate Slowdown (NDS) regime.
\blue{This regime was formally introduced by Atar~\cite{Atar12}. However, as mentioned in \cite{Atar12}, earlier than that the regime was also considered by Mandelbaum~\cite{AM08}, Mandelbaum and Shaikhet~\cite{MS04ab}, Whitt~\cite{ww03}, and Gurvich~\cite{IG04}, where the authors did asymptotic analysis of the M/M/$N$ system motivated by call centers. 
Atar~\cite{Atar12} discussed related convergence under non-degenerate slowdown while analyzing heterogeneous many-server systems.}
Later, this regime was also considered by Maglaras et al.~\cite{MYZ18} from a revenue maximization perspective.
Gupta and Walton~\cite{GW19} analyzed the transient limit of the JSQ policy in this scaling regime and proved that the total queue-length process, scaled by $N$, has a diffusion limit which is similar to Bessel process with a constant drift.
The interchange of $t\to\infty$ and $N\to\infty$ limits in the NDS regime requires establishing the tightness of the sequence of scaled total queue-lengths in steady state indexed by $N$, which is not established in~\cite{GW19}. 
However, the results in~\cite{GW19}, in combination with the results in the current paper that analyze the case $N- \lambda(N) = O(N^{\frac{1}{2} - \varepsilon})$ for $\varepsilon \in (0, 0.5)$, hint at a second phase transition at $\varepsilon = 0.5$, where for all $i\geq 3$, the value of $Q_i^{(N)}$ in steady state jumps from 0 (for $\varepsilon < 0.5$) to $O(N)$ (for $\varepsilon = 0.5$).
In fact, if we pretend that the interchange of limits holds for the NDS regime, then the steady-state maximum queue length distribution has an exponential tail~\cite[Theorem 1]{GW19},  whereas for $\varepsilon < 0.5$, its value is 2 with probability tending to 1 as $N\to\infty$.

Recently, Hurtado-Lange and Maguluri~\cite{HM20, HM21} considered a class of power-of-$d$ policies in the super-slowdown regime, where $N - \lambda(N) = O(N^{1-\alpha})$ for some $\alpha > 1$ and established state-space collapse results.
In this regime, the average queue length at each server scales as $N^{\alpha - 1}$, which grows with $N$.
For the JSQ policy, the results in~\cite{HM20, HM21} imply the total queue length in steady state, scaled by $N^{\alpha}$, converges in distribution to an exponential random variable with mean 1 as $N\to\infty$, when $\alpha>2$.
To the best of our knowledge, this is the first heavy-traffic scaling window where the conventional heavy-traffic behavior coincides with the many-server heavy-traffic.
For a general many-server heavy-traffic regime ($\lambda(N)/N\rightarrow1$), Budhiraja et al.~\cite{BFW19} established a large deviation principle for the occupancy process which states that for large $N$ and $T$,
starting from the state $Q^{(N)}_1(0)=N$ and $Q^{(N)}_j(0)=0$ for all $j\geq 2$,
$\mathbb{P}(\sup_{0\leq t\leq T} Q^{(N)}_i(t)\geq 1)\approx \exp (-\frac{N(i-2)^2}{4T})$ for $i\geq 3$.
The works in various regimes mentioned above, have been summarized in Table~\ref{tab:lit}, which is an expanded version of the one presented in~\cite{HM20} (associated notations are described in Section \ref{notsecn}).
Also, see~\cite{BBLM21} for a recent survey on load balancing algorithms and their performance in various asymptotic regimes.
\begin{table}[htb]
\def\arraystretch{1.2}
  \begin{tabular}{C{1.75cm}|L{4cm}|L{6.75cm}|L{2.5cm}}
    Value of $\alpha$ & Regime & Asymptotic behavior& References\\
    \hline
    $0$ & Meanfield & $Q_1^{(N)} = N\lambda_N \pm \ThetaP(\sqrt{N\lambda_N}),$ $Q_i^{(N)}= \oP(1)$ for $i\geq 2$  &\cite{MBLW16-3}\\
    $(0, \frac{1}{2})$ & Sub-Halfin-Whitt & $\sum_{i=1}^b Q_i^{(N)} = N\lambda_N + \OP(\sqrt{N}\log N)$ & \cite{LY19}\\
    $\frac{1}{2}$ & Halfin-Whitt & $Q_1^{(N)} = N - \ThetaP(\sqrt{N}),$ $Q_2^{(N)} = \Theta(\sqrt{N})$, $Q_i^{(N)}=\oP(1)$ for $i\geq 3$ &\cite{Braverman18, EG15, BM19a, BM19b} \\
    $(\frac{1}{2}, 1)$ & Super-Halfin-Whitt & $Q_1^{(N)} = N - \Theta(N^{1-\alpha}),$ $Q_2^{(N)} = \ThetaP(N^\alpha)$, $Q_i^{(N)}=\oP(1)$ for $i\geq 3$ & \cite{LY21}, current paper\\
    $1$ & Non-Degenerate Slowdown (NDS) &  $Q_i = \ThetaP(N)$ for all $i\geq 1$ &\cite{GW19} \\
    $(1, \infty)$ & Super Slowdown & Unknown for $\alpha\in (1,2]$.
    For $\alpha>2$, $\sum_{i=1}^\infty Q_i^{(N)} = \ThetaP\big(N^\alpha\big)$ & \cite{HM20, HM21}
  \end{tabular}
  \caption{Analysis of JSQ in various regimes: Load per server is $\lambda_N = 1 - \frac{\beta}{N^{\alpha}}$ with
  $\beta\in (0,1)$ for $\alpha=0$ and
  $\beta>0$ for $\alpha>0$. The random variables in the third column are steady state random variables. $b \in [1, \infty)$ denotes the buffer size, when it is assumed to be finite.
  \label{tab:lit}}
\end{table}

\subsection{Our contributions.}
In this paper, we obtain the diffusion limit for the centered and scaled total number of tasks in the system $S^{(N)}(t) = \sum_{i=1}^\infty Q_i^{(N)}(t)$ under the super-Halfin-Whitt regime and characterize the limit of its stationary distribution as $N\rightarrow\infty$.
Specifically, we assume that the total arrival rate is $ \lambda(N)=N-\beta N^{\frac{1}{2}-\varepsilon}$ for $\varepsilon\in (0, 0.5)$ and the service times are exponentially distributed with mean 1.
Our main contributions are two-fold:

\paragraph{(a)~Process-level convergence.}
In Theorem~\ref{thm:PROCESS-LEVEL}, we show that 
$X^{(N)}(t)=N^{-(\frac{1}{2}+\varepsilon)}(S^{(N)}(N^{2\varepsilon}t)-N)$
converges to a certain Bessel process with negative drift, uniformly on compact intervals.
Since the difference between the total arrival rate and the maximum departure rate is $O\big( N^{\frac{1}{2}-\varepsilon}\big)$ and $S^{(N)}(t)$ is scaled by $N^{\frac{1}{2}+\varepsilon}$, we need the time-scaling of $N^{2\varepsilon}$ to obtain the process-level convergence.   
From a high-level perspective, we follow the same broad approach used in~\cite{GW19} to prove the transient limit result. 
However, our technique differs significantly in several places as some of the estimates in~\cite{GW19} only apply for $\varepsilon$ values close to $0.5$, but we need estimates that uniformly apply for $\varepsilon \in (0,0.5)$ (for example, the proof of Lemma~\ref{M/M/1 behavior} requires the estimate \eqref{e2}). \blue{The proof also involves some characteristics of pre-limit systems in the super-Halfin-Whitt regime (for example, with certain initial state, for any $0< t<\infty$, $\sup_{0\leq s\leq t}Q^{(N)}_2(s)$ will be less than $N$ with probability tending to 1 as $N\rightarrow\infty$, which is very different from the scenario in~\cite{GW19}).}

The key challenge in establishing the diffusion limit is to obtain precise asymptotics of the following integral of the idleness process $I^{(N)}(\cdot)$ (which equals $N - Q_1^{(N)}(\cdot)$): $N^{-(\frac{1}{2}+\varepsilon)}\int_0^{N^{2\varepsilon}t}I^{(N)}(s)ds$, uniformly for $t$ in an appropriate (random) compact interval.
As it turns out, starting from suitable states, the process $I^{(N)}$ is negligible in magnitude, on the scale $N^{\frac{1}{2}+\varepsilon}$ (Lemma~\ref{lem:MM1-IDLE}).
However, the above integral is not negligible. 
In fact, we show that it is asymptotically close to $\int_0^t \frac{1}{X^{(N)}(s)}ds$ (Proposition~\ref{prop:INT-IDLE-2}). 
We note that for a fixed $N$, $X^{(N)}(s)$ can hit 0, in which case, $\frac{1}{X^{(N)}(s)}$ will be undefined. 
However, under some assumptions on the initial state $X^{(N)}(0)$, it can be shown that for any $0<t<\infty$, $X^{(N)}(s)$ will not hit 0 in the interval $[0,t]$ with probability tending to 1 as $N\rightarrow\infty$. Thus, without loss of generality, we can take $\frac{1}{X^{(N)}(s)}$ to be zero if $X^{(N)}(s)=0$.
The analysis of the aforementioned integral relies on the fact that $I^{(N)}(t)$ evolves on a faster time scale compared to $X^{(N)}$ and achieves a local stationarity for any fixed value of $X^{(N)}$.
The (local) steady-state expectation of $I^{(N)}(s)$ is approximately $\frac{1}{X^{(N)}(s)}$.
Other parts of the evolution equation of the limiting diffusion in Theorem~\ref{thm:PROCESS-LEVEL} are obtained by standard martingale decomposition and their convergence.
We then show that this limiting process is ergodic and has $\newGamma(2, \beta)$ as the unique stationary distribution (Proposition~\ref{prop:STEADY-STATE}).

\paragraph{(b)~Tightness of the sequence of \blue{diffusion-scaled} pre-limit stationary distributions.}
The next major challenge, proving the interchange of $t\to\infty$ and $N\to\infty$ limits, requires establishing the tightness of the sequence of steady-state random variables $\{X^{(N)}(\infty), N^{-\frac{1}{2} + \varepsilon}I^{(N)}(\infty), N^{-\frac{1}{2} - \varepsilon}Q^{(N)}_2(\infty)\}_{N=1}^\infty$ as well as showing $\sum_{i=3}^{\infty}Q^{(N)}_3(\infty) \xrightarrow{P} 0$ as $N \rightarrow \infty$.
This is provided by Theorem~\ref{thm:TIGHTNESS-XN}.
In fact, Theorem~\ref{thm:TIGHTNESS-XN} tells us much more about the prelimit stationary distribution than tightness.
We use the theory of regenerative processes to obtain tail probability bounds on $X^{(N)}(\infty)$ for all large but fixed $N$. More precisely, we identify renewal times along the path of the Markov process $(I^{(N)}(\cdot), \{Q^{(N)}_i(\cdot)\}_{i \ge 2})$ and use them to obtain a representation \eqref{eq:repre-stat} of the stationary measure of the process. Tail behavior of $X^{(N)}(\infty)$ is then studied by carefully analyzing these renewal times and fluctuations of the process between these times. This renewal representation also gives the other tightness and convergence results stated above.

Two key technical steps in the renewal time analysis are to obtain sharp asymptotics for the following: 
    (i)~\emph{Down-crossing estimates}: Tail-probability bounds on the time $Q_2^{(N)}$ takes to hit $BN^{\frac{1}{2}+\varepsilon}$ starting from $2BN^{\frac{1}{2}+\varepsilon}$ (Proposition~\ref{prop:DOWNCROSS}), where $B$ is a large fixed constant that does not depend on~$N$.
    (ii)~\emph{Up-crossing estimates}:  Tail-probability bounds on the time $Q_2^{(N)}$ takes to hit $2BN^{\frac{1}{2}+\varepsilon}$ starting from $BN^{\frac{1}{2}+\varepsilon}$ (Proposition~\ref{prop:UPCROSS}).

Analyzing the tail behavior of $X^{(N)}(\infty)$ for fixed large $N$ requires very different techniques than ones required to prove the process level convergence. 
First, \emph{there is no `state space collapse' in the pre-limit}, in the sense that one cannot directly relate the idleness process $I^{(N)}(\cdot)$ to $X^{(N)}(\cdot)$ as in Proposition~\ref{prop:INT-IDLE-2}. 
Therefore, we take an excursion-theoretic approach where one performs a piece-wise analysis of excursions of the joint process $(I^{(N)}(\cdot), \{Q^{(N)}_i(\cdot)\}_{i \ge 2})$ in different parts of the state space. This, along with some novel concentration inequalities (Lemma \ref{lem:LEMMA-5}), leads to quantitative probability bounds on the supremum and time integral of $I^{(N)}(\cdot)$ (Lemma \ref{lem:LEMMA-6}). This is a crucial step in proving Theorem~\ref{thm:TIGHTNESS-XN}.

Second, note that, for the process level limit, one can `ignore' the contributions of $Q^{(N)}_i(\cdot)$ for $i \ge 3$. This is because, if $Q^{(N)}_3(0)=0$, for fixed $T>0$, the probability that $Q^{(N)}_3(t)>0$ for any $t \in [0, N^{2\varepsilon}T]$ becomes small as $N \rightarrow \infty$ (see \eqref{eq:prop-A3-3}). However, for fixed $N$, the processes $\{Q^{(N)}_i(\cdot)\}_{i \ge 3}$ eventually become non-zero, and for obtaining probabilistic bounds on the renewal times discussed above (see Lemma \ref{prop:RENEWAL-TIME} and Proposition \ref{cor:MOEMNT-RENEWAL}), one needs precise quantitative control on $\sum_{i=3}^{\infty}Q^{(N)}_i(\cdot)$ (eg. Lemma \ref{lem:LEMMA-8}). 

\blue{Finally, we briefly explain why the existing methods in the literature on proving tightness of the diffusion-scaled pre-limit stationary distributions fail to apply in our case. The celebrated method of \cite{budhiraja2009} involves analyzing stability properties of the associated `noiseless' system to construct suitable Lyapunov functions in terms of hitting times of certain compact sets in the state space. Unlike Jackson networks, the stability analysis of the noiseless system in the JSQ setting is quite involved and obtaining Lyapunov functions based on this approach thus becomes very intricate, even in the conventional Halfin-Whitt setting (see \cite{Anton20}). Moreover, the `singular' nature of the dynamics (arrivals only increase $Q_1^{(N)}(\cdot)$ when there are idle servers) makes the `small set' coupling analysis (characteristic of the \cite{meyn2012} approach used in \cite{budhiraja2009}) very challenging. This motivates a more pathwise analysis taken in this article. But the above discussion also sheds light on why analyzing suitable hitting times lies at the heart of our methods.}

\subsection{Notation and organization.}\label{notsecn}
For a metric space $S$, 
denote by $D=D([0,\infty),S)$ the space of functions from $[0,\infty)$ to $S$ that are right continuous and have left limits everywhere. 
For $x, y\in \R$, $x\vee y$ and $x\wedge y$ denote $\max(x, y)$ and $\min(x,y)$, respectively.
$x^+ = \max(x, 0)$.
For a positive deterministic sequence $(f(N))_{N\geq 1}$, a sequence of random variables 
$(X(N))_{N\geq 1}$ is said to be $\OP(f(N))$, $\oP(f(N))$, respectively if the sequence $(X(N)/f(N))_{N\geq 1}$
is tight, and $X(N)/f(N)\xrightarrow{\sss\PR} 0$, as $N\to\infty$.
Also, $(X(N))_{N\geq 1}$ is said to be $\ThetaP(f(N))$ if it is $\OP(f(N))$ and there is a constant $c>0$ such that $\liminf_{N\to\infty}(f(N))^{-1}\E(X(N)) \geq c$.
We will use the symbol `$\dto$' to denote convergence in distribution of random variables and `$\Rightarrow$' to denote weak convergence of stochastic processes uniformly on any compact time interval.

The rest of the sections are organized as follows: 
In Section~\ref{sec:MAIN} we present the model, main results, and discuss their ramifications.
Section~\ref{sec:HITTIME} contains a sample-path analysis of the process and several hitting time estimates. 
These estimates will be used in Section~\ref{sec:STEAYSTATE} to establish the tightness of the sequence of random variables corresponding to appropriately centered and scaled steady-state total number of tasks, in the number of servers $N$. In Section~\ref{sec:PROCESS-LEVEL}, we prove the process-level convergence.
Proofs of some results 
are moved to the appendix for better readability.

\section{Model Description and main results}\label{sec:MAIN}
Consider a system with $N$ parallel single-server queues and one dispatcher. 
Tasks with independent unit-mean exponentially distributed service requirements arrive at the dispatcher as a Poisson process of rate $\lambda(N)$. 
Denote the per-server load by $\lambda_N:= \lambda(N)/N$.
Each arriving task is assigned instantaneously and irrevocably to the shortest queue at the time of arrival. Ties are broken arbitrarily.
The service discipline at each server is oblivious to the actual service requirements (viz., FCFS).
We will analyze the system in the so-called \emph{super-Halfin-Whitt} heavy-traffic regime, where 
\begin{equation}\label{eq:SHW-def}
    \lambda_N=1-\frac{\beta}{N^{\frac{1}{2}+\varepsilon}}
\end{equation}
with fixed parameters $\beta>0$ and $\varepsilon\in(0,\frac{1}{2})$. 
Our goal is to characterize the behavior of the queue-length process as $N\to \infty$.
At time $t$,  $S^{(N)}(t)$ denotes the total number of tasks in the system, $I^{(N)}(t)$ denotes the number of idle servers, and $Q^{(N)}_i(t)$ denotes the number of servers of queue length at least $i$, $i\geq 1$.
Note that $\big(Q_1^{(N)}, Q_2^{(N)},\ldots\big)$ provides a Markovian description of the system state and $I^{(N)} = N - Q_1^{(N)}$, which we call the idleness process.
In the following, we will often consider $\big(I^{(N)}, Q_2^{(N)},\ldots\big)$ as the Markovian state descriptor instead. 
\blue{For integers $x \in [0,N]$, $y \in [0, N-x]$ and a vector of non-negative integers $\underline{z} = (z_1,z_2,\dots) \in \mathbb{N}_0^{\infty}$ with $y \ge z_1 \ge z_2 \ge \dots$, we will denote the system state by $(x, y, \underline{z})$ to mean that $I^{(N)} = x$, $Q_2^{(N)} = y$, and $(Q_3^{(N)}, Q_4^{(N)}, \ldots) = \underline{z}$.}

Let $A(\cdot)$ and $D(\cdot)$ be two independent unit-rate Poisson processes. We will write the arrival and departure processes as random time change of $A$ and $D$ respectively; see~\cite[Section 2.1]{pang07}. 
Hence, the cumulative number of arrivals up to time $t$ can be written as $A(N\lambda_Nt)$ and $D\br{\int_0^{t}(N-I^{(N)}(s))ds}$, respectively. 
We introduce the scaled process $\big\{X^{(N)}(t),t\geq 0\big\}$ as 
\begin{equation*}
    X^{(N)}(t):=\frac{S^{(N)}(N^{2\varepsilon}t)-N}{N^{\frac{1}{2}+\varepsilon}},
\end{equation*}
and thus,
\begin{equation}\label{eq:represent-XN}
    X^{(N)}(t)=X^{(N)}(0)+N^{-\frac{1}{2}-\varepsilon}\left[A(N^{1+2\varepsilon}\lambda_Nt) - D\Big(\int_0^{N^{2\varepsilon}t}(N-I^{(N)}(s))ds\Big)\right].
\end{equation}
Note that the process $X^{(N)}$ is not Markovian. However, we can view it as a function of the irreducible Markov process $\big(Q_1^{(N)}, Q_2^{(N)},\ldots\big)$.
\blue{Moreover, since the aggregate arrival rate $N\lambda_N$ is less than $N$, the system is \emph{subcritical} according to the definition given in~\cite{Bramson11}.
Further, since the arrivals occur as a Poisson process and the service times are i.i.d.~exponential, one can verify that the conditions of~\cite[Theorem 1.2]{Bramson11} are satisfied.
Thus, the process is $\big(Q_1^{(N)}, Q_2^{(N)},\ldots\big)$ is positive recurrent (e.g., the hitting time of the state $(0,0,\ldots)$ starting from any state has finite expectation), and hence ergodic for any fixed $N$.}
Denote by $X^{(N)}(\infty)$ a random variable distributed as the centered and scaled total number of tasks in the $N$-th system in steady state.
We will characterize the weak-limit of $X^{(N)}(\infty)$ and convergence of its moments by establishing the process-level limit, ergodicity of the limiting process, and the tightness of the random variables $\{X^{(N)}(\infty)\}_{N \ge 1}$. The next result states that the sequence of processes $\big\{X^{(N)}(t),t\geq 0\big\}_{N=1}^{\infty}$  converges weakly to the process $X$, uniformly on compact time intervals, 
where $\big\{X(t),t\geq 0\big\}$ is a certain Bessel process with negative drift. 
\begin{theorem}\label{thm:PROCESS-LEVEL-NEW}
\blue{Fix $\beta>0$ and $\varepsilon \in (0, \frac{1}{2})$. 
Assume the sequence of initial states $(X^{(N)}(0), I^{(N)}(0), Q^{(N)}_2(0), Q^{(N)}_3(0), \ldots)$
satisfy the following:
\begin{enumerate}[\normalfont (i)]
    \item $X^{(N)}(0)\dto X^*$ as $N\to \infty$, where $X^*$ is a positive random variable.
    \item $\big(I^{(N)}(0)/N^{\frac{1}{2}-\varepsilon}\big)_N$ is a tight sequence of random variables.
    \item $\big(Q^{(N)}_2(0)/N^{\frac{1}{2}+\varepsilon}\big)_N$ is a tight sequence of random variables.
    \item  $Q_3^{(N)}(0) \pto 0$ as $N\to\infty$. 
\end{enumerate}
Then, for any fixed $T>0$, the scaled process $X^{(N)}$ converges weakly to the path-wise unique solution of the following stochastic differential equation, uniformly on $[0,T]$:
\begin{equation}\label{langevin}
    dX(t)=\Big(\frac{1}{X(t)}-\beta\Big)dt+\sqrt{2}dW(t),
\end{equation}
with $X(0)\disteq X^*$,
where $W=\big(W(t),t\geq 0\big)$ is the standard Brownian motion independent of $X^*$.}
\end{theorem}

Theorem~\ref{thm:PROCESS-LEVEL-NEW} is proved in Section~\ref{sec:PROCESS-LEVEL}.

Observe that the SDE in~\eqref{langevin} is a certain Langevin diffusion \blue{(see \cite[Section 1.2]{RT96} for a discussion on Langevin diffusions)} and is ergodic.
Its stationary distribution can be explicitly characterized as follows. 
\begin{prop}\label{prop:STEADY-STATE}
The SDE
in~\eqref{langevin} 
has a unique stationary distribution $\pi$ with probability density function 
$    \frac{d\pi}{dx}=\beta^2 xe^{-\beta x},\  x>0,$ i.e., $\mathrm{Gamma}(2,\beta)$,
having $p$-th moment $\Gamma(p+2)/\beta^p$.
\end{prop}
The proof of Proposition~\ref{prop:STEADY-STATE} is given in Section~\ref{sec:PROCESS-LEVEL}.
\begin{remark}
The limiting diffusion in~\eqref{langevin} behaves like a Bessel process of dimension $2$ for small values of $X$ and a Brownian motion with negative drift for large values of $X$. In particular, this diffusion almost surely never hits zero (see Lemma \ref{lem:SDE-SOL}). This is a consequence of the fact that the idleness process in the prelimit moves in scale $N^{\frac{1}{2} - \varepsilon}$ and thus vanishes under the scaling $N^{\frac{1}{2} + \varepsilon}$ appearing in $X^{(N)}$. Moreover, the drift of the above diffusion process is smooth on $(0,\infty)$, unlike the diffusion limit in the NDS regime~\cite{GW19}. This results in the stationary density of the diffusion in \cite{GW19} being $C^1$ but not $C^2$ on $(0,\infty)$, whereas our stationary density is smooth on $(0,\infty)$.
\end{remark}

The following theorem gives tail estimates for $X^{(N)}(\infty)$ for all fixed large $N$ \blue{and shows that any subsequential weak-limit of $X^{(N)}(\infty)$ must have a strictly positive support}. 
It also partially characterizes the steady state behavior of the individual coordinates in $(I^{(N)}, Q^{(N)}_2,\sum_{i=3}^{\infty}Q^{(N)}_i)$ required for the application of Theorem \ref{thm:PROCESS-LEVEL} in the limit interchange argument of Theorem \ref{thm:INTERCHANGE}.
\begin{theorem}\label{thm:TIGHTNESS-XN}
Fix $\beta>0$ and $\varepsilon \in [0, \frac{1}{2})$. 
\begin{enumerate}[\normalfont (i)]
\item There exist positive constants $N_0,B, C_1,C_2$ such that for any $N\geq N_0$,
\begin{equation}
\label{eq:thm-2.3-tail}
\begin{split}
    \PR\big(X^{(N)}(\infty)\geq x\big)&
    \leq \begin{cases}C_1\exp\big\{-C_2x^{1/5}\big\},\quad 4B\leq x\leq 2N^{\frac{1}{2}-\varepsilon},\\
    C_1\exp \big\{-C_2x^{1/44}\big\},\quad x\geq 2N^{\frac{1}{2}-\varepsilon}.
    \end{cases}
\end{split}
\end{equation}
\item 
\blue{\begin{equation}\label{eq:steady-state-lb}
\lim_{\delta\downarrow0}\varliminf_{N\rightarrow\infty}\PP\big(X^{(N)}(\infty)\ge \delta\big)= 1.
\end{equation}}
\item $\sup_{N \ge 1} \E\left[N^{-\frac{1}{2} - \varepsilon}Q^{(N)}_2(\infty)\right] < \infty$, 
\item $\E\left[N^{-\frac{1}{2} + \varepsilon}I^{(N)}(\infty)\right] = \beta$ for any $N \ge 1$, and 
\item $\sum_{i=3}^{\infty}Q^{(N)}_i(\infty) \xrightarrow{P} 0$ as $N \rightarrow \infty$.
\end{enumerate}
\end{theorem}
Theorem~\ref{thm:TIGHTNESS-XN} is proved in Section~\ref{sec:STEAYSTATE}.
\begin{remark}\label{rendes}
The analysis of the steady-state tail behavior is based on a renewal theoretic representation of the stationary measure. We consider the continuous time Markov process $(I^{(N)}(\cdot), \{Q^{(N)}_i(\cdot)\}_{i \ge 2})$ starting from $(0, \lfloor 2B N^{\frac{1}{2} + \varepsilon}\rfloor, \underline{0})$ (that is, $I^{(N)}(0) = 0$, $Q_2^{(N)}(0)=\lfloor 2B N^{\frac{1}{2} + \varepsilon}\rfloor$ and $Q_3^{(N)}(0) = Q_4^{(N)}(0)=\dots=0$) for a suitably large $B>0$. We then wait for $Q_2^{(N)}$ to fall to $\lfloor B N^{\frac{1}{2} + \varepsilon}\rfloor$, then for $\sum_{i=3}^{\infty}Q^{(N)}_i$ to drop to zero, and subsequently for $Q_2^{(N)}$ to climb back to $\lfloor 2B N^{\frac{1}{2} + \varepsilon}\rfloor$. More formally,
define the stopping times $\sigma^{(N)}_0=0$, and for $i\geq 0$,
\begin{equation}\label{eq:sigma_i-def}
    \begin{split}
    \sigma^{(N)}_{2i+1}& \coloneqq\inf\{t\geq \sigma^{(N)}_{2i}: Q^{(N)}_2(t)\leq  \lfloor BN^{\frac{1}{2}+\varepsilon}\rfloor\},\\
    \sigma^{(N)}_{2i+2}& \coloneqq\inf\{t\geq \sigma^{(N)}_{2i+1}: Q^{(N)}_2(t)\geq  \lfloor 2BN^{\frac{1}{2}+\varepsilon} \rfloor \},\\
    \bar{K}^{(N)}&\coloneqq\inf \{k\geq 1:\Bar{Q}^{(N)}_3(\sigma^{(N)}_{2k})=0\}\quad\text{and}\quad \Theta^{(N)}\coloneqq \sigma^{(N)}_{2\bar{K}^{(N)}}.
    \end{split}
\end{equation}
Note that $I^{(N)}(\Theta^{(N)}) = 0$, $Q_2^{(N)}(\Theta^{(N)}) = \lfloor 2BN^{\frac{1}{2}+\varepsilon} \rfloor$ and $\Bar{Q}^{(N)}_3(\Theta^{(N)})=0$.
Therefore, $\Theta^{(N)}$ is a renewal time point.
Therefore, the process observed from time $\Theta^{(N)}$ onward has the same distribution as the one starting from time~$0$. On showing that $\Theta^{(N)}$ has finite expectation, the stationary distribution of $(I^{(N)}(\cdot), \{Q^{(N)}_i(\cdot)\}_{i \ge 2})$ admits the following representation: 
 \begin{equation}\label{eq:repre-stat}
    \pi\br{(I^{(N)}(\infty),Q^{(N)}_2(\infty),\dots)\in A}=\frac{\mathbb{E}_{(0, \, \lfloor 2BN^{\frac{1}{2}+\varepsilon} \rfloor, \,\underline{0})}\br{\int_0^{\Theta^{(N)}}\mathds{1}\br{(I^{(N)}(s),Q^{(N)}_2(s),\dots)\in A}ds}}{\mathbb{E}_{(0, \, \lfloor 2BN^{\frac{1}{2}+\varepsilon} \rfloor, \,\underline{0})}\br{\Theta^{(N)}}}
\end{equation}
for any Borel set $A \subseteq \{0,1,2,\ldots\}^{\infty}$, where $\mathbb{E}_{(0, \, \lfloor 2BN^{\frac{1}{2}+\varepsilon} \rfloor, \,\underline{0})}(\cdot)$ refers to the expectation given the starting state is $(0, \, \lfloor 2BN^{\frac{1}{2}+\varepsilon} \rfloor, \,\underline{0})$. 
Showing this is standard renewal theory (see, for example, proof of Theorem 3.3 in \cite{BM19a}): one expresses the stationary distribution in terms of the long time average occupancy measure which, in turn, can be expressed in terms of the i.i.d.~regenerative cycles of the process between successive renewal times. Application of the strong law of large numbers yields~\eqref{eq:repre-stat}.

Hence, obtaining tail behavior of the stationary distribution reduces to quantifying the extremal behavior of the process path before time $\Theta^{(N)}$. \blue{The exponents  in the tail bounds above arise from applying a concentration result (Lemma \ref{lem:LEMMA-5}) for sums of random variables with stretched-exponential tails that stochastically dominate the integral of the idleness process (see Lemma \ref{lemma 6}).
We do not believe that these exponents are optimal.
However, they show that the steady-state tails are sufficiently light to have finiteness of all moments.}
\bigskip
\end{remark}

\begin{remark} Note that Theorem \ref{thm:TIGHTNESS-XN} also applies to $\varepsilon =0$ (the Halfin-Whitt regime). Tightness of the diffusion-scaled steady state occupancy measures in the Halfin-Whitt regime was shown by~\cite{Braverman18} using a generator expansion method in conjunction with Stein's method. However, this requires apriori knowledge of the diffusion limit and an elaborate construction of Lyapunov functions based on the dynamics of the noiseless system associated with this diffusion limit. Moreover, the results in~\cite{Braverman18} provide tightness by bounding the expected values of $Q^{(N)}_i(\infty), \, i \ge 1$, but finiteness of higher moments remained open. Further, the analysis is highly tied to the Halfin-Whitt regime as the diffusion limit behavior drastically changes as one moves to the super-Halfin-Whitt regime.
In contrast, the renewal theoretic method for studying stationary distributions developed in the current paper can be applied directly to the pre-limit process without knowledge of its diffusion limit. In addition to giving a much more detailed picture of the tail behavior of the steady state, it can be applied universally for all $\varepsilon \in [0, \frac{1}{2})$ as the above phase transition does not appear in the pre-limit behavior. Recently, the sample path approach in the current paper has been utilized in~\cite{braverman2022join} to obtain rates of convergence of the JSQ system in Halfin-Whitt regime to its diffusion limit by directly analyzing the generator of the pre-limit process.
\end{remark}

Theorem~\ref{thm:TIGHTNESS-XN} implies that $\big\{X^{(N)}(\infty)\big\}_{N=1}^{\infty}$ is tight in $\R$. 
By Theorem~\ref{thm:PROCESS-LEVEL}, Proposition~\ref{prop:STEADY-STATE}, and Theorem~\ref{thm:TIGHTNESS-XN}, the interchangeability of the $t\to\infty$ and $N\to\infty$ limits in Theorem~\ref{thm:INTERCHANGE} can be established using routine arguments.
The proof of Theorem~\ref{thm:INTERCHANGE} is given at the end of this section.
\begin{theorem}\label{thm:INTERCHANGE}
Fix $\beta>0$ and $\varepsilon \in (0, \frac{1}{2})$. The sequence of random variables $\big\{X^{(N)}(\infty)\big\}_{N\geq 1}$ converges weakly to the $\newGamma\br{2,\beta}$ distribution as $N\rightarrow\infty$, where the probability density function of\ $\newGamma\br{2,\beta}$ is given by 
$
    f(x)=\beta^2xe^{-\beta x},\ x\in(0,\infty).
$
Moreover, for any $p>0$, $\E \Big[X^{(N)}(\infty)\Big]^p\rightarrow \frac{\Gamma(p+2)}{\beta^p}$
as $N\to\infty$, where $\Gamma$ denotes the standard Gamma function. 
\end{theorem}


Let $W^{(N)}$ be a random variable denoting the waiting time of a typical task in steady state for the $N$-th system.
By Little's law~\cite[$\S$ 6.4~(a)]{LST19}, note that $\E\big(W^{(N)}\big) = \big(N\lambda_N\big)^{-1}\sum_{i = 2}^\infty \E\big(Q_i^{(N)}(\infty)\big) = \big(N\lambda_N\big)^{-1}\E\big(S^{(N)}(\infty)+I^{(N)}(\infty)-N\big).$
Now, since in steady state, expected total arrival rate equals expected total service rate, we have $\E(I^{(N)}(\infty)) = N(1-\lambda_N) = \beta N^{\frac{1}{2}- \varepsilon}$. 
Therefore, we have the following immediate corollary of Theorem~\ref{thm:INTERCHANGE}.
\begin{corollary}\label{cor:main-1}
 $$\lim_{N\to\infty}N^{\frac{1}{2} - \varepsilon}\E\big(W^{(N)}\big) = \frac{2}{\beta}.$$
\end{corollary}


\begin{remark}[Contrast with centralized systems]
An interesting aspect of the many-server limit of the JSQ policy is revealed as we compare it with the corresponding M/M/$N$ system with the same load per server.
As briefly mentioned in the introduction, the key difference between the JSQ dynamics and the M/M/$N$ dynamics lies in the non-idling nature of the latter.
In the JSQ dynamics, there can be idle servers in the system, while having waiting tasks. 
However, as the system load becomes closer to 1, 
one may expect that most of the servers must remain busy to keep the system stable. 
Consequently, JSQ should behave similarly to the centralized queueing system. 
Theorem~\ref{thm:INTERCHANGE} shows that this is not the case even in the super-Halfin-Whitt regime. The centered and scaled total number of tasks for JSQ converges to $\newGamma\br{2,\beta}$ with mean $2/\beta$, instead of $\mathrm{Exponential}(\beta)$ with mean $1/\beta$ for the corresponding M/M/$N$ system.
Interestingly, indeed, if the load is much heavier, that is, $\lambda_N = 1 - O(N^{-\alpha})$ with $\alpha>2$, then the result in~\cite{HM21} implies that the appropriately scaled total number of tasks under the JSQ policy has the same limiting distribution as the centralized system.
It will be an interesting future direction to identify the precise scaling of the system load where this transition of behavior occurs.
\end{remark}


\begin{proof}{Proof of Theorem~\ref{thm:INTERCHANGE}.}
For brevity of notation, we will use the following notation in this proof:
\blue{$$\bX^{(N)}:=\Big(X^{(N)}(0), \frac{I^{(N)}(0)}{N^{\frac{1}{2} - \varepsilon}}, \frac{Q^{(N)}_2(0)}{N^{\frac{1}{2} + \varepsilon}}, \sum_{i=3}^{\infty}Q^{(N)}_i(0)\Big).$$}
Now assume that the system is initiated at the steady state, that is, $$\Big(X^{(N)}(0), I^{(N)}(0), Q_2^{(N)}(0), \sum_{i=3}^{\infty}Q^{(N)}_i(0)\Big)\disteq \Big(X^{(N)}(\infty), I^{(N)}(\infty), Q_2^{(N)}(\infty), \sum_{i=3}^{\infty}Q^{(N)}_i(\infty)\Big).$$ 
\blue{Note that by Theorem~\ref{thm:TIGHTNESS-XN} 
the sequence of random variables $\{\bX^{(N)}\}_{N\geq 1}$
is tight and moreover, $\sum_{i=3}^{\infty}Q^{(N)}_i(\infty) \xrightarrow{P} 0$ as $N \rightarrow \infty$. Hence, we can extract a subsequence $(\hat{N})$ of $(N)$ for which $X^{(\hat{N})}(0) \xrightarrow{d} X^*$ for some finite random variable. 
Furthermore, Theorem~\ref{thm:TIGHTNESS-XN}~(ii) implies that $\PP(X^*>0) = 1$ and thus, the conditions of Theorem~\ref{thm:PROCESS-LEVEL-NEW} hold. Thus, by Theorem~\ref{thm:PROCESS-LEVEL-NEW} applied to this subsequence, the process $X^{(\hat{N})}(\cdot)$ converges weakly to the path-wise unique solution of \eqref{langevin}, uniformly on $[0,T]$ with $X(0)\disteq X^*$. Note that since $X^{(\hat{N})}(t)$ is also distributed as $X^{(\hat{N})}(\infty)$ for all $t\geq 0$, by Theorem \ref{thm:PROCESS-LEVEL-NEW} and Proposition \ref{prop:STEADY-STATE}, $X^*$ must be distributed as the unique stationary distribution of the diffusion~\eqref{langevin}. 
This proves the weak convergence of $X^{(N)}(\infty)$ in $\R$.}

For the convergence of the $p$-th moment, note that $X^{(N)}(\infty)$'s are nonnegative random variables and hence,
$\E \Big[(X^{(N)}(\infty))^p\Big]=p\int_0^{\infty}x^{p-1}\PP (X^{(N)}(\infty)>x)dx.$
Take $B,N_0$ as in Theorem~\ref{thm:TIGHTNESS-XN}.
From the tail-probability bound in Theorem~\ref{thm:TIGHTNESS-XN}, we have that for any $p>0$ and $\tilde{\varepsilon}>0$,
\begin{equation}\label{eq:p-eps-moment}
    \begin{split}
        \sup_{N \ge N_0}\E \Big[(X^{(N)}(\infty))^{p+\tilde{\varepsilon}}\Big]&\leq\sup_{N \ge N_0} \Big((4B\vee1)^{p+\tilde{\varepsilon}}+(p+\tilde{\varepsilon})\int_{(4B\vee1)}^{\infty}x^{p+\tilde{\varepsilon}-1}\PP (X^{(N)}(\infty)>x)dx\Big)\\
        &\leq \Big((4B\vee1)^{p+\tilde{\varepsilon}}+(p+\tilde{\varepsilon})\int_{(4B\vee1)}^{\infty}x^{p+\tilde{\varepsilon}-1}C_1\exp\big\{-C_2x^{1/44}\big\}dx\Big)
        <\infty.
    \end{split}
\end{equation}
Since $X^{(N)}(\infty)\xrightarrow{d} \newGamma\br{2,\beta}$ and \eqref{eq:p-eps-moment} holds, by \cite[Corollary of Theorem 25.12]{pb12}, for any $p>0$, 
$\E \big[X^{(N)}(\infty)\big]^p\rightarrow \Gamma(p+2)/\beta^p$
as $N\to\infty$.
\end{proof}

\section{Sample-path analysis of the pre-limit process.}\label{sec:HITTIME}
In this section, we will obtain quantitative estimates on the sample path behavior of the pre-limit process that will be useful in Sections~\ref{sec:STEAYSTATE} and~\ref{sec:PROCESS-LEVEL}.
A significant part of our work, especially studying the steady-state behavior, involves a detailed pathwise analysis of the pre-limit process. 
Recall the discussion in Remark~\ref{rendes} in this regard.
The key quantity to analyze is $\Theta^{(N)}$ and, in the interval $[0, \Theta^{(N)}]$, how much time the process spends in different parts of the state space.
Such estimates are obtained in this section.
The challenge in this analysis comes from the fact that this process behaves very differently in different parts of the state space (no \emph{state space collapse}). 
Loosely speaking, the behavior of the process can be categorized into two qualitatively different phases:

\textbf{(a) Heavily-loaded phase.}
\blue{Loosely speaking, when the system is heavily loaded, i.e., $Q^{(N)}_2$ is large, the evolution of $Q^{(N)}_2$ is qualitatively similar to that of $X^{(N)}$.  
Thus,~\eqref{eq:represent-XN} hints that for $Q^{(N)}_2$ to fall below a certain threshold (we call this a `down-crossing'), the integral of the idleness process has to grow at a small enough rate with time so as to create a `negative drift' for $Q^{(N)}_2$ (see \eqref{q2ub}).
Quantifying this takes significant work and is the content of Section \ref{idlesec}. 
Also, note that the true evolution of $Q^{(N)}_2$ is more complicated than that of $X^{(N)}$ given in ~\eqref{eq:represent-XN}. This is because the evolution of $Q^{(N)}_2$ is governed by two very different types of effects. When $I^{(N)}>0$, $Q^{(N)}_2$ can only decrease and this happens when a customer is served at a busy server (with two or more customers). In contrast, the increase in $Q^{(N)}_2$ happens only when $I^{(N)}=0$ and additional arrivals take place. This is a more `local time like' effect. This leads to added technical challenges.}
In Section \ref{qthree}, this drift is exploited to bound the supremum and estimate the decay rate of $S^{(N)}$ and $\sum_{i=3}^{\infty}Q^{(N)}_i$ in the heavily loaded phase.  In Section \ref{downsec}, a probability bound is obtained for the time taken for the \blue{down-crossing}. 

\textbf{(b) Lightly-loaded phase.}
On the other hand, in the lightly loaded phase, \blue{when $Q^{(N)}_2$ is small, the fluctuations in the arrival and departure processes for the system eventually push $Q^{(N)}_2$ above a certain level (see \eqref{eq:lem4.11-2}). This is what we call an `up-crossing' and quantify it in Section \ref{upsec}. }

\blue{Later, in Section \ref{sec:STEAYSTATE}, the above two phases are combined to create a renewal structure which sheds light on the tail behavior of the stationary distribution of $X^{(N)}$ for fixed large $N$. 
An overview of how various results of Section~\ref{sec:HITTIME} are used in Section~\ref{sec:STEAYSTATE} is schematically presented in Figure~\ref{fig:results-dependence}.}
\begin{figure}
    \centering
    \includegraphics[width=\textwidth]{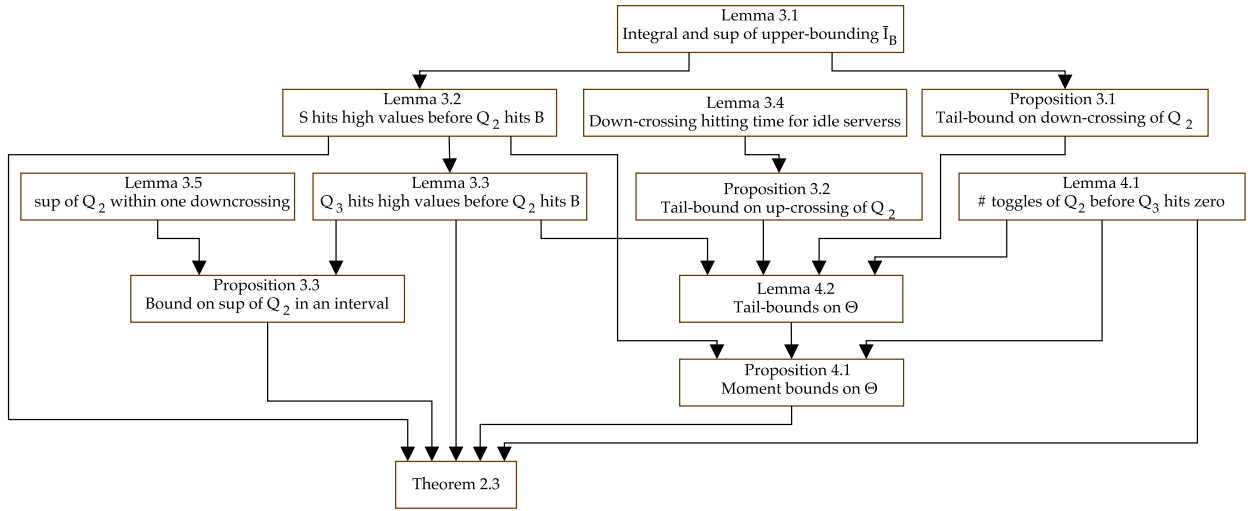}
    \caption{\blue{Interdependence of various results of Sections~\ref{sec:HITTIME} and~\ref{sec:STEAYSTATE}.}}
    \label{fig:results-dependence}
\end{figure}

In Section \ref{ssec:supremum Q2}, we estimate the probability of $Q^{(N)}_2$ crossing large values on large compact time intervals. This is crucially used in the process-level convergence proofs in Sections \ref{sec:PROCESS-LEVEL} and \ref{sec:proof-prop-5.2} (in particular, Proposition \ref{prop:INT-IDLE-2}).

The general technique to achieve the above involves decomposing the process sample path into excursions between suitably chosen stopping times. Between successive stopping times, the behavior is more uniform, and comparable to some process that is easier to analyze.

Now, we set up some notation. Throughout this article, we fix an $\varepsilon \in [0,1/2)$. Define $\Bar{Q}^{(N)}_3(t):=\sum_{i=3}^{\infty}Q^{(N)}_i(t)$
and, for $x, y, z \geq 0$, define the stopping times
\begin{equation}\label{eq:tau-def}
    \begin{split}
        \tau^{(N)}_1(x)&:=\inf\big\{t\geq 0:I^{(N)}(t)= \lfloor xN^{\frac{1}{2}-\varepsilon} \rfloor \big\},\\
    \tau^{(N)}_2(y)&:=\inf\big\{t\geq 0:Q^{(N)}_2(t)= \lfloor yN^{\frac{1}{2}+\varepsilon} \rfloor\big\},\\
    \tau^{(N)}_s(z)&:=\inf \big\{t\geq 0:S^{(N)}(t)= \lfloor zN^{\frac{1}{2}+\varepsilon}\rfloor\big\}.
    \end{split}
\end{equation}

\subsection{Upper bound for idleness process in heavily loaded phase.}\label{idlesec}
Intuitively speaking, when $Q_2^{(N)}$ is large, the process denoting the number of idle servers, $I^{(N)}(t)$, can be upper bounded by a birth-death process.
In particular, if $Q_2^{(N)}\geq BN^{\frac{1}{2}+\varepsilon}$, then the rate of decrease of $I^{(N)}$ when it is positive, is $N-\beta N^{\frac{1}{2}-\varepsilon}$ (whenever there is an arrival), and the rate of increase is $Q_1^{(N)}-Q_2^{(N)}$, which is at most $N-BN^{\frac{1}{2}+\varepsilon}$.
For any fixed $B>0$, let us define such a birth-death process $\Bar{I}^{(N)}_B$,  with  
 increase rate $N-BN^{\frac{1}{2}+\varepsilon}$ and decrease rate $N-\beta N^{\frac{1}{2}-\varepsilon}$.
We will assume that $N$ is large enough so that $N>BN^{\frac{1}{2}+\varepsilon}>\beta N^{\frac{1}{2}-\varepsilon}$.
From the above discussion, if $Q^{(N)}_2(0)>BN^{\frac{1}{2}+\varepsilon}$, then there exists a natural coupling such that
$
    I^{(N)}(t\wedge \tau^{(N)}_2(B))\leq \Bar{I}^{(N)}_B(t\wedge \tau^{(N)}_2(B)),
$
for all $t\geq 0$. 
Thus, in the analysis of $I^{(N)}$, a key ingredient is to study this upper-bounding birth death process. 
In particular, as can be understood from the evolution equation of $X^{(N)}(t)$ in~\eqref{eq:represent-XN}, two quantities involving $I^{(N)}$ that will be pivotal in successive analysis are its integral and supremum over appropriately scaled time intervals.
For that, the following lemma will be \blue{crucial in quantifying the `negative drift' created due to the integral of $I^{(N)}$ being `small' in the heavily loaded phase (when $Q^{(N)}_2 > B N^{\frac{1}{2}+\varepsilon}$).}
\begin{lemma}\label{lem:LEMMA-6}
Assume $\Bar{I}^{(N)}_B(0)=0$.  There exist $t_0>0,B_1 \ge 1$ such that for all $B\geq B_1$, \blue{there exists $\tilde N_B$ such that for all} $N\geq \tilde N_B$ and $t\geq t_0$, the following hold: \begin{enumerate}[\normalfont (i)]
    \item \label{7i}$\mathbb{P}\big(\int_0^{N^{2\varepsilon}s}\Bar{I}_B^{(N)}(u)du \ge \frac{\beta}{2} N^{\frac{1}{2}+\varepsilon}s\text{ for some }s\geq t\big)\leq c_1\exp \{-c_2 B^{\frac{1}{5}} N^{\frac{4\varepsilon}{5}}t^{\frac{1}{5}}\}$;
    \item \label{7ii} $\mathbb{P}\big(\sup_{s\leq t}\Bar{I}_B^{(N)}(N^{2\varepsilon}s)\geq \frac{\beta}{4}N^{\frac{1}{2}-\varepsilon}t\big)\leq \exp \{-\tilde{c}_1 B N^{4\varepsilon}t\}+\tilde{c}_2N^{4\varepsilon}t \exp \{-\tilde{c}_3 \sqrt{B} t\}$.
\end{enumerate}
The above constants $c_i,\tilde{c}_i$ may depend on $\beta$, but not on $B,N,t$.
\end{lemma}
The proof of Lemma~\ref{lem:LEMMA-6} is given in Appendix~\ref{sec:app-int-IB}.
Here we provide a brief outline of the proof idea.
We will estimate $\int_0^{N^{2\varepsilon}t}\Bar{I}_B^{(N)}(s)ds$ for large $B,N$ and $t$ by identifying certain excursions in the path of the process $\Bar{I}^{(N)}_B$. The integral will then be estimated by controlling the duration of each such excursion and the supremum of the process $\Bar{I}^{(N)}_B$ on this excursion.
Let $\Bar{I}^{(N)}_B(0)=0$. Given $\eta>0$, define the following stopping times: $\Bar{\sigma}^{(N)}_0\coloneqq0$, and for $i\geq 0$,
\begin{align*}
    \Bar{\sigma}^{(N)}_{2i+1} &\coloneqq\inf \big\{t\geq \Bar{\sigma}^{(N)}_{2i}: \Bar{I}^{(N)}_B(t)\geq \lfloor \eta N^{\frac{1}{2}-\varepsilon}\rfloor\big\},\text{ and }\\
    \Bar{\sigma}^{(N)}_{2i+2} &\coloneqq\inf \big\{t\geq \Bar{\sigma}^{(N)}_{2i+1}: \Bar{I}^{(N)}_B(t)\leq \big\lfloor \frac{\eta}{2} N^{\frac{1}{2}-\varepsilon} \big\rfloor\big\}.
\end{align*}
Also define:
    $$\Bar{\xi}^{(N)}_i :=\Bar{\sigma}^{(N)}_{2i}-\Bar{\sigma}^{(N)}_{2i-1}, \ \ \
    \Bar{u}^{(N)}_i :=\sup_{s\in[\Bar{\sigma}^{(N)}_{2i-1},\Bar{\sigma}^{(N)}_{2i}]}(\Bar{I}_B^{(N)}(s)-\lfloor \eta N^{\frac{1}{2}-\varepsilon}\rfloor), \ i\geq 1,
    $$
    and
    $\bar{K}^{(N)}_t :=\inf\big\{k:\Bar{\sigma}^{(N)}_{2k}\geq N^{2\varepsilon}t\big\}$.
 Note that 
 \begin{equation}\label{eq:app:int-IB}
      \int_0^{N^{2\varepsilon}t}\Bar{I}_B^{(N)}(s)ds \le \eta N^{\frac{1}{2}+\varepsilon} t+ \sum_{i=1}^{\bar{K}^{(N)}_{t}}\Bar{u}^{(N)}_i\Bar{\xi}^{(N)}_i.
 \end{equation}
Therefore, the proof of Lemma~\ref{lem:LEMMA-6} will rely on analyzing the behavior of $\Bar{u}^{(N)}_i\Bar{\xi}^{(N)}_i$ (Lemma~\ref{lem:LEMMA-3}), that of $\bar{K}^{(N)}_{t}$ (Lemma~\ref{lem:LEMMA-4}), and of the random sum $\sum_{i=1}^{\bar{K}^{(N)}_{t}}\Bar{u}^{(N)}_i\Bar{\xi}^{(N)}_i$ (Lemma~\ref{lemma 6}).
To analyze the random sum, we will use the concentration inequality for independent random variables with stretched exponential tails stated in Lemma~\ref{lem:LEMMA-5}.
\blue{The form of exponents of $N, B,$ and $t$ appearing on the right hand side of Lemma~\ref{lem:LEMMA-6} (i) and (ii) is an artifact of using Lemma~\ref{lem:LEMMA-5} along with the tail estimates for $\Bar{u}^{(N)}_i\Bar{\xi}^{(N)}_i$ from Lemma~\ref{lem:LEMMA-3}. 
We do not expect these exponents to be optimal, but they will serve our purpose of showing that the integral has sufficiently light tail.}

\subsection{Controlling long queues in heavily loaded phase.}\label{qthree}
Recall that for any $t \ge 0$, $S^{(N)}(t)$ denotes the total number of tasks in the system at time $t$, and $\Bar{Q}^{(N)}_3(t):=\sum_{i=3}^{\infty}Q^{(N)}_i(t)$ denotes the number of tasks waiting behind at least two other tasks at time $t$. Also recall that $\tau_2^{(N)}(B)$ denotes the first time $Q_2^{(N)}$ hits $\lfloor B N^{\frac{1}{2} + \varepsilon}\rfloor$. 
As described in Remark \ref{rendes}, the renewal time $\Theta^{(N)}$ is crucial to the steady-state analysis.
For this, we need to have sharp control on how large the quantities $S^{(N)}$ and $\Bar{Q}^{(N)}_3$ can become during the phase when $Q_2^{(N)}$ hits $\lfloor B N^{\frac{1}{2} + \varepsilon}\rfloor$ starting from $\lfloor 2B N^{\frac{1}{2} + \varepsilon}\rfloor$.
This is achieved in Lemma \ref{lem:S-hit-time} and Lemma~\ref{lem:LEMMA-8} in this section.
They will be used in the subsequent sections.

First, let us introduce the following notation.
For integers $x \in [0,N]$, $y \in [0, N-x]$ and a vector of non-negative integers $\underline{z} = (z_1,z_2,\dots) \in \mathbb{N}_0^{\infty}$ with $y \ge z_1 \ge z_2 \ge \dots$, denote 
\begin{align*}
    \mathbb{P}_{(x,y,\underline{z})}(\cdot) &:=\mathbb{P}(\ \cdot\ |\ I^{(N)}(0)=x,Q^{(N)}_2(0)=y, Q^{(N)}_i(0)= z_{i-2} \text{ for } i \ge 3)\quad \text{and}\\
    \sup_{\underline{z}} \mathbb{P}_{(x,y,\underline{z})}(\cdot) &:= \sup\Big\{ \mathbb{P}_{(x,y,\underline{z})}(\cdot):  \underline{z} \in \mathbb{N}_0^{\infty}, \, z_1 \ge z_2 \ge \dots, \, z_1 \le y\Big\}.
\end{align*}
\blue{The following lemma quantifies the probability of the total queue length becoming large in the heavily loaded phase. The negative drift created on account of Lemma \ref{lem:LEMMA-6} ensures that this probability is small.}
\begin{lemma}\label{lem:S-hit-time}
There exist $t_0,B_1,$ such that for all $B\geq B_1, N\geq \tilde N_B$, and $x\geq t_0$,
    \begin{align*}
        & \sup_{\underline{z}} \mathbb{P}_{(0, \, \lfloor 2BN^{\frac{1}{2}+\varepsilon} \rfloor, \,\underline{z})}\big(S^{(N)}\text{ hits }S^{(N)}(0)+x\beta N^{\frac{1}{2}+\varepsilon}\text{ before time }\tau_2^{(N)}(B)\big)\\
        \leq & \hat{c}_1\exp \{-\hat{c}_2 B^{\frac{1}{5}}N^{\frac{4\varepsilon}{5}}x^{\frac{1}{5}}\}+\hat{c}_3\exp \{-\hat{c}_4x\},
    \end{align*}
    where constants $\hat{c}_i$, $i\in \{1,2,3,4\}$ do not depend on $B$ and $N$.
\end{lemma}
\blue{The first term in the above bound arises from the control on the integral of the idleness process achieved in Lemma \ref{lem:LEMMA-6} (i), while the second term arises from hitting probabilities of processes that are martingales plus a negative drift (which approximates $S^{(N)}$ in this phase).}
\begin{proof}{Proof of Lemma~\ref{lem:S-hit-time}.}
 Note that there is a natural coupling between $I^{(N)}(t)$ and $\bar{I}^{(N)}_B(t) $, such that for all $t\leq \tau_2^{(N)}(B)$, $I^{(N)}(t)\leq \bar{I}^{(N)}_B(t)$. 
 Thus, for all $t\leq \tau_2^{(N)}(B)$,
    \begin{align*}
        S^{(N)}(t) &\leq  S^{(N)}(0)+A((N-\beta N^{\frac{1}{2}-\varepsilon})t)-D\Big(\int_0^t(N-\Bar{I}^{(N)}_B(s))ds\Big)\\
        & = S^{(N)}(0)+\hat{A}((N-\beta N^{\frac{1}{2}-\varepsilon})t)-\hat{D}\Big(\int_0^t(N-\Bar{I}^{(N)}_B(s))ds\Big)+\int_0^t \Bar{I}_B^{(N)}(s)ds-\beta N^{\frac{1}{2}-\varepsilon}t,
    \end{align*}
    where recall $A(\cdot)$ and $D(\cdot)$ are independent unit-rate Poisson processes representing arrivals and departures respectively, and $\hat{A}(s)=A(s)-s$ and $\hat{D}(s)=D(s)-s$.
    Therefore,
    \begin{align}\label{eq:s-hit-x-tau2}
        & \sup_{\underline{z}} \mathbb{P}_{(0, \, \lfloor 2BN^{\frac{1}{2}+\varepsilon} \rfloor, \,\underline{z})}\big(S^{(N)} \text{ hits }S^{(N)}(0)+x\beta N^{\frac{1}{2}+\varepsilon}\text{ before }\tau^{(N)}_2(B))\big)\nonumber\\
        &\leq  \mathbb{P}\Big(\int_0^{N^{2\varepsilon}s}\Bar{I}_B^{(N)}(u)du\geq \frac{\beta x }{4}N^{\frac{1}{2}+\varepsilon}+\frac{\beta s}{2}N^{\frac{1}{2}+\varepsilon}\text{ for some }s\geq 0\Big)\\
        &\hspace{2cm}+\mathbb{P}\Big(\mathcal{M}^{*}(t)-\frac{\beta t}{2}N^{\frac{1}{2}-\varepsilon}\text{ hits }\frac{3\beta x}{4}N^{\frac{1}{2}+\varepsilon}\text{ for some }t\geq 0\Big),\nonumber
    \end{align}
     where $\mathcal{M}^{*}(t)=\hat{A}((N-\beta N^{\frac{1}{2}-\varepsilon})t)-\hat{D}(\int_0^t(N-\Bar{I}_B(s))ds)$.
    Recall $t_0,B_1$ from Lemma \ref{lem:LEMMA-6}. For $x\geq t_0$, $B \ge B_1$, $N \ge \tilde N_B$,
    \begin{equation}\label{eq:lem4.7-2}
        \begin{split}
        &\mathbb{P}\Big(\int_0^{N^{2\varepsilon}s}\Bar{I}_B^{(N)}(u)du\geq \frac{\beta x}{4}N^{\frac{1}{2}+\varepsilon}+\frac{\beta s}{2}N^{\frac{1}{2}+\varepsilon}\text{ for some }s\geq 0\Big)\\
        &\leq  \mathbb{P}\Big(\int_0^{N^{2\varepsilon}x}\Bar{I}_B^{(N)}(u)du \geq \frac{\beta x}{4}N^{\frac{1}{2}+\varepsilon}\Big)+\mathbb{P}\Big(\int_0^{N^{2\varepsilon}s}\Bar{I}_B^{(N)}(u)du \geq \frac{\beta s}{2}N^{\frac{1}{2}+\varepsilon}\text{ for some }s\geq x\Big)\\
        & \leq  \hat{c}\exp \{-\hat{c}' B^{\frac{1}{5}}N^{\frac{4\varepsilon}{5}}x^{\frac{1}{5}}\}, 
        \end{split}
    \end{equation}
    for some constants $\hat{c},\hat{c}'>0$,
    where the last inequality follows by \eqref{eq:lem4.7-int-IB} and Lemma~\ref{lem:LEMMA-6} (i).
   For the second term on the right-hand-side of~\eqref{eq:s-hit-x-tau2}, denote $\Bar{s}_k=kN^{2\varepsilon}x$, $k\geq 0$.
    For any $k\geq 0$, 
    \begin{align*}
        \mathbb{P}\Big(\sup_{s\in[\Bar{s}_k,\Bar{s}_{k+1}]}\Big(\mathcal{M}^{*}(s)-\frac{\beta s}{2}N^{\frac{1}{2}-\varepsilon}\Big)\geq \frac{3\beta x}{4}N^{\frac{1}{2}+\varepsilon}\Big) 
        & \leq  \mathbb{P}\Big(\sup_{s\in[0,\Bar{s}_{k+1}]}\mathcal{M}^{*}(s)\geq \frac{2 N^{-2\varepsilon} \Bar{s}_{k}+3x}{4}\beta N^{\frac{1}{2}+\varepsilon}\Big)\\
        & \leq  \mathbb{P}\Big(\sup_{s\in[0,\Bar{s}_{k+1}]}\mathcal{M}^{*}(s)\geq \frac{2k+5}{4} x\beta N^{\frac{1}{2}+\varepsilon}\Big)
        \leq  c'' e^{-\bar{c}''(k+1)x},
    \end{align*}
    for some constants $c'',\bar{c}''>0$,
    where the last inequality follows from Lemma~\ref{lem:sup-poi}.
     Summing over $k$, we obtain \begin{equation}\label{eq:lem4.7-3}
         \mathbb{P}\Big(\mathcal{M}^{*}(t)-\frac{\beta t}{2}N^{\frac{1}{2}-\varepsilon}\text{ hits }\frac{x}{2}N^{\frac{1}{2}+\varepsilon}\text{ for some }t\geq 0\Big)\leq c''e^{-\bar{c}''x}.
     \end{equation}
Plugging \eqref{eq:lem4.7-2} and \eqref{eq:lem4.7-3} into \eqref{eq:s-hit-x-tau2}  and choosing appropriate constants complete the proof. 
\end{proof}
\bigskip

\blue{Now recall that $\Bar{Q}^{(N)}_3(t)=\sum_{i=3}^{\infty}Q^{(N)}_i(t), \, t \ge 0$ and the renewal time $\Theta^{(N)}$.
Note that for such a renewal event to occur, all the following three things need to happen: 
Starting from $(0, \lfloor 2B N^{\frac{1}{2} + \varepsilon}\rfloor, \underline{0})$,

(a)~$Q_2^{(N)}$ falls to $\lfloor B N^{\frac{1}{2} + \varepsilon}\rfloor$, then 

(b)~$\Bar{Q}^{(N)}_3$ drops to zero, and subsequently, 

(c)~$Q_2^{(N)}$ climbs back to $\lfloor 2B N^{\frac{1}{2} + \varepsilon}\rfloor$. 

\noindent
Part (b) requires special care as, although the probability of $\Bar{Q}^{(N)}_3$ becoming positive before time $\tau_2^{(N)}(B)$ is small for large $N$, on the event that this happens, it takes a potentially long time for $\Bar{Q}^{(N)}_3$ to fall back to zero. 
This is quantified in Lemma \ref{lem:LEMMA-8}, which gives tail estimates on the supremum of $\Bar{Q}^{(N)}_3$ on $[0,\tau_2^{(N)}(B)]$ and shows that with probability at least $1/2$, $\Bar{Q}^{(N)}_3$ falls by $\Bar{Q}^{(N)}_3(0)\wedge \frac{B}{4\beta}N^{2\varepsilon}$ on this time interval.
Define 
\begin{equation}\label{eq:zbar}
        \Bar{Z}^{(N)}_{*} = \sup_{s\in [0,\tau^{(N)}_2(B)]}\big(\Bar{Q}^{(N)}_3(s)-\Bar{Q}^{(N)}_3(0)\big)_+,
    \quad \text{and}\quad
    \Bar{Z}^{(N)} =\Bar{Q}^{(N)}_3(\tau^{(N)}_2(B))-\Bar{Q}^{(N)}_3(0).
\end{equation}}
\begin{lemma}\label{lem:LEMMA-8}
There exists $B_2>0$, such that for all $B\geq B_2$, we can obtain $\tilde N^*_B \in \mathbb{N}$ for which the following hold for all $N \geq \tilde N^*_B$ and $x>0$.
\begin{enumerate}[\normalfont (i)]
    \item There exist constants $\bar{c}_i$, $i\in\{1,2,3,4\}$, not depending on $N,x$, such that 
    \begin{equation}\label{eq:lem3.3-1}
        \begin{split}
            \sup_{\underline{z}}\mathbb{P}_{(0, \, \lfloor 2BN^{\frac{1}{2}+\varepsilon} \rfloor, \,\underline{z})}
        &\big(\Bar{Z}^{(N)}_{*}\geq x\beta N^{\frac{1}{2}+\varepsilon}\big)\\
        &\leq \bar{c}_1\exp \big\{-\bar{c}_2N^{(\frac{1}{2}-\varepsilon)/5}\big\}\big(\exp \{-\bar{c}_3B^{\frac{1}{5}}N^{\frac{4\varepsilon}{5}}x^{\frac{1}{5}}\}+\exp \{-\bar{c}_4x \}\big).
        \end{split}
    \end{equation}
    \item $\inf_{\underline{z}} \mathbb{P}_{(0, \, \lfloor 2BN^{\frac{1}{2}+\varepsilon} \rfloor, \,\underline{z})}\Big(\Bar{Z}^{(N)}\leq -\big[\Bar{Q}^{(N)}_3(0)\wedge \frac{B}{4\beta}N^{2\varepsilon}\big]\Big)\geq \frac{1}{2}$.
\end{enumerate}
\end{lemma}
\blue{As we will see in the proof of Lemma~\ref{lem:LEMMA-8} below, the exponents in the right hand side of~\eqref{eq:lem3.3-1} are inherited from those in Lemma~\ref{lem:S-hit-time} as we will relate the hitting probabilities of $\Bar{Z}^{(N)}_{*}$ with that of $S^{(N)}$.}
\begin{proof}[Proof of Lemma~\ref{lem:LEMMA-8}.]
(i) Recall $t_0$ from Lemma~\ref{lem:S-hit-time}. Let $\bar{N}_B\in\mathbb{N}$ be such that $N-2BN^{\frac{1}{2}+\varepsilon}\geq  \frac{N}{2}$ and $N^{\frac{1}{2}-\varepsilon}/(2\beta) >t_0$ for all $N\geq \bar{N}_B$.
If $\Bar{Z}^{(N)}_{*}\geq x\beta N^{\frac{1}{2}+\varepsilon}$, then there exists $s\in[0,\tau^{(N)}_2(B)]$ such that $Q^{(N)}_2(s)=N, \Bar{Q}^{(N)}_3(s)=\Bar{Q}^{(N)}_3(0)+x\beta N^{\frac{1}{2}+\varepsilon},$ and $I^{(N)}(0)=0$. 
Hence, for all $N \ge \bar{N}_B$,
\begin{align*}
        S^{(N)}(s)=&N+N+\Bar{Q}^{(N)}_3(0)+x\beta N^{\frac{1}{2}+\varepsilon}\\
        &\ge (N-I^{(N)}(0)+Q_2^{(N)}(0)+\Bar{Q}^{(N)}_3(0))+(N-2B N^{\frac{1}{2}+\varepsilon}+x\beta N^{\frac{1}{2}+\varepsilon})\\
        &= S^{(N)}(0)+(N-2BN^{\frac{1}{2}+\varepsilon}+x\beta N^{\frac{1}{2}+\varepsilon})\\
        &\geq  S^{(N)}(0)+\Big(\frac{N}{2}+x\beta N^{\frac{1}{2}+\varepsilon}\Big).
    \end{align*}
Take $B_1$, $\tilde N_B$ as in Lemma~\ref{lem:S-hit-time} and define $\Tilde{N}'_B :=\tilde N_B\vee \bar{N}_B$.
    Thus, for $B\geq B_1$, $N\geq\Tilde{N}'_B$, $x>0$,
    \begin{align*}
        &\sup_{\underline{z}}\mathbb{P}_{(0, \, \lfloor 2BN^{\frac{1}{2}+\varepsilon} \rfloor, \,\underline{z})}\big(\Bar{Z}^{(N)}_{*}\geq x\beta N^{\frac{1}{2}+\varepsilon}\big)\\
        &\leq  \sup_{\underline{z}}\mathbb{P}_{(0, \, \lfloor 2BN^{\frac{1}{2}+\varepsilon} \rfloor, \,\underline{z})}\big(S^{(N)}\text{ hits }S^{(N)}(0)+\frac{N}{2}+x\beta N^{\frac{1}{2}+\varepsilon}\text{ before time }\tau^{(N)}_2(B)\big)\\
        &\leq  \hat{c}_1\exp \big\{-\hat{c}_2 B^{\frac{1}{5}}N^{\frac{4\varepsilon}{5}}(N^{\frac{1}{2}-\varepsilon}/(2\beta)+x)^{\frac{1}{5}}\big\}+\hat{c}_3\exp \{-\hat{c}_4(N^{\frac{1}{2}-\varepsilon}/(2\beta)+x)\}\\
        & \leq  \bar{c}_1\exp \big\{-\bar{c}_2N^{(\frac{1}{2}-\varepsilon)/5}\big\}\big(\exp \{-\bar{c}_3B^{\frac{1}{5}}N^{\frac{4\varepsilon}{5}}x^{\frac{1}{5}}\}+\exp \{-\bar{c}_4x \}\big),
    \end{align*}
    where the second inequality is due to Lemma~\ref{lem:S-hit-time} upon recalling $N^{\frac{1}{2}-\varepsilon}/(2\beta)+x>t_0$ by the definition of $\bar{N}_B$. The last inequality is due to $(a+b)^{\frac{1}{5}}\leq a^{\frac{1}{5}}+b^{\frac{1}{5}}$, for all $a,b\geq 0$, and $B_1 \ge 1$.\\

\noindent    
(ii) From part (i), for $B\geq B_1$ and $N\geq\Tilde{N}'_B$, 
\begin{equation}\label{eq:lem4.8-z-star}
        \sup_{\underline{z}}\mathbb{P}_{(0, \, \lfloor 2BN^{\frac{1}{2}+\varepsilon} \rfloor, \,\underline{z})}\big(\Bar{Z}^{(N)}_{*}>0\big)\leq 2\bar{c}_1 \exp\big(-\bar{c}_2N^{(\frac{1}{2}-\varepsilon)/5}\big).
    \end{equation}
    Starting from $I^{(N)}(0)=0,Q_2^{(N)}(0)= \lfloor 2BN^{\frac{1}{2}+\varepsilon}\rfloor$ and any $\Bar{Q}^{(N)}_3(0)$,
    on the event $\Bar{Z}^{(N)}_{*}=0$, $\Bar{Q}^{(N)}_3(\tau^{(N)}_2(B))\leq \Bar{Q}^{(N)}_3(0)$ and $Q_2^{(N)}(\tau^{(N)}_2(B))\le Q_2^{(N)}(0)-BN^{\frac{1}{2}+\varepsilon} + 1$. 
    Thus, we have $S^{(N)}(\tau^{(N)}_2(B))\leq S^{(N)}(0)-BN^{\frac{1}{2}+\varepsilon} + 1.$
    Also, for all $t\geq 0$,
    \begin{align*}
        S^{(N)}(t)\geq &S^{(N)}(0)+A((N-\beta N^{\frac{1}{2}-\varepsilon})t)-D(Nt)\\
        &= S^{(N)}(0)+\hat{A}((N-\beta N^{\frac{1}{2}-\varepsilon})t)-\hat{D}(Nt)-\beta N^{\frac{1}{2}-\varepsilon}t,
    \end{align*}
    where $\hat{A}(s)=A(s)-s$ and $\hat{D}(s)=D(s)-s$. Let $\hat{M}(s)=\hat{A}((N-\beta N^{\frac{1}{2}-\varepsilon})s)-\hat{D}(Ns)$.
    Thus, by Lemma~\ref{lem:sup-poi}, there exists a constant $\tilde{c}>0$, such that
    \begin{equation}\label{eq:4.8-1}
    \begin{split}
        &\sup_{\underline{z}}\mathbb{P}_{(0, \, \lfloor 2BN^{\frac{1}{2}+\varepsilon} \rfloor, \,\underline{z})}\big(\tau^{(N)}_2(B)<\frac{B}{2\beta}N^{2\varepsilon},\Bar{Z}^{(N)}_{*}=0\big)\\
        & \leq \mathbb{P}\big(\inf_{s\leq (B/2\beta)N^{2\varepsilon}}\big(\hat{M}(s)-\beta N^{\frac{1}{2}-\varepsilon}s\big)\leq -BN^{\frac{1}{2}+\varepsilon} + 1\big)\\
        & \leq  \mathbb{P}\big(\inf_{s\leq (B/2\beta)N^{2\varepsilon}}\hat{M}(s)\leq -\frac{B}{2}N^{\frac{1}{2}+\varepsilon} + 1\big)
        \leq 4\exp \{-\Tilde{c}B\}.
    \end{split}
    \end{equation}
    Furthermore, since the instantaneous rate of decrease of $\Bar{Q}^{(N)}_3$ is at least 1 when $\Bar{Q}^{(N)}_3$ is positive, and that $\Bar{Q}^{(N)}_3$ can only decrease when $\Bar{Z}^{(N)}_{*}=0$, observe that, using Lemma \ref{lem:sup-poi} for Poisson processes, 
    \begin{align}\label{eq:4.8-2}
        &\sup_{\underline{z}}\mathbb{P}_{(0, \, \lfloor 2BN^{\frac{1}{2}+\varepsilon} \rfloor, \,\underline{z})}\Big(\Bar{Q}^{(N)}_3(\tau^{(N)}_2(B))>\big(\Bar{Q}^{(N)}_3(0)-\frac{B}{4\beta}N^{2\varepsilon}\big)_+, \Bar{Z}^{(N)}_{*}=0,\tau^{(N)}_2(B)\geq \frac{B}{2\beta}N^{2\varepsilon}\Big)\nonumber\\
        &\leq  \mathbb{P}\Big(\mathrm{Po}\big(\frac{B}{2\beta}N^{2\varepsilon}\big)<\frac{B}{4\beta}N^{2\varepsilon}\Big)\leq 2e^{-c'B N^{2\varepsilon}}, \text{ for some constant }c'>0\text{ dependent on }\beta,
    \end{align}
    where $\mathrm{Po}(\frac{B}{2\beta}N^{2\varepsilon})$ is a Poisson random variable with parameter $\frac{B}{2\beta}N^{2\varepsilon}$. 
    Inequalities~\eqref{eq:lem4.8-z-star}, \eqref{eq:4.8-1} and \eqref{eq:4.8-2}, yield
    \begin{align*}
        &\inf_{\underline{z}} \mathbb{P}_{(0, \, \lfloor 2BN^{\frac{1}{2}+\varepsilon} \rfloor, \,\underline{z})}\big(\Bar{Z}^{(N)}\leq -\big[\Bar{Q}^{(N)}_3(0) \wedge \frac{B}{4\beta}N^{2\varepsilon}\big]\big)\\
        &=\inf_{\underline{z}} \mathbb{P}_{(0, \, \lfloor 2BN^{\frac{1}{2}+\varepsilon} \rfloor, \,\underline{z})}\big(\Bar{Q}^{(N)}_3(\tau^{(N)}_2(B)) \le \big(\Bar{Q}^{(N)}_3(0)-\frac{B}{4\beta}N^{2\varepsilon}\big)_+\big)\\
        &\geq  1-\sup_{\underline{z}}\mathbb{P}_{(0, \, \lfloor 2BN^{\frac{1}{2}+\varepsilon} \rfloor, \,\underline{z})}\big(\Bar{Z}^{(N)}_{*}>0\big)-\sup_{\underline{z}}\mathbb{P}_{(0, \, \lfloor 2BN^{\frac{1}{2}+\varepsilon} \rfloor, \,\underline{z})}\big(\tau^{(N)}_2(B)<\frac{B}{2\beta}N^{2\varepsilon},\Bar{Z}^{(N)}_{*}=0\big)\\
        &\qquad-\sup_{\underline{z}}\mathbb{P}_{(0, \, \lfloor 2BN^{\frac{1}{2}+\varepsilon} \rfloor, \,\underline{z})}\big(\Bar{Q}^{(N)}_3(\tau^{(N)}_2(B))>\big(\Bar{Q}^{(N)}_3(0)-\frac{B}{4\beta}N^{2\varepsilon}\big)_+, \Bar{Z}^{(N)}_{*}=0,\tau^{(N)}_2(B)\geq \frac{B}{2\beta}N^{2\varepsilon}\big)\\
        &\ge 1 - 2\bar{c}_1 \exp\big(-\bar{c}_2N^{(\frac{1}{2}-\varepsilon)/5}\big) - 4\exp \{-\Tilde{c}B\} - 2\exp\{-c'B N^{2\varepsilon}\}.
    \end{align*}
Let $B_2 \ge B_1$ such that $4\exp \{-\Tilde{c}B\} \le 1/4$ for all $B \ge B_2$. For $B \ge B_2$, we can obtain $\tilde N^*_B \ge \Tilde{N}'_B$ such that  the lower bound above is at least $1/2$. This completes the proof of the lemma.
\end{proof}

\subsection{Down-crossing estimate.}\label{downsec}
The goal of this section is to obtain a probability bound on $\tau^{(N)}_2(B)$, that is, starting from $Q^{(N)}_2(0)=\lfloor 2BN^{\frac{1}{2}+\varepsilon} \rfloor$, the time taken for $Q^{(N)}_2$ to hit $\lfloor BN^{\frac{1}{2}+\varepsilon} \rfloor$.
\begin{prop}\label{prop:DOWNCROSS}
Recall $B_1 \ge 1$ and $\tilde N_B$ for $B \ge B_1$ from Lemma \ref{lem:LEMMA-6}. There exist constants $ t_0, c'_i>0$ , $i\in\{0,1,2,3,4\}$, such that for any $B\geq B_1, N\geq \tilde N_B$, and $x \ge 0$,
\begin{equation}\label{eq:prop3.1}
    \begin{split}
         &\sup_{\underline{z}: \sum_iz_i \leq xN^{\frac{1}{2}+\varepsilon}}\mathbb{P}_{(0, \, \lfloor 2BN^{\frac{1}{2}+\varepsilon} \rfloor, \,\underline{z})}\big(\tau^{(N)}_2(B)\geq N^{2\varepsilon}t\big)\\
    &\leq 4e^{-c'_0t}+c'_1\exp \{-c'_2 B^{\frac{1}{5}}N^{\frac{4\varepsilon}{5}}t^{\frac{1}{5}}\}+c'_3N^{4\varepsilon}t\exp \{-c'_4\sqrt{B}N^{2\varepsilon}t\},\quad\forall t\geq t_0\vee \frac{8}{\beta}(B+x).
    \end{split}
\end{equation}
\end{prop}
Recall the upper bounding birth-death process  $\bar{I}_B^{(N)}$ from Section~\ref{idlesec}.
\blue{As we will see in the proof below, the exponents in the right hand side of~\eqref{eq:prop3.1} are inherited from Lemma~\ref{lem:LEMMA-6} as we will require control over the integral and supremum of $\bar{I}_B^{(N)}$ in bounding the tail probability for $\tau^{(N)}_2(B)$.}
\begin{proof}[Proof of Proposition~\ref{prop:DOWNCROSS}.]
Starting from $I^{(N)}(0)=0,Q^{(N)}_2(0)=2BN^{\frac{1}{2}+\varepsilon},\sum_{i=3}^{
    \infty}Q^{(N)}_i(0)) \leq xN^{\frac{1}{2}+\varepsilon}$,
for $t\leq \tau_2^{(N)}(B)$, 
\begin{align}\label{q2ub}
    Q_2^{(N)}(t)\leq &S^{(N)}(t)-Q^{(N)}_1(t)  =S^{(N)}(t)-(N-I^{(N)}(t))\notag\\
    =&S^{(N)}(0)+A\big((N-\beta N^{\frac{1}{2}-\varepsilon})t\big)-D\big(Nt-\int_0^t I^{(N)}(s)ds\big)-N+I^{(N)}(t)\notag\\
    \leq & S^{(N)}(0)-N+\bar{I}_B^{(N)}(t)+A\big((N-\beta N^{\frac{1}{2}-\varepsilon})t\big)-D\big(Nt-\int_0^t \bar{I}_B^{(N)}(s)ds\big)\notag\\
    \leq & (2B+x)N^{\frac{1}{2}+\varepsilon}+\bar{I}_B^{(N)}(t)+\hat{A}\big((N-\beta N^{\frac{1}{2}-\varepsilon})t\big)-\hat{D}\big(Nt-\int_0^t \bar{I}_B^{(N)}(s)ds\big)\notag\\
    &\hspace{8cm}+\int_0^t\Bar{I}_B^{(N)}(s)ds-\beta N^{\frac{1}{2}-\varepsilon}t,
\end{align}
where $\hat{A}(s)=A(s)-s$ and $\hat{D}(s)=D(s)-s$. Recall
 $$\mathcal{M}^{*}(t)= \hat{A}\big((N-\beta N^{\frac{1}{2}-\varepsilon})t\big)-\hat{D}\big( Nt-\int_0^t\bar{I}_B^{(N)}(s)ds\big).$$
 Thus, using the above upper bound for $Q_2^{(N)}$, we obtain
 \begin{align}\label{eq:lem4.9-1}
     &\sup_{\underline{z}: \sum_iz_i \leq xN^{\frac{1}{2}+\varepsilon}}\mathbb{P}_{(0, \, \lfloor 2BN^{\frac{1}{2}+\varepsilon} \rfloor, \,\underline{z})}\big(\tau^{(N)}_2(B)\geq N^{2\varepsilon}t\big)\nonumber\\
     =&\sup_{\underline{z}: \sum_iz_i \leq xN^{\frac{1}{2}+\varepsilon}}\mathbb{P}_{(0, \, \lfloor 2BN^{\frac{1}{2}+\varepsilon} \rfloor, \,\underline{z})}\big(Q^{(N)}_2(N^{2\varepsilon}t)\geq \lfloor BN^{\frac{1}{2}+\varepsilon}\rfloor\ \text{and}\  \tau^{(N)}_2(B)\geq N^{2\varepsilon}t\big)\nonumber\\
     \leq &\mathbb{P}\Big(\int_0^{N^{2\varepsilon}t}\Bar{I}_B^{(N)}(s)ds\geq\frac{\beta}{2}N^{\frac{1}{2}+\varepsilon}t\Big)+\mathbb{P}\Big(\Bar{I}_B^{(N)}(N^{2\varepsilon}t)\geq \frac{\beta}{4}N^{\frac{1}{2}+\varepsilon}t\Big)\\
     &+\mathbb{P}\Big((2B+x)N^{\frac{1}{2}+\varepsilon}+\big(\mathcal{M}^{*}(N^{2\varepsilon }t)-\frac{\beta}{2}N^{\frac{1}{2}+\varepsilon}t\big)+\frac{\beta}{4}N^{\frac{1}{2}+\varepsilon}t\geq BN^{\frac{1}{2}+\varepsilon}\Big).\nonumber
 \end{align}
Then, for $B\geq B_1,N\geq \tilde N_B,t\geq t_0$, by Lemma~\ref{lem:LEMMA-6}\ref{7i}, 
\begin{equation}\label{eq:lem4.9-2}
    \mathbb{P}\Big(\int_0^{N^{2\varepsilon}t}\Bar{I}_B^{(N)}(s)ds>\frac{\beta}{2}N^{\frac{1}{2}+\varepsilon}t\Big)\leq c_1\exp \{-c_2 B^{\frac{1}{5}}N^{\frac{4\varepsilon}{5}}t^{\frac{1}{5}}\},
\end{equation}
and moreover, by Lemma~\ref{lem:LEMMA-6}\ref{7ii}, \begin{equation}\label{eq:lem4.9-3}
    \mathbb{P}\big(\Bar{I}_B^{(N)}(N^{2\varepsilon}t)\geq \frac{\beta}{4}N^{\frac{1}{2}+\varepsilon}t\big)\leq \exp\{-\tilde{c}_1BN^{4\varepsilon}t\}+\tilde{c}_2N^{4\varepsilon}t \exp \{-\tilde{c}_3 \sqrt{B} N^{2\varepsilon}t\}.
\end{equation}
For $t\geq \frac{8}{\beta}(B+x)$, by Lemma~\ref{lem:sup-poi},
\begin{align}\label{eq:lem4.9-4}
    &\mathbb{P}\Big((2B+x)N^{\frac{1}{2}+\varepsilon}+\big(\mathcal{M}^{*}(N^{2\varepsilon }t)-\frac{\beta}{2}N^{\frac{1}{2}+\varepsilon}t\big)+\frac{\beta}{4}N^{\frac{1}{2}+\varepsilon}t\geq BN^{\frac{1}{2}+\varepsilon}\Big)\nonumber \\
    = & \mathbb{P}\big(\mathcal{M}^{*}(N^{2\varepsilon }t)\geq \frac{\beta}{4}N^{\frac{1}{2}+\varepsilon}t-(B+x)N^{\frac{1}{2}+\varepsilon}\big)
    \leq \mathbb{P}\big(\mathcal{M}^{*}(N^{2\varepsilon }t)\geq \frac{\beta}{8}N^{\frac{1}{2}+\varepsilon}t\big)\leq 4e^{-c't},
\end{align}
where $c'$ is a positive constant not depending on $N$.
Plugging \eqref{eq:lem4.9-2}, \eqref{eq:lem4.9-3} and \eqref{eq:lem4.9-4} into \eqref{eq:lem4.9-1} and choosing appropriate constants complete the proof.
\end{proof}

\subsection{Up-crossing estimate.}\label{upsec}
In this subsection, we estimate the time taken for $Q^{(N)}_2$ to cross the level $2B N^{\frac{1}{2} + \varepsilon}$, starting from $(x, \, \lfloor BN^{\frac{1}{2}+\varepsilon} \rfloor, \,\underline{z})$, uniformly in $x, \underline{z}$.
The proof involves identifying excursions in the process path based on the $I^{(N)}$ process falling below a certain threshold (see \eqref{iexc}). The length of each excursion is controlled using Lemma~\ref{lem:LEMMA-9}. During each excursion, it is shown that $Q_2^{(N)}$ hits the level $2BN^{\frac{1}{2} + \varepsilon}$ with a positive probability that does not depend on $N$. This leads to a geometric number of such excursions required for $Q_2^{(N)}$ to hit the level $2BN^{\frac{1}{2} + \varepsilon}$. These estimates are combined to prove Proposition \ref{prop:UPCROSS}, the main result of this subsection.

Recall $\tau^{(N)}_1(2\beta)=\inf \{t\geq 0: I^{(N)}(t)= \lfloor 2\beta N^{\frac{1}{2}-\varepsilon}\rfloor\}$. We will write $\underline{Q}^{(N)}_3 := (Q_3^{(N)}, Q_4^{(N)}, \dots)$.

\begin{lemma}\label{lem:LEMMA-9}
Assume $I^{(N)}(0)=x$, $Q^{(N)}_2(0)=y$, and $\underline{Q}^{(N)}_3(0)=\underline{z}$. For any $N$ such that $\lfloor 2\beta N^{\frac{1}{2}-\varepsilon}\rfloor \ge 1$ and any $x\geq 2\beta N^{\frac{1}{2}-\varepsilon}$,
\begin{equation*}
    \sup_{y,\underline{z}}\mathbb{E}_{(x,y,\underline{z})}\big(e^{\tau^{(N)}_1(2\beta)/2}\big)\leq \frac{x}{\lfloor 2\beta N^{\frac{1}{2}-\varepsilon}\rfloor},
\end{equation*}
where $\mathbb{E}_{(x,y,\underline{z})}(\cdot)=\mathbb{E}(\cdot|I^{(N)}(0)=x,Q^{(N)}_2(0)=y,\underline{Q}^{(N)}_3(0)=\underline{z})$.
\end{lemma}
\begin{proof}
Define $W^{(N)}(t)=e^{\frac{t}{2}}I^{(N)}(t)$. 
Since the rate of increase of $I^{(N)}(t)$ is at most $N-I^{(N)}(t)$ and the rate of decrease is $N-\beta N^{\frac{1}{2}-\varepsilon}$ if $I^{(N)}(t)>0$, therefore
\begin{align*}
    \mathcal{L}W^{(N)}(t)&\leq \frac{1}{2}W^{(N)}(t)+e^{\frac{t}{2}}\big[(N-I^{(N)}(t))-(N-\beta N^{\frac{1}{2}-\varepsilon})\big]\\
    &=e^{\frac{t}{2}}\big(-\frac{1}{2} I^{(N)}(t)+\beta N^{\frac{1}{2}-\varepsilon}\big),
\end{align*}
where $\mathcal{L}(\cdot)$ is the infinitesimal generator.
For $t < \tau^{(N)}_1(2\beta)$, $\beta N^{\frac{1}{2}-\varepsilon}\leq \frac{I^{(N)}(t)}{2}$, and so $\mathcal{L}W^{(N)}(t)\mathds{1}\big[t < \tau^{(N)}_1(2\beta)\big]\leq 0.$
This implies that for all $y,\underline{z}\geq0$,
\begin{equation*}
\mathbb{E}_{(x,y,\underline{z})}(W^{(N)}(t\wedge \tau^{(N)}_1(2\beta)))\leq \mathbb{E}_{(x,y,\underline{z})}(W^{(N)}(0))=x,\quad\forall t\geq 0.
\end{equation*}
By Fatou's lemma and the observation that, almost surely, $(t\wedge \tau^{(N)}_1(2\beta))=\tau^{(N)}_1(2\beta)$ for sufficiently large $t$, we have that for all $y,\underline{z}\geq0$, 
\begin{equation}
    \mathbb{E}_{(x,y,\underline{z})}(W^{(N)}(\tau^{(N)}_1(2\beta)))\leq \liminf_{t\rightarrow\infty}\mathbb{E}(W^{(N)}(t\wedge \tau^{(N)}_1(2\beta)))\leq x,
\end{equation} 
and therefore, 
$ \sup_{y,\underline{z}}\mathbb{E}_{(x,y,\underline{z})}\big(e^{\tau^{(N)}_1(2\beta)/2}\big)\leq \frac{x}{\lfloor 2\beta N^{\frac{1}{2}-\varepsilon}\rfloor}.$
\end{proof}

\begin{prop}\label{prop:UPCROSS}
For any fixed $B > 0$,  there exist $p_B, t'_B, N'_B>0$ such that $\forall t\geq t'_B, N\geq N'_B$,
\begin{equation*}
    \blue{\sup_{x,\underline{z}}\mathbb{P}_{(x, \, \lfloor BN^{\frac{1}{2}+\varepsilon} \rfloor, \,\underline{z})}\big(\tau^{(N)}_2(2B)>N^{2\varepsilon}t\big)\leq \frac{\sqrt{t} N^{\frac{1}{2}+\varepsilon}}{2\beta} \exp \{-N^{2\varepsilon}\sqrt{t}/2\}+(1-p_B)^{\lfloor\sqrt{t}/2\rfloor}.}
\end{equation*}
\end{prop}
\blue{We outline the idea of the proof here. For $Q_2^{(N)}$ to exceed level $\lfloor 2B N^{\frac{1}{2}+\varepsilon}\rfloor$, $I^{(N)}$ has to first become small (fall below $\lfloor 2\beta N^{\frac{1}{2}-\varepsilon}\rfloor$). The time taken for this is quantified by Lemma \ref{lem:LEMMA-9}. The first term in the right hand side above arises from bounding the probability of less than $\lfloor\sqrt{t}/2\rfloor$ many such `returns' of $I^{(N)}$ to $\lfloor 2\beta N^{\frac{1}{2}-\varepsilon}\rfloor$ before time $N^{2\varepsilon}t$. Each time $I^{(N)}$ becomes small, the fluctuation in the arrival and departure processes ensures that with at least positive probability $p_B$, the total queue length $S^{(N)}$ (and hence $Q_2^{(N)}$) exceeds $\lfloor 2B N^{\frac{1}{2}+\varepsilon}\rfloor$ within the next $N^{2\varepsilon}$ time units. Thus, if there are at least $\lfloor\sqrt{t}/2\rfloor$ returns of $I^{(N)}$ by time $N^{2\varepsilon}t$, the probability of $Q_2^{(N)}$ not hitting $\lfloor 2B N^{\frac{1}{2}+\varepsilon}\rfloor$ in any of the time intervals of length $N^{2\varepsilon}$ following such returns is at least $(1-p_B)^{\lfloor\sqrt{t}/2\rfloor}$, giving the second term in the above bound.}
\begin{proof}[Proof of Proposition~\ref{prop:UPCROSS}.]
Define the stopping times: $\theta^{(N)}_0\coloneqq0$ and for $k\geq 0$, 
\begin{align}
    &\theta^{(N)}_{2k+1}\coloneqq\inf \{t\geq \theta^{(N)}_{2k}:I^{(N)}(t)\leq  \lfloor 2\beta N^{\frac{1}{2}-\varepsilon}\rfloor\},\label{iexc}\\
    &\theta^{(N)}_{2k+2}\coloneqq\theta^{(N)}_{2k+1}+N^{2\varepsilon}.
\end{align}
Note that if $I^{(N)}(\theta^{(N)}_{2k})\leq \lfloor 2\beta N^{\frac{1}{2}-\varepsilon}\rfloor$, then $\theta^{(N)}_{2k+1}=\theta^{(N)}_{2k}$; equivalently, if $\theta^{(N)}_{2k+1}-\theta^{(N)}_{2k}>0$, then $I^{(N)}(\theta^{(N)}_{2k})> \lfloor 2\beta N^{\frac{1}{2}-\varepsilon} \rfloor$.
Thus, for any $k\geq 0$ and $t\geq0$, any $N$ such that $\lfloor 2\beta N^{\frac{1}{2}-\varepsilon}\rfloor \ge 1$,
\begin{equation}\label{eq:lem4.11-1}
    \begin{split}
        \mathbb{P}\big(\theta^{(N)}_{2k+1}-\theta^{(N)}_{2k}>N^{2\varepsilon}t\big)
    &\leq  \sup_{x>\lfloor 2\beta N^{\frac{1}{2}-\varepsilon}\rfloor, y,\underline{z}}\mathbb{P}_{(x,y,\underline{z})}\big(\tau^{(N)}_1(2\beta)>N^{2\varepsilon}t\big)\\
    &\leq  e^{-N^{2\varepsilon}t/2}\sup_{x}\sup_{y,\underline{z}}\mathbb{E}_{(x,y,\underline{z})}\big(e^{\tau_1(2\beta)/2}\big)
    \leq  \frac{N}{\lfloor 2\beta N^{\frac{1}{2}-\varepsilon} \rfloor}\exp \{-N^{2\varepsilon}t/2\},
    \end{split}
\end{equation}
where the last inequality is due to Lemma~\ref{lem:LEMMA-9} and $x\leq N$.
Next, we claim the following:
\begin{claim}\label{claim:4.12-a}
Fix any $N > \max\{(5B)^{-(\frac{1}{2} - \varepsilon)}, (2\beta / B)^{\frac{1}{2\varepsilon}}\}$.
For any $k \ge 0$, the following inclusion relation holds on the event  $\sup_{s\in[0,\theta^{(N)}_{2k+1}]}Q^{(N)}_2(s)<2BN^{\frac{1}{2}+\varepsilon}:$ 
\begin{align}\label{eq:s-q2-subset}
   \Big\{\sup_{s\in[\theta^{(N)}_{2k+1},\theta^{(N)}_{2k+2}]}S^{(N)}(s)\geq 3BN^{\frac{1}{2}+\varepsilon}+S^{(N)}(\theta^{(N)}_{2k+1})\Big\}
   \subseteq \Big\{\sup_{s\in[\theta^{(N)}_{2k+1},\theta^{(N)}_{2k+2}]}Q^{(N)}_2(s)\geq 2BN^{\frac{1}{2}+\varepsilon}\Big\},
\end{align} 
\end{claim}
\emph{Proof.}
 Suppose that for some $k\geq 0$, $\sup_{s\in[0,\theta^{(N)}_{2k+1}]}Q^{(N)}_2(s)<2BN^{\frac{1}{2}+\varepsilon}$ and $\exists s\in [\theta^{(N)}_{2k+1},\theta^{(N)}_{2k+2}]$ such that $S^{(N)}(s)\geq 3BN^{\frac{1}{2}+\varepsilon}+S^{(N)}(\theta^{(N)}_{2k+1})$.
Let  $\bar{s}\coloneqq \inf\{s\in [\theta^{(N)}_{2k+1},\theta^{(N)}_{2k+2}]:S^{(N)}(s)\geq 3BN^{\frac{1}{2}+\varepsilon}+S^{(N)}(\theta^{(N)}_{2k+1})\}.$
Since $N > 5B N^{\frac{1}{2}+\varepsilon}$, we claim that $\Bar{Q}^{(N)}_3(\bar{s})\leq \bar{Q}^{(N)}_3(\theta^{(N)}_{2k+1})$. If this was not the case, $\bar s > \theta^{(N)}_{2k+1}$ and there would be $\tilde s \in [\theta_{2k+1}, \bar s)$ such that $\tilde s$ is a `point of increase' of $\bar Q_3^{(N)}$, that is, $Q_2^{(N)}(\tilde s) = N$, $\bar Q_3^{(N)}(\tilde s) = \bar Q_3^{(N)}(\theta^{(N)}_{2k+1})$ and $I^{(N)}(\tilde s) = 0$. Hence, recalling that $\sup_{s\in[0,\theta^{(N)}_{2k+1}]}Q^{(N)}_2(s)<2BN^{\frac{1}{2}+\varepsilon}$,
\begin{align*}
S^{(N)}(\tilde s) &= N + \bar Q_3^{(N)}(\theta^{(N)}_{2k+1}) + N \ge N - 2BN^{\frac{1}{2}+\varepsilon} + Q^{(N)}_2(s) + \bar Q_3^{(N)}(\theta^{(N)}_{2k+1}) + (N- I^{(N)}(\theta^{(N)}_{2k+1}))\\
&=N - 2BN^{\frac{1}{2}+\varepsilon} + S^{(N)}(\theta^{(N)}_{2k+1}) > 3BN^{\frac{1}{2}+\varepsilon} + S^{(N)}(\theta^{(N)}_{2k+1}),
\end{align*}
which is a contradiction to the fact that $\tilde s < \bar s$. Hence, we conclude that $N > 5B N^{\frac{1}{2}+\varepsilon}$ implies $\Bar{Q}^{(N)}_3(\bar{s})\leq \bar{Q}^{(N)}_3(\theta^{(N)}_{2k+1})$.
Using this observation and the definition of $\bar s$, we obtain
\begin{align*}
    3BN^{\frac{1}{2}+\varepsilon}=&S^{(N)}(\bar{s})-S^{(N)}(\theta^{(N)}_{2k+1})\\
    =&\big(Q^{(N)}_2(\bar{s})+\Bar{Q}^{(N)}_3(\bar{s})+N-I^{(N)}(\bar{s})\big)-\big(Q^{(N)}_2(\theta^{(N)}_{2k+1})+\Bar{Q}^{(N)}_3(\theta^{(N)}_{2k+1})+N-I^{(N)}(\theta^{(N)}_{2k+1})\big)\\
    \leq & Q_2^{(N)}(\bar{s})-Q_2^{(N)}(\theta^{(N)}_{2k+1})+I^{(N)}(\theta^{(N)}_{2k+1})
    \leq  Q_2^{(N)}(\bar{s})-Q_2^{(N)}(\theta^{(N)}_{2k+1})+2\beta N^{\frac{1}{2}-\varepsilon}.
\end{align*}
Further, note that $2\beta N^{\frac{1}{2}-\varepsilon}<BN^{\frac{1}{2}+\varepsilon}$, due to the choice of $N$. 
Thus, the above yields 
$Q^{(N)}_2(\bar{s})\geq Q^{(N)}_2(\theta^{(N)}_{2k+1})+2BN^{\frac{1}{2}+\varepsilon}\geq 2BN^{\frac{1}{2}+\varepsilon}.$
This completes the proof of Claim~\ref{claim:4.12-a}.

Therefore, due to Claim~\ref{claim:4.12-a},
\begin{align}\label{eq:lem4.11-2}
    &\inf_{x,\underline{z}}\mathbb{P}_{(x, \, \lfloor BN^{\frac{1}{2}+\varepsilon} \rfloor, \,\underline{z})}\Big(\sup_{s\in[\theta^{(N)}_{2k+1},\theta^{(N)}_{2k+2}]}Q^{(N)}_2(s)\geq 2BN^{\frac{1}{2}+\varepsilon} \, | \, Q^{(N)}_2(\theta^{(N)}_{2k+1})<2BN^{\frac{1}{2}+\varepsilon}\Big)\nonumber\\
    &\geq  \inf_{x,\underline{z}}\mathbb{P}_{(x, \, \lfloor BN^{\frac{1}{2}+\varepsilon} \rfloor, \,\underline{z})}\Big(\sup_{s\in[\theta^{(N)}_{2k+1},\theta^{(N)}_{2k+2}]}(S^{(N)}(s)-S^{(N)}(\theta^{(N)}_{2k+1}))\geq 3BN^{\frac{1}{2}+\varepsilon} \, | \,Q^{(N)}_2(\theta^{(N)}_{2k+1})<2BN^{\frac{1}{2}+\varepsilon}\Big)\nonumber\\
    &\geq  \mathbb{P}\Big(\sup_{s\in[0,N^{2\varepsilon}]}A\big((N-\beta N^{\frac{1}{2}-\varepsilon})s\big)-D\big(Ns\big)\geq 3BN^{\frac{1}{2}+\varepsilon}\Big)\nonumber\\
    &= \mathbb{P}\Big(\sup_{s\in[0,N^{2\varepsilon}]}\hat{A}\big((N-\beta N^{\frac{1}{2}-\varepsilon})s\big)-\hat{D}\big(Ns\big)-\beta N^{\frac{1}{2}-\varepsilon}s\geq 3BN^{\frac{1}{2}+\varepsilon}\Big)\nonumber\\
    &\geq  \mathbb{P}\Big(\hat{A}\big((N-\beta N^{\frac{1}{2}-\varepsilon})N^{2\varepsilon}\big)-\hat{D}\big(N^{1+2\varepsilon}\big)-\beta N^{\frac{1}{2}+\varepsilon}\geq 3BN^{\frac{1}{2}+\varepsilon}\Big)\nonumber\\
    &= \mathbb{P}\Big(N^{-\frac{1}{2}-\varepsilon}\big(\hat{A}\big((N-\beta N^{\frac{1}{2}-\varepsilon})N^{2\varepsilon}\big)-\hat{D}\big(N^{1+2\varepsilon}\big)\big)-\beta \geq 3B\Big)\nonumber\\
    &\geq  p_B>0, \text{ for sufficient large }N.
\end{align}
Observe that $p_B$ does not depend on $N$ since $N^{-\frac{1}{2}-\varepsilon}\big(\hat{A}\big((N-\beta N^{\frac{1}{2}-\varepsilon})N^{2\varepsilon}\big)-\hat{D}\big(N^{1+2\varepsilon}\big)\big)\pto N(0,2)$ by the  Martingale FCLT.
Therefore, for sufficiently large $N,t$,
\begin{align*}
    &\sup_{x,\underline{z}} \mathbb{P}_{(x, \, \lfloor BN^{\frac{1}{2}+\varepsilon} \rfloor, \,\underline{z})}\big(\tau^{(N)}_2(2B)>N^{2\varepsilon}t\big)
    \leq  \sum_{k=0}^{\lfloor\sqrt{t}/2\rfloor}\sup_{x,\underline{z}} \mathbb{P}_{(x, \, \lfloor BN^{\frac{1}{2}+\varepsilon} \rfloor, \,\underline{z})}\big(\theta^{(N)}_{2k+1}-\theta^{(N)}_{2k}>N^{2\varepsilon}\sqrt{t}\big)\\
    &+\sup_{x,\underline{z}} \mathbb{P}_{(x, \, \lfloor BN^{\frac{1}{2}+\varepsilon} \rfloor, \,\underline{z})}\big(\sup_{s\in[\theta^{(N)}_{2k+1},\theta^{(N)}_{2k+2}]}Q^{(N)}_2(s)<2BN^{\frac{1}{2}+\varepsilon},Q^{(N)}_2(\theta^{(N)}_{2k+1})<2BN^{\frac{1}{2}+\varepsilon},\forall\ k\leq \lfloor\sqrt{t}/2\rfloor\big)\\
    \leq & \frac{\sqrt{t} N^{\frac{1}{2}+\varepsilon}}{2\beta} \exp \{-N^{2\varepsilon}\sqrt{t}/2\}+(1-p_B)^{\lfloor\sqrt{t}/2\rfloor},
\end{align*}
where the last inequality is due to \eqref{eq:lem4.11-1} and \eqref{eq:lem4.11-2}.
\end{proof}

\subsection{Supremum of \texorpdfstring{$Q^{(N)}_2$}{Q2}.}\label{ssec:supremum Q2}
In this section, we will give bounds on the supremum of the process $Q^{(N)}_2(\cdot)$ on a time interval $[0,N^{2\varepsilon}T]$ for fixed $T>0$. This bound will be used in Section \ref{sec:STEAYSTATE} to show that, $\bar{Q}^{(N)}_3(\infty) \xrightarrow{P} 0$ as $N \rightarrow \infty$. It will also be used for the process-level convergence proofs in Sections \ref{sec:PROCESS-LEVEL} and \ref{sec:proof-prop-5.2}. In particular, it is a key tool in proving Proposition \ref{prop:INT-IDLE-2}.
Let $I^{(N)}(0)=0$, $Q^{(N)}_2(0)= \lfloor2BN^{\frac{1}{2}+\varepsilon}\rfloor$, and $Q^{(N)}_3(0)=0$.
Recall $\sigma^{(N)}_i$ and $K^{(N)}_T$ from~\eqref{eq:sigma_i-def} and 
fix $B\geq(B_1\vee 2\beta\vee 5)$, where $B_1$ was introduced in Lemma \ref{lem:LEMMA-6}. The next lemma bounds the supremum of $Q^{(N)}$ on $[\sigma^{(N)}_{2i},\sigma^{(N)}_{2i+1}]$ for $i \ge 0$. Proof is given in Appendix \ref{sampleapp}.
\begin{lemma}\label{lem:B1}
Recall $\tilde N_B$ from Lemma~\ref{lem:LEMMA-6}. There exist $x_B \ge 2B$ such that for all $i\geq 0$,  $N\geq \tilde N_B$, $N^{\frac{1}{2}-\varepsilon}\geq x\geq x_B$,
\begin{equation*}
    \mathbb{P}\Big(\sup_{s\in[\sigma^{(N)}_{2i},\sigma^{(N)}_{2i+1}]}Q^{(N)}_2(s)\geq xN^{\frac{1}{2}+\varepsilon}\ \Big|\ Q^{(N)}_3(\sigma^{(N)}_{2i})=0\Big)
    \leq c^*_1\exp\{-c^*_2x\}+c^*_3\exp\{-c^*_4N^{\frac{4\varepsilon}{5}}x^{\frac{1}{5}}\},
\end{equation*}
where $c^*_j$, $j\in\{1,2,3,4\}$, are constants that do not depend on $i,N,x$.
\end{lemma}
\begin{prop}\label{prop:bdd-Q2}
There exist constants $x'$, $N'>0$, $\bar{c}'$, $\bar{c}_1$, $\bar{c}_2$ and $c^{*}_i,i=\{1,2,3,4\}$, depending only on $B$, such that for all $N\geq N_B$, $x\geq x'_B$, $T \ge x^{-1}$,
\begin{align*}
    &\mathbb{P}_{(0, \, \lfloor 2BN^{\frac{1}{2}+\varepsilon} \rfloor, \,\underline{0})}\Big(\sup_{s\in[0,N^{2\varepsilon}T]}Q^{(N)}_2(s)\geq xN^{\frac{1}{2}+\varepsilon}\Big)\\
    \leq&  \exp\{-\bar{c}'Tx\}+Tx\big(\bar{c}_1\exp \big\{-\bar{c}_2N^{(\frac{1}{2}-\varepsilon)/5}\big\}+c^*_1\exp\{-c^*_2x\}+c^*_3\exp\{-c^*_4N^{\frac{4\varepsilon}{5}}x^{\frac{1}{5}}\}\big).
\end{align*}
\end{prop}
\begin{proof}[Proof of Proposition~\ref{prop:bdd-Q2}.]
Define $\phi^{(N)}_i=\mathds{1}[\sigma^{(N)}_{2i+1}-\sigma^{(N)}_{2i}\geq  N^{2\varepsilon}]$, $i \ge 1$. For $i \ge 1$, $s \ge 0$,
\begin{equation}\label{eq:prop-4.14-1}
    \begin{split}
    &S^{(N)}(s+\sigma^{(N)}_{2i})=N+ \lfloor2BN^{\frac{1}{2} + \varepsilon}\rfloor+\bar{Q}^{(N)}_3(\sigma^{(N)}_{2i})\\
    &\quad+ \left[A\br{(N-\beta N^{\frac{1}{2}-\varepsilon})(s+ \sigma^{(N)}_{2i})} - A\br{(N-\beta N^{\frac{1}{2}-\varepsilon})\sigma^{(N)}_{2i}}\right] - D\br{\int_{\sigma^{(N)}_{2i}}^{\sigma^{(N)}_{2i}+s}(N-I^{(N)}(u))du}\\
    &=N+ \lfloor2BN^{\frac{1}{2} + \varepsilon}\rfloor +\bar{Q}^{(N)}_3(\sigma^{(N)}_{2i}) + \int_{\sigma^{(N)}_{2i}}^{\sigma^{(N)}_{2i}+s}I^{(N)}(u)du -\beta N^{\frac{1}{2}-\varepsilon}s\\
    &+ \left[\hat A\br{(N-\beta N^{\frac{1}{2}-\varepsilon})(s+ \sigma^{(N)}_{2i})} - \hat A\br{(N-\beta N^{\frac{1}{2}-\varepsilon})\sigma^{(N)}_{2i}}\right] - \hat D\br{\int_{\sigma^{(N)}_{2i}}^{\sigma^{(N)}_{2i}+s}(N-I^{(N)}(u))du}.
    \end{split}
\end{equation}
where $\hat{A}(s)=A(s)-s$ and $\hat{D}(s)=D(s)-s$. Recall that 
$\hat{\mathcal{M}}(s) = \hat{A}\br{(N-\beta N^{\frac{1}{2}-\varepsilon})s}-\hat{D}\br{\int_0^s(N-I^{(N)}(u))du}$ for $s \ge 0.$
For $i \ge 1$, $B\geq(2\beta\vee 5)$,
\begin{align*}
    & \mathbb{P}\Big(\phi^{(N)}_i=0\ \Big|\ Q^{(N)}_3(\sigma^{(N)}_{2i})=0\Big)\le \mathbb{P}\Big(\inf_{s\in[\sigma^{(N)}_{2i},\sigma^{(N)}_{2i}+N^{2\varepsilon}]}Q^{(N)}_2(s)\leq \lfloor BN^{\frac{1}{2}+\varepsilon} \rfloor\ \Big|\ Q^{(N)}_3(\sigma^{(N)}_{2i})=0\Big)\\
    &\leq \mathbb{P}\Big(\inf_{s\in[\sigma^{(N)}_{2i},\sigma^{(N)}_{2i}+N^{2\varepsilon}]}Q_2^{(N)}(s)\leq \lfloor BN^{\frac{1}{2}+\varepsilon}\rfloor, \sup_{s\in[\sigma^{(N)}_{2i},\sigma^{(N)}_{2i+1}]}Q^{(N)}_3(s)=0\ \Big|\ Q^{(N)}_3(\sigma^{(N)}_{2i})=0\Big)\\
    &\hspace{4cm}+\mathbb{P}\Big(\sup_{s\in[\sigma^{(N)}_{2i},\sigma^{(N)}_{2i+1}]}Q^{(N)}_3(s)>0\ \Big|\ Q^{(N)}_3(\sigma^{(N)}_{2i})=0\Big)\\
    &\leq \mathbb{P}\Big(\inf_{s\in[\sigma^{(N)}_{2i},\sigma^{(N)}_{2i}+N^{2\varepsilon}]}S^{(N)}(s)\leq N+ \lfloor BN^{\frac{1}{2}+\varepsilon}\rfloor\ \Big|\ Q^{(N)}_3(\sigma^{(N)}_{2i})=0\Big)\\
    &\hspace{4cm}+\mathbb{P}_{(0, \, \lfloor 2BN^{\frac{1}{2}+\varepsilon} \rfloor, \,\underline{0})}\Big(\sup_{s\in[0,\tau_2^{(N)}(B)]}\bar Q^{(N)}_3(s)>0\Big)\\
    &\leq  \mathbb{P}\Big(\inf_{s\in[0,N^{2\varepsilon}]}(\lfloor2BN^{\frac{1}{2}+\varepsilon}\rfloor +\hat{\mathcal{M}}(s)-\beta N^{\frac{1}{2}-\varepsilon}s)\leq \lfloor BN^{\frac{1}{2}+\varepsilon} \rfloor\Big) 
    +\bar{c}_1\exp \big\{-\bar{c}_2N^{(\frac{1}{2}-\varepsilon)/5}\big\}\\
    &\leq  \mathbb{P}\Big(\inf_{s\in[0,N^{2\varepsilon}]}\hat{\mathcal{M}}(s)\leq -\frac{B}{2}N^{\frac{1}{2}+\varepsilon}\Big)+\bar{c}_1\exp \big\{-\bar{c}_2N^{(\frac{1}{2}-\varepsilon)/5}\big\}\\
    &\leq  \frac{8N^{1+2\varepsilon}}{B^2N^{1+2\varepsilon}}+\bar{c}_1\exp \big\{-\bar{c}_2N^{(\frac{1}{2}-\varepsilon)/5}\big\}\leq \frac{1}{2}, \text{ for sufficient large }N,
\end{align*}
where the third inequality uses the strong Markov property for the second term and the observation that $S^{(N)}(s) \le N + Q_2^{(N)}(s)$ if $Q_3^{(N)}(s)=0$ for the first term. The fourth inequality is due to the strong Markov property, \eqref{eq:prop-4.14-1} and Lemma~\ref{lem:LEMMA-8} \textcolor{black}{(on taking a limit as $x \downarrow 0$)}, and the last inequality follows from Doob's $L^2$-maximal inequality.
Thus, $\mathbb{E}(\phi^{(N)}_i\ |\ Q^{(N)}_3(\sigma^{(N)}_{2i})=0)\geq 1/2$ for sufficiently large $N$.
Now, for $T >0$, recall $K^{(N)}_T=\inf\big\{k:\Bar{\sigma}^{(N)}_{2k}\geq N^{2\varepsilon}T\big\}$.
By Azuma's inequality and Lemma~\ref{lem:LEMMA-8}, taking $a\geq 8$, $T \ge a^{-1}$, for sufficiently large $N$,
\begin{equation*}\label{eq:B2-1}
    \begin{split}
    &\mathbb{P}\Big(K^{(N)}_T\geq aT\Big)\\
    &\leq \mathbb{P}\Big(\sum_{i=1}^{\lfloor aT \rfloor}(\sigma^{(N)}_{2i+1}-\sigma^{(N)}_{2i})\leq N^{2\varepsilon}T\Big)\\
    &\leq  \mathbb{P}\Big(\sum_{i=1}^{\lfloor aT \rfloor}\phi^{(N)}_i
    \leq T\ \ \text{and}\ \sup_{s\in[\sigma^{(N)}_{2i},\sigma^{(N)}_{2i+1}]}Q^{(N)}_3(s)=0,\forall 1\leq i\leq \lfloor aT \rfloor \Big)\\
    &\hspace{8cm}+\mathbb{P}\Big(\exists 1\leq i\leq \lfloor aT \rfloor, \sup_{s\in[\sigma^{(N)}_{2i},\sigma^{(N)}_{2i+1}]}Q^{(N)}_3(s)>0\Big)
    \end{split}
\end{equation*}
\begin{equation}
    \begin{split}
        &\leq  \mathbb{P}\bigg(\sum_{i=1}^{\lfloor aT \rfloor}\Big(\phi^{(N)}_i-\mathbb{E}(\phi^{(N)}_i\ |\ Q^{(N)}_3(\sigma^{(N)}_{2i})=0)\Big)\leq -\frac{aT}{8}\\ 
        &\hspace{5cm}\text{and}\ \sup_{s\in[\sigma^{(N)}_{2i},\sigma^{(N)}_{2i+1}]}Q^{(N)}_3(s)=0,\ \forall 1\leq i\leq \lfloor aT \rfloor\bigg)  \\
    &\hspace{5cm}+aT\cdot\sup_i\mathbb{P}\Big(\sup_{s\in[\sigma^{(N)}_{2i},\sigma^{(N)}_{2i+1}]}Q^{(N)}_3(s)>0 \, \Big| \, Q^{(N)}_3(\sigma^{(N)}_{2i})=0\Big)\\
    &\leq e^{-\bar{c}'aT}+aT\bar{c}_1\exp \big\{-\bar{c}_2N^{(\frac{1}{2}-\varepsilon)/5}\big\}.
    \end{split}
\end{equation}
Take $a=x$. Then, using \eqref{eq:B2-1}, Lemma~\ref{lem:B1}, and union bound, for $N\geq \tilde N_B$, $ N^{\frac{1}{2}-\varepsilon}\geq x\geq(x_B\vee 8)$, $T \ge x^{-1}$,
\begin{align*}
    &\mathbb{P}\Big(\sup_{s\in[0,N^{2\varepsilon}T]}Q^{(N)}_2(s)
    \geq xN^{\frac{1}{2}+\varepsilon}\Big)\\
    &\leq \mathbb{P}\Big(K^{(N)}_T\geq xT\Big)+Tx\cdot \sup_i\mathbb{P}\Big(\sup_{s\in[\sigma^{(N)}_{2i},\sigma^{(N)}_{2i+1}]}Q^{(N)}_2\geq xN^{\frac{1}{2}+\varepsilon}\ \Big|\ Q^{(N)}_3(\sigma^{(N)}_{2i})=0\Big)\\
    &\leq \exp\{-\bar{c}'Tx\}+Tx\Big(\bar{c}_1\exp \big\{-\bar{c}_2N^{(\frac{1}{2}-\varepsilon)/5}\big\}+c^*_1\exp\{-c^*_2x\}+c^*_3\exp\{-c^*_4N^{\frac{4\varepsilon}{5}}x^{\frac{1}{5}}\}\Big).
\end{align*}
This proves the proposition.

\end{proof}

\section{Steady state analysis.}\label{sec:STEAYSTATE}
Our goal is to identify points in the state space which are hit infinitely often by the process and 
the length between successive hitting times has a finite expectation. 
This will provide a renewal-theoretic representation of the stationary measure.

Fix $\varepsilon \in [0, 1/2)$. Throughout this section, we will choose and fix a number $B=B_0$ as the maximum of the lower bounds on $B$ given in the results in Section~\ref{sec:HITTIME}. In Section \ref{renbd}, we will formally define certain renewal times in the process path and obtain estimates on their first and second moments (Proposition \ref{cor:MOEMNT-RENEWAL}). In Section \ref{renre}, these estimates will be used to obtain a renewal representation \eqref{eq:repre-stat} for the stationary measure of $(I^{(N)}(\cdot), \{Q^{(N)}_i(\cdot)\}_{i \ge 2})$ and will be combined with the sample path analysis of Section \ref{sec:HITTIME} to prove Theorem \ref{thm:TIGHTNESS-XN}. 

\subsection{Renewal times and their moment bounds.}\label{renbd}
Recall $\sigma^{(N)}_i, \bar{K}^{(N)}$ and $\Theta^{(N)}$ from~\eqref{eq:sigma_i-def}.
The next lemma gives a probability bound on $\bar{K}^{(N)}$. 
Starting from $(0, \lfloor 2B N^{\frac{1}{2} + \varepsilon}\rfloor, \underline{0})$, if $Q^{(N)}_2$ hits $\lfloor BN^{\frac{1}{2}+\varepsilon}\rfloor$ without making $\Bar{Q}^{(N)}_3$ positive, then $\bar{K}^{(N)}=1$.
In fact, as discussed in Section \ref{qthree}, this trajectory occurs with high probability for large $N$.
However, on the event where $\bar{K}^{(N)} >1$ (implying $\bar{Q}^{(N)}_3$ becomes positive before time $\sigma^{(N)}_1$), it could take a potentially large number of `toggles' of $Q^{(N)}_2$ between the levels  $\lfloor BN^{\frac{1}{2}+\varepsilon}\rfloor$ and $\lfloor 2BN^{\frac{1}{2}+\varepsilon} \rfloor$ for $\bar{Q}^{(N)}_3$ to drop back to zero. This is reflected in the statement of the lemma; \textcolor{black}{note the additive $k+k^2N^{\frac{1}{2}-\varepsilon}$ term within the probability statement}.

\begin{lemma}\label{lem:LEMMA-5.1}
There exist $ N_0, c^*_1,c^*_2>0$, such that for all $N\geq N_0,k\geq 1$,
\begin{equation*}
    \mathbb{P}_{(0, \, \lfloor 2BN^{\frac{1}{2}+\varepsilon} \rfloor, \,\underline{0})}\big(\bar{K}^{(N)}\geq 1+k+k^2N^{\frac{1}{2}-\varepsilon}\big)\leq kc^*_1\exp \big\{-c^*_2N^{(\frac{1}{2}-\varepsilon)/11}\big\}\exp \Big\{-c^*_2\big(k/N^{\frac{1}{2}-\varepsilon}\big)^{\frac{1}{11}}\Big\}+2^{-k}.
\end{equation*}
\end{lemma}
Lemma~\ref{lem:LEMMA-5.1} is proved in Appendix \ref{app:steady}. \blue{At the heart of the proof is Lemma~\ref{lem:LEMMA-8} which quantifies the increments of $\Bar{Q}^{(N)}_3$ on the excursion intervals $[\sigma_{2i}^{(N)},\sigma_{2i+2}^{(N)}]$. In particular, it shows that with probability at least $1/2$, $\Bar{Q}^{(N)}_3$ decreases by $\frac{B}{4\beta}N^{2\varepsilon}$ or more during each such excursion until $\Bar{Q}^{(N)}_3$ is $O(N^{2\varepsilon})$. This is the content of Lemma~\ref{lem:bar-q3-positive}. Using this observation, it is shown that over a sufficiently large number of such excursion intervals, $\Bar{Q}^{(N)}_3$ drops to zero with high probability.}
\blue{From Lemma~\ref{lem:LEMMA-8}, note that the supremum of $\Bar{Q}^{(N)}_3$ on $[\sigma_{2i}^{(N)},\sigma_{2i+2}^{(N)}]$ has a stretched-exponential tail.
The constant $1/11$ in the exponent above thus appears due to the use of concentration inequality for sums of independent random variables with stretched-exponential tails (Lemma~\ref{lem:LEMMA-5}) in this setup.}

Recall $\Theta^{(N)} = \sigma^{(N)}_{2\bar{K}^{(N)}}$. 
To obtain bounds on the first and second moments of $\Theta^{(N)}$, we will first establish a tail bound for $\mathbb{P}_{(0, \, \lfloor 2BN^{\frac{1}{2}+\varepsilon} \rfloor, \,\underline{0})}\big(\Theta^{(N)} >N^{2\varepsilon}t\big)$.
The next lemma combines the bound on the number of excursions obtained in Lemma \ref{lem:LEMMA-5.1} with the sample path estimates from Section \ref{sec:HITTIME} to provide the necessary tail estimates.

\begin{lemma}\label{prop:RENEWAL-TIME}
There exist constants $ \bar{c}_1,\bar{c}_2,N_0,t_0>0$, such that for all $N\geq N_0,t\geq t_0$,
\begin{equation*}
   \mathbb{P}_{(0, \, \lfloor 2BN^{\frac{1}{2}+\varepsilon} \rfloor, \,\underline{0})}\big(\Theta^{(N)} >N^{2\varepsilon}t\big)\leq \bar{c}_1e^{-\bar{c}_2t^{1/5}}+\bar{c}_1e^{-\bar{c}_2N^{(\frac{1}{2}-\varepsilon)/11}}e^{-\bar{c}_2\big(\frac{t}{N^{4(1/2-\varepsilon)}}\big)^{1/44}}.
\end{equation*}
\end{lemma}
\begin{proof}
Note that for any $k\geq 1$,
\begin{align}\label{eq:5.2-1}
    &\mathbb{P}_{(0, \, \lfloor 2BN^{\frac{1}{2}+\varepsilon} \rfloor, \,\underline{0})}\big(\Theta^{(N)} >N^{2\varepsilon}t\big)\nonumber\\
    &\leq \mathbb{P}_{(0, \, \lfloor 2BN^{\frac{1}{2}+\varepsilon} \rfloor, \,\underline{0})}\big(\sigma^{(N)}_2>N^{2\varepsilon}t/2\big)+\mathbb{P}_{(0, \, \lfloor 2BN^{\frac{1}{2}+\varepsilon} \rfloor, \,\underline{0})}\big(\bar{K}^{(N)}>1+k+k^2N^{\frac{1}{2}-\varepsilon}\big)\\
    &\quad+\mathbb{P}_{(0, \, \lfloor 2BN^{\frac{1}{2}+\varepsilon} \rfloor, \,\underline{0})}\big(\sigma^{(N)}_{2\bar{K}^{(N)}}-\sigma^{(N)}_2>N^{2\varepsilon}t/2,1<\bar{K}^{(N)}\leq 1+k+k^2N^{\frac{1}{2}-\varepsilon}\big).\nonumber
\end{align}
We will upper bound each of the terms on the RHS above. Note that there exist $t_0, N_0>0$, such that
for all $N\geq N_0, x > 0$ and $t \ge t_0 \vee 8\beta^{-1}(B+ x)$,
\begin{align}\label{eq:5.2-2}
    &\sup_{\underline{z}: \, \sum_iz_i \leq xN^{\frac{1}{2}+\varepsilon}}\mathbb{P}_{(0, \, \lfloor 2BN^{\frac{1}{2}+\varepsilon} \rfloor, \,\underline{z})}\big(\sigma^{(N)}_2> N^{2\varepsilon}t/2\big)\nonumber\\
    &\leq  \sup_{\underline{z}: \, \sum_iz_i \leq xN^{\frac{1}{2}+\varepsilon}}\mathbb{P}_{(0, \, \lfloor 2BN^{\frac{1}{2}+\varepsilon} \rfloor, \,\underline{z})}\big(\sigma^{(N)}_1>N^{2\varepsilon}t/4\big)\nonumber\\
    &\qquad +\sup_{\underline{z}: \, \sum_iz_i \leq xN^{\frac{1}{2}+\varepsilon}}\mathbb{P}_{(0, \, \lfloor 2BN^{\frac{1}{2}+\varepsilon} \rfloor, \,\underline{z})}\big(\sigma^{(N)}_2>N^{2\varepsilon}t/2,\sigma^{(N)}_1\leq N^{2\varepsilon}t/4\big)\\
    & \leq \sup_{\underline{z}: \sum_iz_i \leq xN^{\frac{1}{2}+\varepsilon}}\mathbb{P}_{(0, \, \lfloor 2BN^{\frac{1}{2}+\varepsilon} \rfloor, \,\underline{z})}\big(\tau^{(N)}_2(B)\geq N^{2\varepsilon}t/4\big)
    +  \sup_{x,\underline{z}}\mathbb{P}_{(x, \, \lfloor BN^{\frac{1}{2}+\varepsilon} \rfloor, \,\underline{z})}\big(\tau^{(N)}_2(2B)>N^{2\varepsilon}t/4\big)\nonumber\\
    &\leq  ce^{-c'\sqrt{t}}+c e^{-c'N^{\frac{4\varepsilon}{5}}t^{1/5}},\nonumber
\end{align}
where the last inequality is from Proposition~\ref{prop:DOWNCROSS} and Proposition~\ref{prop:UPCROSS} and by collecting the leading order terms.
Write $l(k,N)=1+k+k^2N^{\frac{1}{2}-\varepsilon}$ and
recall $\bar{Z}^{(N)}_{j}=\bar{Q}^{(N)}_3(\sigma_{2j+2})-\bar{Q}^{(N)}_3(\sigma_{2j})$.
Then we have for any $x>0$,
\begin{equation}\label{eq:5.3-local-0}
    \begin{split}
     &\mathbb{P}_{(0, \, \lfloor 2BN^{\frac{1}{2}+\varepsilon} \rfloor, \,\underline{0})}\big(\sigma^{(N)}_{2\bar{K}^{(N)}}-\sigma^{(N)}_2>N^{2\varepsilon}t/2,1<\bar{K}^{(N)}\leq l(k,N)\big)\\
    \leq&\sum_{j=1}^{l(k,N)}\mathbb{P}_{(0, \, \lfloor 2BN^{\frac{1}{2}+\varepsilon} \rfloor, \,\underline{0})}\big(\sigma^{(N)}_{2j+2}-\sigma^{(N)}_{2j}>\frac{N^{2\varepsilon}t}{2l(k,N)},\Bar{Z}^{(N)}_0>0\big)\\
    \leq& \sum_{j=1}^{l(k,N)}\left(\mathbb{P}_{(0, \, \lfloor 2BN^{\frac{1}{2}+\varepsilon} \rfloor, \,\underline{0})}\big(\Bar{Z}^{(N)}_j>xN^{\frac{1}{2}+\varepsilon}\big)\right.\\
 &\qquad\quad \left. + \, \mathbb{P}_{(0, \, \lfloor 2BN^{\frac{1}{2}+\varepsilon} \rfloor, \,\underline{0})}\big(\sigma^{(N)}_{2j+2}-\sigma^{(N)}_{2j}>\frac{N^{2\varepsilon}t}{2l(k,N)},\Bar{Z}^{(N)}_0>0,\Bar{Z}^{(N)}_j\leq xN^{\frac{1}{2}+\varepsilon}\big)\right).
    \end{split}
\end{equation}
By Lemma \ref{lem:LEMMA-8} (i), for all large enough $N$,
\begin{equation}\label{eq:5.3-local-1}
    \sum_{j=1}^{l(k,N)}\mathbb{P}_{(0, \, \lfloor 2BN^{\frac{1}{2}+\varepsilon} \rfloor, \,\underline{0})}\big(\Bar{Z}^{(N)}_j>xN^{\frac{1}{2}+\varepsilon}\big)\leq c_1l(k,N)e^{-c_2N^{(\frac{1}{2}-\varepsilon)/5}}e^{-c_2x^{1/5}}.
\end{equation}
Also, for all $N$ large enough, $1\leq j\leq l(k,N)$, $x > 0$ and $t\geq [t_0 \vee 8\beta^{-1}(B+ x)] l(k,N)$,
\begin{equation}\label{eq:5.3-local-2}
    \begin{split}
    &\mathbb{P}_{(0, \, \lfloor 2BN^{\frac{1}{2}+\varepsilon} \rfloor, \,\underline{0})}\big(\sigma^{(N)}_{2j+2}-\sigma^{(N)}_{2j}>\frac{N^{2\varepsilon}t}{2l(k,N)},\Bar{Z}^{(N)}_0>0,\Bar{Z}^{(N)}_j\leq xN^{\frac{1}{2}+\varepsilon}\big)\\
    &\quad \leq\mathbb{P}_{(0, \, \lfloor 2BN^{\frac{1}{2}+\varepsilon} \rfloor, \,\underline{0})}\big(\Bar{Z}^{(N)}_0>0\big)\cdot\sup_{\underline{z}: \, \sum_iz_i \leq xN^{\frac{1}{2}+\varepsilon}}\mathbb{P}_{(0, \, \lfloor 2BN^{\frac{1}{2}+\varepsilon} \rfloor, \,\underline{z})}\big(\sigma^{(N)}_2>\frac{N^{2\varepsilon}t}{2l(k,N)}\big)\\
    &\quad \leq  c_1e^{-c_2 N^{(\frac{1}{2}-\varepsilon)/5}}\big[ce^{-c'\big(\frac{t}{l(k,N)}\big)^{1/2}}+ce^{-c'N^{\frac{4\varepsilon}{5}}\big(\frac{t}{l(k,N)}\big)^{1/5}}\big],
    \end{split}
\end{equation}
where the first inequality is due to the strong Markov property and the last inequality is due to Lemma~\ref{lem:LEMMA-8} \blue{(on taking a limit as $x \downarrow 0$)} and \eqref{eq:5.2-2}.
Hence, taking $x=\frac{t}{2l(k,N)}$ and using~\eqref{eq:5.3-local-1} and~\eqref{eq:5.3-local-2} in~\eqref{eq:5.3-local-0}, we have for  $t\geq 2(t_0 + 8\beta^{-1}B) l(k,N),$
\begin{equation}\label{eq:geq-t0l}
    \mathbb{P}_{(0, \, \lfloor 2BN^{\frac{1}{2}+\varepsilon} \rfloor, \,\underline{0})}\big(\sigma^{(N)}_{2\bar{K}^{(N)}}-\sigma^{(N)}_2>N^{2\varepsilon}t/2,1< \bar{K}^{(N)}\leq l(k,N)\big)\leq c_1l(k,N)e^{-c_2N^{(\frac{1}{2}-\varepsilon)/5}}e^{-c_3\big(\frac{t}{l(k,N)}\big)^{1/5}}.
\end{equation}
Also, for $t_0\leq t\leq 2(t_0 + 8\beta^{-1}B) l(k,N)$,
\begin{equation}\label{eq:leq-t0l}
\begin{split}
    &\mathbb{P}_{(0, \, \lfloor 2BN^{\frac{1}{2}+\varepsilon} \rfloor, \,\underline{0})}\big(\sigma^{(N)}_{2\bar{K}^{(N)}}-\sigma^{(N)}_2>N^{2\varepsilon}t/2,1< \bar{K}^{(N)}\leq l(k,N)\big)\\
     &\hspace{3cm}\leq  \mathbb{P}_{(0, \, \lfloor 2BN^{\frac{1}{2}+\varepsilon} \rfloor, \,\underline{0})}\big(\Bar{Z}^{(N)}_0>0\big)\leq c_1e^{-c_2N^{(\frac{1}{2}-\varepsilon)/5}},
\end{split}
\end{equation}
where the last inequality from Lemma~\ref{lem:LEMMA-8}.
Combining \eqref{eq:geq-t0l} and \eqref{eq:leq-t0l}, we have for all $t\geq t_0$,
\begin{equation}\label{eq:5.2-3}
    \mathbb{P}_{(0, \, \lfloor 2BN^{\frac{1}{2}+\varepsilon} \rfloor, \,\underline{0})}\big(\sigma^{(N)}_{2\bar{K}^{(N)}}-\sigma^{(N)}_2>N^{2\varepsilon}t/2,1< \bar{K}^{(N)}\leq l(k,N)\big)\leq c'_1l(k,N)e^{-c'_2N^{(\frac{1}{2}-\varepsilon)/5}}e^{-c'_3\big(\frac{t}{l(k,N)}\big)^{1/5}}.
\end{equation}
Finally, taking $k=t^{1/4}$ and $l(k,N)=1+t^{1/4}+\sqrt{t}N^{\frac{1}{2}-\varepsilon}$, we obtain the proposition by plugging \eqref{eq:5.2-2}, 
\eqref{eq:5.2-3}, and the bound on $\mathbb{P}_{(0, \, \lfloor 2BN^{\frac{1}{2}+\varepsilon} \rfloor, \,\underline{0})}\big(\bar{K}^{(N)}\geq l(k,N)\big)$ from Lemma \ref{lem:LEMMA-5.1} into \eqref{eq:5.2-1}.
\end{proof}

Now we come to the main result of this subsection which shows that the renewal time $\Theta^{(N)}$ has finite second moment and gives bounds for the first and second moments.

\begin{prop}\label{cor:MOEMNT-RENEWAL} 
There exist $ N_0,c,c'>0$ such that for all $N\geq N_0$,
\begin{align*}
    \mathbb{E}_{(0, \, \lfloor 2BN^{\frac{1}{2}+\varepsilon} \rfloor, \,\underline{0})}\big((\Theta^{(N)})^2\big)&\leq cN^{4\varepsilon},\\
    \mathbb{E}_{(0, \, \lfloor 2BN^{\frac{1}{2}+\varepsilon} \rfloor, \,\underline{0})}\big(\Theta^{(N)}\big)&\geq c'N^{2\varepsilon}.
\end{align*}
In particular, $c'N^{2\varepsilon} \le \mathbb{E}_{(0, \, \lfloor 2BN^{\frac{1}{2}+\varepsilon} \rfloor, \,\underline{0})}\big(\Theta^{(N)}\big) \le \sqrt{c}N^{2\varepsilon}$ for all $N \ge N_0$.
\end{prop}
\begin{proof}
Take $t_0$ as in Lemma~\ref{prop:RENEWAL-TIME}.
Then
\begin{equation*}
\begin{split}
    \mathbb{E}_{(0, \, \lfloor 2BN^{\frac{1}{2}+\varepsilon} \rfloor, \,\underline{0})}\big((\Theta^{(N)})^2/N^{4\varepsilon}\big)
    &\leq t_0^2+\int_{t_0}^{\infty}\mathbb{P}_{(0, \, \lfloor 2BN^{\frac{1}{2}+\varepsilon} \rfloor, \,\underline{0})}\big((\Theta^{(N)})^2/N^{4\varepsilon}>t\big)dt\\
    &\leq t_0^2+\int_{t_0}^{\infty}\mathbb{P}_{(0, \, \lfloor 2BN^{\frac{1}{2}+\varepsilon} \rfloor, \,\underline{0})}\big(\Theta^{(N)} >N^{2\varepsilon}\sqrt{t}\big)dt.
\end{split}
\end{equation*}
The upper bound on $\mathbb{E}_{(0, \, \lfloor 2BN^{\frac{1}{2}+\varepsilon} \rfloor, \,\underline{0})}\big((\Theta^{(N)})^2\big)$ now follows from 
Lemma~\ref{prop:RENEWAL-TIME} with replacing $t$ by $\sqrt{t}$ and integrating the RHS.
To obtain the lower bound on $\mathbb{E}_{(0, \, \lfloor 2BN^{\frac{1}{2}+\varepsilon} \rfloor, \,\underline{0})}\big(\Theta^{(N)}\big)$, recall $\tau_s^{(N)}$ from~\eqref{eq:tau-def} and note that 
\begin{align}\label{eq:5.3-1}
    \mathbb{E}_{(0, \, \lfloor 2BN^{\frac{1}{2}+\varepsilon} \rfloor, \,\underline{0})}\big(\Theta^{(N)} \big)&\geq \mathbb{E}_{(0, \, \lfloor 2BN^{\frac{1}{2}+\varepsilon} \rfloor, \,\underline{0})}\big(\tau^{(N)}_2(B)\big)\nonumber\\
    &\geq \mathbb{E}_{(0, \, \lfloor 2BN^{\frac{1}{2}+\varepsilon} \rfloor, \,\underline{0})}\big(\tau^{(N)}_2(B)\mathds{1}\big[\tau^{(N)}_2(B) < \tau^{(N)}_2(N^{\frac{1}{2}-\varepsilon}))\big]\big)\nonumber\\
    &= \mathbb{E}_{(0, \, \lfloor 2BN^{\frac{1}{2}+\varepsilon} \rfloor, \,\underline{0})}\big(\tau^{(N)}_2(B)\mathds{1}\big[\tau^{(N)}_2(B) < \tau^{(N)}_s(2N^{\frac{1}{2}-\varepsilon}))\big]\big)\nonumber\\
    &\geq \mathbb{E}_{(0, \, \lfloor 2BN^{\frac{1}{2}+\varepsilon} \rfloor, \,\underline{0})}\big(\tau^{(N)}_{s}(B+N^{\frac{1}{2}-\varepsilon})\mathds{1}\big[\tau^{(N)}_2(B) < \tau^{(N)}_s(2N^{\frac{1}{2}-\varepsilon}))\big]\big)\nonumber\\
    &=\mathbb{E}_{(0, \, \lfloor 2BN^{\frac{1}{2}+\varepsilon} \rfloor, \,\underline{0})}\big(\tau^{(N)}_s(B+N^{\frac{1}{2}-\varepsilon})\big)\nonumber\\
    &\ \ \ \ -\mathbb{E}_{(0, \, \lfloor 2BN^{\frac{1}{2}+\varepsilon} \rfloor, \,\underline{0})}\big(\tau_s^{(N)}(B+N^{\frac{1}{2}-\varepsilon})\mathds{1}\big[\tau_s^{(N)}(2N^{\frac{1}{2}-\varepsilon})<\tau_2^{(N)}(B)\big]\big).
\end{align}
Now, 
for any $t\geq 0$,
\begin{align*}
    &\mathbb{P}_{(0, \, \lfloor 2BN^{\frac{1}{2}+\varepsilon} \rfloor, \,\underline{0})}\big(\tau^{(N)}_s(B+N^{\frac{1}{2}-\varepsilon}) \ge N^{2\varepsilon}t\big)\\
    &= \mathbb{P}_{(0, \, \lfloor 2BN^{\frac{1}{2}+\varepsilon} \rfloor, \,\underline{0})}\big(\inf_{s\leq t}S^{(N)}(N^{2\varepsilon}s)\geq \lfloor N+ BN^{\frac{1}{2}+\varepsilon} \rfloor\big)\\
    &= \mathbb{P}_{(0, \, \lfloor 2BN^{\frac{1}{2}+\varepsilon} \rfloor, \,\underline{0})}\big(\inf_{s\leq t}\big(S^{(N)}(N^{2\varepsilon}s)-S^{(N)}(0)\big)\geq \lfloor N+ BN^{\frac{1}{2}+\varepsilon} \rfloor- \lfloor N+2 BN^{\frac{1}{2}+\varepsilon}\rfloor\big)\\
    &\ge \mathbb{P}_{(0, \, \lfloor 2BN^{\frac{1}{2}+\varepsilon} \rfloor, \,\underline{0})}\Big(\inf_{s\leq t}\big(A((N-\beta N^{\frac{1}{2}-\varepsilon})N^{2\varepsilon}s\big)-D\big(\int_0^{N^{2\varepsilon}s}(N-I^{(N)}(u))du\big)\big)\geq -B N^{\frac{1}{2}+\varepsilon} + 1\Big) 
\end{align*}
\begin{align*}
    &\geq  \mathbb{P}_{(0, \, \lfloor 2BN^{\frac{1}{2}+\varepsilon} \rfloor, \,\underline{0})}\big(\inf_{s\leq t}\big(\hat{A}\big((N-\beta N^{\frac{1}{2}-\varepsilon})N^{2\varepsilon}s\big)-\hat{D}\big(\int_0^{N^{2\varepsilon}s}(N-I^{(N)}(u))du\big)\\
    &\hspace{11cm}-\beta N^{\frac{1}{2}+\varepsilon}s\big)\geq -BN^{\frac{1}{2}+\varepsilon} + 1\big)\\
    &\ge  \mathbb{P}_{(0, \, \lfloor 2BN^{\frac{1}{2}+\varepsilon} \rfloor, \,\underline{0})}\big(\inf_{s\leq t}\big(\hat{A}\big((N-\beta N^{\frac{1}{2}-\varepsilon})N^{2\varepsilon}s\big)-\hat{D}\big(\int_0^{N^{2\varepsilon}s}(N-I^{(N)}(u))du\big)\\
    &\hspace{10cm}-\beta N^{\frac{1}{2}+\varepsilon}s\big)\geq (\beta t-B)N^{\frac{1}{2}+\varepsilon} + 1\big),
\end{align*}
where $\hat{A}(s)=A(s)-s$ and $\hat{D}(s)=D(s)-s$.
Using Lemma \ref{lem:sup-poi} in the above lower bound, there exist $\underline{t}>0, \underline{N}>0$ such that for all $t\leq \underline{t},N\geq \underline{N}$,
\begin{equation*}
    \mathbb{P}_{(0, \, \lfloor 2BN^{\frac{1}{2}+\varepsilon} \rfloor, \,\underline{0})}\br{\tau^{(N)}_s\br{B+N^{\frac{1}{2}-\varepsilon}}>N^{2\varepsilon}t}\geq \frac{1}{2},
\end{equation*}
and consequently, \begin{equation}\label{eq:5.3-2}
    \mathbb{E}_{(0, \, \lfloor 2BN^{\frac{1}{2}+\varepsilon} \rfloor, \,\underline{0})}\br{\tau^{(N)}_s\br{B+N^{\frac{1}{2}-\varepsilon}}}\geq \frac{1}{2}\underline{t}N^{2\varepsilon}.
\end{equation}
Next, we will show that the second term in the bound \eqref{eq:5.3-1} is much smaller than the first term. To show this, recall the stopping times $\{\sigma^{(N)}_{j}\}$ from \eqref{eq:sigma_i-def} and define
$
\hat{K}^{(N)}\coloneqq\inf \{k\geq 0:\Bar{Q}^{(N)}_3(\sigma^{(N)}_{2k+1})=0\}.
$
\blue{Lemma \ref{lem:LEMMA-5.1} and Lemma \ref{prop:RENEWAL-TIME} readily extend to $\hat{K}^{(N)}$ in place of $\bar{K}^{(N)}$ and $\sigma^{(N)}_{2\hat{K}^{(N)} + 1}$ in place of $\Theta^{(N)}$. To elaborate, the estimates in Lemma \ref{lem:LEMMA-8} can be used exactly as in the geometric trials based argument in the proof of Lemma \ref{lem:LEMMA-5.1} to obtain tail bounds for $\hat{K}^{(N)}$. Using these bounds along the same lines as in the proof of Lemma \ref{prop:RENEWAL-TIME}, one obtains tail bounds for $\sigma^{(N)}_{2\hat{K}^{(N)} + 1}$ similar to those for $\Theta^{(N)}$.}
Hence, using the same argument used to bound the second moment of $(\Theta^{(N)})^2/N^{4\epsilon}$, we obtain
$
\mathbb{E}_{(0, \, \lfloor 2BN^{\frac{1}{2}+\varepsilon} \rfloor, \,\underline{0})}\left[\left(\sigma_{2\hat{K}^{(N)} + 1}^{(N)}\right)^2\right] \le cN^{4\epsilon}.
$
Now, observe that 
$
    \tau^{(N)}_s\big( B+N^{\frac{1}{2}-\varepsilon}\big) \le \sigma^{(N)}_{2\hat{K}^{(N)} + 1}.
$
Using this observation along with Lemma~\ref{lem:S-hit-time},
\begin{align}\label{eq:5.3-3}
    &\mathbb{E}_{(0, \, \lfloor 2BN^{\frac{1}{2}+\varepsilon} \rfloor, \,\underline{0})}\br{\tau^{(N)}_s(B+N^{\frac{1}{2}-\varepsilon})\mathds{1}\big[\tau^{(N)}_s(2N^{\frac{1}{2}-\varepsilon})<\tau^{(N)}_2(B)\big]}\nonumber\\
    & \leq \sqrt{\mathbb{E}_{(0, \, \lfloor 2BN^{\frac{1}{2}+\varepsilon} \rfloor, \,\underline{0})}\left[\left(\sigma_{2\hat{K}^{(N)} + 1}^{(N)}\right)^2\right]}\sqrt{\mathbb{P}_{(0, \, \lfloor 2BN^{\frac{1}{2}+\varepsilon} \rfloor, \,\underline{0})}\br{\tau^{(N)}_s(2N^{\frac{1}{2}-\varepsilon})<\tau^{(N)}_2(B)}}\nonumber\\
    & \leq  cN^{2\varepsilon}\br{c\exp\br{-c'N^{4\varepsilon/5}N^{(\frac{1}{2}-\varepsilon)/5}}+c\exp\br{-c'N^{\frac{1}{2}-\varepsilon}}},
\end{align}
where the last inequality is from Lemma~\ref{lem:S-hit-time}.
Finally, the lower bound on $\mathbb{E}_{(0, \, \lfloor 2BN^{\frac{1}{2}+\varepsilon} \rfloor, \,\underline{0})}\big(\Theta^{(N)} \big)$ claimed in the corollary is established by plugging \eqref{eq:5.3-2} and \eqref{eq:5.3-3} into \eqref{eq:5.3-1}.
\end{proof}

\subsection{Renewal representation of stationary measure.}\label{renre} The finiteness of expectation of the renewal time $\Theta^{(N)}$ implies  the representation of the stationary measure as stated in~\eqref{eq:repre-stat}.
This representation is useful as it gives a tractable handle on a highly implicit stationary measure by expressing it in terms of the sample paths observed until a finite random (renewal) time. 
Upon combining this with the sample path analysis of Section \ref{sec:HITTIME}, we now prove Theorem~\ref{thm:TIGHTNESS-XN}.

\begin{proof}[Proof of Theorem~\ref{thm:TIGHTNESS-XN}.]
Recall the stopping time $\tau^{(N)}_s$ from~\eqref{eq:tau-def}.
Note that due to \eqref{eq:repre-stat},
\begin{align}\label{eq:5.4-1}
    \pi\br{\textcolor{black}{S^{(N)}(\infty)} \geq N+xN^{\frac{1}{2}+\varepsilon}}&\leq \frac{\mathbb{E}_{(0, \, \lfloor 2BN^{\frac{1}{2}+\varepsilon} \rfloor, \,\underline{0})}\br{\mathds{1}\br{\tau^{(N)}_s(x+N^{\frac{1}{2}-\varepsilon})<\Theta^{(N)} }\Theta^{(N)}}}{\mathbb{E}_{(0, \, \lfloor 2BN^{\frac{1}{2}+\varepsilon} \rfloor, \,\underline{0})}\br{\Theta^{(N)} }}\nonumber\\
    &\leq \frac{\sqrt{\mathbb{E}_{(0, \, \lfloor 2BN^{\frac{1}{2}+\varepsilon} \rfloor, \,\underline{0})}\br{(\Theta^{(N)})^2}}\sqrt{\mathbb{P}_{(0, \, \lfloor 2BN^{\frac{1}{2}+\varepsilon} \rfloor, \,\underline{0})}\br{\tau^{(N)}_s(x+N^{\frac{1}{2}-\varepsilon})<\Theta^{(N)} }}}{\mathbb{E}_{(0, \, \lfloor 2BN^{\frac{1}{2}+\varepsilon} \rfloor, \,\underline{0})}\br{\Theta^{(N)} }}.
\end{align}
The bounds on $\mathbb{E}_{(0, \, \lfloor 2BN^{\frac{1}{2}+\varepsilon} \rfloor, \,\underline{0})}\br{\Theta^{(N)} }$ and $\mathbb{E}_{(0, \, \lfloor 2BN^{\frac{1}{2}+\varepsilon} \rfloor, \,\underline{0})}\br{(\Theta^{(N)})^2}$ will be used from Proposition~\ref{cor:MOEMNT-RENEWAL}. In the following, we will derive upper bound on $\mathbb{P}_{(0, \, \lfloor 2BN^{\frac{1}{2}+\varepsilon} \rfloor, \,\underline{0})}\br{\tau^{(N)}_s(x+N^{\frac{1}{2}-\varepsilon})<\Theta^{(N)} }$.
\blue{This upper bound depends on the magnitude of $x$ and that is why we will consider three cases: (a)~$x\in [4B, N^{\frac{1}{2}-\varepsilon}]$,  (b)~$x \in (N^{\frac{1}{2} - \epsilon}, 2N^{\frac{1}{2} - \epsilon}]$, and (c)~$x\geq 2N^{\frac{1}{2}-\varepsilon}$.}

\noindent
\blue{\textbf{Case 1:}} $x\in [4B, N^{\frac{1}{2}-\varepsilon}]$.
Recall $t_0$ from Lemma~\ref{lem:S-hit-time}. For $x\in [t_0\beta + 2B, N^{\frac{1}{2}-\varepsilon}]$,
\begin{align}\label{eq:5.4-2}
    &\mathbb{P}_{(0, \, \lfloor 2BN^{\frac{1}{2}+\varepsilon} \rfloor, \,\underline{0})}\br{\tau^{(N)}_s(x+N^{\frac{1}{2}-\varepsilon})<\Theta^{(N)} }
    =\mathbb{P}_{(0, \, \lfloor 2BN^{\frac{1}{2}+\varepsilon} \rfloor, \,\underline{0})}\br{\tau^{(N)}_s(x+N^{\frac{1}{2}-\varepsilon})<\tau^{(N)}_2(B)}\nonumber\\
    &= \mathbb{P}_{(0, \, \lfloor 2BN^{\frac{1}{2}+\varepsilon} \rfloor, \,\underline{0})}\br{S^{(N)} \text{ hits } S^{(N)}(0) + (x-2B)N^{\frac{1}{2}+\varepsilon} \text{ before time }\tau^{(N)}_2(B)}\nonumber\\
    & \leq  c\exp\big\{-c'x\big\}+c\exp\big\{-c'N^{4\varepsilon/5}x^{1/5}\big\},
\end{align}
where
the second equality is due to the fact that $S^{(N)}(0) = N + 2BN^{\frac{1}{2}+\varepsilon}$,
the inequality is due to Lemma~\ref{lem:S-hit-time}, and $c'$ is a fixed positive constant that depends on $B$ but not on $N$ and $x$. 
The bound of \eqref{eq:5.4-2} can be extended to $x\in[4B,N^{\frac{1}{2}-\varepsilon}]$ by adjusting the constants. 

\noindent
\blue{\textbf{Case 2:}} For $x \in (N^{\frac{1}{2} - \epsilon}, 2N^{\frac{1}{2} - \epsilon}]$,
note that at $t=\tau^{(N)}_s(x+N^{\frac{1}{2}-\varepsilon})$, $S^{(N)}(t) > 2N$ and thus $\Bar{Q}^{(N)}_3(t)>0$.
Therefore, by Lemma \ref{lem:LEMMA-8} (i),
\begin{align}\label{middler}
 \mathbb{P}_{(0, \, \lfloor 2BN^{\frac{1}{2}+\varepsilon} \rfloor, \,\underline{0})}\br{\tau^{(N)}_s(x+N^{\frac{1}{2}-\varepsilon})<\Theta^{(N)} } &\le \mathbb{P}_{(0, \, \lfloor 2BN^{\frac{1}{2}+\varepsilon} \rfloor, \,\underline{0})}\bigg(\sup_{s \in [0,\tau^{(N)}_2(B)]}\big(\Bar{Q}^{(N)}_3(s)-\Bar{Q}^{(N)}_3(0)\big) >0\bigg)\nonumber\\
 &\le c_1 \exp\left\lbrace -c_2N^{(\frac{1}{2}-\varepsilon)/5}\right\rbrace \le c_1 \exp\left\lbrace - c_2' x^{1/5}\right\rbrace,
\end{align}
for positive constants $c_1, c_2, c_2'$ that can depend on $B$ but not on $N$ and $x$.

\noindent
\blue{\textbf{Case 3:} $x\geq 2N^{\frac{1}{2}-\varepsilon}$. Values of $S^{(N)}$ in the associated range result from $\Bar{Q}^{(N)}_3$ hitting large values (see \eqref{eq:local-case3-1}) and hence need a different treatment from those in the above cases.}

Recall from~\eqref{eq:sigma_i-def} the stopping times $\sigma^{(N)}_{2i}$ and $\sigma^{(N)}_{2i+1}$ for $i\geq 0$ and $\bar{K}^{(N)}=\inf \{k\geq 1:\Bar{Q}^{(N)}_3(\sigma^{(N)}_{2k})=0\}$.
In this case, write 
$
    \Bar{Z}^{(N)}_{*i}\coloneqq \sup_{s \in [\sigma^{(N)}_{2i},\sigma^{(N)}_{2i+2}]}\big(\Bar{Q}^{(N)}_3(s)-\Bar{Q}^{(N)}_3(\sigma^{(N)}_{2i})\big)_+, \, i\geq 0.
$
Take any $k \ge 1$ and recall $l(k,N)=1+k+k^2N^{\frac{1}{2}-\varepsilon}$. Now observe the following:
Since $x\geq 2N^{\frac{1}{2}-\varepsilon}$, at $t=\tau^{(N)}_s(x+N^{\frac{1}{2}-\varepsilon})$, 
$S^{(N)}(t) > N + xN^{\frac{1}{2}+\varepsilon}$ and thus 
\begin{equation}\label{eq:local-case3-1}
    \Bar{Q}^{(N)}_3(t)\geq S^{(N)}(t) -2N >xN^{\frac{1}{2}+\varepsilon} -N\geq \frac{xN^{\frac{1}{2}+\varepsilon}}{2}>0.
\end{equation}
Also, note that starting from $(0, \, \lfloor 2BN^{\frac{1}{2}+\varepsilon} \rfloor, \,\underline{0})$, if $\Bar{Z}^{(N)}_{*0} = 0$, then $\Theta^{(N)}$ must be smaller than $\tau^{(N)}_s(x+N^{\frac{1}{2}-\varepsilon})$, otherwise it will contradict~\eqref{eq:local-case3-1}.
From the above, we can write
\begin{align}\label{eq:5.4-3}
    &\mathbb{P}_{(0, \, \lfloor 2BN^{\frac{1}{2}+\varepsilon} \rfloor, \,\underline{0})}
    \br{\tau^{(N)}_s(x+N^{\frac{1}{2}-\varepsilon})<\Theta^{(N)} }\nonumber\\
    &\quad\leq \mathbb{P}_{(0, \, \lfloor 2BN^{\frac{1}{2}+\varepsilon} \rfloor, \,\underline{0})}\br{\bar{K}^{(N)}\geq l(k,N)}+\mathbb{P}_{(0, \, \lfloor 2BN^{\frac{1}{2}+\varepsilon} \rfloor, \,\underline{0})}\Big(\sup_{0\leq i\leq l(k,N)}\sup_{s\in[\sigma^{(N)}_{2i},\sigma^{(N)}_{2i+2}]} S^{(N)}(s) > N + xN^{\frac{1}{2}+\varepsilon}\Big)\nonumber\\
    &\quad\leq \mathbb{P}_{(0, \, \lfloor 2BN^{\frac{1}{2}+\varepsilon} \rfloor, \,\underline{0})}\br{\bar{K}^{(N)}\geq l(k,N)}\\
    &\qquad+\mathbb{P}_{(0, \, \lfloor 2BN^{\frac{1}{2}+\varepsilon} \rfloor, \,\underline{0})}\Big(\sup_{0\leq i\leq l(k,N)}\sup_{s\in[\sigma^{(N)}_{2i},\sigma^{(N)}_{2i+2}]}\Bar{Q}^{(N)}_3(s)\geq \frac{xN^{\frac{1}{2}+\varepsilon}}{2},\Bar{Z}^{(N)}_{*0}>0\Big),\nonumber
\end{align}
where the last inequality follows from~\eqref{eq:local-case3-1}.
By Lemma \ref{lem:LEMMA-5.1}, for sufficiently large $N$, for all $k\geq 1$,
\begin{equation}\label{eq:5.4-4}
    \mathbb{P}_{(0, \, \lfloor 2BN^{\frac{1}{2}+\varepsilon} \rfloor, \,\underline{0})}\br{\bar{K}^{(N)}\geq l(k,N)}\leq c_1 k e^{-c_2N^{(\frac{1}{2}-\varepsilon)/11}}e^{-c_2\br{\frac{k}{N^{\frac{1}{2}-\varepsilon}}}^{1/11}}+2^{-k}.
\end{equation}
Next, we will upper-bound the second term on the RHS of~\eqref{eq:5.4-3}.
Note that since $\Bar{Q}_3^{(N)}(0)=0$,
\begin{align*}
    \sup_{s\in [\sigma^{(N)}_{2i},\sigma^{(N)}_{2i+2}]}\Bar{Q}_3^{(N)}(s)
    &\le \sup_{s\in [\sigma^{(N)}_{2i},\sigma^{(N)}_{2i+2}]}\br{\Bar{Q}_3^{(N)}(s)-\Bar{Q}_3^{(N)}(\sigma^{(N)}_{2i})}+\sum_{j=0}^{i-1}\br{\Bar{Q}_3^{(N)}(\sigma^{(N)}_{2j+2})-\Bar{Q}_3^{(N)}(\sigma^{(N)}_{2j})}\\
    &\leq  \sum_{j=0}^{i}\sup_{s\in [\sigma^{(N)}_{2j},\sigma^{(N)}_{2j+2}]}\br{\Bar{Q}_3^{(N)}(s)-\Bar{Q}_3^{(N)}(\sigma^{(N)}_{2j})}\leq \sum_{j=0}^{i}\Bar{Z}^{(N)}_{*j}.
\end{align*}
Since  $\sum_{j=0}^{i}\Bar{Z}^{(N)}_{*j}$ is non-decreasing in $i$, we can write
\begin{align}\label{eq:5.4-5}
    &\mathbb{P}_{(0, \, \lfloor 2BN^{\frac{1}{2}+\varepsilon} \rfloor, \,\underline{0})}\Big(\sup_{0\leq i\leq l(k,N)}\sup_{s\in [\sigma^{(N)}_{2i},\sigma^{(N)}_{2i+2}]}\Bar{Q}_3^{(N)}(s)\geq \frac{xN^{\frac{1}{2}+\varepsilon}}{2},\Bar{Z}^{(N)}_{*0}>0\Big)\nonumber\\
    &\leq  \mathbb{P}_{(0, \, \lfloor 2BN^{\frac{1}{2}+\varepsilon} \rfloor, \, \underline{0})}\Big(\sum_{j=0}^{l(k,N)}\Bar{Z}^{(N)}_{*j}\geq \frac{xN^{\frac{1}{2}+\varepsilon}}{2}\Big)
    =\mathbb{P}_{(0, \, \lfloor 2BN^{\frac{1}{2}+\varepsilon} \rfloor, \, \underline{0})}\Big(\sum_{j=0}^{l(k,N)}\frac{\Bar{Z}^{(N)}_{*j}}{N^{\frac{1}{2}+\varepsilon}}\geq \frac{x}{2}\Big).
\end{align}
To upper bound the RHS of~\eqref{eq:5.4-5}, we will use Lemma~\ref{lem:LEMMA-5}.
For that, first we need to have appropriate tail bounds on $\frac{\Bar{Z}^{(N)}_{*j}}{N^{\frac{1}{2}+\varepsilon}}$, which is given by Lemma \ref{lem:LEMMA-8} (i) as follows:
Letting $k= k(x) = \lfloor \sqrt{x} \rfloor$, we have $l'(x,N) :=l(k(x),N)=1+\lfloor \sqrt{x} \rfloor +\lfloor \sqrt{x} \rfloor^2N^{\frac{1}{2}-\varepsilon}$.
Therefore, by Lemma \ref{lem:LEMMA-8} (i), for any $j\geq 1, x>0$, and sufficient large~$N$,
\begin{equation}
    \mathbb{P}\Big(\frac{\Bar{Z}^{(N)}_{*j}}{N^{\frac{1}{2}+\varepsilon}}\geq x\ \big|\ \mathcal{F}_{\sigma^{(N)}_{2j}}\Big)\leq c_1e^{-c_2N^{(\frac{1}{2}-\varepsilon)/5}}e^{-c_3x^{1/5}},
\end{equation}
where $c_1,c_2,c_3>0$ are constants. Also, there exists an $N''$ such that for all $N\geq N''$,
$\int_0^{\infty}c_1e^{-c_2N^{(\frac{1}{2}-\varepsilon)/5}}e^{-c_3y^{1/5}}dy\leq \frac{x}{4\times 2(l'(x,N)+1)}.$
Take any $N\geq N''$. 
We will use Lemma \ref{lem:LEMMA-5}  with trivial indexing set $R$ and
$\theta =\frac{1}{5}, a=\frac{x}{2(l'(x,N)+1)},n=l'(x,N) + 1,\ \text{and}\  \hat{\Phi}_j=\frac{\Bar{Z}^{(N)}_{*(j-1)}}{N^{\frac{1}{2}+\varepsilon}}, \, j \ge 1,$
with starting configuration $(I^{(N)}(0), Q_2^{(N)}(0), \bar{Q}^{(N)}_3(0)) = (0, \, \lfloor 2BN^{\frac{1}{2}+\varepsilon} \rfloor, \,\underline{0})$, and associated filtration $\big\{\mathcal{F}_j : j\geq 1\big\}$ being the natural filtration generated by the above random variables.
Thus, we get
\begin{align}\label{eq:5.4-6}
    & \mathbb{P}_{(0, \, \lfloor 2BN^{\frac{1}{2}+\varepsilon} \rfloor, \, \underline{0})}\Big(\sum_{j=0}^{l'(x,N)}\Bar{Z}^{(N)}_{*j}\geq \frac{xN^{\frac{1}{2}+\varepsilon}}{2}\Big)\nonumber\\
    &=\mathbb{P}\Big(\sum_{j=1}^{n}\hat{\Phi}_j\geq an\Big)
    \leq  c'_1\Big(1+\frac{\big(l'(x,N) + 1\big)^{5/11}}{\big(\frac{x}{ 2(l'(x,N)+1)}\big)^{\frac{1}{11}}}\Big)\exp \Big[-c'_2\Big(\frac{x^2}{4(l'(x,N)+1)}\Big)^{1/11}\Big]\\\nonumber
    &\leq \hat{c}_1'N^{\frac{1}{2}-\varepsilon}\exp\big\{-\hat{c}'_2\Big(\frac{x}{N^{\frac{1}{2}-\varepsilon}}\Big)^{1/11}\big\}.\nonumber
\end{align}
Moreover, for all large enough $N$, for any $k\geq 1$,
\begin{align}\label{eq:sum-zpr}
     \mathbb{P}_{(0, \, \lfloor 2BN^{\frac{1}{2}+\varepsilon} \rfloor, \,\underline{0})}\Big(\sum_{j=0}^{l'(x,N)}\Bar{Z}^{(N)}_{*j}\geq \frac{xN^{\frac{1}{2}+\varepsilon}}{2}\Big)
     &\leq (l'(x,N) + 1) \sup_{\underline{z}}\mathbb{P}_{(0, \, \lfloor 2BN^{\frac{1}{2}+\varepsilon} \rfloor, \,\underline{z})}\Big(\Bar{Z}^{(N)}_{*0}\geq 1 \Big)\notag\\
    &\leq  c_1 xN^{\frac{1}{2}-\varepsilon}\exp \big\{-c_2N^{(\frac{1}{2}-\varepsilon)/5}\big\},
\end{align}
where the first inequality follows from union bound and the fact that $l'(x,N) + 1\leq xN^{\frac{1}{2}+\varepsilon}/2$ for all $N$ large enough, and the last inequality follows from Lemma \ref{lem:LEMMA-8} (i).
Using \eqref{eq:5.4-6} and \eqref{eq:sum-zpr} and proceeding exactly as in deriving \eqref{eq:sum-z-bar-plus} from \eqref{nex0}, we obtain
\begin{equation}\label{nex2}
 \mathbb{P}_{(0, \, \lfloor 2BN^{\frac{1}{2}+\varepsilon} \rfloor, \,\underline{0})}\Big(\sum_{j=0}^{l'(x,N)}\Bar{Z}^{(N)}_{*j}\geq \frac{xN^{\frac{1}{2}+\varepsilon}}{2}\Big) \le c_1''\exp \big\{-c_2''N^{(\frac{1}{2}-\varepsilon)/11}\big\}\exp\big\{-c''_2\Big(\frac{x}{N^{\frac{1}{2}-\varepsilon}}\Big)^{1/11}\big\},
\end{equation}
for positive constants $c_1'', c_2''$.
Plugging \eqref{nex2} into \eqref{eq:5.4-5}, we obtain $N'''$ large enough such that for all $N\geq N''', \, x\geq 2N^{\frac{1}{2}-\varepsilon}$,
\begin{multline}\label{eq:5.4-8}
    \mathbb{P}_{(0, \, \lfloor 2BN^{\frac{1}{2}+\varepsilon} \rfloor, \,\underline{0})}\Big(\sup_{0\leq i\leq l(k,N)}\sup_{s\in [\sigma^{(N)}_{2i},\sigma^{(N)}_{2i+2}]}\Bar{Q}_3^{(N)}(s)\geq \frac{xN^{\frac{1}{2}+\varepsilon}}{2},\Bar{Z}^{(N)}_{*0}>0\Big)\\
   \leq c_1''\exp \big\{-c_2''N^{(\frac{1}{2}-\varepsilon)/11}\big\}\exp\big\{-c''_2\Big(\frac{x}{N^{\frac{1}{2}-\varepsilon}}\Big)^{1/11}\big\}.
\end{multline}
Moreover, plugging \eqref{eq:5.4-8} and \eqref{eq:5.4-4} (with $k = \lfloor \sqrt{x} \rfloor$) into \eqref{eq:5.4-3}, we have that for sufficiently large $N$, for $x\geq 2N^{\frac{1}{2}-\varepsilon}$,
\begin{align}\label{eq:5.4-9}
    &\mathbb{P}_{(0, \, \lfloor 2BN^{\frac{1}{2}+\varepsilon} \rfloor, \,\underline{0})}\br{\tau^{(N)}_s(x+N^{\frac{1}{2}-\varepsilon})<\Theta^{(N)}}\nonumber\\
    \leq & \bar{c}\exp \Big\{-\bar{c}'\Big(N^{(\frac{1}{2}-\varepsilon)/11}+\Big(\frac{\sqrt{x}}{N^{\frac{1}{2}-\varepsilon}}\Big)^{1/11}\Big)\Big\}+\bar{c}e^{-\bar{c}'\sqrt{x}}\nonumber\\
    \leq & \bar{c}\exp \big\{-2\bar{c}'x^{1/44}\big\}+\bar{c}e^{-\bar{c}'\sqrt{x}}.
\end{align}
Equations \eqref{eq:5.4-2}, \eqref{middler} and \eqref{eq:5.4-9} imply that there exist $N_0 \in \mathbb{N}$ and positive constants $C'_1, C'_2$ such that for all $N\geq N_0$,
\begin{equation}\label{eq:5.4-10}
    \mathbb{P}_{(0, \, \lfloor 2BN^{\frac{1}{2}+\varepsilon} \rfloor, \,\underline{0})}\br{\tau^{(N)}_s(x+N^{\frac{1}{2}-\varepsilon})<\Theta^{(N)} } \le \begin{cases} C'_1\exp\big\{-C'_2x^{1/5}\big\},\quad 4B\leq x\leq 2N^{\frac{1}{2}-\varepsilon},\\
    C'_1\exp \big\{-C'_2x^{1/44}\big\},\quad x\geq 2N^{\frac{1}{2}-\varepsilon}.
    \end{cases}
\end{equation}
Thus, the tail estimate on $X^{(N)}(\infty)$ in the theorem follows upon using \eqref{eq:5.4-10} and Proposition~\ref{cor:MOEMNT-RENEWAL} in \eqref{eq:5.4-1}.

(ii) \blue{Next, we will use the $M/M/N$ system to lower bound a JSQ system. 
Specifically, consider an $M/M/N$ queue with total arrival rate $N\lambda_N=N-\beta N^{\frac{1}{2}-\varepsilon}$ and unit-mean exponential service times, and denote by $\underline{S}^{(N)}(t)$ the total number of tasks in the system at time $t$. Let $\underline{X}^{(N)}(t):=\frac{\underline{S}^{(N)}(t)-N}{N^{\frac{1}{2}+\varepsilon}}$. By a natural coupling, if $X^{(N)}(0)=\underline{X}^{(N)}(0)$, $X^{(N)}(t)\geq \underline{X}^{(N)}(t)$, $\forall\ t\geq 0$, which implies that $X^{(N)}(\infty)$ stochastically dominates $\underline{X}^{(N)}(\infty)$. Hence, it is sufficient to show that 
\begin{equation}\label{eq:mmn-geq-0}
\lim_{\delta\downarrow0}\varliminf_{N\rightarrow\infty}\PP\big(\underline{X}^{(N)}(\infty)\ge \delta\big)= 1.
\end{equation}
Denote $\pi^{(N)}_k=\PP\big(\underline{X}^{(N)}(\infty)=k\big)$, $k\in\N_0$. In the following, $c_i(N), \, i \in \mathbb{N},$ are positive constants depending only on $N$ (not $k$) such that $c_i(N)\xrightarrow{N\rightarrow\infty}1, \, i \in \mathbb{N}$. By the explicit formula of the stationary distribution of an $M/M/N$ queue (\cite[Section~3.5]{KL75}), we have that 
\begin{equation}\label{eq:pi-0}
    \pi^{(N)}_0=\Big[\sum_{k=0}^{N-1}\frac{N^k(1-\beta/N^{\frac{1}{2}+\varepsilon})^k}{k!}+\frac{N^N(1-\beta/N^{\frac{1}{2}+\varepsilon)^N}}{N!}\frac{N^{\frac{1}{2}+\varepsilon}}{\beta}\Big]^{-1},
\end{equation}
and for all $k\geq N$,
\begin{equation}\label{pi-k}
    \pi^{(N)}_k=\pi^{(N)}_0\frac{(N\lambda_N)^kN^{N-k}}{N!}=\pi^{(N)}_0\frac{(1-\beta/N^{\frac{1}{2}+\varepsilon})^kN^{N}}{N!}.
\end{equation}
By Stirling's formula (\cite[Section~27]{pb12}), we get that for all large enough $N$,
\begin{equation}\label{eq:low-bound-pi-0}
    \begin{split}
        \pi^{(N)}_0&\geq \Big[e^{N}+c_0(N)\frac{e^NN^{\varepsilon}}{\sqrt{2\pi}\beta}(1-\beta/N^{\frac{1}{2}+\varepsilon})^N\Big]^{-1}\\
        &\geq c_1(N)\sqrt{2\pi}\beta N^{-\varepsilon}e^{-N}(1-\beta/N^{\frac{1}{2}+\varepsilon})^{-N},
    \end{split}
\end{equation}
and for all $k\geq N$,
\begin{equation}\label{eq:low-bound-pi-k}
    \begin{split}
        \pi^{(N)}_k&\geq  c_2(N)\beta N^{-\frac{1}{2}-\varepsilon}\big(1-\frac{\beta}{N^{\frac{1}{2}+\varepsilon}}\big)^{k-N}\\
        &\geq c_2(N)\beta N^{-\frac{1}{2}-\varepsilon}\exp\big(-\frac{(k-N)\beta}{N^{\frac{1}{2}+\varepsilon}}c_3(N)\big).
    \end{split}
\end{equation}
For any fixed $\delta>0$, denote $N_{\delta}=N+\delta N^{\frac{1}{2}+\varepsilon}$. By \eqref{eq:low-bound-pi-k}, we have that for large enough $N$,
\begin{equation}\label{eq:prob-N-delta}
    \begin{split}
        \PP\big(\underline{X}^{(N)}(\infty)\ge \delta\big)&=\sum_{k\geq \lceil N_{\delta} \rceil}c_2(N)\beta N^{-\frac{1}{2}-\varepsilon}\exp\big(-\frac{(k-N)\beta}{N^{\frac{1}{2}+\varepsilon}}c_3(N)\big)\\
        &\geq c_2(N)\int_{\beta\delta c_4(N)}^{\infty}\exp(-xc_3(N))dx\\
        &=(c_2(N)/c_3(N))\exp(-\beta\delta c_3(N)c_4(N))\\
        &\geq 1-2\beta\delta,
    \end{split}
\end{equation}
where the first inequality comes from approximating the sum by an integral and, for the last inequality, we have used $\exp(-x)\geq (1-x)$, $\forall x \in \R$.
Since \eqref{eq:prob-N-delta} holds for all fixed $\delta>0$, we conclude that \eqref{eq:mmn-geq-0} holds.}\\

(iii) To obtain the expression for expectation of $I^{(N)}(\infty)$, observe that, if we start the process at stationarity, then by \eqref{eq:represent-XN}, for any $t > 0$,
\begin{align*}
0&= N^{\frac{1}{2} + \varepsilon}\E[X^{(N)}(t) - X^{(N)}(0)] =  \E\left[A(N^{1+2\varepsilon}\lambda_Nt) - D\Big(\int_0^{N^{2\varepsilon}t}(N-I^{(N)}(s))ds\Big)\right]\\
&= -\beta N^{\frac{1}{2}+\varepsilon} t + \int_0^{N^{2\varepsilon}t}\E[I^{(N)}(s)]ds = -\beta N^{\frac{1}{2}+\varepsilon} t + N^{2\varepsilon}t \E[I^{(N)}(\infty)],
\end{align*}
from which we conclude $\E\left[I^{(N)}(\infty)\right] = \beta N^{\frac{1}{2}-\varepsilon}$.\\

(iv) The observation $\sup_{N \ge 1}\E\left[N^{-\frac{1}{2}-\varepsilon}Q^{(N)}_2(\infty)\right] < \infty$ follows upon noting that $N^{-\frac{1}{2}-\varepsilon}Q^{(N)}_2(\infty) \le X^{(N)}(\infty) + N^{-\frac{1}{2} - \varepsilon}I^{(N)}(\infty)$ and using the tail estimate on $X^{(N)}(\infty)$ in~\eqref{eq:thm-2.3-tail} and the above explicit expectation for $I^{(N)}(\infty)$.\\

(v) Now we show that 
\begin{equation}\label{qthreezero}
\bar{Q}^{(N)}_3(\infty) \xrightarrow{P} 0 \ \mbox{ as } \ N \rightarrow \infty.
\end{equation}
By the renewal representation \eqref{eq:repre-stat}, note that for any fixed $T>0$,
\begin{align*}
&\pi(\bar{Q}^{(N)}_3(\infty)>0) =\frac{\mathbb{E}_{(0, \, \lfloor 2BN^{\frac{1}{2}+\varepsilon} \rfloor, \,\underline{0})}\br{\int_0^{\Theta^{(N)}}\mathds{1}\br{\bar{Q}^{(N)}_3(s)>0}ds}}{\mathbb{E}_{(0, \, \lfloor 2BN^{\frac{1}{2}+\varepsilon} \rfloor, \,\underline{0})}\br{\Theta^{(N)}}}\\
&\le \frac{\mathbb{E}_{(0, \, \lfloor 2BN^{\frac{1}{2}+\varepsilon} \rfloor, \,\underline{0})}\br{\Theta^{(N)} \mathds{1}(\Theta^{(N)} \ge N^{2\varepsilon}T)}}{\mathbb{E}_{(0, \, \lfloor 2BN^{\frac{1}{2}+\varepsilon} \rfloor, \,\underline{0})}\br{\Theta^{(N)}}}
+ \frac{\mathbb{E}_{(0, \, \lfloor 2BN^{\frac{1}{2}+\varepsilon} \rfloor, \,\underline{0})}\br{\int_0^{N^{2\varepsilon}T}\mathds{1}\br{\bar{Q}^{(N)}_3(s)>0}ds}}{\mathbb{E}_{(0, \, \lfloor 2BN^{\frac{1}{2}+\varepsilon} \rfloor, \,\underline{0})}\br{\Theta^{(N)}}}\\
&\le \frac{\mathbb{E}_{(0, \, \lfloor 2BN^{\frac{1}{2}+\varepsilon} \rfloor, \,\underline{0})}\br{\blue{(\Theta^{(N)})^2}}}{N^{2\varepsilon}T \ \mathbb{E}_{(0, \, \lfloor 2BN^{\frac{1}{2}+\varepsilon} \rfloor, \,\underline{0})}\br{\Theta^{(N)}}}
+ \frac{N^{2\varepsilon}T \, \mathbb{P}_{(0, \, \lfloor 2BN^{\frac{1}{2}+\varepsilon} \rfloor, \,\underline{0})}\br{\sup_{s \le N^{2\varepsilon}T}\bar{Q}^{(N)}_3(s)>0}}{\mathbb{E}_{(0, \, \lfloor 2BN^{\frac{1}{2}+\varepsilon} \rfloor, \,\underline{0})}\br{\Theta^{(N)}}}\\
&\le \frac{\mathbb{E}_{(0, \, \lfloor 2BN^{\frac{1}{2}+\varepsilon} \rfloor, \,\underline{0})}\br{\blue{(\Theta^{(N)})^2}}}{N^{2\varepsilon}T \ \mathbb{E}_{(0, \, \lfloor 2BN^{\frac{1}{2}+\varepsilon} \rfloor, \,\underline{0})}\br{\Theta^{(N)}}}
+ \frac{N^{2\varepsilon}T \, \mathbb{P}_{(0, \, \lfloor 2BN^{\frac{1}{2}+\varepsilon} \rfloor, \,\underline{0})}\br{\sup_{s \le N^{2\varepsilon}T}Q^{(N)}_2(s)\ge N}}{\mathbb{E}_{(0, \, \lfloor 2BN^{\frac{1}{2}+\varepsilon} \rfloor, \,\underline{0})}\br{\Theta^{(N)}}}.
\end{align*}
Now, applying Proposition \ref{cor:MOEMNT-RENEWAL} for the first term and Proposition \ref{prop:bdd-Q2} along with Proposition \ref{cor:MOEMNT-RENEWAL} for the second term above, we conclude that there exists a finite constant $c_\circ>0$ independent of $N,T$ such that
$$
\limsup_{N \rightarrow \infty} \ \pi(\bar{Q}^{(N)}_3(\infty)>0) \le \frac{c_\circ}{T}.
$$
As $T>0$ is arbitrary, \eqref{qthreezero} follows.

\end{proof}

\section{Proof of process-level limit.}
\label{sec:PROCESS-LEVEL}
In this section, we will analyze the process-level limit of the scaled occupancy process, and in particular, prove Theorem~\ref{thm:PROCESS-LEVEL-NEW}. \blue{Throughout this section, we consider fixed $\varepsilon \in (0, 1/2)$.
First, we claim that it is enough to establish the following `deterministic initial state' version of 
Theorem~\ref{thm:PROCESS-LEVEL-NEW}.}
\begin{theorem}\label{thm:PROCESS-LEVEL}
Fix $\beta>0$ and $\varepsilon \in (0, \frac{1}{2})$ \blue{ and assume that the initial states are deterministic. Suppose there exists a subsequence $(\hat{N}) \subset \mathbb{N}$
such that, as $\hat{N}\rightarrow\infty$, $X^{(\hat{N})}(0)\to x_0$  where $x_0\in (0,\infty)$ is a constant, 
$Q_3^{(\hat{N})}(0)\to 0$, and 
$\limsup_{\hat{N}\rightarrow\infty} \hat{N}^{-(\frac{1}{2}-\varepsilon)}I^{(\hat{N})}(0)<\infty$.}
Then, for any fixed $T>0$, the scaled process $X^{(\hat{N})}$ converges weakly as $\hat{N} \rightarrow \infty$ to the path-wise unique solution of the following stochastic differential equation, uniformly on $[0,T]$:
\begin{equation}
    dX(t)=\Big(\frac{1}{X(t)}-\beta\Big)dt+\sqrt{2}dW(t),
\end{equation}
with $X(0)= x_0$,
where $W=\big(W(t),t\geq 0\big)$ is the standard Brownian motion.
\end{theorem}
Assuming Theorem~\ref{thm:PROCESS-LEVEL}, we now prove Theorem~\ref{thm:PROCESS-LEVEL-NEW} below.
\blue{\begin{proof}[Proof of Theorem~\ref{thm:PROCESS-LEVEL-NEW}.]
Note that, under the assumptions of Theorem~\ref{thm:PROCESS-LEVEL-NEW}, for any subsequence of $(N)$, there exists a further subsequence $(\hat{N})$, such that 
$$\big(X^{(\hat{N})}(0), \hat{N}^{-\frac{1}{2}+\varepsilon}I^{(\hat{N})}(0), \hat{N}^{-\frac{1}{2}-\varepsilon}Q^{(\hat{N})}_2(0), Q^{(\hat{N})}_3(0)\big) \dto \big(X^*, I^*, Q^*_2, 0\big) \text{ as } \hat{N}\rightarrow\infty$$ 
where $X^*$ is as in Theorem~\ref{thm:PROCESS-LEVEL-NEW} and $I^*$, $Q^*_2$ are finite random variables.
Therefore, by the Skorohod representation theorem, there exists a probability space
$(\Omega, \mathcal{F},\mathbb{P})$ such that this convergence happens almost surely. 
Let $\hat{\Omega}\in \mathcal{F}$ be a subset of $\Omega$ with $\mathbb{P}(\hat{\Omega}) = 1$ such that for all $\omega\in \hat{\Omega}$, 
\begin{equation}\label{eq:initial}
\big(X^{(\hat{N})}(0, \omega), \hat{N}^{-\frac{1}{2}+\varepsilon} I^{(\hat{N})}(0, \omega), \hat{N}^{-\frac{1}{2}-\varepsilon} Q^{(\hat{N})}_2(0, \omega), Q^{(\hat{N})}_3(0, \omega)\big) \to \big(X^*(\omega), I^*(\omega), Q^*_2(\omega), 0\big)    
\end{equation}
as $\hat{N}\rightarrow\infty$
and $X^*(\omega), I^*(\omega), Q^*_2(\omega)$ are finite.
Also, note that the primitive stochastic processes driving the dynamics of the system are independent of the initial state.
Thus, append the probability space $(\Omega, \mathcal{F},\mathbb{P})$ with countably many independent Poisson processes of unit intensity and countably many $\operatorname{Uniform}(0,1)$ random variables, independent of each other and the Poisson processes. 
Note that these independent Poisson processes and $\operatorname{Uniform}(0,1)$ random variables can be used to construct the primitive processes that govern the arrivals and departures, and hence determines the evolution of $\{Q_i^{(N)}(t): t\ge 0,\ i=1,2,\ldots\}$ (e.g., see the random time change representation in \cite[Section 4.2]{EG15} and note that the independent Uniform random variables above can be used to appropriately thin the Poisson processes). 
Then, we can construct the process $\{X^{(\hat{N})}(t,\omega): t \ge 0\}$ on this space for each $\omega \in \hat{\Omega}$ and $\hat{N}$ such that, for any $\omega \in \hat{\Omega}$, the sequence of initial states (which now form a deterministic sequence) converge as in~\eqref{eq:initial}, as $\hat{N} \rightarrow \infty$.
Therefore, in the above probability space, the (sub-)sequence of initial states satisfies the conditions stated in Theorem~\ref{thm:PROCESS-LEVEL} for each fixed $\omega \in \hat{\Omega}$. 
Hence, applying Theorem~\ref{thm:PROCESS-LEVEL}, 
we deduce that for any $\omega \in \hat{\Omega}$ and any fixed $T>0$, $X^{(\hat{N})}$ converges weakly to the path-wise unique solution $X$ of \eqref{langevin}, with $X(0)=X^*(\omega)$, uniformly on $[0,T]$ as $\hat{N} \rightarrow \infty$. To reemphasize, fixing $\omega \in \hat{\Omega}$ specifies only the starting configuration of the process and, conditional on this choice, the stated convergence is in the sense of traditional weak convergence uniformly on compact time intervals (e.g., \cite[Chapter 3]{billingsley2013convergence}).
Finally, since $\mathbb{P}(\hat{\Omega}) = 1$, we thus conclude that $X^{(N)}$ converges weakly to the solution of \eqref{langevin} uniformly on $[0,T]$ as $N \rightarrow \infty$.\\
\end{proof}
}

The rest of the section is devoted to the proof of Theorem~\ref{thm:PROCESS-LEVEL}. \blue{For notational convenience, we will write $(N)$ for the subsequence $(\hat{N})$ in the theorem.
Also, observe from the initial conditions that there exist constants $N_0\geq 1$ and $K_1>0$, such that $I^{(N)}(0)\leq K_1N^{\frac{1}{2}-\varepsilon}$, and $Q^{(N)}_3(0)=0$ (and hence, $\sum_{i\geq 3}Q^{(N)}_i = 0$) for all $N\geq N_0$. 
Furthermore, since $Q^{(N)}_2= N^{\frac{1}{2}+\varepsilon}X^{(N)} + I^{(N)} - \sum_{i\geq 3}Q^{(N)}_i$, we also have $K_2>0$ such that 
$Q^{(N)}_2(0)\leq K_2 N^{\frac{1}{2}+\varepsilon}$ for all $N\geq N_0$.
Henceforth, we will assume that $N$ is large enough to satisfy these bounds.}

The main ingredient is to analyze the idle-server process $I^{(N)}$.
We start with the martingale representation of the process $X^{(N)}$.
\paragraph{Martingale representation.}
Recall the representation \eqref{eq:represent-XN} of $X^{(N)}(t)$, the arrival process $A(N^{1+2\varepsilon}\lambda_Nt)$, and the departure process  $D\br{\int_0^{N^{2\varepsilon}t}(N-I^{(N)}(s))ds}$.
Let us introduce the related filtrations $\mathbf{F}=\big\{\mathcal{F}_{N,t}: N\in\mathbb{N}, t \in [0, \infty]\big\}$ where
\begin{equation}\label{eq:filtration}
    \mathcal{F}_{N,t}:=\sigma\Big(S^{(N)}(0), A(N^{1+2\varepsilon}\lambda_N s), D\Big(\int_0^{N^{2\varepsilon}s}(N-I^{(N)}(u))du\Big), \, 0\leq s\leq t\Big),\quad t\geq 0,
\end{equation}
and $\mathcal{F}_{N,\infty}:=\sigma(\cup_{t\geq 0}\mathcal{F}_{N,t})$.
We write
\begin{align}
    X^{(N)}(t)-X^{(N)}(0)&= N^{-\frac{1}{2}-\varepsilon}\left[A(N^{1+2\varepsilon}\lambda_Nt) - D\Big(\int_0^{N^{2\varepsilon}t}(N-I^{(N)}(s))ds\Big)\right]\nonumber\\
    &=\mathcal{M}_a^{(N)}(\lambda_N t)-\mathcal{M}_d^{(N)}\Big(t-\frac{1}{N^{1+2\varepsilon}} \int_0^{N^{2\varepsilon}t} I^{(N)}(s)ds\Big) \label{eqthm1}\\
    &\quad+\frac{1}{N^{\frac{1}{2}+\varepsilon}} \int_0^{N^{2\varepsilon}t} I^{(N)}(s)ds-\int_0^t\frac{1}{X^{(N)}(s)}ds\label{eqthm2}\\ 
    &\quad-\beta t+\int_0^t\frac{1}{X^{(N)}(s)}ds\label{eqthm3}
\end{align}
where 
$\mathcal{M}_a^{(N)}(t)=\frac{A(N^{1+2\varepsilon}t)-N^{1+2\varepsilon}t}{N^{\frac{1}{2}+\varepsilon}},\quad \mathcal{M}_d^{(N)}(t)=\frac{D(N^{1+2\varepsilon}t)-N^{1+2\varepsilon}t}{N^{\frac{1}{2}+\varepsilon}}.$
Note that $\mathcal{M}_a^{(N)}(t)$ and \blue{$\mathcal{M}_d^{(N)}\Big(t-\frac{1}{N^{1+2\varepsilon}} \int_0^{N^{2\varepsilon}t} I^{(N)}(s)ds\Big)$} are martingales adapted to the filtration $\mathbf{F}$.
We will proceed by first showing in Proposition~\ref{prop:IDLE-INTEGRAL} that the 
integral in~\eqref{eqthm1} converges to 0 uniformly on any (scaled) finite time interval.
Using this, we will be able to show that the difference of martingales in~\eqref{eqthm1} convergence weakly to $\sqrt{2}W$ as $N\to\infty$, where $W$ is the standard Brownian motion.
The next major challenge is to show that 
the difference of the two terms in~\eqref{eqthm2} converges to 0 as $N\to\infty$.
This is achieved in Proposition~\ref{prop:INT-IDLE-2}.
Finally, a continuous mapping theorem-type result will complete the proof of Theorem~\ref{thm:PROCESS-LEVEL}.

In its core, the analysis of $I^{(N)}$ will be done by upper and lower bounding it with suitable birth-and-death processes. 
Now, for any fixed $B>0$, recall the stopping time $\tau^{(N)}_2(B)$ from \eqref{eq:tau-def} and the process $\bar{I}^{(N)}_B$, and note that
if $Q^{(N)}_2(0)>BN^{\frac{1}{2}+\varepsilon}$, then for all $t\leq \tau^{(N)}_2(B)$, $I^{(N)}(t)$ can be stochastically upper bounded by $\bar{I}^{(N)}_B(t)$ with $\bar{I}^{(N)}_B(0) = I^{(N)}(0)$.
As before, throughout we assume $N$ to be large enough so that $N> BN^{\frac{1}{2}+\varepsilon}>\beta N^{\frac{1}{2}-\varepsilon}$. We emphasize that, unlike in Section~\ref{sec:HITTIME}, we will be interested in \emph{small} values of $B$ for the process level limit and thus cannot directly apply the estimates in Section~\ref{sec:HITTIME} which deals with \emph{large} values of $B$.
\begin{lemma}\label{lem:MM1-IDLE}
There exist $N_0, a, b > 0$ depending only on $T$, $B$, and $\beta$, such that 
for all $N\geq N_0$ and $\delta>0$,
$$\mathbb{P}\Big(\sup_{0\leq t\leq N^{2\varepsilon}T}\bar{I}^{(N)}_B(t)\geq \frac{5}{2} N^{\frac{1}{2}-\varepsilon+\delta}\ \big|\ \bar{I}^{(N)}_B(0)=0\Big)\leq a e^{-bN^{\frac{\delta}{2}}}.$$
\end{lemma}
\begin{proof}
The proof follows in three steps: first, we upper bound the tail probability of the stationary distribution of $\bar{I}^{(N)}_B$. Next, we upper bound the tail probability of $\sup_{0\leq t\leq N^{2\varepsilon }T}\bar{I}^{(N)}_B(t)$ when $\bar{I}^{(N)}_B(0)$ is a random variable having the same distribution as the steady state of $\bar{I}^{(N)}_B$, and finally, we consider $\sup_{0\leq t\leq N^{2\varepsilon }T}\bar{I}^{(N)}_B(t)$ when $\bar{I}^{(N)}_B(0)=0$.

\begin{claim}\label{claim:6.2}
Let $\bar{I}^{(N)}_B(\infty)$ denote a random variable having stationary distribution of $\bar{I}^{(N)}_B$. 
Then there exist constants $N_0'$, $a_1$ and $b_1$, that only depend on $B$ and $\beta$, such that for all $N\geq N_0'$,  
$$\mathbb{P}\big(\bar{I}^{(N)}_B(\infty)\geq N^{\frac{1}{2}-\varepsilon+\delta}\big)\leq a_1 e^{-b_1N^{\delta}}.$$ 
\end{claim}
\emph{Proof.}
 Note that the stationary distribution of $\bar{I}^{(N)}_B$ is given by $\mathbb{P}(\bar{I}^{(N)}_B(t)=k)=(1-\rho)\rho^{k}$ for $k\geq 0$, where $\rho=\frac{N-B N^{\frac{1}{2}+\varepsilon}}{N-\beta N^{\frac{1}{2}-\varepsilon}}$.
 Therefore, there exists $N_0'>0$\blue{, independent of $\delta$, } such that for all $N\geq N_0'$, \blue{$B/2\ge \beta N^{-2\varepsilon}$, and }
 \blue{\begin{equation}\label{eq:Idle-B-maximal}
     \begin{split}
\mathbb{P}\big(\bar{I}^{(N)}_B(\infty)\geq N^{\frac{1}{2}-\varepsilon+\delta}\big)
 &=\Big(\frac{N-B N^{\frac{1}{2}+\varepsilon}}{N-\beta N^{\frac{1}{2}-\varepsilon}}\Big)^{N^{\frac{1}{2}-\varepsilon+\delta}}
 =\Big(1-\frac{B N^{\frac{1}{2}+\varepsilon}-\beta N^{\frac{1}{2}-\varepsilon}}{N-\beta N^{\frac{1}{2}-\varepsilon}}\Big)^{N^{\frac{1}{2}-\varepsilon+\delta}}\\
 &\leq \Big(1-\frac{B }{2N^{\frac{1}{2}-\varepsilon}}\Big)^{N^{\frac{1}{2}-\varepsilon+\delta}}\leq a_1 e^{-b_1N^{\delta}},
     \end{split}
 \end{equation}}
 for appropriate constants $a_1$ and $b_1$ that depend only on $B$ and $\beta$\blue{, where the last inequality is due to the observation that $\exists n_0>0$ s.t. $\forall n\geq n_0$, $(1-1/n)^n\leq 2e^{-1}$, derived from $\lim_{n\rightarrow\infty}(1-1/n)^n=e^{-1}$}.
This completes the proof of Claim~\ref{claim:6.2}.

\begin{claim}\label{claim:stat-proc}
Assume that $\{\bar{I}^{(N)}_B(t),t\geq 0\}$ is an equilibrium process. Then there exist positive constants $N_1, a_2,$ and $b_2$ which only depend on $T$, $B$ and $\beta$ such that 
for all $N\geq N_1$,
$$\mathbb{P}\big(\sup_{0\leq t\leq N^{2\varepsilon}T}\bar{I}^{(N)}_B(t)\geq \frac{5}{2} N^{\frac{1}{2}-\varepsilon+\delta}\big)\leq a_2 e^{-b_2N^{\frac{\delta}{2}}}.$$
\end{claim}
\emph{Proof.}
Let $j=\lceil\frac{N^{2\varepsilon}T}{N^{-\varepsilon-\frac{1}{2}+\delta}}\rceil$ and consider the times $t_i=iN^{-\varepsilon-\frac{1}{2}+\delta}$, $i=0,1,...,j-1$. 
Denote the number of positive increments in $\bar{I}^{(N)}_B$ in a subinterval $[t_{i},t_{i+1})$ by $\zeta^{(N)}_i$.
Then $\zeta^{(N)}_i$
has a Poisson distribution with parameter $N^{\frac{1}{2}-\varepsilon+\delta}-B N^{\delta}$. 
For any Poisson random variable $\mathrm{Po}(\lambda)$ with parameter $\lambda$, we have (see \cite[Theorem 2.3(b)]{MC98}) that for $0\leq \xi\leq 1$,
$\mathbb{P}\br{\mathrm{Po}(\lambda)-\lambda\geq \xi \lambda}\leq e^{-\frac{3}{8}\xi^2\lambda}.$
Hence, for all $i=0,1,...,j-1$,
$\mathbb{P}\big(\zeta^{(N)}_i\geq \frac{3}{2}N^{\frac{1}{2}-\varepsilon+\delta}\big)\leq e^{-\frac{3}{32}N^{\frac{1}{2}-\varepsilon+\delta}}.$
Take $N_0'$ as in Claim~\ref{claim:6.2}.
Since we are considering the equilibrium process, due to Claim~\ref{claim:6.2}, we know for all $N\geq N_0'$ and $t\geq 0$, $\mathbb{P}\big(\bar{I}^{(N)}_B(t)\geq N^{\frac{1}{2}-\varepsilon+\delta}\big)\leq a_1 e^{-b_1N^{\delta}}.$
Now, note that
$$\Big\{\sup_{t_i\leq t\leq t_{i+1}}\bar{I}^{(N)}_B(t)\geq \frac{5}{2}N^{\frac{1}{2}-\varepsilon+\delta} \Big\}\subseteq \Big\{\bar{I}^{(N)}_B(t_i)\geq N^{\frac{1}{2}-\varepsilon+\delta}\Big\}\cup\Big\{\zeta^{(N)}_i\geq \frac{3}{2}N^{\frac{1}{2}-\varepsilon+\delta}\Big\},$$ 
for $i=0,2,...,j-1$, and
we have,
\begin{align*}
    \mathbb{P}\big(\sup_{0\leq t\leq N^{2\varepsilon}T}\bar{I}^{(N)}_B(t)\geq \frac{5}{2}N^{\frac{1}{2}-\varepsilon+\delta}\big)
    &\leq  \sum_{i=0}^{j-1}\mathbb{P}\big(\bar{I}^{(N)}_B(t_i)\geq N^{\frac{1}{2}-\varepsilon+\delta}\big)+\sum_{i=0}^{j-1}\mathbb{P}\big(\zeta^{(N)}_i\geq \frac{3}{2}N^{\frac{1}{2}-\varepsilon+\delta}\big)\\
    &\leq  \big(N^{\frac{1}{2}+3\varepsilon-\delta}T +1\big)\big(a_1 e^{-b_1N^{\delta}}+e^{-\frac{3}{32}N^{\frac{1}{2}-\varepsilon+\delta}}\big).
\end{align*}
Thus, there exist constants $a_2, b_2>0$, and $N_1\geq N_0'$, which depend only on $B$, $\beta$, and $T$, such that for all $N\geq N_1$,
$\mathbb{P}\big(\sup_{0\leq t\leq N^{2\varepsilon}T}\bar{I}^{(N)}_B(t)\geq \frac{5}{2}N^{\frac{1}{2}-\varepsilon+\delta}\big) \leq a_2 e^{-b_2N^{\frac{\delta}{2}}}.$ 
This completes the proof of Claim~\ref{claim:stat-proc}.

\noindent
Finally, the proof for $\bar{I}^{(N)}_B(0)=0$ follows from Claim~\ref{claim:stat-proc} by observing that for any $k\geq 1$,
\begin{equation*}
    \mathbb{P}\big(\sup_{0\leq t\leq N^{2\varepsilon}T}\bar{I}^{(N)}_B(t)\geq \frac{5}{2} N^{\frac{1}{2}-\varepsilon+\delta}\ \big|\ \bar{I}^{(N)}_B(0)=0\big)\leq \mathbb{P}\big(\sup_{0\leq t\leq N^{2\varepsilon}T}\bar{I}^{(N)}_B(t)\geq \frac{5}{2} N^{\frac{1}{2}-\varepsilon+\delta}\ \big|\ \bar{I}^{(N)}_B(0)=k\big).
\end{equation*}
This completes the proof of the lemma.
\end{proof}

\begin{prop}\label{prop:IDLE-INTEGRAL}
Fix any $T>0$ and $0<\delta<\frac{1}{2}+\varepsilon$, and take $N_0$ as in Lemma \ref{lem:MM1-IDLE}.
For any $K_1>0$, there exist constants $N_1,a, b>0$ that depend only on $B$, $\beta$, and $T$, such that
for all $N\geq N_1, x\leq K_1N^{\frac{1}{2}-\varepsilon}, y\geq B N^{\frac{1}{2}+\varepsilon}$,
\begin{equation}\label{eq:idle-sup}
   \sup_{\underline{z}} \mathbb{P}_{(x, y, \underline{z})}\Big(\sup_{0\leq t\leq (N^{2\varepsilon}T)\wedge\tau^{(N)}_2(B)}I^{(N)}(t)\geq 5N^{\frac{1}{2}-\varepsilon+\delta}\Big)\leq ae^{-bN^{\frac{\delta}{2}}},
\end{equation}
and consequently, for all $ \Tilde{\varepsilon}>0$,
\begin{equation}\label{eq:idle-integral}
    \lim_{N\rightarrow\infty}
    \sup_{\underline{z}} \mathbb{P}_{(x, y, \underline{z})}\Big(\frac{1}{N^{1+2\varepsilon}}\int_0^{(N^{2\varepsilon}T)\wedge\tau^{(N)}_2(B)} I^{(N)}(s)ds\geq \tilde{\varepsilon}\Big)=0.
\end{equation}
\end{prop}
\begin{proof}
Recall $N_0$ as in Lemma \ref{lem:MM1-IDLE}.
Take $N_1 \ge N_0$ such that $\frac{5}{2}N^{\frac{1}{2}-\varepsilon+\delta}\geq K_1 N^{\frac{1}{2}-\varepsilon}$ and consider $N\geq N_1$. Note that for any $x\leq K_1N^{\frac{1}{2}-\varepsilon}$, $y\geq B N^{\frac{1}{2}+\varepsilon}$ and $\underline{z} \in \mathbb{N}_0^{\infty}$ with $z_1 \ge z_2 \ge \dots$,
\begin{align*}
    &\mathbb{P}_{(x, y, \underline{z})}\Big(\sup_{0\leq t\leq (N^{2\varepsilon}T)\wedge\tau^{(N)}_2(B)}I^{(N)}(t)\geq 5N^{\frac{1}{2}-\varepsilon+\delta}\Big)\\
    &\leq \mathbb{P}\Big(\sup_{0\leq t\leq N^{2\varepsilon}T}\left(K_1N^{\frac{1}{2}-\varepsilon} + \bar{I}^{(N)}_B(t)\right)\geq 5N^{\frac{1}{2}-\varepsilon+\delta}\ \big|\ \bar{I}^{(N)}_B(0)=0\Big)\\
   &\leq \mathbb{P}\Big(\sup_{0\leq t\leq N^{2\varepsilon}T}\bar{I}^{(N)}_B(t)\geq \frac{5}{2}N^{\frac{1}{2}-\varepsilon+\delta}\ \big|\ \bar{I}^{(N)}_B(0)=0\Big)
    \leq ae^{-bN^{\frac{\delta}{2}}}.
\end{align*}
The first inequality follows from the fact that for $t \le \tau^{(N)}_2(B)$, the process $I^{(N)}(\cdot)$ starting from $x \le K_1N^{\frac{1}{2}-\varepsilon}$ is stochastically dominated by $K_1N^{\frac{1}{2}-\varepsilon} + \bar I^{(N)}_B(\cdot)$ with $\bar I^{(N)}_B(0)=0$. The last inequality follows from Lemma \ref{lem:MM1-IDLE}.
Next, for all $N\geq N_1$, $x\leq K_1N^{\frac{1}{2}-\varepsilon}$, $y\geq B N^{\frac{1}{2}+\varepsilon}$, and feasible $\underline{z}$,
\begin{align*}
    &\mathbb{P}_{(x, y, \underline{z})}\Big(\frac{1}{N^{1+2\varepsilon}}\int_0^{(N^{2\varepsilon}T)\wedge\tau^{(N)}_2(B)} I^{(N)}(s)ds\geq 5N^{-\frac{1}{2}-\varepsilon+\delta}T\Big)\\
    &\hspace{3cm}\quad\leq \mathbb{P}_{(x, y, \underline{z})}\Big(\sup_{0\leq s\leq (N^{2\varepsilon}T)\wedge\tau^{(N)}_2(B)}I^{(N)}(s)\geq 5N^{\frac{1}{2}-\varepsilon+\delta}\Big)\leq ae^{-bN^{\frac{\delta}{2}}},
\end{align*}
and thus,~\eqref{eq:idle-integral} holds for any $0 < \delta < \frac{1}{2}+\varepsilon$. 
\end{proof}

\begin{prop}\label{prop:INT-IDLE-2}
For any fixed $B>0$, recall $\tau^{(N)}_2(B)$ from \eqref{eq:tau-def}.  Under the assumptions on the initial conditions stated in Theorem~\ref{thm:PROCESS-LEVEL}, the following holds as $N\to\infty$:
$$\sup_{0\leq t\leq T\wedge(N^{-2\varepsilon}\tau^{(N)}_2(B))}\Big|\frac{1}{N^{\frac{1}{2}+\varepsilon}}\int_0^{N^{2\varepsilon}t}I^{(N)}(s)ds-\int_0^t \frac{1}{X^{(N)}(s)}ds\Big|\pto 0.$$
\end{prop}
The proof of Proposition~\ref{prop:INT-IDLE-2} is given in Section~\ref{sec:proof-prop-5.2}.

\begin{lemma}\label{lem:SDE-SOL}
The stochastic differential equation
\begin{equation}\label{eq:SDE-SOL}
    dX(t)=\Big(\frac{1}{X}-\beta\Big)dt +\sqrt{2}dW(t),
\end{equation}
with $X(0)>0$,
has a path-wise unique strong solution. Also, if $\tau_{\varepsilon}\coloneqq \inf\{t>0:X(t) \le \varepsilon\}$, $\varepsilon>0$, and $\tau :=\lim_{\varepsilon\rightarrow0}\tau_{\varepsilon}$, then $\tau=\infty$ almost surely.
\end{lemma}
\begin{proof}
For $\varepsilon>0$, the process $X(t)$, for $t\leq \tau_{\varepsilon}$ satisfies an SDE with Lipschitz coefficients. 
Such SDE are known to have path-wise unique strong solutions (see \cite[Theorem V.11.2]{RW00}).

Next, to show that $\tau=\infty$ almost surely, consider the SDE 
$d\hat{X}(t)=\frac{1}{\hat{X}(t)} +\sqrt{2}dW(t),$
and define the analogous quantities $\hat{\tau}_{\varepsilon}$, $\varepsilon>0$, and $\hat {\tau}$ for $\hat{X}$.
Note that $\frac{\hat{X}(t)}{\sqrt{2}} $ is a Bessel process of dimension 2. 
By Girsanov's Theorem, we can add and remove the drift $\beta t$ to $\hat{X}$ with an exponential change of measure. Hence, the law of $X$ and that of $\hat{X}$ are mutually absolutely continuous on compact time intervals. 
For $n\geq 2$, the $n$-dimensional Bessel process  is transient from its starting point with probability one, i.e., the $n$-dimensional Bessel process will be greater than 0 for all $t>0$ almost surely~\cite[\S 3.3.C]{KS14}.
Thus, for any $a>0$, $\mathbb{P}(\tau \le a) = \mathbb{P}(\hat{\tau} \le a)=0$.
Therefore, $\tau=\infty$ almost surely.
\end{proof}

\begin{proof}
We will proceed as in the proof of \cite[Theorem 1]{GW19}, using Proposition \ref{prop:INT-IDLE-2} stated above in place of their Proposition EC.3. 
Recall the martingale representation of $X^{(N)}$ in~\eqref{eqthm1}--\eqref{eqthm3}.
First consider (\ref{eqthm1}). 
By the assumption on $\lambda_N$ and Proposition \ref{prop:upper-bound-I}, we have that, as $N\rightarrow\infty$,
$
    \lambda_N t\rightarrow t\text{ and } t-\frac{1}{N^{1+2\varepsilon}} \int_0^{N^{2\varepsilon}t} I^{(N)}(s)ds\pto t,
$
uniformly on the interval $[0,T]$.
By the Martingale FCLT \cite{whitt07} and the independence of $\mathcal{M}_a$ and $\mathcal{M}_d$, we have that as $N\rightarrow\infty$,
\begin{align*}
    \Big\{\mathcal{M}_a^{(N)}(\lambda_N t)-\mathcal{M}_d^{(N)}\Big( t-\frac{1}{N^{1+2\varepsilon}} \int_0^{N^{2\varepsilon}t} I^{(N)}(s)ds\Big): t\geq0\Big\}\Rightarrow
    \Big\{\sqrt{2}W(t): t\geq0\Big\},
\end{align*} 
where $W$ is a standard Brownian motion.
Next, we will consider \eqref{eqthm2}. 
For any fixed $B>0$, define $\hat{\tau}^{(N)}(B)\coloneqq\inf\{t\geq0:X^{(N)}(t)\leq B\}$ and $\hat{\tau}(B)\coloneqq\inf\{t\geq 0: X(t)\leq B\}$, where $X(t)$ is the unique strong solution to the S.D.E.~\eqref{eq:SDE-SOL} with initial value $X(0)>0$.
The claim below establishes a relation between $\hat{\tau}^{(N)}$ and $\tau^{(N)}_2$.
\begin{claim}\label{claim:tau-tau2}
For any fixed $B>0$,
$$\lim_{N\rightarrow\infty}\PP\br{N^{2\varepsilon}\hat{\tau}^{(N)}(B)\wedge (N^{2\varepsilon}T)\leq \tau^{(N)}_2(B)\wedge (N^{2\varepsilon}T)}=1.$$ 
\end{claim}
\begin{claimproof}[Proof.]
First, note that on the event $\big\{\tau^{(N)}_2(B)\geq  N^{2\varepsilon}T\big\}$, trivially,
$N^{2\varepsilon}\hat{\tau}^{(N)}(B)\wedge (N^{2\varepsilon}T)\leq \tau^{(N)}_2(B)\wedge (N^{2\varepsilon}T).$
Now, on the event $\big\{\tau^{(N)}_2(B)< N^{2\varepsilon}T\big\}\cap
\big\{\sup_{0\leq t\leq N^{2\varepsilon}T}Q^{(N)}_3(t)=0\big\}$, we have that for $0\leq t\leq N^{2\varepsilon}T$,
$S^{(N)}(t)-N=Q^{(N)}_2(t)-I^{(N)}(t)\leq Q^{(N)}_2(t),$
and thus,
$S^{(N)}(\tau^{(N)}_2(B))-N\leq Q^{(N)}_2(\tau^{(N)}_2(B))= \lfloor BN^{\frac{1}{2}+ \varepsilon}\rfloor,$
which implies that $N^{2\varepsilon}\hat{\tau}^{(N)}(B)\leq \tau^{(N)}_2(B)$.
Hence, we have 
\begin{align}\label{eq:claim-6.7}
\PP\br{N^{2\varepsilon}\hat{\tau}^{(N)}(B)\wedge (N^{2\varepsilon}T)\leq \tau^{(N)}_2(B)\wedge (N^{2\varepsilon}T)}\geq \PP\Big(\sup_{0\leq t\leq N^{2\varepsilon}T}Q^{(N)}_3(t)=0\Big).
\end{align}
By Proposition~\ref{prop:bdd-Q2}, the right hand side of \eqref{eq:claim-6.7} tends to 1 as $N\rightarrow\infty$. This completes the proof of Claim~\ref{claim:tau-tau2}.
\end{claimproof}
Using Claim~\ref{claim:tau-tau2} and Proposition \ref{prop:INT-IDLE-2}, we can conclude
\begin{equation*}
    \sup_{0\leq t\leq \hat{\tau}^{(N)}(B)\wedge T}  \left|\frac{1}{N^{\frac{1}{2}+\varepsilon}} \int_0^{N^{2\varepsilon}t} I^{(N)}(s)ds-\int_0^t\frac{1}{X^{(N)}(s)}ds\right|\pto0\quad \text{as}\quad N\rightarrow\infty.
\end{equation*}
Therefore, defining 
$
    \delta^{(N)}(t):=\frac{1}{N^{\frac{1}{2}+\varepsilon}} \int_0^{N^{2\varepsilon}t} I^{(N)}(s)ds-\int_0^{t}\frac{1}{X^{(N)}(s)}ds, \ t \ge 0,
$
(when the integrals are well defined) we have that for any fixed $B>0$, the process
$\big( \delta^{(N)}(t\wedge \hat{\tau}^{(N)}(B)):t\geq 0\big)$
converges weakly to a process that is identically equal to 0, as $N\rightarrow\infty$.
Also, recall that $X^{(N)}(0)\pto X(0)$ where $X(0)$ is a positive constant.
Thus, by the Skorohod representation theorem, there exists a probability space 
$(\Omega, \mathcal{F},\mathbb{P})$ such that, almost surely, the following convergence holds
\small\begin{align}
    &\Big(X^{(N)}(0),\Big\{\mathcal{M}_a^{(N)}(\lambda_N t)-\mathcal{M}_d^{(N)}\Big(t-\frac{1}{N^{1+2\varepsilon}} \int_0^{N^{2\varepsilon}t} I^{(N)}(s)ds\Big),\notag\\
    &\hspace{7cm}\delta^{(N)}(t\wedge \hat{\tau}^{(N)}(K^{-1})\wedge \hat{\tau}(K^{-1})\big) :t\in [0, T], K \in \mathbb{N}\Big\} \Big)\nonumber \\
    &\xrightarrow{u.o.c.}\Big(X(0),\big\{\sqrt{2}W(t),0\big):t\in [0, T], K \in \mathbb{N}\big\}\Big)\quad \text{as}\quad N\rightarrow\infty, \label{eq:AS-CONV}
\end{align}
\normalsize
where `u.o.c.' denotes convergence of the associated processes uniformly on compact subsets of $[0,T]$, and the above random variables are seen as $\mathbb{R}_+ \times \left(\mathcal{D}\left([0,T] : \mathbb{R}^2\right)\right)^{\mathbb{N}}$ valued random variables.
For $K \in \mathbb{N}$, define the event
$
\Omega_K := \{\inf_{t \in [0,T]}X(t) \ge K^{-1}\}.
$
By Lemma \ref{lem:SDE-SOL}, $\lim_{K \rightarrow \infty}\mathbb{P}\left(\Omega_K\right) = 1$. Let
$
b^{(N)}(t) :=\mathcal{M}_a^{(N)}(\lambda_N t) -\mathcal{M}_d^{(N)}\Big(t-\frac{1}{N^{1+2\varepsilon}} \int_0^{N^{2\varepsilon}t} I^{(N)}(s)ds\Big) +\delta^{(N)}(t)(\omega)
$
and $b(t) :=\sqrt{2}W(t)$.
Define the following subsets of $\Omega$: 
$$
\mathcal{S}^{(N)}_{\varepsilon, K} := \left\{ \left|X^{(N)}(0)-X(0)\right|+\sup_{0\leq s\leq T \wedge \hat{\tau}^{(N)}(K^{-1})\wedge \hat{\tau}(K^{-1}) }\left|b^{(N)}(s) -b(s)\right| < \varepsilon\right\}, \ \varepsilon>0, \, K \in \mathbb{N}.
$$
By \eqref{eq:AS-CONV}, for any $\varepsilon>0, \, K \in \mathbb{N}$,
$
\lim_{N \rightarrow \infty}\mathbb{P}\left(\mathcal{S}^{(N)}_{\varepsilon, K} \right) = 1.
$

Using the triangle inequality, we have that for all $ t\in[0,T\wedge\hat{\tau}^{(N)}(B) \wedge \hat{\tau}(B)]$, 
\begin{equation*}
     \left|X^{(N)}(t)-X(t)\right|\leq \left|X^{(N)}(0)-X(0)\right|+\left|b^{(N)}(t)-b(t)\right|+\int_0^t\left|\frac{1}{X^{(N)}(s)}-\frac{1}{X(s)}\right|ds.
\end{equation*}
Observe that, for any $B >0$, the map $x \mapsto x^{-1}$ is Lipschitz on $[B, \infty)$ with Lipschitz constant $B^{-2}$. Thus, for $t\in[0,T\wedge\hat{\tau}^{(N)}(K^{-1}/2) \wedge \hat{\tau}(K^{-1}/2)]$,
\begin{align*}
    &\sup_{0\leq s\leq t}\left|X^{(N)}(s) -X(s)\right|\\
    \leq & \left|X^{(N)}(0)-X(0)\right|+\sup_{0\leq s\leq t}\left|b^{(N)}(s)-b(s)\right|+ 4K^2\int_0^t\sup_{0\leq s\leq \mu}\left|X^{(N)}(s)-X(s))\right|d\mu.
\end{align*}
Applying Gronwall's inequality (see \cite[Lemma 4.1]{pang07}), we have on the set $\mathcal{S}^{(N)}_{\varepsilon, 2K}$,
$$
    \sup_{0\leq t\leq T\wedge\hat{\tau}^{(N)}(K^{-1}/2) \wedge \hat{\tau}(K^{-1}/2)} |X^{(N)}(t)-X(t)|\leq \varepsilon e^{4K^2T}.
$$
Set $\varepsilon = \varepsilon_K = \frac{X(0)}{2} \wedge \frac{1}{4K} e^{-4K^2T}$. Let $K_0:= \lceil 4/X(0)\rceil$. Take any $K \ge K_0$. The above bound implies that, on the event $\mathcal{S}^{(N)}_{\varepsilon_K, 2K}$, $X^{(N)}(s) > \frac{1}{2K}$ for all $s \in[0,T\wedge\hat{\tau}^{(N)}(K^{-1}/2) \wedge \hat{\tau}(K^{-1}/2)]$. Moreover, on $\Omega_K$, $\hat{\tau}(K^{-1}/2) \ge \hat{\tau}(K^{-1}) \ge T$. Hence, on $\Omega_K \cap\mathcal{S}^{(N)}_{\varepsilon_K, 2K}$, $\hat{\tau}^{(N)}(K^{-1}/2)  \wedge \hat{\tau}(K^{-1}/2)  > T$. Thus, the above bound gives on the event $\Omega_K \cap\mathcal{S}^{(N)}_{\varepsilon_K, 2K}$,
$
 \sup_{0\leq t\leq T} |X^{(N)}(t)-X(t)|\leq \frac{1}{4K}.
$
Therefore, for any $K \ge K_0$,
$
\limsup_{N \rightarrow \infty}\mathbb{P}\left( \sup_{0\leq t\leq T} |X^{(N)}(t)-X(t)| > \frac{1}{4K}\right) \le \mathbb{P}\left(\Omega_K^c\right) + \limsup_{N \rightarrow \infty}\mathbb{P}\left(\left(\mathcal{S}^{(N)}_{\varepsilon_K, 2K}\right)^c\right) = \mathbb{P}\left(\Omega_K^c\right).
$
On recalling $\lim_{K \rightarrow \infty}\mathbb{P}\left(\Omega_K\right) = 1$, we obtain
$
\sup_{0\leq t\leq T} |X^{(N)}(t)-X(t)| \pto 0 \ \text{ as } \ N \rightarrow \infty,
$
proving the theorem.
\end{proof}

\begin{proof}[Proof of Proposition~\ref{prop:STEADY-STATE}.]
Write the SDE as
\begin{align}\label{eq:stat-1}
        dX(t)&=\Big(\frac{1}{X}-\beta \Big)dt+\sqrt{2}dW(t)
        =\sqrt{2}dW(t)-V'(X)dt,
\end{align}
where $V(X)$ is a function with derivative 
$
    V'(X)=-\frac{1}{X}+\beta.
$
The diffusion \eqref{eq:stat-1} is a Langevin diffusion and, for any $B>0$, it has an invariant measure with density given by 
$\frac{d\Tilde{\pi}}{dx}=\exp\{-V(x)\},$ where
$
    V(x)=\int_{B}^{x}\Big(\beta-\frac{1}{u}\Big)du
    =\beta(x-B)+[\ln{B}-\ln{x}]
    =\beta(x-B)+\ln{\frac{B}{x}}, \ x >0.
$
Therefore, we have an invariant distribution $\pi(x)$ with density
$
    \frac{d\pi}{dx}=C\frac{x}{B}e^{-\beta (x-B)},
$
where $C$ satisfies that $\frac{1}{C}=\int_0^{\infty} \frac{x}{B}e^{-\beta (x-B)}dx=\frac{e^{\beta B}}{\beta^2 B}$. The computation of moments of $\pi$ is routine. This completes the proof.
\end{proof}

\section{Proof of Proposition \ref{prop:INT-IDLE-2}.}
\label{sec:proof-prop-5.2}
For convenience, denote $T(N,B)\coloneqq \tau^{(N)}_2(B)\wedge (N^{2\varepsilon}T) $.
The key idea to estimate the integral of $I^{(N)}$ is to consider its excursions in the interval $[0,  T(N,B)]$ and estimate the integral of  $I^{(N)}$ within each excursion.
For that, define the following stopping times: $\xi^{(N)}_0=0$ and for $i\geq 0$,
\begin{equation}\label{eq:xi-i-def}
    \begin{split}
        \xi^{(N)}_{2i+1}&=\inf\big\{t\geq\xi^{(N)}_{2i}:I^{(N)}(t)=0\big\},\\
    \xi^{(N)}_{2i+2}&=\inf\big\{t\geq\xi^{(N)}_{2i+1}:I^{(N)}(t)>0\big\}.
    \end{split}
\end{equation}
Thus, $(\xi_1^{(N)}, \xi_3^{(N)}, \xi_5^{(N)},\ldots)$ constitute successive excursions of $I^{(N)}$, and $\xi_{2i+2}^{(N)} - \xi_{2i+1}^{(N)}$, for $i\geq 0$ are the intervals of time $I^{(N)}$ spends at $0$.
Also, let $i^*_N:=\min \big\{i\geq 0: \xi^{(N)}_{2i+3}\geq T(N,B) \big\}$ be the number of excursions in $[0,  T(N,B)]$.

Recall that $I^{(N)}$ has instantaneous transition rates at time $t$ as follows:
\begin{equation}\label{eq:HAT-I}
\begin{split}
    &I^{(N)}(t)\nearrow I^{(N)}(t)+1 \text{  at rate  }Q_1^{(N)}(t)-Q_2^{(N)}(t),\\
    &I^{(N)}(t)\searrow (I^{(N)}(t)-1)_+ \text{  at rate  }N-\beta N^{\frac{1}{2}-\varepsilon}.
\end{split}
\end{equation}
As before, fix any $B>0$ and we will consider $N$ to be large enough so that $N>BN^{\frac{1}{2}+\varepsilon}>\beta N^{\frac{1}{2}-\varepsilon}$.
The main challenge in estimating the integral $I^{(N)}$ is that the above rates are state-dependent.
That is why, our first step is to approximate, for each excursion $[\xi^{(N)}_{2i+1},\xi^{(N)}_{2i+3})$, the term $Q_1^{(N)}(t)-Q_2^{(N)}(t)$ by $Q_1^{(N)}(\xi^{(N)}_{2i+1})-Q_2^{(N)}(\xi^{(N)}_{2i+1})$, which is static within that excursion.
The approximation is justified in two steps: First, by showing that 
the length of each excursion of $I^{(N)}$ is short with high probability uniformly in $[0,  T(N,B)]$ (Lemma~\ref{excursion time}) and second, by showing that the fluctuation of $S^{(N)}$ within each of these excursions is also uniformly small (Proposition~\ref{prop:change-st}).
Finally, Lemma~\ref{M/M/1 behavior} and Lemma~\ref{approx. rate} show that the error in using the above approximation for the integral of $I^{(N)}$ is asymptotically negligible.

In the following, the set $\{0,-1\}$ is assumed to be the null set.
\begin{lemma}\label{excursion time}
For any $0<\delta<\frac{1}{2}-\varepsilon$, there exist constants $a_1, b_1, N_1>0$, such that for all $N\geq N_1$,
\begin{equation}
\sup_{\substack{x\leq K_1N^{\frac{1}{2} - \varepsilon},\\ BN^{\frac{1}{2} + \varepsilon} <y\leq K_2N^{\frac{1}{2} + \varepsilon}}}\mathbb{P}_{(x,y, \underline{0})}\Big(\exists\ i \in\big\{0,...,i^*_N-1\big\}\text{  such that  } \xi^{(N)}_{2i+3}-\xi^{(N)}_{2i+1}\geq 2N^{\delta-2\varepsilon}\Big)\leq a_1e^{-b_1N^{\frac{\delta}{5}}}\label{idle excursion time}.
\end{equation}
\end{lemma}
The proof is given in Appendix~\ref{appendix:prop3.2} and we provide the high-level proof idea here. 
We consider the excursion of $I^{(N)}$ under two mutually exclusive events: $\sup_{t \in [0, T(N,B)]} Q^{(N)}_2(t)\geq N^{\frac{1}{2}+\varepsilon+\delta}$ or not. By Proposition~\ref{prop:bdd-Q2}, we have $\PP(\sup_{t \in [0, T(N,B)]}Q^{(N)}_2(t)\geq N^{\frac{1}{2}+\varepsilon+\delta})\xrightarrow{N\rightarrow\infty}0$. 
Conditional on the event $\{\sup_{t \in [0, T(N,B)]} Q^{(N)}_2(t) < N^{\frac{1}{2}+\varepsilon+\delta}\}$, we can bound the excursion time of $I^{(N)}$ by coupling it with an appropriate $M/M/1$ queue. 

As mentioned above, the next  proposition states that the fluctuation of $S^{(N)}$ within each of these excursions is also uniformly small.
\begin{prop}\label{prop:change-st}
For any $0<\delta<\frac{1}{2}-\varepsilon$, there exist constants $N_2, a_2, b_2>0$, such that for all $N\geq N_2$, 
$$\sup_{\substack{x\leq K_1N^{\frac{1}{2} - \varepsilon},\\ BN^{\frac{1}{2} + \varepsilon} <y\leq K_2N^{\frac{1}{2} + \varepsilon}}}\mathbb{P}_{(x,y, \underline{0})}\Big(\sup_{i\in \{0,...,i^*_N-1\}}\sup_{\xi^{(N)}_{2i+1}\leq t\leq \xi^{(N)}_{2i+3}}\left|S^{(N)}(t)-S^{(N)}(\xi^{(N)}_{2i+1})\right|\geq 13 N^{\frac{1}{2}-\varepsilon+\delta} \Big)\leq a_2 e^{-b_2 N^{\frac{\delta}{5}}}.$$
\end{prop}
Here also, the idea is to condition on the event $\{\sup_{t \in [0, T(N,B)]} Q^{(N)}_2(t)\leq N^{\frac{1}{2}+\varepsilon+\delta}\}$ and then use Lemma~\ref{excursion time} to bound the change in $Q^{(N)}_2$ in this `short' interval 
$[\xi^{(N)}_{2i+1}, \xi^{(N)}_{2i+3}]$ and then use 
Lemma~\ref{lem:MM1-IDLE} to bound the maximum change in $I^{(N)}$.
The proof is given in Appendix~\ref{appendix:prop3.2}

From now on, we will work under the assumptions stated in Theorem~\ref{thm:PROCESS-LEVEL} for fixed $T>0$ and $\delta<(\varepsilon/2) \wedge \left(\frac{1}{2} - \varepsilon\right)$. 
Proposition~\ref{prop:change-st} and the properties of $I^{(N)}$ and $Q_2^{(N)}$ established in Sections~\ref{sec:HITTIME} and~\ref{sec:STEAYSTATE} hint that we can define the following high-probability event, which is stated in Lemma~\ref{lem:EN-to-1}:
\begin{equation}
\begin{split}
    \bar{E}^{(N)}\coloneqq\Big\{\sup_{0\leq t\leq T(N,B)}I^{(N)}(t)&\leq 5N^{\frac{1}{2}-\varepsilon+\delta},\sup_{0\leq t\leq T(N,B)}Q^{(N)}_2(t)\leq  N^{\frac{1}{2}+\varepsilon+\delta},\\
    &\sup_{\xi^{(N)}_{2i+1}\leq t\leq \xi^{(N)}_{2i+3}}\left|S^{(N)}(t)-S^{(N)}(\xi_{2i+1})\right|\leq 13 N^{\frac{1}{2}-\varepsilon+\delta},\forall i\leq i^*_N-1\Big\}.
\end{split}
\end{equation}
\begin{lemma}\label{lem:EN-to-1}
$\lim_{N\rightarrow\infty}\mathbb{P}(\bar{E}^{(N)})=1.$
\end{lemma}
\begin{proof}
The result follows from Proposition~\ref{prop:bdd-Q2}, Proposition~\ref{prop:IDLE-INTEGRAL} and Proposition~\ref{prop:change-st}.
\end{proof}
Note that on the event $\bar{E}^{(N)}$, the following equation holds:
\begin{equation}
    2N=S^{(N)}(t)+Q_1^{(N)}(t)-Q_2^{(N)}(t)+2I^{(N)}(t),\quad t\in [0,T(N,B)].
\end{equation}
Hence, for $t\in[0,T(N,B)]$, the evolution of $I^{(N)}$ can be described as follows:
\begin{align*}
    I^{(N)}(t)&\nearrow I^{(N)}(t)+1\text{  at rate  }2N-S^{(N)}(t)- 2I^{(N)}(t),\\
    I^{(N)}(t)&\searrow (I^{(N)}(t)-1)_+\text{  at rate  }N-\beta N^{\frac{1}{2}-\varepsilon}.
\end{align*}
Also, on the event $\bar{E}^{(N)}$, we have the following bounds on the rate of increase of $I^{(N)}(t)$.
\begin{lemma}\label{lem:u-l-bund}
On the event $\bar{E}^{(N)}$, for each $i\in\{0,...,i^*_N-1\}$ and $t\in[\xi^{(N)}_{2i+1},\xi^{(N)}_{2i+3})$, we have \begin{equation*}
    \lambda_{l,i}\leq 2N-S^{(N)}(t)-2 I^{(N)}(t)\leq \lambda_{u,i},
\end{equation*}
where $\lambda^{(N)}_{l,i}=2N-S^{(N)}(\xi^{(N)}_{2i+1})- 23 N^{\frac{1}{2}-\varepsilon+\delta}$ and $\lambda_{u,i}=2N-S^{(N)}(\xi^{(N)}_{2i+1})+ 13 N^{\frac{1}{2}-\varepsilon+\delta}$.
\end{lemma}
Lemma~\ref{lem:u-l-bund} is easy to check (we omit its proof).
Thus, Lemma~\ref{lem:u-l-bund} hints that on $\bar{E}^{(N)}$, for each excursion interval $[\xi^{(N)}_{2i+1},\xi^{(N)}_{2i+3})$, we can construct two $M/M/1$ queues, $I^{(N)}_{u,i}$ and $I^{(N)}_{l,i}$, on the same probability space as $I^{(N)}$, and both starting from zero, to bound the process $I^{(N)}$ from above and below respectively. 
For $t\in [\xi^{(N)}_{2i+1},\xi^{(N)}_{2i+3})$, we define $I^{(N)}_{u,i}$ and $I^{(N)}_{l,i}$ with the following transition rates:
\begin{align*}
    I^{(N)}_{u,i}(t)\nearrow I^{(N)}_{u,i}(t)+1&\text{  at rate  }\lambda^{(N)}_{u,i},\\
    I^{(N)}_{u,i}(t)\searrow I^{(N)}_{u,i}(t)-1&\text{  at rate  }\mu^{(N)}_{u,i}=N-\beta N^{\frac{1}{2}-\varepsilon};\\
    I^{(N)}_{l,i}(t)\nearrow I^{(N)}_{l,i}(t)+1&\text{  at rate  }\lambda^{(N)}_{l,i},\\
    I^{(N)}_{l,i}(t)\searrow I^{(N)}_{l,i}(t)-1&\text{  at rate  }\mu^{(N)}_{l,i}=N.
\end{align*}
Due to Lemma~\ref{lem:u-l-bund}, we can naturally couple $I^{(N)}_{u,i}$ and $I^{(N)}_{l,i}$ with $I^{(N)}$ by setting $I^{(N)}_{u,i}(\xi^{(N)}_{2i})=I^{(N)}_{l,i}(\xi^{(N)}_{2i+1})=I^{(N)}(\xi^{(N)}_{2i+1})=0$, so that for $t\in[\xi^{(N)}_{2i+1},\xi^{(N)}_{2i+3})$, $I^{(N)}_{l,i}(t)\leq I^{(N)}(t) \leq I^{(N)}_{u,i}(t)$.
Note that since $\big\{(I^{(N)}(t),t\geq 0)\big\}_{N=1}^{\infty}$ are defined on the same probability space (see the representation \eqref{eq:filtration}), $\big\{(I_{u,i}^{(N)}(t),t\in [\xi^{(N)}_{2i+1},\xi^{(N)}_{2i+3})),i\in\{0,...,i^*_N-1\}\big\}_{N=1}^{\infty}$ and $\big\{(I_{l,i}^{(N)}(t),t\in [\xi^{(N)}_{2i+1},\xi^{(N)}_{2i+3})),i\in\{0,...,i^*_N-1\}\big\}_{N=1}^{\infty}$ are on the same probability space as well.

Having the above sample path-wise bounds, Lemma~\ref{M/M/1 behavior} below provides an approximation of the integral of the upper and lower-bounding process at the end of each excursion.

\begin{lemma}\label{M/M/1 behavior}
The following hold:
$$\sup_{j\in\{0,...,i^*_N-1\}}\frac{1}{N^{\frac{1}{2}+\varepsilon}}\sum_{i=0}^j\int_{\xi^{(N)}_{2i+1}}^{\xi^{(N)}_{2i+3}}\Big(I^{(N)}_{u,i}(s)-\frac{\lambda^{(N)}_{u,i}}{\mu^{(N)}_{u,i}-\lambda^{(N)}_{u,i}}\Big)ds\pto0,$$
$$\inf_{j\in\{0,...,i^*_N-1\}}\frac{1}{N^{\frac{1}{2}+\varepsilon}}\sum_{i=0}^j\int_{\xi^{(N)}_{2i+1}}^{\xi^{(N)}_{2i+3}}\Big(I^{(N)}_{l,i}(s)-\frac{\lambda^{(N)}_{l,i}}{\mu^{(N)}_{l,i}-\lambda^{(N)}_{l,i}}\Big)ds\pto0.$$
\end{lemma}
The proof is given in Appendix~\ref{appendix:prop3.2}.
Here we provide the proof scheme of the first convergence since the second one is similar. 
The idea is similar to the one used in~\cite{GW19}, with one technical difference: The term
$\int_{\xi^{(N)}_{2i+1}}^{\xi^{(N)}_{2i+3}}\Big(I^{(N)}_{u,i}(s)-\frac{\lambda^{(N)}_{u,i}}{\mu^{(N)}_{u,i}-\lambda^{(N)}_{u,i}}\Big)ds$ can be approximated by $Z^{(N)}_j\coloneqq\int_{\xi^{(N)}_{2i+1}}^{\xi^{(N)}_{u,i}}\Big(I^{(N)}_{u,i}(s)-\frac{\lambda^{(N)}_{u,i}}{\mu^{(N)}_{u,i}-\lambda^{(N)}_{u,i}}\Big)ds$ 
plus 
$\int^{\xi^{(N)}_{u,i}}_{\xi^{(N)}_{2i+3}}\Big(I^{(N)}_{u,i}(s)-\frac{\lambda^{(N)}_{u,i}}{\mu^{(N)}_{u,i}-\lambda^{(N)}_{u,i}}\Big)ds$ where $\xi^{(N)}_{u,i}$ is the time that $I^{(N)}_{u,i}$ hits 0. 
Although $\E(Z^{(N)}_j)=0$, we cannot use the LLN argument (as used in \cite{GW19}) since the order of $i^*_N$ is higher than the scaling $N^{\frac{1}{2}+\varepsilon}$ with high probability. 
Hence, we estimate the second moment of $\int_{\xi^{(N)}_{2i-1}}^{\xi^{(N)}_{u,i-1}}I^{(N)}_{u,i-1}(s)ds$ by Lemma~\ref{EC19} and use Doob's $L^2$ inequality to get $\sup_j\frac{1}{N^{\frac{1}{2}+\varepsilon}}\sum_{i=0}^jZ^{(N)}_j\pto0$. 
For the second term, $|\xi^{(N)}_{u,i}-\xi^{(N)}_{2i+3}|$ can be bounded due to the known expectation of excursion time of an $M/M/1$ queue, and we use Markov's inequality to get the desired convergence.

Next, Lemma~\ref{approx. rate} connects the term that is used in the integrand as an approximation in Lemma~\ref{M/M/1 behavior} to the integral of certain functional of $S^{(N)}(t)$, which will later be shown to be close to the integral of $I^{(N)}(t)$.
\begin{lemma}\label{approx. rate}
The following hold as $N\to\infty$:
\begin{equation}\label{eq:A7-1}
    \sup_{j\in\{0,...,i^*_N-1\}} \frac{1}{N^{\frac{1}{2}+\varepsilon}}\Big|\sum_{i=0}^j\int_{\xi^{(N)}_{2i+1}}^{\xi^{(N)}_{2i+3}} \Big(\frac{\lambda^{(N)}_{l,i}}{\mu^{(N)}_{l,i}-\lambda^{(N)}_{l,i}}-\frac{2N-S^{(N)}(s)}{S^{(N)}(s)-N}\Big)ds\Big|\pto 0,
\end{equation}
\begin{equation}\label{eq:A7-2}
    \sup_{j\in\{0,...,i^*_N-1\}} \frac{1}{N^{\frac{1}{2}+\varepsilon}}\Big|\sum_{i=0}^j\int_{\xi^{(N)}_{2i+1}}^{\xi^{(N)}_{2i+3}} \Big(\frac{\lambda^{(N)}_{u,i}}{\mu^{(N)}_{u,i}-\lambda^{(N)}_{u,i}}-\frac{2N-S^{(N)}(s)}{S^{(N)}(s)-N}\Big)ds\Big|\pto 0.
\end{equation}
\end{lemma}
The key step in proving Lemma~\ref{approx. rate} is to condition on the event $\bar{E}^{(N)}$, where $$\sup_{i\in\{0,...,i^*_N-1\}}\sup_{\xi^{(N)}_{2i+1}\leq t\leq \xi^{(N)}_{2i+3}}\left|S^{(N)}(t)-S^{(N)}(\xi_{2i+1})\right|\leq 13 N^{\frac{1}{2}-\varepsilon+\delta}$$ and to use the fact that 
$\frac{2N-x}{x-N}$ (which is the form of two terms of the integrand) is Lipschitz continuous on the interval $[N+\frac{1}{2}BN^{\frac{1}{2}-\varepsilon+\delta},\infty)$ with Lipschitz constant $\frac{4}{B^2}N^{-2\varepsilon}$. The proof is given in Appendix~\ref{appendix:prop3.2}.

Note that both Lemma~\ref{M/M/1 behavior} and Lemma~\ref{approx. rate} concerns approximation of respective integrals where the upper limits are excursion end points.  
Lemma~\ref{each excursion behavior} shows that it is sufficient to consider the latter case, since the contribution of the integrals within a single excursion is asymptotically negligible.
\begin{lemma} \label{each excursion behavior}
Under the assumptions on the initial state as given in Theorem~\ref{thm:PROCESS-LEVEL}, the following holds as $N\to\infty$:
\begin{equation}
    \label{eq:lem-A4}
    \sup_{i\in \{0,...,i^*_N-1\}}\sup_{\xi^{(N)}_{2i+1}\leq t\leq \xi^{(N)}_{2i+3}}\left[\frac{1}{N^{\frac{1}{2}+\varepsilon}}\int_{\xi^{(N)}_{2i+1}}^t I^{(N)}(s)ds+\frac{1}{N^{\frac{1}{2}+\varepsilon}}\int_{\xi^{(N)}_{2i+1}}^t\frac{2N-S^{(N)}(s)}{S^{(N)}(s)-N}ds\right] \pto 0.
\end{equation}
\end{lemma}
The proof of Lemma~\ref{each excursion behavior} is given in Appendix~\ref{appendix:prop3.2}.
We now have all the ingredients to prove Proposition~\ref{prop:INT-IDLE-2}.
\begin{proof}[Proof of Proposition~\ref{prop:INT-IDLE-2}.]
Again, we will consider events within $\bar{E}^{(N)}$.
Observe that for any $t \le T(N,B)$, either $t \le \xi^{(N)}_1$, or $t \in [\xi^{(N)}_{2j+1}, \xi^{(N)}_{2j+3}]$ for some $0 \le j \le i^*_N$, or $t \in [\xi^{(N)}_{2i^*_N+1}, T(N,B)]$. Hence,  
\begin{equation}\label{eq:Prop6.5-1}
    \begin{split}
    &\sup_{0\leq t\leq T(N,B)}\Big|\frac{1}{N^{\frac{1}{2}+\varepsilon}}\int_0^{t}\Big(I^{(N)}(s)-\frac{2N-S^{(N)}(s)}{S^{(N)}(s)-N}\Big)ds\Big|\\
    \leq & \frac{1}{N^{\frac{1}{2}+\varepsilon}}\int_0^{\xi^{(N)}_1}\left| I^{(N)}(s)-\frac{2N-S^{(N)}(s)}{S^{(N)}(s)-N}\right| ds\\
    &+\sup_{j\in\{0,...,i^*_N-1\}}\frac{1}{N^{\frac{1}{2}+\varepsilon}}\left|\int_{\xi^{(N)}_1}^{\xi^{(N)}_{2j+3}}\Big(I^{(N)}(s)-\frac{2N-S^{(N)}(s)}{S^{(N)}(s)-N}\Big)ds\right|\\
    &+\sup_{j\in\{0,...,i^*_N-1\}}\sup_{\xi^{(N)}_{2j+1}\leq t\leq\xi^{(N)}_{2j+3} }\frac{1}{N^{\frac{1}{2}+\varepsilon}}\left|\int_{\xi^{(N)}_{2j+1}}^t\Big(I^{(N)}(s)-\frac{2N-S^{(N)}(s)}{S^{(N)}(s)-N}\Big)ds\right|\\
    &+ \frac{1}{N^{\frac{1}{2}+\varepsilon}}\int_{\xi^{(N)}_{2i^*_N+1}}^{T(N,B)}\left| I^{(N)}(s)-\frac{2N-S^{(N)}(s)}{S^{(N)}(s)-N}\right| ds.
    \end{split}
\end{equation}
For $t\in[0,T(N,B)]$, $I^{(N)}(t)$ can be upper bounded by $\bar{I}^{(N)}_B(t)$ as defined in Section~\ref{sec:HITTIME}. Then, there exist constants $C, C'$ such that 
\begin{equation}\label{eq:Prop6.5-2}
   \begin{split}
        \mathbb{P}\big(\xi^{(N)}_1\geq N^{\delta-2\varepsilon}\big)
        \leq \mathbb{P}(\inf\{t>0:\bar{I}^{(N)}_B(t)=0\} \geq N^{\delta-2\varepsilon}\ \big|\ \bar{I}^{(N)}_B(0)=\lfloor K_1N^{\frac{1}{2}-\varepsilon}\rfloor)\leq C e^{-C'N^{\delta}},
   \end{split}
\end{equation}
where the first inequality is due to the assumption $I^{(N)}(0)\leq K_1N^{\frac{1}{2}-\varepsilon}$ and the second inequality is due to Lemma~\ref{lem:LEMMA-1} and Markov's inequality.
Also, on the event $\bar{E}^{(N)}$, we have 
\begin{equation}\label{inert}
\sup_{0\leq t\leq T(N,B)}I^{(N)}(t)\leq 5 N^{\frac{1}{2}-\varepsilon+\delta}\quad\text{and}\quad \sup_{0\leq t\leq T(N,B)}\frac{2N-S^{(N)}(t)}{S^{(N)}(t)-N}\leq \frac{2N^{\frac{1}{2}-\varepsilon}}{B}.
\end{equation}
Hence, on 
$\big\{\xi^{(N)}_1<N^{\delta-2\varepsilon}\big\}\cap \bar{E}^{(N)}$,  
we have \begin{equation*}
     \frac{1}{N^{\frac{1}{2}+\varepsilon}}\int_0^{\xi^{(N)}_1}\Big|I^{(N)}(s)-\frac{2N-S^{(N)}(s)}{S^{(N)}(s)-N}\Big|ds\leq C N^{-4\varepsilon+2\delta}\to 0, \quad\text{as}\quad N\to\infty.
\end{equation*}
With $\lim_{N\rightarrow\infty}\mathbb{P}\big(\xi^{(N)}_1< N^{\delta-2\varepsilon}\big)=1$ due to~\eqref{eq:Prop6.5-2} and $\lim_{N\to\infty}\mathbb{P}(\bar{E}^{(N)})=1$ due to Lemma~\ref{lem:EN-to-1}, we have
\begin{equation}\label{eq:Prop6.5-3}
     \frac{1}{N^{\frac{1}{2}+\varepsilon}}\int_0^{\xi^{(N)}_1}\Big|I^{(N)}(s)-\frac{2N-S^{(N)}(s)}{S^{(N)}(s)-N}\Big|ds\pto0,
\end{equation}
as $N\to\infty$.
Now, consider the second term in the right hand side of \eqref{eq:Prop6.5-1}. For $j\in\{0,...,i^*_N-1\}$,
\begin{align*}
    &\frac{1}{N^{\frac{1}{2}+\varepsilon}}\sum_{i=0}^{j}\int_{\xi^{(N)}_{2i+1}}^{\xi^{(N)}_{2i+3}}\Big(I^{(N)}_{l,i}(s)-\frac{\lambda^{(N)}_{l,i}}{\mu^{(N)}_{l,i}-\lambda^{(N)}_{l,i}}\Big)ds+ \frac{1}{N^{\frac{1}{2}+\varepsilon}}\sum_{i=0}^{j}\int_{\xi^{(N)}_{2i+1}}^{\xi^{(N)}_{2i+3}} \Big(\frac{\lambda^{(N)}_{l,i}}{\mu^{(N)}_{l,i}-\lambda^{(N)}_{l,i}}-\frac{2N-S^{(N)}(s)}{S^{(N)}(s)-N}\Big)ds\\
    \leq & \frac{1}{N^{\frac{1}{2}+\varepsilon}}\int_{\xi^{(N)}_1}^{\xi^{(N)}_{2j+3}}\Big(I^{(N)}(s)-\frac{2N-S^{(N)}(s)}{S^{(N)}(s)-N}\Big)ds\\
    \leq & \frac{1}{N^{\frac{1}{2}+\varepsilon}}\sum_{i=0}^{j}\int_{\xi^{(N)}_{2i+1}}^{\xi^{(N)}_{2i+3}}\Big(I^{(N)}_{u,i}(s)-\frac{\lambda^{(N)}_{u,i}}{\mu^{(N)}_{u,i}-\lambda^{(N)}_{u,i}}\Big)ds+ \frac{1}{N^{\frac{1}{2}+\varepsilon}}\sum_{i=0}^{j}\int_{\xi^{(N)}_{2i+1}}^{\xi^{(N)}_{2i+3}} \Big(\frac{\lambda^{(N)}_{u,i}}{\mu^{(N)}_{u,i}-\lambda^{(N)}_{u,i}}-\frac{2N-S^{(N)}(s)}{S^{(N)}(s)-N}\Big)ds.
\end{align*}
Thus, by Lemma \ref{approx. rate} and Lemma \ref{M/M/1 behavior}, 
\begin{equation}\label{eq 12}
    \sup_{j\in\{0,...,i^*_N-1\}}\frac{1}{N^{\frac{1}{2}+\varepsilon}}\left|\int_{\xi^{(N)}_1}^{\xi^{(N)}_{2j+3}}\Big(I^{(N)}(s)-\frac{2N-S^{(N)}(s)}{S^{(N)}(s)-N}\Big)ds\right|\pto0\text{ as }N\rightarrow\infty .
\end{equation}
By Lemma~\ref{each excursion behavior}, we have 
\begin{equation}\label{eq:Prop6.5-4}
    \sup_{j\in\{0,...,i^*-1\}}\sup_{\xi^{(N)}_{2j+1}\leq t\leq\xi^{(N)}_{2j+3} }\frac{1}{N^{\frac{1}{2}+\varepsilon}}\left|\int_{\xi^{(N)}_{2j+1}}^{t}\Big(I^{(N)}(s)-\frac{2N-S^{(N)}(s)}{S^{(N)}(s)-N}\Big)ds\right|\pto0.
\end{equation}
To estimate the last term in the bound \eqref{eq:Prop6.5-1}, note that using the same argument used to derive \eqref{eq:lemA2-3} and \eqref{eq:lemA2-4},
$
\mathbb{P}\left(T(N,B) - \xi^{(N)}_{2i^*_N + 1} \ge N^{\delta-2\varepsilon}\right)\leq Ce^{-C'N^{\delta}}.
$
From this and \eqref{inert}, we conclude
\begin{equation}\label{laste}
 \frac{1}{N^{\frac{1}{2}+\varepsilon}}\int_{\xi^{(N)}_{2i^*_N+1}}^{T(N,B)}\left| I^{(N)}(s)-\frac{2N-S^{(N)}(s)}{S^{(N)}(s)-N}\right| ds \pto0.
\end{equation}
Using \eqref{eq:Prop6.5-3}, \eqref{eq 12}, \eqref{eq:Prop6.5-4} and \eqref{laste} in \eqref{eq:Prop6.5-1}, we have
\begin{equation}\label{eq:Prop6.5-5}
    \sup_{0\leq t\leq T(N,B)}\Big|\frac{1}{N^{\frac{1}{2}+\varepsilon}}\int_0^{t}\Big(I^{(N)}(s)-\frac{2N-S^{(N)}(s)}{S^{(N)}(s)-N}\Big)ds\Big|\pto0.
\end{equation}
For any $t$ such that $N^{2\varepsilon}t\leq T(N,B)$, by the triangle inequality,
\begin{equation}
\begin{split}
    &\Big|\frac{1}{N^{\frac{1}{2}+\varepsilon}} \int_0^{N^{2\varepsilon}t} I^{(N)}(s)ds-\int_0^t\frac{1}{X^{(N)}(s)}ds\Big|\\
    \leq& \Big|\frac{1}{N^{\frac{1}{2}+\varepsilon}} \int_0^{N^{2\varepsilon}t} I^{(N)}(s)ds-\frac{1}{N^{\frac{1}{2}+\varepsilon}}\int_0^{N^{2\varepsilon}t}\frac{2N-S^{(N)}(s)}{S^{(N)}(s)-N}ds\Big|\\
    &\hspace{2cm}+\Big|\frac{1}{N^{\frac{1}{2}+\varepsilon}}\int_0^{N^{2\varepsilon}t}\frac{2N-S^{(N)}(s)}{S^{(N)}(s)-N}ds-\int_0^t\frac{1}{X^{(N)}(s)}ds\Big|\label{limit of 1/X^N}.
\end{split}
\end{equation}
Note that
\begin{align*}
    \frac{1}{N^{\frac{1}{2}+\varepsilon}}\int_0^{N^{2\varepsilon}t}\frac{2N-S^{(N)}(s)}{S^{(N)}(s)-N}ds
    =&\frac{1}{N^{\frac{1}{2}-\varepsilon}}\int_0^t \frac{2N-S^{(N)}(N^{2\varepsilon}s)}{S^{(N)}(N^{2\varepsilon}s)-N}ds\\
    =&\frac{1}{N^{\frac{1}{2}-\varepsilon}}\int_0^t \frac{N^{\frac{1}{2}-\varepsilon}+\frac{N-S^{(N)}(N^{2\varepsilon}s)}{N^{\frac{1}{2}+\varepsilon}}}{\frac{S^{(N)}(N^{2\varepsilon}s)-N}{N^{\frac{1}{2}+\varepsilon}}}ds
    =-\frac{t}{N^{\frac{1}{2}-\varepsilon}}+\int_0^t\frac{1}{X^{(N)}}ds,
\end{align*}
which implies that for all $t$ such that $N^{2\varepsilon}t\leq T(N,B)$,
\begin{align*}
    &\Big|\frac{1}{N^{\frac{1}{2}+\varepsilon}} \int_0^{N^{2\varepsilon}t} I^{(N)}(s)ds-\int_0^t\frac{1}{X^{(N)}(s)}ds\Big|\\
    &\hspace{2cm}\leq \Big|\frac{1}{N^{\frac{1}{2}+\varepsilon}} \int_0^{N^{2\varepsilon}t} I^{(N)}(s)ds-\frac{1}{N^{\frac{1}{2}+\varepsilon}}\int_0^{N^{2\varepsilon}t}\frac{2N-S^{(N)}(s)}{S^{(N)}(s)-N}ds\Big| + \frac{t}{N^{\frac{1}{2}-\varepsilon}}.
\end{align*}
This, combined with \eqref{eq:Prop6.5-5}, yields that as $N\to\infty$,
$\sup_{0\leq t\leq T \wedge (N^{-2\varepsilon}\tau^{(N)}_2(B))}\Big|\frac{1}{N^{\frac{1}{2}+\varepsilon}}\int_0^{N^{2\varepsilon}t}I^{(N)}(s)ds-\int_0^t \frac{1}{X^{(N)}(s)}ds\Big|\pto 0,$
which concludes the proof of Proposition~\ref{prop:INT-IDLE-2}.
\end{proof}

\begin{appendix}

\section{Useful concentration estimates.}\label{app:conc}

The following elementary lemma gives a supremum bound on the path of a compensated Poisson process.
\begin{lemma}\label{lem:sup-poi}
Suppose $A(\cdot)$ is a Poisson process with unit rate. For any $T>0$ and $ x\in [0, 4T]$,
$$\mathbb{P}\big(\sup_{s\in[0,T]}\left|A(s)-s\right|\geq x\big)\leq 2 e^{-\frac{x^2}{8T}}.$$
\end{lemma}
\begin{proof}
By \cite[Theorem 5, Chapter 4, Page 351]{LS}, for any $x \ge 0$,
$\mathbb{P}\big(\sup_{s\in[0,T]}\big(A(s)-s\big)\geq x\big)\leq \exp\big\{-\sup_{\lambda>0}\big(\lambda x-2T(e^{\lambda}-1-\lambda)\big)\big\}.$
Since $e^{\lambda}-1-\lambda\leq \lambda^2$, for all $\lambda\in [0,1],$
we have for any $x \in [0,4T]$, $\sup_{\lambda>0}\big(\lambda x- 2T(e^{\lambda}-1-\lambda)\big)\geq \sup_{\lambda\in[0,1]}\big(\lambda x-2T\lambda^2\big)=\frac{x^2}{8T}.$
We can perform the same calculation with $-(A(s)-s)$ in place of $(A(s)-s)$. This proves the lemma. 
\end{proof}

The next lemma gives a concentration bound on sums of random variables with stretched exponential tails. It is used multiple times in this article, and is of independent interest.
\begin{lemma}\label{lem:LEMMA-5}
Suppose $\{\Phi^{(r)}_i:i\geq 1,r\in R\}$ are independent nonnegative random variables adapted to a filtration $\big\{\mathcal{F}^{(r)}_i:i\geq 1, r\in R\big\}$, where $R$ is an arbitrary indexing set. Let $\mathcal{F}^{(r)}_0\subseteq \mathcal{F}^{(r)}_1$ be an arbitrary $\sigma$-field. Suppose there exist deterministic constants $c,c_1>0$, $\theta \in (0,1)$ not dependent on $r$  such that for any $x\geq0$ and $i\geq 1$, $\mathbb{P}\big(\Phi^{(r)}_i\geq x|\mathcal{F}^{(r)}_{i-1}\big)\leq ce^{-c_1x^{\theta}},x\geq 0.$
Define $\mu :=\int_0^{\infty}ce^{-c_1x^{\theta}} dx\geq \sup_{r\in R}\mathbb{E}\big(\Phi^{(r)}_i|\mathcal{F}^{(r)}_{i-1}\big),\forall i\geq 1$.
Then there exist $c_2,c_3>0$, not depending on $r$, such that for any $a\geq 4\mu $,
$$\sup_{r\in R}\mathbb{P}\big(\sum_{i=1}^{n}\Phi^{(r)}_i \geq an|\mathcal{F}^{(r)}_0\big)\leq c_2\left(1+\frac{n^{\frac{1}{2+\theta}}}{a^{\frac{\theta}{2+\theta}}}\right)\exp \big\{-c_3(a^2n)^{\frac{\theta}{2+\theta}}\big\},\quad n\geq 1.$$
\end{lemma}
\begin{proof}
Let $\Tilde{\Phi}^{(r)}_i:=\Phi^{(r)}_i\mathds{1}[\Phi^{(r)}_i\leq d]$ for some $d>0$ that does not depend on $r$, to be chosen later. Then by Azuma's inequality,
\begin{equation}
\begin{split}\label{eq:lem4.5-1}
    \sup_{r\in R}\mathbb{P}\big(\sum_{i=1}^{n}\Tilde{\Phi}^{(r)}_i\geq \frac{an}{2}|\mathcal{F}^{(r)}_0\big)
    \leq & \sup_{r\in R}\mathbb{P}\big(\sum_{i=1}^{n}(\Tilde{\Phi}^{(r)}_i-\mathbb{E}(\Tilde{\Phi}^{(r)}_i|\mathcal{F}_{i-1}))\geq \frac{an}{2}-n\mu|\mathcal{F}^{(r)}_0\big)\\
    \leq & \sup_{r\in R}\mathbb{P}\big(\sum_{i=1}^{n}(\Tilde{\Phi}^{(r)}_i-\mathbb{E}(\Tilde{\Phi}^{(r)}_i|\mathcal{F}_{i-1}))\geq \frac{an}{4}|\mathcal{F}^{(r)}_0\big)\\
    \leq &\exp \big\{-\frac{a^2n}{32d^2 }\big\}.
\end{split}
\end{equation}
Moreover, using Markov's inequality and the union bound 
\begin{equation}
    \label{eq:lem4.5-2}
    \sup_{r\in R} \mathbb{P}\big(\sum_{i=1}^n (\Phi^{(r)}_i-\Tilde{\Phi}^{(r)}_i)\geq \frac{an}{2}|\mathcal{F}^{(r)}_0\big)\leq \frac{2}{a}\left[dce^{-c_1d^{\theta}} + \int_d^{\infty}ce^{-c_1x^{\theta}}dx\right]
    \leq \frac{\Tilde{c}d e^{-\Tilde{c}'d^{\theta}}}{a},
\end{equation}
where $\Tilde{c}$ and $\tilde{c}'$ are positive constants not depending on $d$. 
Using ~\eqref{eq:lem4.5-1} and~\eqref{eq:lem4.5-2} with $d=(a^2n)^{\frac{1}{2+\theta}}$, the lemma follows upon using
$$\sup_{r\in R}\mathbb{P}\big(\sum_{i=1}^{n}\Phi^{(r)}_i \geq an|\mathcal{F}^{(r)}_0\big)\leq \sup_{r\in R}\mathbb{P}\big(\sum_{i=1}^{n}\Tilde{\Phi}^{(r)}_i\geq \frac{an}{2}|\mathcal{F}^{(r)}_0\big)+\sup_{r\in R} \mathbb{P}\big(\sum_{i=1}^n (\Phi^{(r)}_i-\Tilde{\Phi}^{(r)}_i)\geq \frac{an}{2}|\mathcal{F}^{(r)}_0\big),$$
and choosing appropriate constants $c_2$ and $c_3$.
\end{proof}

\section{Integral and supremum of the upper-bounding birth-death process}
\label{sec:app-int-IB}
The goal of this appendix is to prove Lemma \ref{lem:LEMMA-6}.
Define the following stopping time and expectation notation: For any $x \ge 0$ and $\eta>0$,
$$\Bar{\tau}^{(N)}_B(x):=\inf \big\{t\geq 0:\Bar{I}^{(N)}_B(t)= \lfloor xN^{\frac{1}{2}-\varepsilon} \rfloor\big\}
\quad\text{and}\quad \mathbb{E}_{\eta N^{\frac{1}{2}-\varepsilon}}\big(\cdot\big) :=\mathbb{E}\big(\cdot|\bar{I}_B^{(N)}(0)=\lfloor \eta N^{\frac{1}{2}-\varepsilon} \rfloor\big).$$
The next lemma is used to control the time taken by $\bar{I}_B^{(N)}$ to hit a level $\lfloor yN^{\frac{1}{2}-\varepsilon} \rfloor$ starting from a larger point $\lfloor \eta N^{\frac{1}{2}-\varepsilon} \rfloor$ with $\eta> y$.
\begin{lemma} \label{lem:LEMMA-1}
There exists $B_0 \ge 1, N_0 \ge 1$ such that for any $\eta>0$, $B\geq B_0$, $N\geq N_0$, and $y\in[0,\eta)$,
$$\mathbb{E}_{\eta N^{\frac{1}{2}-\varepsilon}}\left(\exp\left\lbrace \frac{B N^{2\varepsilon}}{8}\, \Bar{\tau}^{(N)}_B(y)\right\rbrace\right)\leq e^{\eta-y}.$$
\end{lemma}
\begin{proof}
Take any $B > 2\beta$. Let us denote $\hat{I}^{(N)}_B(t)=\frac{\Bar{I}^{(N)}_B(t)}{N^{\frac{1}{2}-\varepsilon}}$ and $\sigma^{(N)}=\Bar{\tau}^{(N)}_B(y)$. 
For $\theta>0$ to be chosen later, define
$Z^{(N)}_B(t):=\exp\big\{\hat{I}^{(N)}_B(t)+\frac{\theta}{2}t\big\}.$
For $t<\bar{\tau}_B^{(N)}(0)$,
\begin{align*}
    \mathcal{L}Z^{(N)}_B(t)&=\frac{\theta}{2}Z^{(N)}_B(t)
    +e^{\frac{\theta }{2}t}\Big[\big(e^{\hat{I}^{(N)}_B(t)+\frac{1}{N^{\frac{1}{2}-\varepsilon}}}-e^{\hat{I}^{(N)}_B(t)}\big)\big(N-B N^{\frac{1}{2}+\varepsilon}\big) \\  
    &\hspace{5.5cm}+\big(e^{\hat{I}^{(N)}_B(t)-\frac{1}{N^{\frac{1}{2}-\varepsilon}}}-e^{\hat{I}^{(N)}_B(t)}\big)\big(N-\beta N^{\frac{1}{2}-\varepsilon}\big)\Big]\\
    &=\frac{\theta}{2}Z^{(N)}_B(t)+Z^{(N)}_B(t)N(e^{\frac{1}{N^{\frac{1}{2}-\varepsilon}}}+e^{-\frac{1}{N^{\frac{1}{2}-\varepsilon}}}-2)\\
    &\hspace{3cm} +Z^{(N)}_B(t)\big[-\big(e^{\frac{1}{N^{\frac{1}{2}-\varepsilon}}}-1\big)BN^{\frac{1}{2}+\varepsilon}-\big(e^{-\frac{1}{N^{\frac{1}{2}-\varepsilon}}}-1\big)\beta N^{\frac{1}{2}-\varepsilon}\big],
\end{align*}
where $\mathcal{L}(\cdot)$ is the infinitesimal generator of the associated continuous time Markov process. 
Fix two constants $c, a>0$, such that for all $x\in(-a,a)$,
$e^x+e^{-x}-2\leq cx^2,$
and note that for all $x\in \mathbb{R}$,
$e^{x}-1\geq x.$
Now, let $N_0\geq 1$ be such that  $N^{-\frac{1}{2}+\varepsilon}<a$ for all $N\geq N_0$. Also, as $B > 2\beta$, $\frac{1}{2}BN^{2\varepsilon}>\beta$ for all $N \ge 1$.
Then for  all $N\geq N_0$,
\begin{align*}
\mathcal{L}Z^{(N)}_B(t\wedge \sigma^{(N)})&\leq \frac{\theta}{2}Z^{(N)}_B(t\wedge \sigma^{(N)})+cN^{2\varepsilon}Z^{(N)}_B(t\wedge \sigma^{(N)})\\
&\quad +Z^{(N)}_B(t\wedge \sigma^{(N)})\big[-\big(e^{\frac{1}{N^{\frac{1}{2}-\varepsilon}}}-1\big)BN^{\frac{1}{2}+\varepsilon}-\big(e^{-\frac{1}{N^{\frac{1}{2}-\varepsilon}}}-1\big)\beta N^{\frac{1}{2}-\varepsilon}\big]\\
&\leq \frac{\theta}{2}Z^{(N)}_B(t\wedge \sigma^{(N)})+cN^{2\varepsilon}Z^{(N)}_B(t\wedge \sigma^{(N)})-\frac{ BN^{\frac{1}{2}+\varepsilon}}{N^{\frac{1}{2}-\varepsilon}}Z^{(N)}_B(t\wedge \sigma^{(N)})+\beta Z^{(N)}_B(t\wedge \sigma^{(N)})\\
&\leq \frac{\theta}{2}Z^{(N)}_B(t\wedge \sigma^{(N)})+cN^{2\varepsilon}Z^{(N)}_B(t\wedge \sigma^{(N)})-\frac{ BN^{2\varepsilon}}{2}Z^{(N)}_B(t\wedge \sigma^{(N)}).
\end{align*}
Let $B_0\geq 1 + 2\beta$ be such that $c-\frac{B}{2}\leq -\frac{B}{4}$, $\forall B\geq B_0$.
Take $\theta=\frac{B}{4}N^{2\varepsilon}$.
Then $\mathcal{L}Z^{(N)}_B(t\wedge \sigma^{(N)})\leq 0$, $\forall t\geq 0$, implying for all $B \ge B_0$,
$$\mathbb{E}_{\eta N^{\frac{1}{2}-\varepsilon}}\big(Z^{(N)}_B(t\wedge \sigma^{(N)})\big)\leq \mathbb{E}_{\eta N^{\frac{1}{2}-\varepsilon}}\big(Z^{(N)}_B(0)\big),\quad \forall t\geq 0.$$
By Fatou's lemma and the observation $\lim_{t\rightarrow\infty}Z^{(N)}_B(t\wedge \sigma^{(N)})=Z^{(N)}_B(\sigma^{(N)})$ a.s., we get
$$\mathbb{E}_{\eta N^{\frac{1}{2}-\varepsilon}}\big(Z^{(N)}_B(\sigma^{(N)})\big)\leq \liminf_{t\rightarrow\infty}\mathbb{E}_{\eta N^{\frac{1}{2}-\varepsilon}}\big(Z^{(N)}_B(t\wedge \sigma^{(N)})\big)\leq \mathbb{E}_{\eta N^{\frac{1}{2}-\varepsilon}}\big(Z^{(N)}_B(0)\big) \le e^{\eta},$$
implying that
$\mathbb{E}_{\eta N^{\frac{1}{2}-\varepsilon}}\big(e^{\frac{\theta}{2} \Bar{\tau}^{(N)}_B(y)}\big)\leq e^{\eta-y}.$
This completes the proof of the lemma.
\end{proof}

Fix $\eta=1 \wedge (\beta/8)$, and 
recall the notations introduced below the statement of Lemma~\ref{lem:LEMMA-3}, such as $\Bar{\sigma}^{(N)}_{i}, \Bar{\xi}^{(N)}_i, \Bar{u}^{(N)}_i,$ and $\bar{K}^{(N)}_t$, and the following upper bound on the integral of $\Bar{I}_B^{(N)}$:
$$\int_0^{N^{2\varepsilon}t}\Bar{I}_B^{(N)}(s)ds \le \eta N^{\frac{1}{2}+\varepsilon} t+ \sum_{i=1}^{\bar{K}^{(N)}_{t}}\Bar{u}^{(N)}_i\Bar{\xi}^{(N)}_i
    \le \frac{\beta}{8}N^{\frac{1}{2}+\varepsilon}t + \sum_{i=1}^{\bar{K}^{(N)}_{t}}\Bar{u}^{(N)}_i\Bar{\xi}^{(N)}_i.$$
In the next few lemmas, probability bounds on the integral are obtained by analyzing the sum in the above upper bound.
Lemma~\ref{lem:LEMMA-3} furnishes control on $\Bar{u}^{(N)}_1\Bar{\xi}^{(N)}_1$ which stochastically dominates the integral of $\Bar{I}^{(N)}_B$ over the excursion interval $[\Bar{\sigma}^{(N)}_{2i-1}, \Bar{\sigma}^{(N)}_{2i}]$. Lemma \ref{lem:LEMMA-4} bounds the number of excursions $\bar{K}^{(N)}_t$ of $\bar{I}_B^{(N)}$ on the time interval $[0,N^{2\varepsilon}t]$ for large enough $t$.
\begin{lemma}\label{lem:LEMMA-3}
Take $B_0$ as in Lemma~\ref{lem:LEMMA-1}. There exist positive constants $\Bar{c}$, $c_1$, $c_2$, such that
for any $B\geq B_0$, $\exists \tilde N_B>0$ such that for all $N\geq \tilde N_B$,
\begin{enumerate}[\normalfont (i)]
    \item $\mathbb{P}\big(\Bar{u}^{(N)}_1\Bar{\xi}^{(N)}_1\geq \frac{xN^{\frac{1}{2}-3\varepsilon}}{B^{\frac{3}{2}}}\big)\leq c_1e^{-c_2\sqrt{x}},\quad\forall x\geq 0$;
    \item $\mathbb{E}\big(\Bar{u}^{(N)}_1\Bar{\xi}^{(N)}_1\big)\leq \frac{\Bar{c}N^{\frac{1}{2}-3\varepsilon}}{B^{\frac{3}{2}}}$.
\end{enumerate}
\end{lemma}
\begin{proof}
Note that
\begin{equation}\label{eq:lem4.3-1}
    \mathbb{P}\big(\Bar{u}^{(N)}_1\Bar{\xi}^{(N)}_1\geq \frac{xN^{\frac{1}{2}-3\varepsilon}}{B^{\frac{3}{2}}} \big)\leq \mathbb{P}\big(\Bar{\xi}^{(N)}_1 \geq \frac{\sqrt{x}}{B}N^{-2\varepsilon}\big)+\mathbb{P}\big(\Bar{u}^{(N)}_1\geq \frac{\sqrt{x}}{\sqrt{B}}N^{\frac{1}{2}-\varepsilon},\Bar{\xi}^{(N)}_1 < \frac{\sqrt{x}}{B}N^{-2\varepsilon}\big).
\end{equation}
By Lemma \ref{lem:LEMMA-1} and Markov's inequality, for $B\geq B_0$, $N\geq N_0$, $x \ge 0$,
\begin{equation}\label{eq:lem4.3-2}
    \mathbb{P}\big(\Bar{\xi}^{(N)}_1\geq \frac{\sqrt{x}}{B}N^{-2\varepsilon}\big)\leq e^{-\frac{\theta \sqrt{x}}{2B N^{2\varepsilon}}}\mathbb{E}_{\eta N^{\frac{1}{2}-\varepsilon}}\big(e^{\frac{\theta}{2}\Bar{\tau}^{(N)}_B(\frac{\eta}{2})}\big)\leq e^{\eta}e^{-\sqrt{x}/8},
\end{equation}
where $\theta= BN^{2\varepsilon}/4$ as in Lemma~\ref{lem:LEMMA-1}.
Now, for $s\in[0, \sigma^{(N)}_2 - \sigma^{(N)}_1)$,
\begin{align*}
    \Bar{I}_B^{(N)}\big(\sigma^{(N)}_1+s\big)&= \lfloor \eta N^{\frac{1}{2}-\varepsilon}\rfloor +A_1\big((N-BN^{\frac{1}{2}+\varepsilon})s\big)-A_2\big((N-\beta N^{\frac{1}{2}-\varepsilon})s\big)\\
    &= \lfloor \eta N^{\frac{1}{2}-\varepsilon}\rfloor +\hat{A}_1\big((N-BN^{\frac{1}{2}+\varepsilon})s\big)-\hat{A}_2\big((N-\beta N^{\frac{1}{2}-\varepsilon})s\big)-BN^{\frac{1}{2}+\varepsilon}s+\beta N^{\frac{1}{2}-\varepsilon}s,
\end{align*}
where $A_1$, $A_2$ are i.i.d.~Poisson processes with unit rate and $\hat{A}_i(s)=A_i(s)-s$, $i=1,2$.
Therefore, choosing $\tilde N_B \ge N_0$ such that $\frac{\sqrt{x}}{2\sqrt{B}}N^{\frac{1}{2}-\varepsilon} \le \frac{4\sqrt{x}}{B}N^{1-2\varepsilon}$ for all $N \ge \tilde N_B$, we obtain for any $N \ge \tilde N_B$, $x \ge 0$,
\begin{equation}\label{eq:sup-IB}
\begin{split}
    &\mathbb{P}\big(\Bar{u}^{(N)}_1\geq \frac{\sqrt{x}}{\sqrt{B}}N^{\frac{1}{2}-\varepsilon},\Bar{\xi}^{(N)}_1<\frac{\sqrt{x}}{B}N^{-2\varepsilon}\big)\\
    &=\mathbb{P}\big(\sup_{s\in[0, \sigma^{(N)}_2 - \sigma^{(N)}_1)}\Bar{I}_B^{(N)}(\sigma^{(N)}_1+s)\geq \lfloor \eta N^{\frac{1}{2}-\varepsilon}\rfloor +\frac{\sqrt{x}}{\sqrt{B}}N^{\frac{1}{2}-\varepsilon},\  \sigma^{(N)}_2 - \sigma^{(N)}_1 < \frac{\sqrt{x}}{B}N^{-2\varepsilon}\big)\\
     &\leq \mathbb{P}\big(\sup_{s\in[0,\sqrt{x}/(BN^{2\varepsilon}))} \hat{A}_1\big((N-BN^{\frac{1}{2}+\varepsilon})s\big)-\hat{A}_2\big((N-\beta N^{\frac{1}{2}-\varepsilon})s\big)-BN^{\frac{1}{2}+\varepsilon}s+\beta N^{\frac{1}{2}-\varepsilon}s\geq \frac{\sqrt{x}}{\sqrt{B}}N^{\frac{1}{2}-\varepsilon}\big)\\
     & \leq \mathbb{P}\big(\sup_{s\in[0,\sqrt{x}/(BN^{2\varepsilon}))}\hat{A}_1\big((N-BN^{\frac{1}{2}+\varepsilon})s\big)-\hat{A}_2\big((N-\beta N^{\frac{1}{2}-\varepsilon})s\big)\geq \frac{\sqrt{x}}{\sqrt{B}}N^{\frac{1}{2}-\varepsilon}\big)\\
     & \leq 2\mathbb{P}\big(\sup_{s\in[0,\sqrt{x}/(BN^{2\varepsilon}))} |\hat{A}_1\big(Ns\big)| \geq \frac{\sqrt{x}}{2\sqrt{B}}N^{\frac{1}{2}-\varepsilon}\big)
    \leq  4\exp \big\{-\frac{\sqrt{x}}{32}\big\},
\end{split}
\end{equation}
where the second inequality is due to $BN^{\frac{1}{2}+\varepsilon}>\beta N^{\frac{1}{2}-\varepsilon}$ that we assumed while defining $\bar{I}^{(N)}_B$, and the last inequality is due to Lemma~\ref{lem:sup-poi}.
Plugging \eqref{eq:lem4.3-2} and \eqref{eq:sup-IB} into \eqref{eq:lem4.3-1}, we get part (i). 
Part (ii) follows directly from part (i).
\end{proof}

\begin{lemma}\label{lem:LEMMA-4}
There exist $ \Tilde{c}, t_0>0$ such that for all $B \ge 1$, $a\geq \frac{256B}{\eta^2}$,  $ N\geq 1$, and $t\geq t_0$,
$$\mathbb{P}\big(\bar{K}^{(N)}_t\geq aN^{4\varepsilon}t\big)\leq \exp \{-\Tilde{c}aN^{4\varepsilon}t\}.$$
\end{lemma}
\begin{proof}
Let $H^{(N)}_i=\mathds{1}\big[\Bar{\sigma}^{(N)}_{2i+1}-\Bar{\sigma}^{(N)}_{2i}>\frac{\eta^2}{64B}N^{-2\varepsilon}\big]$, $i\geq 1$. 
Recall $A_i$, $i=1,2$ as defined in the proof of Lemma~\ref{lem:LEMMA-3}. Also, define the one-dimensional Skorohod map $\Psi$ as follows: for any function $x:\R \to \R$ having c\`adl\`ag paths, $\Psi[x](t) := x(t) + \sup_{0\leq s\leq t} \max\{-x(s), 0\}, \, t \ge 0.$ 
Also write
$$
x^{(N)}(s) := \lfloor\frac{\eta}{2}N^{\frac{1}{2}-\varepsilon}\rfloor +\hat{A}_1\big((N-BN^{\frac{1}{2}+\varepsilon})s\big)-\hat{A}_2\big((N-\beta N^{\frac{1}{2}-\varepsilon})s\big)-BN^{\frac{1}{2}+\varepsilon}s+\beta N^{\frac{1}{2}-\varepsilon}s, \, s \ge 0.
$$
Then, note that
\begin{align}\label{en1}
    &\mathbb{P}\Big(H^{(N)}_i=0\Big)\notag=\mathbb{P}\Big(\sup_{s\in [0,\Bar{\sigma}^{(N)}_{2i+1}-\Bar{\sigma}^{(N)}_{2i}]}\Bar{I}_B^{(N)}\big(\sigma^{(N)}_{2i}+s\big)\geq \eta N^{\frac{1}{2}-\varepsilon}, \Bar{\sigma}^{(N)}_{2i+1}-\Bar{\sigma}^{(N)}_{2i}\leq \frac{\eta^2}{32}N^{-2\varepsilon}\Big)\notag\\
    &\leq \mathbb{P}\Big(\sup_{s\in [0,\frac{\eta^2}{64B}N^{-2\varepsilon}]}\Psi\Big[x^{(N)}\Big](s) \geq \eta N^{\frac{1}{2}-\varepsilon}\Big)
    \le \mathbb{P}\Big(\sup_{s\in [0,\frac{\eta^2}{64B}N^{-2\varepsilon}]}\Big|x^{(N)}(s)\Big| \geq \frac{\eta}{2} N^{\frac{1}{2}-\varepsilon}\Big),
\end{align}
where, in the last step, we used the fact that $\sup_{s \in [0,T]} \Psi[x^{(N)}](s) \le 2 \sup_{s \in [0,T]}|x^{(N)}(s)|$ for any $T \ge 0$.
Now, since $BN^{\frac{1}{2}+\varepsilon}>\beta N^{\frac{1}{2}-\varepsilon}$ as before and $\eta \le 1$, we have
\begin{align*}
\sup_{s\in [0,\frac{\eta^2}{64B}N^{-2\varepsilon}]}\Big|x^{(N)}(s)\Big| &\le \sup_{s\in [0,\frac{\eta^2}{64B}N^{-2\varepsilon}]}\big|\hat{A}_1\big((N-BN^{\frac{1}{2}+\varepsilon})s\big)-\hat{A}_2\big((N-\beta N^{\frac{1}{2}-\varepsilon})s\big)\big| + \frac{\eta}{2}N^{\frac{1}{2}-\varepsilon} + \frac{\eta^2}{64}N^{\frac{1}{2}-\varepsilon}\\
&\le \sup_{s\in [0,\frac{\eta^2}{64B}N^{-2\varepsilon}]}\big|\hat{A}_1\big((N-BN^{\frac{1}{2}+\varepsilon})s\big)-\hat{A}_2\big((N-\beta N^{\frac{1}{2}-\varepsilon})s\big)\big| + \frac{3\eta}{4}N^{\frac{1}{2}-\varepsilon}.
\end{align*}
Using this in \eqref{en1}, we obtain
\begin{align*}
\mathbb{P}\Big(H^{(N)}_i=0\Big)
    &\leq \mathbb{P}\Big(\sup_{s\in [0,\frac{\eta^2}{64B}N^{-2\varepsilon}]}\big|\hat{A}_1\big((N-BN^{\frac{1}{2}+\varepsilon})s\big)-\hat{A}_2\big((N-\beta N^{\frac{1}{2}-\varepsilon})s\big)\big|\geq \frac{\eta}{4}N^{\frac{1}{2}-\varepsilon}\Big)\\
    &\leq \frac{2(\eta^2/64B)N^{1-2\varepsilon}}{(\eta^2/16)N^{1-2\varepsilon}} \le \frac{1}{2},
\end{align*}
where the last inequality follows from Doob's $L^2$-maximal inequality and the assumption $B \ge 1$. Then, for $a\geq \frac{256B}{\eta^2}$ we can write for $t\geq t_0$ sufficiently large,
\begin{align*}
    \mathbb{P}\Big(\bar{K}^{(N)}_t\geq a N^{4\varepsilon}t\Big)&\leq \mathbb{P}\Big(\sum_{i=1}^{\left\lfloor{atN^{4\varepsilon}}\right\rfloor}(\Bar{\sigma}^{(N)}_{2i+1}-\Bar{\sigma}^{(N)}_{2i})<N^{2\varepsilon}t\Big)\\
    &\leq \mathbb{P}\Big(\sum_{i=1}^{\left\lfloor{atN^{4\varepsilon}}\right\rfloor}H^{(N)}_i<\frac{64B}{\eta^2}N^{4\varepsilon}t\Big)\\
    &\leq \mathbb{P}\Big(\sum_{i=1}^{\left\lfloor{atN^{4\varepsilon}}\right\rfloor}(H^{(N)}_i-\mathbb{E}(H^{(N)}_i))<\frac{64B}{\eta^2}N^{4\varepsilon}t-\frac{1}{2}\left\lfloor{atN^{4\varepsilon}}\right\rfloor\Big)\\
    &\leq \mathbb{P}\Big(\sum_{i=1}^{\left\lfloor{atN^{4\varepsilon}}\right\rfloor}(H^{(N)}_i-\mathbb{E}(H^{(N)}_i))<-\frac{1}{8}atN^{4\varepsilon}\Big)\\
    &\leq \exp \{-\frac{\Tilde{c}a^2N^{8\varepsilon}t^2}{aN^{4\varepsilon}t}\}=\exp \{-\Tilde{c}aN^{4\varepsilon}t\}
\end{align*}
for some $\Tilde{c}>0$, where the last inequality above follows from Azuma's inequality. 
\end{proof}
\bigskip
The following lemma gives bounds on a random sum that appears in~\eqref{eq:app:int-IB}. 
It will be used to bound the integral $\int_0^{N^{2\varepsilon}t}\Bar{I}_B^{(N)}(s)ds$.

\begin{lemma}\label{lemma 6}
Take $B_0$ as in Lemma~\ref{lem:LEMMA-1}, $\tilde N_B$ for $B \ge B_0$ as in Lemma~\ref{lem:LEMMA-3}, and $t_0$ as in Lemma~\ref{lem:LEMMA-4}.
There exist constants $c, c^{*}, c'>0$, such that the following holds for all $B\geq B_0$, $N\geq \tilde N_B $, $t\geq t_0$,
$$\mathbb{P}\big(\sum_{i=1}^{\bar{K}^{(N)}_{t}}\Bar{u}^{(N)}_i\Bar{\xi}^{(N)}_i\geq \frac{c^{*}}{\eta^2B^{\frac{1}{2}}}N^{\frac{1}{2}+\varepsilon}t\big)\leq c\exp \{-c'B^{\frac{1}{5}}N^{\frac{4\varepsilon}{5}}t^{\frac{1}{5}}\}.$$
\end{lemma}
\begin{proof}
Write
\begin{equation}\label{eq:lem4.6-1}
    \mathbb{P}\big(\sum_{i=1}^{\bar{K}^{(N)}_t}\Bar{u}^{(N)}_i\Bar{\xi}^{(N)}_i\geq \frac{c^{*}}{\eta^2B^{\frac{3}{2}}}N^{\frac{1}{2}+\varepsilon}t \big)
\leq  \mathbb{P}\big(\bar{K}^{(N)}_t\geq a'N^{4\varepsilon}t\big) + \mathbb{P}\big(\sum_{i=1}^{\lfloor a'N^{4\varepsilon}t\rfloor}\Bar{u}^{(N)}_i\Bar{\xi}^{(N)}_i\geq \frac{c^{*}}{\eta^2B^{\frac{3}{2}}}N^{\frac{1}{2}+\varepsilon}t\big).
\end{equation}
For the first term, we take $a'=\frac{256B}{\eta^2}$, and thus, Lemma~\ref{lem:LEMMA-4} yields
\begin{equation}
    \label{eq:lem4.6-2}
    \mathbb{P}\big(\bar{K}^{(N)}_t\geq a'N^{4\varepsilon}t\big) \leq \exp \{-\Tilde{c}a'N^{4\varepsilon}t\}.
\end{equation}
For the second term, define 
$\Phi_i^{(N)}:= \frac{B^{3/2}}{N^{\frac{1}{2}- 3\varepsilon}}\Bar{u}^{(N)}_i\Bar{\xi}^{(N)}_i, \quad i\geq 1,$
and a sequence of filtrations $\mathcal{F}_{N,i} :=\sigma\{\Phi_j^{(N)} : j \le i\}$, $i\geq 1$.
Then, as $\{\Phi_j^{(N)} : i \in \mathbb{N}\}$ are i.i.d., from Lemma~\ref{lem:LEMMA-3} we know for $B\geq B_0$ and $N\geq \tilde N_B$,
$
\mathbb{P}\big(\Phi_i^{(N)}\geq x|\mathcal{F}_{N,i-1}\big)\leq c_1e^{-c_2\sqrt{x}},\quad\forall x\geq 0.
$
We will use Lemma~\ref{lem:LEMMA-5} for the random variables $\{\Phi_i^{(N)} : i \ge 1, N \in \mathbb{N}\}$. Note that $\mu$ in the lemma takes the form $\mu = \int_0^{\infty}c_1e^{-c_2\sqrt{x}}dx$ in our case. Write $c^*= 4\mu \times 256$. Applying Lemma~\ref{lem:LEMMA-5} with $\theta = 1/2$, $n = \lfloor a'N^{4\varepsilon}t\rfloor$, and $a = 4\mu$, we get for some constants $c'_1, c'_2>0$,
\begin{equation}
    \begin{split}\label{eq:lem4.6-3}
         \mathbb{P}\big(\sum_{i=1}^{\lfloor a'N^{4\varepsilon}t\rfloor}&\Bar{u}^{(N)}_i\Bar{\xi}^{(N)}_i\geq \frac{c^{*}}{\eta^2B^{\frac{1}{2}}}N^{\frac{1}{2}+\varepsilon}t|\mathcal{F}_{N,0}\big)
    = \mathbb{P}\big(\sum_{i=1}^{\lfloor a'N^{4\varepsilon}t\rfloor} \frac{N^{\frac{1}{2}- 3\varepsilon}}{B^{\frac{3}{2}}}\Phi_i^{(N)}
    \geq \frac{c^{*}}{\eta^2B^{\frac{1}{2}}}N^{\frac{1}{2}+\varepsilon}t|\mathcal{F}_{N,0}\big)\\
    &= \mathbb{P}\Big(\sum_{i=1}^{\lfloor a'N^{4\varepsilon}t \rfloor} \Phi_i^{(N)}
    \geq \frac{c^{*}B}{\eta^2}N^{4\varepsilon}t|\mathcal{F}_{N,0}\Big)
    \leq c'_1 \exp \Big\{-c'_2 \Big[\Big(4\mu\Big)^2 a' N^{4\varepsilon}t\Big]^{1/5}\Big\}.
    \end{split}
\end{equation}
Plugging the bounds from~\eqref{eq:lem4.6-2} and~\eqref{eq:lem4.6-3} into~\eqref{eq:lem4.6-1} completes the proof of the lemma for appropriately chosen constants~$c$ and $c'$ dependent on $\beta$ but not $B$.
\end{proof}

We now have all the required ingredients to prove Lemma~\ref{lem:LEMMA-6}.

\begin{proof}[Proof of Lemma~\ref{lem:LEMMA-6}.]
(i) Recall $\eta=1 \wedge (\beta/8)$ and
    \begin{equation}\label{intbd}
    \int_0^{N^{2\varepsilon}t}\Bar{I}_B^{(N)}(s)ds 
    \leq  \eta N^{\frac{1}{2}+\varepsilon} t+ \sum_{i=1}^{\bar{K}^{(N)}_{t}}\Bar{u}^{(N)}_i\Bar{\xi}^{(N)}_i
    \le \frac{\beta}{8}N^{\frac{1}{2}+\varepsilon}t + \sum_{i=1}^{\bar{K}^{(N)}_{t}}\Bar{u}^{(N)}_i\Bar{\xi}^{(N)}_i.
    \end{equation}
Now, choose $B_1\geq B_0$ (from Lemma~\ref{lem:LEMMA-1}) such that $\frac{c^{*}}{\eta^2 B_1^{\frac{1}{2}}}\leq \frac{\beta}{8}$, where $c^{*}$ is the constant from Lemma~\ref{lemma 6}. 
By Lemma \ref{lemma 6} and \eqref{intbd}, there exist constants $c_1$ and $c_2$, such that for all $B\geq B_1$, $t\geq t_0$, $N\geq \tilde N_B$,
    \begin{align}\label{eq:lem4.7-int-IB}
        \mathbb{P}\Big(\int_0^{N^{2\varepsilon}t}\Bar{I}_B^{(N)}(s)ds \ge \frac{\beta}{4}N^{\frac{1}{2}+\varepsilon}t\Big)&\leq \mathbb{P}\Big(\sum_{i=1}^{\bar{K}^{(N)}_{t}}\Bar{u}^{(N)}_i\Bar{\xi}^{(N)}_i\geq \frac{\beta}{8} N^{\frac{1}{2}+\varepsilon}t\Big)\nonumber\\
        &\leq \mathbb{P}\Big(\sum_{i=1}^{\bar{K}^{(N)}_{t}}\Bar{u}^{(N)}_i\Bar{\xi}^{(N)}_i\geq \frac{c^{*}}{\eta^2B^{\frac{1}{2}}}N^{\frac{1}{2}+\varepsilon}t\Big)
        \leq c_1\exp \{-c_2 B^{\frac{1}{5}}  N^{\frac{4\varepsilon}{5}}t^{\frac{1}{5}}\}.
    \end{align}
    For any $t\geq t_0$, define $s_k=kt,k\geq 1$. Then, for any $k\geq 1$,
    \begin{align*}
        &\mathbb{P}\Big(\int_0^{N^{2\varepsilon}s}\Bar{I}_B^{(N)}(u)du \ge \frac{\beta s}{2}N^{\frac{1}{2}+\varepsilon}\text{ for some }s \in [s_k,s_{k+1})\Big)\\
        &\leq  \mathbb{P}\Big(\int_0^{N^{2\varepsilon}s_{k+1}}\Bar{I}_B^{(N)}(u)du \ge \frac{\beta s_{k+1}}{4}N^{\frac{1}{2}+\varepsilon}\Big) \leq c_1\exp \{-c_2 B^{\frac{1}{5}}  N^{\frac{4\varepsilon}{5}}s_{k+1}^{\frac{1}{5}}\}.
    \end{align*}
    Therefore, \ref{7i} follows on summing over $k$ and applying the union bound.\\

\noindent (ii)
 Write $a=\frac{256B}{\eta^2}$, and note that 
    \begin{equation}\label{eq:lem4.7-sup-IB}
        \mathbb{P}\Big(\sup_{s\leq t}\Bar{I}_B^{(N)}(N^{2\varepsilon}s)\geq \frac{\beta}{4}N^{\frac{1}{2}-\varepsilon}t\Big)
        \leq \mathbb{P}\Big(\bar{K}^{(N)}_t\geq a N^{4\varepsilon}t\Big)+\mathbb{P}\Big(\sup_{1\leq i\leq aN^{4\varepsilon}t}\Bar{u}^{(N)}_i\geq \frac{\beta}{4}N^{\frac{1}{2}-\varepsilon}t - \eta N^{\frac{1}{2}-\varepsilon}\Big).
    \end{equation}
    By Lemma~\ref{lem:LEMMA-4}, for $t\geq t_0$, 
\begin{equation}\label{eq:lem4.7-kt}
    \mathbb{P}\big(\bar{K}^{(N)}_t\geq a N^{4\varepsilon}t\big)\leq \exp \{-\tilde{c}aN^{4\varepsilon}t\}.
\end{equation}
Recall the upper bounds from Equations~\eqref{eq:lem4.3-2} and~\eqref{eq:sup-IB}, and note that for $B\geq B_0$, and large enough $N\geq \tilde N_B$,
\begin{align}\label{eq:lem4.7-sup-u}
&\mathbb{P}\Big(\sup_{1\leq i\leq aN^{4\varepsilon}t}\Bar{u}^{(N)}_i\geq \frac{\beta}{4}N^{\frac{1}{2}-\varepsilon}t - \eta N^{\frac{1}{2}-\varepsilon}\Big)\nonumber\\
    \le&\mathbb{P}\Big(\sup_{1\leq i\leq aN^{4\varepsilon}t}\Bar{u}^{(N)}_i\geq \frac{\beta}{8}N^{\frac{1}{2}-\varepsilon}t\Big)\nonumber\\
    \leq& aN^{4\varepsilon}t\cdot\Big[\mathbb{P}\big(\Bar{u}^{(N)}_1\geq \frac{\beta}{8}N^{\frac{1}{2}-\varepsilon}t, \bar{\xi}^{(N)}_1< \frac{\beta tN^{-2\varepsilon}}{8\sqrt{B}}\big)+\mathbb{P}\big(\bar{\xi}^{(N)}_1\geq \frac{\beta t N^{-2\varepsilon}}{8\sqrt{B}}\big)\Big]\nonumber\\
    \leq& aN^{4\varepsilon}t\cdot c_0 \exp \{-c'_0 \beta\sqrt{B} t\}\text{ for some constants }c_0, c'_0>0,
\end{align}
where the last inequality is due to~\eqref{eq:lem4.3-2} and~\eqref{eq:sup-IB}.
Plugging \eqref{eq:lem4.7-kt} and \eqref{eq:lem4.7-sup-u} into \eqref{eq:lem4.7-sup-IB} and choosing appropriate constants $\tilde{c}_1,\tilde{c}_2,\tilde{c}_3$ complete the proof.
\end{proof}

\section{Proofs from Section \ref{sec:HITTIME} (sample path analysis).}\label{sampleapp}

Here, we give proofs of some lemmas that were used to prove the main results in Section \ref{sec:HITTIME}.

\begin{proof}[Proof of Lemma \ref{lem:B1}.]
Fix $N^{\frac{1}{2}-\varepsilon}\geq x\geq x_B$, where $x_B \ge 2B$ will be chosen later. 
Define the stopping time 
$\hat{\sigma}^{(N)}_i\coloneqq\inf\{s\geq \sigma^{(N)}_{2i}:Q^{(N)}_2(s)\geq xN^{\frac{1}{2}+\varepsilon}\text{ or }Q^{(N)}_2(s)\leq \lfloor BN^{\frac{1}{2}+\varepsilon}\rfloor\},\ i\geq 0.$
Note that, as $x \le N^{\frac{1}{2}-\varepsilon}$, when $Q^{(N)}_3(\sigma^{(N)}_{2i})=0$, $Q^{(N)}_3(t)=0$ for all $t\in[\sigma^{(N)}_{2i},\hat{\sigma}^{(N)}_i]$, and
in that case, we have $S^{(N)}(t)=N-I^{(N)}(t)+Q^{(N)}_2(t)$. 
Therefore, we obtain for $t\in[0,\hat{\sigma}^{(N)}_i-\sigma^{(N)}_{2i}]$,
\begin{align*}
    &Q^{(N)}_2(t+\sigma^{(N)}_{2i})-I^{(N)}(t+\sigma^{(N)}_{2i})-(Q^{(N)}_2(\sigma^{(N)}_{2i})-I^{(N)}(\sigma^{(N)}_{2i})) =S^{(N)}(t+\sigma^{(N)}_{2i})-S^{(N)}(\sigma^{(N)}_{2i})\\
    &=\left[A\br{(N-\beta N^{\frac{1}{2}-\varepsilon})(t+ \sigma^{(N)}_{2i})} - A\br{(N-\beta N^{\frac{1}{2}-\varepsilon})\sigma^{(N)}_{2i}}\right] - D\br{\int_{\sigma^{(N)}_{2i}}^{\sigma^{(N)}_{2i}+t}(N-I^{(N)}(s))ds}\\
    &=\left[\hat A\br{(N-\beta N^{\frac{1}{2}-\varepsilon})(t+ \sigma^{(N)}_{2i})} - \hat A\br{(N-\beta N^{\frac{1}{2}-\varepsilon})\sigma^{(N)}_{2i}}\right] -  \hat D\br{\int_{\sigma^{(N)}_{2i}}^{\sigma^{(N)}_{2i}+t}(N-I^{(N)}(s))ds}\\
    &\qquad +\int_{\sigma^{(N)}_{2i}}^{\sigma^{(N)}_{2i}+t}I^{(N)}(s)ds-\beta N^{\frac{1}{2}-\varepsilon}t,
\end{align*}
where $\hat{A}(s)=A(s)-s$ and $\hat{D}(s)=D(s)-s$. 
Using the above and the strong Markov property at time $\sigma^{(N)}_{2i}$, we can write
\begin{equation}\label{eq:B1-1}
    \begin{split}
        &\mathbb{P}\Big(\sup_{s\in[\sigma^{(N)}_{2i},\sigma^{(N)}_{2i+1}]}Q^{(N)}_2(s)\geq xN^{\frac{1}{2}+\varepsilon}\ \Big|\ Q^{(N)}_3(\sigma^{(N)}_{2i})=0\Big)\\
    &\leq \mathbb{P}\br{\hat{\sigma}^{(N)}_i-\sigma^{(N)}_{2i}\geq N^{2\varepsilon}t \ \Big|\ Q^{(N)}_3(\sigma^{(N)}_{2i})=0}\\
    & \qquad +\mathbb{P}\Big(\sup_{s\in[\sigma^{(N)}_{2i},\sigma^{(N)}_{2i} + (N^{2\varepsilon} t) \wedge (\hat{\sigma}^{(N)}_i-\sigma^{(N)}_{2i})]}(Q^{(N)}_2(s)-I^{(N)}(s))\geq xN^{\frac{1}{2}+\varepsilon}\ \Big|\ Q^{(N)}_3(\sigma^{(N)}_{2i})=0\Big)\\
    &\leq \mathbb{P}_{(0, \, \lfloor 2BN^{\frac{1}{2}+\varepsilon} \rfloor, \,\underline{0})}\big(\tau^{(N)}_2(B)\geq N^{2\varepsilon}t\big)\\
    &\qquad +\mathbb{P}_{(0, \, \lfloor 2BN^{\frac{1}{2}+\varepsilon} \rfloor, \,\underline{0})}\Big(\sup_{s\in[0,N^{2\varepsilon}t \wedge \tau_2^{(N)}(B)]}\left(2BN^{\frac{1}{2}+\varepsilon}+\hat{\mathcal{M}}(s)+\int_{0}^{s}I^{(N)}(u)du-\beta N^{\frac{1}{2}-\varepsilon}s\right) \geq xN^{\frac{1}{2}+\varepsilon} \Big),
    \end{split}
\end{equation}
 where $\hat{\mathcal{M}}(s) := \hat A\br{(N-\beta N^{\frac{1}{2}-\varepsilon})s} -  \hat D\br{\int_{0}^{s}(N-I^{(N)}(u))du}$, and the first equality above is due to the fact that $Q^{(N)}_2$ can increase only when $I^{(N)}=0$. 
Now, for the first term in the above bound, observe that by Proposition \ref{prop:DOWNCROSS}, for $B\geq B_1$, $N\geq \tilde N_B$, and $t\geq t_0\vee \frac{8B}{\beta}$, we have 
\begin{equation}\label{eq:Lemma4.12-1}
    \mathbb{P}_{(0, \, \lfloor 2BN^{\frac{1}{2}+\varepsilon} \rfloor, \,\underline{0})}\big(\tau^{(N)}_2(B)\geq N^{2\varepsilon}t\big)\leq 4e^{-c'_0t}+c'_1\exp \{-c'_2 B^{\frac{1}{5}}N^{\frac{4\varepsilon}{5}}t^{\frac{1}{5}}\}+c'_3N^{4\varepsilon}t\exp \{-c'_4\sqrt{B}N^{2\varepsilon}t\}.
\end{equation}
For the second term on the right side of \eqref{eq:B1-1},
\begin{align*}
    &\mathbb{P}_{(0, \, \lfloor 2BN^{\frac{1}{2}+\varepsilon} \rfloor, \,\underline{0})}\Big(\sup_{s\in[0,N^{2\varepsilon}t\wedge \tau_2^{(N)}(B)]}\left(2BN^{\frac{1}{2}+\varepsilon}+\hat{\mathcal{M}}(s)+\int_{0}^{s}I^{(N)}(u)du-\beta N^{\frac{1}{2}-\varepsilon}s\right) \geq xN^{\frac{1}{2}+\varepsilon}\Big)\\
    &\qquad\leq \mathbb{P}\Big(\int_0^{N^{2\varepsilon}t}\Bar{I}_B^{(N)}(u)du>\frac{\beta}{2}N^{\frac{1}{2}+\varepsilon}t\Big) + \mathbb{P}\Big(\sup_{s\in[0,N^{2\varepsilon}t]}\hat{\mathcal{M}}(s)\geq (x-\frac{\beta t}{2}-2B)N^{\frac{1}{2}+\varepsilon}\Big).
\end{align*}
Thus, using Lemma \ref{lem:LEMMA-6} (i) for the first term and Lemma \ref{lem:sup-poi} for the second term in the above bound,
we have that  for $N\geq \tilde N_B$, $t\geq t_0$ and $x-\frac{\beta t}{2}-2B\geq0$,
\begin{align}\label{eq:Lemma4.12-2}
    &\mathbb{P}_{(0, \, \lfloor 2BN^{\frac{1}{2}+\varepsilon} \rfloor, \,\underline{0})}\Big(\sup_{s\in[0,N^{2\varepsilon}t]}\left(2BN^{\frac{1}{2}+\varepsilon}+\hat{\mathcal{M}}(s)+\int_{0}^{s}I^{(N)}(u)du-\beta N^{\frac{1}{2}-\varepsilon}s\right) \geq xN^{\frac{1}{2}+\varepsilon}\Big)\nonumber\\
    \leq &c_1\exp \{-c_2 B^{\frac{1}{5}}N^{\frac{4\varepsilon}{5}}t^{\frac{1}{5}}\}+ 2\exp \{-c_3(x- \beta t/2-2B)^2/t\},
\end{align}
for positive constants $c_1,c_2,c_3$ not depending on $x,t,N$.

Thus, taking $t=\frac{x}{\beta}$, for $x\geq x_B :=(t_0\beta \vee8B)$, and putting the upper bounds in \eqref{eq:Lemma4.12-1} and \eqref{eq:Lemma4.12-2} into \eqref{eq:B1-1}, the result holds for appropriately chosen constants.
\end{proof}

\section{Proofs from Section \ref{sec:STEAYSTATE} (steady state analysis).}\label{app:steady}
We start by proving the following lemma that \blue{bounds the number of toggles of $Q^{(N)}_2$ for $\bar{Q}^{(N)}_3$ to fall below level $\frac{B}{4\beta}N^{2\varepsilon}$, which will be used to prove Lemma~\ref{lem:LEMMA-5.1}}. Define $\bar{K}^{*}\coloneqq \inf \{k\geq 1: \Bar{Q}^{(N)}_3(\sigma^{(N)}_{2k})\leq \frac{B}{4\beta}N^{2\varepsilon}\}$.

\begin{lemma}\label{lem:bar-q3-positive}
 There exist $ N_0, c^*_1,c^*_2>0$, such that for all $N\geq N_0,k\geq 1$,
$$\sup_{0\leq\underline{z}\leq \frac{B}{4\beta}N^{2\varepsilon}}\mathbb{P}_{(0, \, \lfloor 2BN^{\frac{1}{2}+\varepsilon} \rfloor, \,\underline{z})}\big(\bar{K}^{*}\geq 1+kN^{\frac{1}{2}-\varepsilon}\big)\leq c^*_1\exp \big\{-c^*_2N^{(\frac{1}{2}-\varepsilon)/11}\big\}\exp \Big\{-c^*_2\big(k/N^{\frac{1}{2}-\varepsilon}\big)^{\frac{1}{11}}\Big\}.$$
\end{lemma}
\begin{proof}
Define $\Bar{Z}^{(N)}_i\coloneqq \Bar{Q}^{(N)}_3(\sigma^{(N)}_{2i+2})-\Bar{Q}^{(N)}_3(\sigma^{(N)}_{2i}),i\geq 0$.
Note that 
\begin{align}\label{eq:lemma5.1-1}
    &\sup_{\underline{z}: \, \sum_iz_i \leq \frac{B}{4\beta}N^{2\varepsilon}}\mathbb{P}_{(0, \, \lfloor 2BN^{\frac{1}{2}+\varepsilon} \rfloor, \,\underline{z})}\big(\Bar{K}^{*}\geq 1+kN^{\frac{1}{2}-\varepsilon}\big)\nonumber\\
    \leq &\sup_{\underline{z}: \, \sum_iz_i \leq \frac{B}{4\beta}N^{2\varepsilon}}\mathbb{P}_{(0, \, \lfloor 2BN^{\frac{1}{2}+\varepsilon} \rfloor, \,\underline{z})}\big(\sum_{j=0}^{i}\Bar{Z}^{(N)}_j>\frac{B}{4\beta}N^{2\varepsilon}-\bar{Q}^{(N)}_3(0),\forall \, 0\leq i\leq \lfloor kN^{\frac{1}{2}-\varepsilon}\rfloor -1\big).
\end{align}
To estimate \eqref{eq:lemma5.1-1}, define $\Phi^{(N)}_i\coloneqq \mathds{1}\big[\Bar{Z}^{(N)}_i\leq -\big(\frac{B}{4\beta}N^{2\varepsilon}\wedge \Bar{Q}^{(N)}_3(\sigma^{(N)}_{2i})\big) \big]$, $i\geq 0$. By Lemma \ref{lem:LEMMA-8} (ii), for sufficiently large $N$,
\begin{equation}\label{eq:lemma5.1-3}
    \inf_{\underline{z}}\mathbb{E}_{(0,2BN^{\frac{1}{2}+\varepsilon},\underline{z})}\big(\Phi^{(N)}_i|\mathcal{F}_{\sigma^{(N)}_{2i}}\big)\geq \frac{1}{2},\quad i\geq 0.
\end{equation}
Using Azuma-Hoeffding inequality, we get for any $k\geq 1$,
\begin{equation}\label{eq:5.2-b}
\begin{split}
     &\sup_{\underline{z}: \, \sum_iz_i \leq \frac{B}{4\beta}N^{2\varepsilon}}\mathbb{P}_{(0, \, \lfloor 2BN^{\frac{1}{2}+\varepsilon} \rfloor, \,\underline{z})}\Big(\sum_{j=0}^{\lfloor kN^{\frac{1}{2}-\varepsilon}\rfloor -1}\Phi^{(N)}_j\leq \frac{kN^{\frac{1}{2}-\varepsilon}}{3}\Big)\\
    \leq & \sup_{\underline{z}: \, \sum_iz_i \leq \frac{B}{4\beta}N^{2\varepsilon}}\mathbb{P}_{(0, \, \lfloor 2BN^{\frac{1}{2}+\varepsilon} \rfloor, \,\underline{z})}\Big(\sum_{j=0}^{\lfloor kN^{\frac{1}{2}-\varepsilon}\rfloor-1}\big(\Phi^{(N)}_j-\mathbb{E}\big(\Phi^{(N)}_j\ \big|\ \mathcal{F}_{\sigma^{(N)}_{2j}}\big)\big)\leq -\frac{kN^{\frac{1}{2}-\varepsilon}}{6}\Big)\leq e^{-ckN^{\frac{1}{2}-\varepsilon}},
\end{split}
\end{equation}
for some $c>0$,
where in the first inequality, we have used \eqref{eq:lemma5.1-3}.
Therefore, for $k\geq 1$,
\begin{equation}\label{eq:sum-zj-z0}
    \begin{split}
        \sup_{\underline{z}: \, \sum_iz_i \leq \frac{B}{4\beta}N^{2\varepsilon}}&\mathbb{P}_{(0, \, \lfloor 2BN^{\frac{1}{2}+\varepsilon} \rfloor, \,\underline{z})}\big(\sum_{j=0}^{i}\Bar{Z}^{(N)}_j>\frac{B}{4\beta}N^{2\varepsilon}-\bar{Q}^{(N)}_3(0),\forall \, 0\leq i\leq \lfloor kN^{\frac{1}{2}-\varepsilon}\rfloor -1\big)\\
    \leq &  \sup_{\underline{z}: \, \sum_iz_i \leq \frac{B}{4\beta}N^{2\varepsilon}}\mathbb{P}_{(0, \, \lfloor 2BN^{\frac{1}{2}+\varepsilon} \rfloor, \,\underline{z})}\Big(\sum_{j=0}^{ \lfloor kN^{\frac{1}{2}-\varepsilon}\rfloor-1}[\Bar{Z}^{(N)}_j]_+-\sum_{j=0}^{ \lfloor kN^{\frac{1}{2}-\varepsilon}\rfloor -1}\Big(\frac{B}{4\beta}N^{2\varepsilon}\Big)\Phi^{(N)}_j>0\Big)\\
    \leq & \sup_{\underline{z}: \, \sum_iz_i \leq \frac{B}{4\beta}N^{2\varepsilon}}\mathbb{P}_{(0, \, \lfloor 2BN^{\frac{1}{2}+\varepsilon} \rfloor, \,\underline{z})}\Big(\sum_{j=0}^{ \lfloor kN^{\frac{1}{2}-\varepsilon}\rfloor-1}\Phi^{(N)}_j\leq \frac{kN^{\frac{1}{2}-\varepsilon}}{3}\Big)\\
    &\hspace{4cm}+\sup_{\underline{z}: \, \sum_iz_i \leq \frac{B}{4\beta}N^{2\varepsilon}}\mathbb{P}_{(0, \, \lfloor 2BN^{\frac{1}{2}+\varepsilon} \rfloor, \,\underline{z})}\Big(\sum_{j=0}^{\lfloor kN^{\frac{1}{2}-\varepsilon}\rfloor-1}[\Bar{Z}^{(N)}_j]_+\geq \frac{B}{12\beta}N^{\frac{1}{2}+\varepsilon}k\Big)\\
    \leq & e^{-ckN^{\frac{1}{2}-\varepsilon}}+\sup_{\underline{z}: \, \sum_iz_i \leq \frac{B}{4\beta}N^{2\varepsilon}}\mathbb{P}_{(0, \, \lfloor 2BN^{\frac{1}{2}+\varepsilon} \rfloor, \,\underline{z})}\Big(\sum_{j=0}^{\lfloor kN^{\frac{1}{2}-\varepsilon}\rfloor-1}[\Bar{Z}^{(N)}_j]_+\geq \frac{B}{12\beta}N^{\frac{1}{2}+\varepsilon}k\Big),
    \end{split}
\end{equation}
where the last inequality uses~\eqref{eq:5.2-b}.
Note that, by Lemma \ref{lem:LEMMA-8} (i), for any $j\geq 1, x>0$, and sufficient large~$N$,
\begin{equation}\label{eq:lemma5.1-4}
    \sup_{\underline{z}: \, \sum_iz_i \leq \frac{B}{4\beta}N^{2\varepsilon}} \mathbb{P}\Big(\frac{[\Bar{Z}^{(N)}_j]_+}{N^{\frac{1}{2}+\varepsilon}}\geq x\ \big|\ \mathcal{F}_{\sigma^{(N)}_{2j}}\Big)\leq c_1e^{-c_2N^{(\frac{1}{2}-\varepsilon)/5}}e^{-c_3x^{1/5}},
\end{equation}
where $c_1,c_2,c_3>0$ are constants.
Choose $N_1$ sufficiently large such that \eqref{eq:lemma5.1-3} and \eqref{eq:lemma5.1-4} hold, and for $N\geq N_1$,
$$\int_0^{\infty}c_1e^{-c_2N^{(\frac{1}{2}-\varepsilon)/5}}e^{-c_3x^{1/5}}dx\leq \frac{B}{4\times 12\beta N^{\frac{1}{2}-\varepsilon}}.$$
Take any $N\geq N_1$, and $k\geq 1$. 
We will use Lemma \ref{lem:LEMMA-5}  with
$\theta =\frac{1}{5}, \, r=\underline{z}, \, R = \{\underline{z}: \, \sum_iz_i \leq \frac{B}{4\beta}N^{2\varepsilon}\}, \, a=\frac{B}{12\beta N^{\frac{1}{2}-\varepsilon}}, \, n= \lfloor kN^{\frac{1}{2}-\varepsilon}\rfloor,$
and 
$\Phi^{(r)}_j=\frac{[\Bar{Z}^{(N)}_j]_+}{N^{\frac{1}{2}+\varepsilon}},$
with starting configuration $$(I^{(N)}(0), Q_2^{(N)}(0), \bar{Q}^{(N)}_3(0)) = (0, \, \lfloor 2BN^{\frac{1}{2}+\varepsilon} \rfloor, \,\underline{z}),$$ and associated filtration $\big\{\mathcal{F}^{(r)}_j : j\geq 1, r \in R\big\}$ being the natural filtration generated by the above random variables.
We get
\begin{align}\label{eq:lemma5.1-5}
   \sup_{\underline{z}: \, \sum_iz_i \leq \frac{B}{4\beta}N^{2\varepsilon}}&\mathbb{P}_{(0, \, \lfloor 2BN^{\frac{1}{2}+\varepsilon} \rfloor, \,\underline{z})}\Big(\sum_{j=0}^{\lfloor kN^{\frac{1}{2}-\varepsilon}\rfloor-1}[\Bar{Z}^{(N)}_j]_+\geq \frac{B}{12\beta}N^{\frac{1}{2}+\varepsilon}k\Big)\nonumber\\
    &\le \sup_{r\in R}\mathbb{P}\Big(\sum_{j=0}^{n}\Phi^{(r)}_j\geq an\Big)
    \leq  c'_1\Big(1+\frac{\big(\lfloor kN^{\frac{1}{2}-\varepsilon}\rfloor\big)^{5/11}}{\big(\frac{B}{12\beta N^{\frac{1}{2}-\varepsilon}}\big)^{\frac{1}{11}}}\Big)\exp \Big[-c'_2\Big(\frac{k}{N^{\frac{1}{2}-\varepsilon}}\Big)^{1/11}\Big]\\\nonumber
    &\leq  c_1''N^{\frac{1}{2}-\varepsilon}\exp\Big[-c''_2\Big(\frac{k}{N^{\frac{1}{2}-\varepsilon}}\Big)^{1/11}\Big].\nonumber
\end{align}
Moreover,  for any $k\geq 1$, from \eqref{eq:lemma5.1-4},
\begin{align}\label{eq:sum-z}
     \sup_{\underline{z}: \, \sum_iz_i \leq \frac{B}{4\beta}N^{2\varepsilon}}&\mathbb{P}_{(0, \, \lfloor 2BN^{\frac{1}{2}+\varepsilon} \rfloor, \,\underline{z})}\Big(\sum_{j=0}^{\lfloor kN^{\frac{1}{2}-\varepsilon}\rfloor-1}[\Bar{Z}^{(N)}_j]_+\geq \frac{B}{12\beta}N^{\frac{1}{2}+\varepsilon}k\Big)\nonumber\\
     &\leq kN^{\frac{1}{2}-\varepsilon} \sup_{\underline{z}: \, \sum_iz_i \leq \frac{B}{4\beta}N^{2\varepsilon}}\mathbb{P}\Big([\Bar{Z}^{(N)}_0]_+\geq \frac{B}{12\beta}N^{2\varepsilon}\Big)\leq  kN^{\frac{1}{2}-\varepsilon}c_1\exp \big\{-c_2N^{(\frac{1}{2}-\varepsilon)/5}\big\}.
\end{align}
Hence, \eqref{eq:lemma5.1-5} and \eqref{eq:sum-z} imply
\begin{align}\label{nex0}
    \sup_{\underline{z}: \, \sum_iz_i \leq \frac{B}{4\beta}N^{2\varepsilon}}&\mathbb{P}_{(0, \, \lfloor 2BN^{\frac{1}{2}+\varepsilon} \rfloor, \,\underline{z})}\Big(\sum_{j=0}^{\lfloor kN^{\frac{1}{2}-\varepsilon}\rfloor-1}[\Bar{Z}^{(N)}_j]_+\geq \frac{B}{12\beta}N^{\frac{1}{2}+\varepsilon}k\Big)\nonumber\\
    \leq& \min \Big\{c_1''N^{\frac{1}{2}-\varepsilon}\exp\Big[-c''_2\Big(\frac{k}{N^{\frac{1}{2}-\varepsilon}}\Big)^{1/11}\Big], \, kN^{\frac{1}{2}-\varepsilon}c_1\exp \big\{-c_2N^{(\frac{1}{2}-\varepsilon)/5}\big\}\Big\}.
\end{align}
Note that, there exists a constant $N_2$, such that for all $N\geq N_2$ and $k\leq N^{1-2\varepsilon}$,
\begin{align*}
    kN^{\frac{1}{2}-\varepsilon}c_1\exp \big[-c_2N^{(\frac{1}{2}-\varepsilon)/5}\big]
    =&\Big[kN^{\frac{1}{2}-\varepsilon}c_1\exp \big\{-\frac{c_2}{2}N^{(\frac{1}{2}-\varepsilon)/5}\big\}\Big]\exp \big[-\frac{c_2}{2}N^{(\frac{1}{2}-\varepsilon)/5}\big]\\
    \leq & \exp \big[-\frac{c_2}{4}k^{\frac{1}{10}}\big]\exp \big[-\frac{c_2}{2}N^{(\frac{1}{2}-\varepsilon)/5})\big],
\end{align*}
and for $k\geq N^{1-2\varepsilon}$,
$$
    N^{\frac{1}{2}-\varepsilon}\exp\Big[-c''_2\Big(\frac{k}{N^{\frac{1}{2}-\varepsilon}}\Big)^{1/11}\Big]
    \leq  N^{\frac{1}{2}-\varepsilon}\exp\Big[-\frac{c''_2}{2}N^{(\frac{1}{2}-\varepsilon)/11}\Big]\exp\Big[-\frac{c''_2}{2}\Big(\frac{k}{N^{\frac{1}{2}-\varepsilon}}\Big)^{1/11}\Big].
$$
Hence, there exist constants $\tilde{c}, \tilde{c}'>0$, such that for all $N\geq N_1\vee N_2$ and $k\geq 1$,
\begin{multline}\label{eq:sum-z-bar-plus}
    \sup_{\underline{z}: \, \sum_iz_i \leq \frac{B}{4\beta}N^{2\varepsilon}}\mathbb{P}_{(0, \, \lfloor 2BN^{\frac{1}{2}+\varepsilon} \rfloor, \,\underline{z})}\Big(\sum_{j=0}^{\lfloor kN^{\frac{1}{2}-\varepsilon}\rfloor-1}[\Bar{Z}^{(N)}_j]_+\geq \frac{B}{12\beta}N^{\frac{1}{2}+\varepsilon}k\Big)\\
    \leq \tilde{c}\exp \big\{-\tilde{c}'N^{(\frac{1}{2}-\varepsilon)/11}\big\}\exp \Big\{-\tilde{c}'\big(k/N^{\frac{1}{2}-\varepsilon}\big)^{\frac{1}{11}}\Big\}.
\end{multline}
Finally, using \eqref{eq:lemma5.1-1} and plugging \eqref{eq:sum-z-bar-plus} into \eqref{eq:sum-zj-z0}, we have 
\begin{align*}
    &\sup_{\underline{z}: \, \sum_iz_i \leq \frac{B}{4\beta}N^{2\varepsilon}}\mathbb{P}_{(0, \, \lfloor 2BN^{\frac{1}{2}+\varepsilon} \rfloor, \,\underline{z})}\big(\Bar{K}^{*}\geq 1+kN^{\frac{1}{2}-\varepsilon}\big)\nonumber\\
    \leq &  \sup_{\underline{z}: \, \sum_iz_i \leq \frac{B}{4\beta}N^{2\varepsilon}}\mathbb{P}_{(0, \, \lfloor 2BN^{\frac{1}{2}+\varepsilon} \rfloor, \,\underline{z})}\big(\sum_{j=0}^{i}\Bar{Z}^{(N)}_j>\frac{B}{4\beta}N^{2\varepsilon}-\bar{Q}^{(N)}_3(0),\forall \, 0\leq i\leq \lfloor kN^{\frac{1}{2}-\varepsilon}\rfloor -1\big)\\
    \leq & c^*_1\exp \big\{-c^*_2N^{(\frac{1}{2}-\varepsilon)/11}\big\}\exp \Big\{-c^*_2\big(k/N^{\frac{1}{2}-\varepsilon}\big)^{\frac{1}{11}}\Big\},
\end{align*}
where $c^*_1$ and $c^*_2$ are appropriate constants.
\end{proof}
\bigskip
\begin{proof}[Proof of Lemma~\ref{lem:LEMMA-5.1}.]
Define $\bar{K}^{*}_0\coloneqq0$ and for $j\geq 0$,
$
    \bar{K}^{*}_{j+1}\coloneqq \inf \big\{l\geq \bar{K}^{*}_{j}+1:\Bar{Q}^{(N)}_3(\sigma^{(N)}_{2l})\leq \frac{B}{4\beta}N^{2\varepsilon}\big\}.
$
Define $\chi^{(N)}_j\coloneqq  \mathds{1}\big[\Bar{Q}^{(N)}_3(\sigma^{(N)}_{2\bar{K}_j^{*}+2})=0\big],\ j\geq 0$.
By Lemma \ref{lem:LEMMA-8} (ii), \begin{equation}\label{eq:lemma5.1-12}
    \mathbb{P}_{(0, \, \lfloor 2BN^{\frac{1}{2}+\varepsilon} \rfloor, \,\underline{0})}\big(\chi^{(N)}_j=1|\mathcal{F}_{\sigma^{(N)}_{2\bar{K}^{*}_j}}\big)\geq \frac{1}{2},\quad \forall j\geq 0.
\end{equation}
Thus, there exist constants $c_2^*, N_0>0$, such that for $k\geq 1, N\geq N_0$,
\begin{align*}
    &\mathbb{P}_{(0, \, \lfloor 2BN^{\frac{1}{2}+\varepsilon} \rfloor, \,\underline{0})}\big(\bar{K}^{(N)}\geq 1+k+k^2N^{\frac{1}{2}-\varepsilon}\big)\\
    \leq & \sum_{j=1}^k\mathbb{P}_{(0, \, \lfloor 2BN^{\frac{1}{2}+\varepsilon} \rfloor, \,\underline{0})}\big(\bar{K}^{*}_{j+1}-\bar{K}^{*}_j\geq 1+kN^{\frac{1}{2}-\varepsilon}\big)\\
    &+\mathbb{P}_{(0, \, \lfloor 2BN^{\frac{1}{2}+\varepsilon} \rfloor, \,\underline{0})}\big(\chi^{(N)}_j=0, \forall 0\leq j\leq k-1\big)\\
    \leq & k c^*_1\exp \big\{-c^*_2N^{(\frac{1}{2}-\varepsilon)/11}\big\}\exp \Big\{-c^*_2\big(k/N^{\frac{1}{2}-\varepsilon}\big)^{\frac{1}{11}}\Big\}+2^{-k},
\end{align*}
where the last inequality comes from Lemma~\ref{lem:bar-q3-positive} and the inequality in~\eqref{eq:lemma5.1-12}.
\end{proof}

\section{Auxiliary results for the proof of Proposition~\ref{prop:INT-IDLE-2}.}
\label{appendix:prop3.2}
In this appendix, we will prove the auxiliary results for the proof of Proposition \ref{prop:INT-IDLE-2}. 
Let us start with the proof of Lemma \ref{excursion time}.
The following technical lemma from~\cite{GW19} will be used in the proof of Lemma \ref{excursion time}.
\begin{lemma}[{\cite[Lemma EC.18]{GW19}}]\label{lemma:ec18}
Let $Q$ be the length of an M/M/1 queue with arrival rate $\alpha$ and service rate $\mu$, with $\mu>\alpha$. Also, let $\Tilde{T}$ be the length of the renewal cycle from the queue being of length 1 to the queue being empty. Then, for $t\geq0$,
$$\mathbb{P}\br{\Tilde{T}\geq t}\leq \sqrt{\frac{\mu}{\alpha}}e^{-(\sqrt{\mu}-\sqrt{\alpha})^2t}.$$
\end{lemma}

\begin{proof}[Proof of Lemma \ref{excursion time}.]
Note that for any $i \ge 0$,
\begin{equation}\label{eq:lemA2-1}
    \begin{split}
        &\sup_{\substack{x\leq K_1N^{\frac{1}{2} - \varepsilon},\\ BN^{\frac{1}{2} + \varepsilon} <y\leq K_2N^{\frac{1}{2} + \varepsilon}}}\mathbb{P}\Big(\xi^{(N)}_{2i+3}-\xi^{(N)}_{2i+1}\geq 2N^{\delta-2\varepsilon}, i <i^*_N\Big)\\
    \leq & \sup_{\substack{x\leq K_1N^{\frac{1}{2} - \varepsilon},\\ BN^{\frac{1}{2} + \varepsilon} <y\leq K_2N^{\frac{1}{2} + \varepsilon}}}\mathbb{P}_{(x,y,\underline{0})}\Big(\sup_{0\leq t\leq T(N,B)} Q^{(N)}_2(t)\geq N^{\frac{1}{2}+\varepsilon+\delta}\Big)\\
    &\qquad+\sup_{\substack{x\leq K_1N^{\frac{1}{2} - \varepsilon},\\ BN^{\frac{1}{2} + \varepsilon} <y\leq K_2N^{\frac{1}{2} + \varepsilon}}}\mathbb{P}_{(x,y,\underline{0})}\Big(\xi^{(N)}_{2i+2}-\xi^{(N)}_{2i+1}\geq N^{\delta-2\varepsilon}, \sup_{0\leq t\leq T(N,B)} Q^{(N)}_2(t)< N^{\frac{1}{2}+\varepsilon+\delta}, i <i^*_N\Big)\\
    &\qquad+\sup_{\substack{x\leq K_1N^{\frac{1}{2} - \varepsilon},\\ BN^{\frac{1}{2} + \varepsilon} <y\leq K_2N^{\frac{1}{2} + \varepsilon}}}\mathbb{P}_{(x,y,\underline{0})}\Big(\xi^{(N)}_{2i+3}-\xi^{(N)}_{2i+2}\geq N^{\delta-2\varepsilon}, \sup_{0\leq t\leq T(N,B)} Q^{(N)}_2(t)< N^{\frac{1}{2}+\varepsilon+\delta}, i <i^*_N\Big).
    \end{split}
\end{equation}
For the first term of the right hand side of \eqref{eq:lemA2-1}, take any $\tilde{K} \ge K_2$ large enough such that the bound in Proposition~\ref{prop:bdd-Q2} holds with $2B$ there replaced by $\tilde{K}$ for large enough $N$. Also assume that $N$ is large enough such that $N^{\delta}> x'_{\tilde{K}/2}$, where the latter constant appears in Proposition~\ref{prop:bdd-Q2}. For such large enough $N$, we obtain constants $c'_1$ and $c'_2$ such that
\begin{align}\label{eq:lemA2-2}
    &\sup_{\substack{x\leq K_1N^{\frac{1}{2} - \varepsilon},\\ BN^{\frac{1}{2} + \varepsilon} <y\leq K_2N^{\frac{1}{2} + \varepsilon}}}\mathbb{P}_{(x,y,\underline{0})}\Big(\sup_{0\leq t\leq T(N,B)} Q^{(N)}_2(t)\geq N^{\frac{1}{2}+\varepsilon+\delta}\Big)\nonumber\\
    \leq& \mathbb{P}_{(0,\lfloor \tilde K N^{\frac{1}{2} + \varepsilon}\rfloor,\underline{0})}\Big(\sup_{0\leq t\leq T(N,B)} Q^{(N)}_2(t)\geq N^{\frac{1}{2}+\varepsilon+\delta}\Big)\nonumber\\
    \leq& \exp\{-\bar{c}'TN^{\delta}\}+TN^{\delta}\big(\bar{c}_1\exp \big\{-\bar{c}_2N^{(\frac{1}{2}-\varepsilon)/5}\big\}+c^*_1\exp\{-c^*_2N^{\delta}\}+c^*_3\exp\{-c^*_4N^{\frac{4\varepsilon}{5}}N^{\frac{\delta}{5}}\}\big)\nonumber\\
     \leq& c'_1e^{-c'_2N^{\delta/5}},
\end{align}
where the second inequality is from Proposition~\ref{prop:bdd-Q2}. For the second term in \eqref{eq:lemA2-1}, there exist constants $c'_3$ and $c'_4$ such that for large enough $N$ and any $i \ge 0$,
\begin{align}\label{eq:lemA2-3}
    &\sup_{\substack{x\leq K_1N^{\frac{1}{2} - \varepsilon},\\ BN^{\frac{1}{2} + \varepsilon} <y\leq K_2N^{\frac{1}{2} + \varepsilon}}}\mathbb{P}_{(x,y,\underline{0})}\Big(\xi^{(N)}_{2i+2}-\xi^{(N)}_{2i+1}\geq N^{\delta-2\varepsilon}, \sup_{0\leq t\leq T(N,B)} Q^{(N)}_2(t)< N^{\frac{1}{2}+\varepsilon+\delta}, i <i^*_N\Big)\nonumber\\
    \leq &\mathbb{P}\Big(\mathrm{Exp}(N-N^{\frac{1}{2}+\varepsilon+\delta})\geq N^{\delta-2\varepsilon}\Big)\leq c'_3e^{-c'_4N^{\delta + 1 - 2\varepsilon}},
\end{align}
where $\mathrm{Exp}(N-N^{\frac{1}{2}+\varepsilon+\delta})$ 
is an exponential random variable with rate $N-N^{\frac{1}{2}+\varepsilon+\delta}$. The first inequality in \eqref{eq:lemA2-3} is due to the fact that the rate of increase of $I^{(N)}(t)$ when $I^{(N)}(t)=0$ is $N-Q^{(N)}_2$. Now, consider the last term in the right hand side of \eqref{eq:lemA2-1}.
Recall the evolution of the process $I^{(N)}(t)$ as in \eqref{eq:HAT-I}.
For $t\leq \tau_2^{(N)}(B)$, $Q^{(N)}_1(t)-Q^{(N)}_2(t)\leq N-BN^{\frac{1}{2}+\varepsilon}$, and thus, there is a natural coupling so that the process $I^{(N)}$ is upper bounded by an M/M/1 queue, $Q$, with arrival rate $\alpha:=N-B^{\frac{1}{2}+\varepsilon}$ and service rate $\mu:=N-\beta N^{\frac{1}{2}-\varepsilon}$. 
Therefore, for any $i \ge 0$,
\begin{equation}\label{eq:exc-upper-bdd}
    \sup_{\substack{x\leq K_1N^{\frac{1}{2} - \varepsilon},\\ BN^{\frac{1}{2} + \varepsilon} <y\leq K_2N^{\frac{1}{2} + \varepsilon}}}\mathbb{P}\big(\xi^{(N)}_{2i+3}-\xi^{(N)}_{2i+2}\geq t, i <i^*_N\big)\leq \mathbb{P}\big(\Tilde{T}\geq t\big),
\end{equation}
where $\Tilde{T}$ for the process $Q$ is as defined in Lemma~\ref{lemma:ec18}. Now, note that there exists $N_0>0$ such that for all $N\geq N_0$,
\begin{equation}\label{eq:sqrt-ratio}
\begin{split}
     \sqrt{\frac{\mu}{\alpha}}&=\sqrt{\frac{N-\beta N^{\frac{1}{2}-\varepsilon}}{N-B^{\frac{1}{2}+\varepsilon}}}\leq 2\\
     \big(\sqrt{\mu}-\sqrt{\alpha}\big)^2&=\big(\sqrt{N-\beta N^{\frac{1}{2}-\varepsilon}}-\sqrt{N-BN^{\frac{1}{2}+\varepsilon}}\big)^2\\
    &= \big(N-\beta N^{\frac{1}{2}-\varepsilon}\big)\Big(1-\sqrt{1-\frac{B-\beta N^{-2\varepsilon}}{N^{\frac{1}{2}-\varepsilon}-\beta N^{-2\varepsilon}}}\Big)^2
    \geq \frac{B^2}{32} N^{2\varepsilon},
\end{split}
\end{equation}
where the last inequality is due to the facts that $1-\sqrt{1-x}\geq \frac{x}{2}$, for $x\in(0,1)$, and for all $N\geq N_0$, 
$N-\beta N^{\frac{1}{2}-\varepsilon} \geq \frac{N}{2}$ and $\frac{B-\beta N^{-2\varepsilon}}{N^{\frac{1}{2}-\varepsilon}-\beta N^{-2\varepsilon}}\geq \frac{B}{2}N^{\varepsilon-\frac{1}{2}}$. 
Thus, plugging \eqref{eq:sqrt-ratio} into Lemma~\ref{lemma:ec18}, we get for any $t \ge 0$, $i \ge 0$,
\begin{equation}\label{eq:lemA2-4}
    \sup_{\substack{x\leq K_1N^{\frac{1}{2} - \varepsilon},\\ BN^{\frac{1}{2} + \varepsilon} <y\leq K_2N^{\frac{1}{2} + \varepsilon}}}\mathbb{P}\big(\xi^{(N)}_{2i+3}-\xi^{(N)}_{2i+2}\geq t,  i <i^*_N\big)
 \leq \mathbb{P}\big(\Tilde{T}\geq t\big)
 \leq 2e^{-\frac{B^2}{32} N^{2\varepsilon}t}.
\end{equation}
Plugging \eqref{eq:lemA2-2}, \eqref{eq:lemA2-3}, and \eqref{eq:lemA2-4}, into \eqref{eq:lemA2-1}, we obtain constants $c'_5$ and $c'_6$ such that for any $i \ge 0$,
\begin{align}\label{eq:lemA2-5}
    &\sup_{\substack{x\leq K_1N^{\frac{1}{2} - \varepsilon},\\ BN^{\frac{1}{2} + \varepsilon} <y\leq K_2N^{\frac{1}{2} + \varepsilon}}}\mathbb{P}\Big(\xi^{(N)}_{2i+3}-\xi^{(N)}_{2i+1}\geq 2N^{\delta-2\varepsilon}, i <i^*_N\Big)\nonumber\\
   \leq & c'_1e^{-c'_2N^{\delta/5}}+c'_3e^{-c'_4N^{\delta + 1 - 2\varepsilon}}+2e^{-\frac{B^2}{32} N^{\delta}}\leq c'_5e^{-c'_6N^{\delta/5}}.
\end{align}
To prove~\eqref{idle excursion time}, note that the value of $i^*_N$ is upper bounded by the number of increments in $I^{(N)}$ in the interval  $[0,\tau^{(N)}_2(B)\wedge (N^{2\varepsilon}T)]$, which is stochastically upper bounded by a Poisson random variable with parameter $N^{1+2\varepsilon}T$.
Hence, using standard concentration inequality for Poisson random variables (see \cite[Theorem 2.3(b)]{MC98}), we have 
\begin{equation}\label{eq:bound-i-star}
   \sup_{\substack{x\leq K_1N^{\frac{1}{2} - \varepsilon},\\ BN^{\frac{1}{2} + \varepsilon} <y\leq K_2N^{\frac{1}{2} + \varepsilon}}} \mathbb{P}\Big(i^*_N\geq \frac{3}{2}N^{1+2\varepsilon}T\Big)\leq\mathbb{P}\Big(\mathrm{Po}(N^{1+2\varepsilon}T)\geq \frac{3}{2}N^{1+2\varepsilon}T\Big)\leq  e^{-\frac{3}{32}N^{1+2\varepsilon}T},
\end{equation}
and therefore, using \eqref{eq:lemA2-5}, there exists $N_1\geq N_0$, depending on $T$, such that for all $N\geq N_1$,
\begin{align}\label{f}
    &\sup_{\substack{x\leq K_1N^{\frac{1}{2} - \varepsilon},\\ BN^{\frac{1}{2} + \varepsilon} <y\leq K_2N^{\frac{1}{2} + \varepsilon}}}\mathbb{P}_{(x,y, \underline{0})}\big(\exists\ i\in\big\{0,...,i^*_N-1\big\}\text{  such that  } \xi^{(N)}_{2i+3}-\xi^{(N)}_{2i+1}\geq N^{\delta-2\varepsilon}\big)\nonumber\\
    &\leq \sup_{\substack{x\leq K_1N^{\frac{1}{2} - \varepsilon},\\ BN^{\frac{1}{2} + \varepsilon} <y\leq K_2N^{\frac{1}{2} + \varepsilon}}}\Big[\mathbb{P}_{(x,y, 0)}\big(i^*_N\geq \frac{3}{2}N^{1+2\varepsilon}T\big)
    +\Big\lceil\frac{3}{2}N^{1+2\varepsilon}T\Big\rceil\cdot c'_5e^{-c'_6N^{\delta/5}}\Big]\nonumber\\
    &\leq e^{-\frac{3}{32}N^{1+2\varepsilon}T} + \Big\lceil\frac{3}{2}N^{1+2\varepsilon}T\Big\rceil\times c'_5e^{-c'_6N^{\delta/5}}
    \leq a_1e^{-b_1N^{\delta/5}},
\end{align}
where $a_1$ and $b_1$ are appropriate constants. 
\end{proof}

\begin{proof}[Proof of Proposition~\ref{prop:change-st}.]
Define the event $w^{(N)}\coloneqq\big\{\sup_{0\leq t\leq T(N,B)} Q^{(N)}_2(t)<N^{\frac{1}{2}+\varepsilon+\delta}\big\}$. Then
\begin{equation}\label{eq:lemA3-1}
    \begin{split}
   & \sup_{\substack{x\leq K_1N^{\frac{1}{2} - \varepsilon},\\ BN^{\frac{1}{2} + \varepsilon} <y\leq K_2N^{\frac{1}{2} + \varepsilon}}}\mathbb{P}_{(x,y, \underline{0})}\Big(\sup_{i\in \{0,...,i^*_N-1\}}\sup_{\xi^{(N)}_{2i+1}\leq t\leq \xi^{(N)}_{2i+3}}\left|S^{(N)}(t)-S^{(N)}(\xi^{(N)}_{2i+1})\right|\geq 13 N^{\frac{1}{2}-\varepsilon+\delta} \Big)\\
    \leq & \sup_{\substack{x\leq K_1N^{\frac{1}{2} - \varepsilon},\\ BN^{\frac{1}{2} + \varepsilon} <y\leq K_2N^{\frac{1}{2} + \varepsilon}}}\mathbb{P}_{(x,y,\underline{0})}\Big(\sup_{0\leq t\leq T(N,B)} Q^{(N)}_2(t)\geq N^{\frac{1}{2}+\varepsilon+\delta}\Big)\\
    &+\sup_{\substack{x\leq K_1N^{\frac{1}{2} - \varepsilon},\\ BN^{\frac{1}{2} + \varepsilon} <y\leq K_2N^{\frac{1}{2} + \varepsilon}}}\mathbb{P}_{(x,y, \underline{0})}\Big(\sup_{i\in \{0,...,i^*_N-1\}}\sup_{\xi^{(N)}_{2i+1}\leq t\leq \xi^{(N)}_{2i+2}}\left|S^{(N)}(t)-S^{(N)}(\xi^{(N)}_{2i+1})\right|\geq 3N^{\frac{1}{2}-\varepsilon+\delta},w^{(N)}\Big)\\
    &+\sup_{\substack{x\leq K_1N^{\frac{1}{2} - \varepsilon}\\ BN^{\frac{1}{2} + \varepsilon} <y\leq K_2N^{\frac{1}{2} + \varepsilon}}}\mathbb{P}_{(x,y, \underline{0})}\Big(\sup_{i\in \{0,...,i^*_N-1\}}\sup_{\xi^{(N)}_{2i+2}\leq t\leq \xi^{(N)}_{2i+3}}\left|S^{(N)}(t)-S^{(N)}(\xi^{(N)}_{2i+2})\right|\geq 10 N^{\frac{1}{2}-\varepsilon+\delta},w^{(N)}\Big).
    \end{split}
\end{equation}
By \eqref{eq:lemA2-2}, for all large enough $N$, 
\begin{equation}\label{eq:lemA3-2}
    \sup_{\substack{x\leq K_1N^{\frac{1}{2} - \varepsilon},\\ BN^{\frac{1}{2} + \varepsilon} <y\leq K_2N^{\frac{1}{2} + \varepsilon}}}\mathbb{P}_{(x,y,\underline{0})}\Big(\sup_{0\leq t\leq T(N,B)} Q^{(N)}_2(t)\geq N^{\frac{1}{2}+\varepsilon+\delta}\Big)
     \leq c'_1e^{-c'_2N^{\delta/5}}.
\end{equation}
For the second term in the right hand side of \eqref{eq:lemA3-1},
\begin{equation}\label{eq:lemA3-3}
    \begin{split}
    &\sup_{\substack{x\leq K_1N^{\frac{1}{2} - \varepsilon},\\ BN^{\frac{1}{2} + \varepsilon} <y\leq K_2N^{\frac{1}{2} + \varepsilon}}}\mathbb{P}_{(x,y, \underline{0})}\Big(\sup_{i\in \{0,...,i^*_N-1\}}\sup_{\xi^{(N)}_{2i+1}\leq t\leq \xi^{(N)}_{2i+2}}\left|S^{(N)}(t)-S^{(N)}(\xi^{(N)}_{2i+1})\right|\geq 3N^{\frac{1}{2}-\varepsilon+\delta},w^{(N)}\Big)\\
    \leq &\sup_{\substack{x\leq K_1N^{\frac{1}{2} - \varepsilon},\\ BN^{\frac{1}{2} + \varepsilon} <y\leq K_2N^{\frac{1}{2} + \varepsilon}}}\mathbb{P}_{(x,y, \underline{0})}\Big(\sup_{i\in \{0,...,i^*_N-1\}}\xi^{(N)}_{2i+2}-\xi^{(N)}_{2i+1}\geq N^{-\frac{1}{2}-\varepsilon+\delta},w^{(N)}\Big)\\
    &+\sup_{\substack{x\leq K_1N^{\frac{1}{2} - \varepsilon},\\ BN^{\frac{1}{2} + \varepsilon} <y\leq K_2N^{\frac{1}{2} + \varepsilon}}}\mathbb{P}_{(x,y, \underline{0})}\Big(\sup_{i\in \{0,...,i^*_N-1\}}\sup_{\xi^{(N)}_{2i+1}\leq t\leq \xi^{(N)}_{2i+2}}\left|S^{(N)}(t)-S^{(N)}(\xi^{(N)}_{2i+1})\right|\geq 3N^{\frac{1}{2}-\varepsilon+\delta},w^{(N)},\tilde{w}^{(N)}\Big),
    \end{split}
\end{equation}
where $\tilde{w}^{(N)}\coloneqq\big\{\sup_{i\in \{0,...,i^*_N-1\}}\xi^{(N)}_{2i+2}-\xi^{(N)}_{2i+1}< N^{-\frac{1}{2}-\varepsilon+\delta}\big\}$.
Since the instantaneous rate of increase of $I^{(N)}$ when $I^{(N)}(t)=0$ is $N-Q^{(N)}_2$, then for each $i\in\{0,...,i^*_N-1\}$, $\xi^{(N)}_{2i+2}-\xi^{(N)}_{2i+1}$ is upper bounded by an exponential random variable, $\mathrm{Exp}(N-N^{\frac{1}{2}+\varepsilon+\delta})$. Hence, proceeding along the same line as \eqref{eq:lemA2-3}, \eqref{eq:bound-i-star} and \eqref{f}, there exist constants $c_1$, $c_2$ such that for large enough $N$,
\begin{align}\label{eq:lemA3-4}
    &\sup_{\substack{x\leq K_1N^{\frac{1}{2} - \varepsilon},\\ BN^{\frac{1}{2} + \varepsilon} <y\leq K_2N^{\frac{1}{2} + \varepsilon}}}\mathbb{P}_{(x,y, \underline{0})}\Big(\sup_{i\in \{0,...,i^*_N-1\}}\xi^{(N)}_{2i+2}-\xi^{(N)}_{2i+1}\geq N^{-\frac{1}{2}-\varepsilon+\delta},w^{(N)}\Big)\nonumber\\
    & \leq  \mathbb{P}(i^*_N\geq \frac{3}{2}N^{1+2\varepsilon}T)+\Big\lceil\frac{3}{2}N^{1+2\varepsilon}T\Big\rceil\times e^{-(N-N^{\frac{1}{2}+\varepsilon+\delta})N^{^{-\frac{1}{2}-\varepsilon+\delta}}}\nonumber\\
    &\leq e^{-\frac{3}{32}N^{1+2\varepsilon}T}+\Big\lceil\frac{3}{2}N^{1+2\varepsilon}T\Big\rceil\times e^{-(N-N^{\frac{1}{2}+\varepsilon+\delta})N^{^{-\frac{1}{2}-\varepsilon+\delta}}}\leq  c_1e^{-c_2N^{\frac{1}{2}-\varepsilon+\delta}}.
\end{align}

For each $i\in\{0,...,i^*_N-1\}$ and $t\in[\xi^{(N)}_{2i+1},\xi^{(N)}_{2i+2})$, \begin{align*}
    \left|S^{(N)}(t)-S^{(N)}(\xi^{(N)}_{2i+1})\right|&\leq A\big((N-\beta N^{\frac{1}{2}-\varepsilon})t\big) - A\big((N-\beta N^{\frac{1}{2}-\varepsilon})\xi^{(N)}_{2i+1}\big)
    +D\big(Nt\big) - D\big(N\xi^{(N)}_{2i+1}\big)\\
    &\leq_{st} \mathrm{Po}\big(2N(\xi^{(N)}_{2i+2}-\xi^{(N)}_{2i+1})\big),
\end{align*}
where $\mathrm{Po}(\lambda)$ is a Poisson random variable with mean $\lambda$.
Therefore, there exist constants $c_3$ and $c_4$ such that for all $N$ large enough,
\begin{align}\label{eq:lemA3-5}
   & \sup_{\substack{x\leq K_1N^{\frac{1}{2} - \varepsilon},\\ BN^{\frac{1}{2} + \varepsilon} <y\leq K_2N^{\frac{1}{2} + \varepsilon}}}\mathbb{P}_{(x,y, \underline{0})}\Big(\sup_{i\in \{0,...,i^*_N-1\}}\sup_{\xi^{(N)}_{2i+1}\leq t\leq \xi^{(N)}_{2i+2}}\left|S^{(N)}(t)-S^{(N)}(\xi^{(N)}_{2i+1})\right|\geq 3N^{\frac{1}{2}-\varepsilon+\delta},w^{(N)},\tilde{w}^{(N)}\Big)\nonumber\\
   &\leq \mathbb{P}(i^*_N\geq \frac{3}{2}N^{1+2\varepsilon}T)+\Big\lceil\frac{3}{2}N^{1+2\varepsilon}T\Big\rceil \mathbb{P}\Big(\mathrm{Po}(2N^{\frac{1}{2}-\varepsilon+\delta})\geq3N^{\frac{1}{2}-\varepsilon+\delta}\Big)\nonumber\\
   &\leq  e^{-\frac{3}{32}N^{1+2\varepsilon}T}+\Big\lceil\frac{3}{2}N^{1+2\varepsilon}T\Big\rceil\times e^{-\frac{3}{16}N^{\frac{1}{2}-\varepsilon+\delta}}\leq c_3e^{-c_4N^{\frac{1}{2}-\varepsilon+\delta}}.
\end{align}
Plugging \eqref{eq:lemA3-4} and \eqref{eq:lemA3-5} into \eqref{eq:lemA3-3}, we have for all $N$ large enough,
\begin{align}\label{eq:lemA3-6}
    \sup_{\substack{x\leq K_1N^{\frac{1}{2} - \varepsilon},\\ BN^{\frac{1}{2} + \varepsilon} <y\leq K_2N^{\frac{1}{2} + \varepsilon}}}&\mathbb{P}_{(x,y, \underline{0})}\Big(\sup_{i\in \{0,...,i^*_N-1\}}\sup_{\xi^{(N)}_{2i+1}\leq t\leq \xi^{(N)}_{2i+2}}\left|S^{(N)}(t)-S^{(N)}(\xi^{(N)}_{2i+1})\right|\geq 3N^{\frac{1}{2}-\varepsilon+\delta},w^{(N)}\Big)\nonumber\\
    \leq& c_5e^{-c_6N^{\frac{1}{2}-\varepsilon+\delta}},
\end{align}
where $c_5$ and $c_6$ are constants. Now, consider the third term of the right hand side of \eqref{eq:lemA3-1}.
For any $x\leq K_1N^{\frac{1}{2} - \varepsilon}$ and $BN^{\frac{1}{2} + \varepsilon} <y\leq K_2N^{\frac{1}{2} + \varepsilon}$,
\begin{align}\label{eq:prop-A3-1}
    &\mathbb{P}_{(x,y, \underline{0})}\Big(\sup_{i\in \{0,...,i^*_N-1\}}\sup_{\xi^{(N)}_{2i+2}\leq t\leq \xi^{(N)}_{2i+3}}\left|S^{(N)}(t)-S^{(N)}(\xi^{(N)}_{2i+2})\right|\geq 10 N^{\frac{1}{2}-\varepsilon+\delta} \Big)\nonumber\\
    &\leq  \mathbb{P}_{(x,y, \underline{0})}(i^*_N>\frac{3}{2}N^{1+2\varepsilon}T)+\sum_{i=0}^{\lceil\frac{3}{2}N^{1+2\varepsilon}T\rceil-1}\mathbb{P}_{(x,y, \underline{0})}\Big(\sup_{\xi^{(N)}_{2i+2}\leq t\leq \xi^{(N)}_{2i+3}}\left|S^{(N)}(t)-S^{(N)}(\xi^{(N)}_{2i+2})\right|\geq 10 N^{\frac{1}{2}-\varepsilon+\delta},i\leq i^*_N-1\Big)\nonumber\\
    &\leq  \mathbb{P}_{(x,y, \underline{0})}(i^*_N>\frac{3}{2}N^{1+2\varepsilon}T)
    +\mathbb{P}_{(x,y, \underline{0})}\Big(\sup_{t\in[0,N^{2\varepsilon}T]}Q^{(N)}_3(t)>0\Big)\nonumber\\
    &+\sum_{i=0}^{\lceil\frac{3}{2}N^{1+2\varepsilon}T\rceil-1}\mathbb{P}_{(x,y, \underline{0})}\Big(\sup_{\xi^{(N)}_{2i+2}\leq t\leq \xi^{(N)}_{2i+3}}\left|S^{(N)}(t)-S^{(N)}(\xi^{(N)}_{2i+2})\right|\geq 10 N^{\frac{1}{2}-\varepsilon+\delta}, \sup_{\xi^{(N)}_{2i+2}\leq t\leq \xi^{(N)}_{2i+3}}Q^{(N)}_3(t)=0,i\leq i^*_N-1\Big).
\end{align}
Recall that, by \eqref{eq:bound-i-star},
\begin{equation}\label{eq:prop-A3-2}
    \sup_{\substack{x\leq K_1N^{\frac{1}{2} - \varepsilon},\\ BN^{\frac{1}{2} + \varepsilon} <y\leq K_2N^{\frac{1}{2} + \varepsilon}}}\mathbb{P}_{(x,y, \underline{0})}(i^*_N>\frac{3}{2}N^{1+2\varepsilon}T)\leq e^{-\frac{3}{32}N^{1+2\varepsilon}T}.
\end{equation}
By Proposition~\ref{prop:bdd-Q2}, there exist constants $B_0, c_1, c_2, N_0'>0$, possibly depending on $T$, such that for all $N\geq N_0'$,
\begin{align}\label{eq:prop-A3-3}
    &\sup_{\substack{x\leq K_1N^{\frac{1}{2} - \varepsilon},\\ BN^{\frac{1}{2} + \varepsilon} <y\leq K_2N^{\frac{1}{2} + \varepsilon}}}\mathbb{P}_{(x,y, \underline{0})}\Big(\sup_{t\in[0,N^{2\varepsilon}T]} Q^{(N)}_3>0\Big)\nonumber\\
    &\hspace{4cm}\leq  \mathbb{P}_{(0,\lfloor B_0N^{\frac{1}{2}+\varepsilon}\rfloor,\underline{0})}\Big(\sup_{t\in[0,N^{2\varepsilon}T]}Q^{(N)}_2(t)=N\Big)\leq c_7e^{-c_8N^{(\frac{1}{2}-\varepsilon)/5}}.
\end{align}
Next, for  $i\leq i^*_N-1$ and $t\in[\xi^{(N)}_{2i+2},\xi^{(N)}_{2i+3})$, observe that on the event $\sup_{t\in [0, N^{2\varepsilon}T]}Q_3^{(N)}(t) =0$, 
\begin{align*}
   \sup_{t\in[\xi^{(N)}_{2i+2},\xi^{(N)}_{2i+3})}\left| S^{(N)}(t)-S^{(N)}(\xi_{2i+2})\right| &\leq \sup_{t\in[\xi^{(N)}_{2i+2},\xi^{(N)}_{2i+3})}\Big(I^{(N)}(t)+Q_2^{(N)}(\xi^{(N)}_{2i+2})-Q_2^{(N)}(\xi^{(N)}_{2i+3})\Big)\\
   & \le \sup_{t\in[\xi^{(N)}_{2i+2},\xi^{(N)}_{2i+3})}I^{(N)}(t) + Q_2^{(N)}(\xi^{(N)}_{2i+2}).
\end{align*}
Indeed, the inequalities are due to the fact that when there is no queue with length greater than $3$, the difference between $S^{(N)}(t)$ and $S^{(N)}(\xi_{2i+2})$ is caused by the change of $I^{(N)}(t)$ and the change of $Q_2^{(N)}(t)$, and for $t\in[\xi^{(N)}_{2i+2},\xi^{(N)}_{2i+3})$, $I^{(N)}(t)$ is always positive so $Q_2^{(N)}(t)$ can only decrease during $[\xi^{(N)}_{2i+2},\xi^{(N)}_{2i+3})$. 
Now, consider the last term of \eqref{eq:prop-A3-1}. For any $x\leq K_1N^{\frac{1}{2} - \varepsilon}$ and $BN^{\frac{1}{2} + \varepsilon} <y\leq K_2N^{\frac{1}{2} + \varepsilon}$, any $i \ge 0$,
\begin{equation}\label{eq:prop-A3-4}
    \begin{split}
     &\mathbb{P}_{(x,y, \underline{0})}\Big(\sup_{\xi^{(N)}_{2i+2}\leq t\leq \xi^{(N)}_{2i+3}}\left|S^{(N)}(t)-S^{(N)}(\xi_{2i+2})\right|\geq 10 N^{\frac{1}{2}-\varepsilon+\delta}, \sup_{\xi^{(N)}_{2i+2}\leq t\leq \xi^{(N)}_{2i+3}}Q^{(N)}_3(t)=0,i\leq i^*_N-1\Big)\\
    &\leq  \mathbb{P}_{(x,y, \underline{0})}\Big(\sup_{\xi^{(N)}_{2i+2}\leq t\leq \xi^{(N)}_{2i+3}}I^{(N)}(t)\geq 5 N^{\frac{1}{2}-\varepsilon+\delta}, i\leq i^*_N-1\Big)\\
    &\hspace{2cm}+\mathbb{P}_{(x,y, \underline{0})}\Big(Q_2^{(N)}(\xi^{(N)}_{2i+2}) \geq 5N^{\frac{1}{2}-\varepsilon+\delta}, i\leq i^*_N-1\Big).
    \end{split}
\end{equation}

Since $\delta<\frac{1}{2}-\varepsilon$,
from Proposition \ref{prop:IDLE-INTEGRAL}, we have that for $N\geq N_0$, 
\begin{align}\label{eq:prop-A3-5}
&\sup_{\substack{x\leq K_1N^{\frac{1}{2} - \varepsilon},\\ BN^{\frac{1}{2} + \varepsilon} <y\leq K_2N^{\frac{1}{2} + \varepsilon}}}\mathbb{P}_{(x,y, \underline{0})}\Big(\sup_{\xi^{(N)}_{2i+2}\leq t\leq \xi^{(N)}_{2i+3}}I^{(N)}(t)\geq 5 N^{\frac{1}{2}-\varepsilon+\delta}, i\leq i^*_N-1\Big)\nonumber\\
    &\le \sup_{\substack{x\leq K_1N^{\frac{1}{2} - \varepsilon},\\ BN^{\frac{1}{2} + \varepsilon} <y\leq K_2N^{\frac{1}{2} + \varepsilon}}}\mathbb{P}_{(x,y, \underline{0})}\big(\sup_{0\leq t\leq (N^{2\varepsilon}T)\wedge \tau^{(N)}_2(B))}I^{(N)}(t)\geq 5 N^{\frac{1}{2}-\varepsilon+\delta}\big)\leq ae^{-bN^{\frac{\delta}{2}}},
\end{align}
where $a$ and $b$ are positive constants depending on $B$, $\beta$ and $T$ only. 

Proceeding as in \eqref{eq:lemA2-2}, by Proposition~\ref{prop:bdd-Q2}, there exist $c_1', c_2', N_0'>0$, such that for all $N\geq N_0'$,
\begin{equation}\label{eq:prop-A3-7}
\begin{split}
    &\sup_{\substack{x\leq K_1N^{\frac{1}{2} - \varepsilon},\\ BN^{\frac{1}{2} + \varepsilon} <y\leq K_2N^{\frac{1}{2} + \varepsilon}}}\mathbb{P}_{(x,y, \underline{0})}\Big(\sup_{\xi^{(N)}_{2i+2}\leq t\leq \xi^{(N)}_{2i+3}}Q^{(N)}_2(t)\geq 5 N^{\frac{1}{2}+\varepsilon+\delta}, i\leq i^*_N-1\Big)\\
     &\leq \sup_{\substack{x\leq K_1N^{\frac{1}{2} - \varepsilon},\\ BN^{\frac{1}{2} + \varepsilon} <y\leq K_2N^{\frac{1}{2} + \varepsilon}}}\mathbb{P}_{(x,y, \underline{0})}\Big(\sup_{0\leq t\leq (N^{2\varepsilon}T)\wedge \tau^{(N)}_2(B))}Q^{(N)}_2(t)\geq 5 N^{\frac{1}{2}+\varepsilon+\delta}\Big)\leq c'_1e^{-c'_2N^{\delta/5}}.
\end{split}
\end{equation} 
Hence, plugging \eqref{eq:prop-A3-5} and \eqref{eq:prop-A3-7} into \eqref{eq:prop-A3-4}, we have that for all sufficiently large $N$, any $i \ge 0$,
\begin{equation}\label{eq:prop-A3-9}
    \mathbb{P}\Big(\sup_{\xi^{(N)}_{2i+2}\leq t\leq \xi^{(N)}_{2i+3}}\left|S^{(N)}(t)-S^{(N)}(\xi_{2i+2})\right|\geq 10 N^{\frac{1}{2}-\varepsilon+\delta}, \sup_{\xi^{(N)}_{2i+2}\leq t\leq \xi^{(N)}_{2i+3}}Q^{(N)}_3(t)=0, i\leq i^*_N-1\Big)\leq c''_1e^{-c''_2N^{\frac{\delta}{5}}},
\end{equation}
where $c''_1$ and $c''_2$ are constants dependent on $K_2$, $B$, $\beta$ and $T$.
Plugging \eqref{eq:prop-A3-2}, \eqref{eq:prop-A3-3}, and \eqref{eq:prop-A3-9} into \eqref{eq:prop-A3-1}, we obtain for sufficiently large $N$,
\begin{align}\label{eq:lemA3-7}
    &\mathbb{P}_{(x,y, \underline{0})}\Big(\sup_{i\in \{0,...,i^*_N-1\}}\sup_{\xi^{(N)}_{2i+2}\leq t\leq \xi^{(N)}_{2i+3}}\left|S^{(N)}(t)-S^{(N)}(\xi^{(N)}_{2i+2})\right|\geq 10 N^{\frac{1}{2}-\varepsilon+\delta} \Big)\nonumber\\
    \leq & e^{-\frac{3}{32}N^{1+2\varepsilon}T}+c_7e^{-c_8N^{(\frac{1}{2}-\varepsilon)/5}}+\Big\lceil\frac{3}{2}N^{1+2\varepsilon}T\Big\rceil\times c''_1e^{-c''_2N^{\frac{\delta}{5}}}\leq c''_3 e^{-c''_4 N^{\frac{\delta}{5}}}.
\end{align}
Finally, plugging \eqref{eq:lemA3-2}, \eqref{eq:lemA3-6} and \eqref{eq:lemA3-7} into \eqref{eq:lemA3-1}, we can choose appropriate constants $a_2$ and $b_2$ such that for all sufficiently large $N$,
\begin{align*}
\sup_{\substack{x\leq K_1N^{\frac{1}{2} - \varepsilon},\\ BN^{\frac{1}{2} + \varepsilon} <y\leq K_2N^{\frac{1}{2} + \varepsilon}}}\mathbb{P}_{(x,y, \underline{0})}\Big(\sup_{i\in \{0,...,i^*_N-1\}}\sup_{\xi^{(N)}_{2i+1}\leq t\leq \xi^{(N)}_{2i+3}}\left|S^{(N)}(t)-S^{(N)}(\xi^{(N)}_{2i+1})\right|\geq 13 N^{\frac{1}{2}-\varepsilon+\delta} \Big)\leq a_2e^{-b_2N^{\frac{\delta}{5}}}.
\end{align*}
\end{proof}

\begin{lemma}\label{lem:bbd-i*}
The following holds for all large enough $N$:
$$\mathbb{P}\Big(\big\{i^*_N\geq \lceil 12TN^{\frac{1}{2}+3\varepsilon+\delta}\rceil\big\}\cap \bar{E}^{(N)}\Big)\leq \frac{1}{9T}N^{-4\varepsilon-2\delta}.$$
\end{lemma}
\begin{proof}[Proof of Lemma~\ref{lem:bbd-i*}.]
On $\bar{E}^{(N)}$, we have that for $t\in[0,T(N,B)]$,
\begin{align*}
    2N-S^{(N)}(t)-2 I^{(N)}(t)&=N-I^{(N)}(t)-Q^{(N)}_2(t)
    \geq N- 5N^{\frac{1}{2}-\varepsilon+\delta}-  N^{\frac{1}{2}+\varepsilon+\delta}
    \geq N-6 N^{\frac{1}{2}+\varepsilon+\delta}.
\end{align*}
On each $[\xi^{(N)}_{2i},\xi^{(N)}_{2i+1})$, we construct an M/M/1 queue $M^{(N)}_i$ with rate of increase $N-6N^{\frac{1}{2}+\varepsilon+\delta}$ and rate of decrease $N$. We can couple $M^{(N)}_i$ with $I^{(N)}$ on $[\xi^{(N)}_{2i},\xi^{(N)}_{2i+1})$ by setting $M^{(N)}_i(\xi^{(N)}_{2i})=I^{(N)}(\xi^{(N)}_{2i})=1$ such that $M^{(N)}_i(t)\leq I^{(N)}(t)$. Define $\xi^{(N)}_{m,i}\coloneqq\inf\{t\geq \xi^{(N)}_{2i}:M^{(N)}_i(t)=0\}$ . Due to the coupling, we have $\xi^{(N)}_{m,i}-\xi^{(N)}_{2i}\leq \xi^{(N)}_{2i+1}-\xi^{(N)}_{2i}$. Now, we have
\begin{align*}
    \mathbb{P}\Big(\big\{i^*_N\geq \lceil 12TN^{\frac{1}{2}+3\varepsilon+\delta}\rceil\big\}\cap \bar{E}^{N}\Big)
    &\leq \mathbb{P}\Big(\Big\{\sum_{i=1}^{\lceil 12TN^{\frac{1}{2}+3\varepsilon+\delta}\rceil}\big(\xi^{(N)}_{2i+1}-\xi^{(N)}_{2i}\big)\leq N^{2\varepsilon}T\Big\}\cap \bar{E}^{N}\Big)\\
    &\leq \mathbb{P}\Big(\sum_{i=1}^{\lceil 12TN^{\frac{1}{2}+3\varepsilon+\delta}\rceil}\big(\xi^{(N)}_{m,i}-\xi^{(N)}_{2i}\big)\leq N^{2\varepsilon}T\Big).
\end{align*}
Let $H^{(N)}_i\coloneqq\xi^{(N)}_{m,i}-\xi^{(N)}_{2i}$. Since $H^{(N)}_i$ is the busy time of the M/M/1 queue $M^{(N)}_i$, we have $\mathbb{E}(H^{(N)}_i)=\frac{1}{6}N^{-\frac{1}{2}-\varepsilon-\delta}$ and  $\mathrm{Var}(H^{(N)}_i)=\frac{2}{6^3}N^{-\frac{1}{2}-3\varepsilon-3\delta}- \frac{1}{6^2}N^{-1-2\varepsilon-2\delta}\leq \frac{2}{6^3}N^{-\frac{1}{2}-3\varepsilon-3\delta}$ \cite[Section 7.9, Page 75]{adan2017queueing}. 
Hence, by Doob's $L^2$-maximal inequality, we have for all large enough $N$,
\begin{align*}
    \mathbb{P}\Big(\big\{i^*_N\geq \lceil12TN^{\frac{1}{2}+3\varepsilon+\delta}\rceil\big\}\cap \bar{E}^{N}\Big)&\leq\mathbb{P}\Big(\sum_{i=1}^{\lceil 12TN^{\frac{1}{2}+3\varepsilon+\delta}\rceil}\big(H^{(N)}_i-\mathbb{E}(H^{(N)}_i)\big)\leq -N^{2\varepsilon}T\Big)\\
    &\leq \frac{\lceil 12TN^{\frac{1}{2}+3\varepsilon+\delta}\rceil\times 2N^{-\frac{1}{2}-3\varepsilon-3\delta}}{6^3 N^{4\varepsilon}T^2} \le \frac{1}{9T}N^{-4\varepsilon-2\delta}.
\end{align*}
\end{proof}

The next lemma will be used in the proof of Lemma~\ref{M/M/1 behavior}.
\begin{lemma}\label{EC19}
Suppose that  $Q$ is an M/M/1 queue with arrival rate $\alpha$ and service rate $\mu$,with $\mu>\alpha$. Let $\Tilde{T}$ be the length of the excursion of the M/M/1 queue started from zero (that is, the time taken for the queue length to become non-zero and then zero again). Then 
\begin{equation}\label{e1}
\mathbb{E}\Big[\int_0^{\Tilde{T}}\big(Q(s)-\frac{\alpha}{\mu-\alpha}\big)ds\Big]=0.
\end{equation}
Moreover, there exists a universal positive constant $C$ not depending on $\alpha, \mu$ such that for any integer $a \ge 2$,
\begin{equation}\label{e2}
\begin{split}
    \mathbb{E}\Big[\Big(\int_0^{\Tilde{T}}Q(s)ds\Big)^2\Big] 
\le C\bigg[ a^2 \left(\frac{1}{\alpha^2} + \frac{\mu}{(\mu - \alpha)^3}\right) &+ \frac{(\mu/\alpha)}{(\log(\mu/\alpha))^2}  \Big(\frac{\alpha}{\mu}\Big)^{a/8}
\bigg(\frac{1}{\alpha^2} +\frac{\big(\mu/\alpha\big)^{1/4}}{(\sqrt{\mu} - \sqrt{\alpha})^4}\bigg)\bigg].
\end{split}
\end{equation}
\end{lemma}

\begin{proof}
Throughout this proof, we will denote by $C'$ a universal positive constant, not depending on $\alpha, \mu$, whose value might change between lines.

\eqref{e1} is proved in \cite[Lemma EC.19]{GW19}. To prove \eqref{e2}, write $\overline{Q} := \sup_{0 \le s \le \tilde{T}}Q(s)$. Note that, for any integer $a \ge 2$,
\begin{align}\label{e21}
\mathbb{E}\Big[\left(\int_0^{\Tilde{T}}Q(s)ds\right)^2\Big] &\le \mathbb{E}\left[\left(a\Tilde{T} + \Tilde{T}\overline{Q}\mathds{1}\br{\overline{Q} \ge a}\right)^2\right] \le 2a^2 \mathbb{E}\left[\tilde{T}^2\right] + 2\mathbb{E}\left[\left(\Tilde{T}\overline{Q}\mathds{1}\br{\overline{Q} \ge a}\right)^2\right]\nonumber\\
&\le 2a^2 \mathbb{E}\left[\tilde{T}^2\right] + 2\sqrt{\mathbb{E}\left[\tilde{T}^4\right]}\sqrt{\mathbb{E}\left[\overline{Q}^4\mathds{1}\br{\overline{Q} \ge a}\right]},
\end{align}
where the last inequality follows from the Cauchy-Schwarz inequality. We can write $\tilde{T} = \tilde{T}_1 + \tilde{T}_2$ where $\tilde{T}_1$ is an $\mathrm{Exp}(\alpha)$ random variable denoting the arrival time of the first task and $\tilde{T}_2$ denotes the time taken after $\tilde{T}_1$ for the queue to become empty again (busy time). Using $\mathbb{E}[\tilde{T}_1^2] = 2\alpha^{-2}$ and \cite[Section 7.9, Page 75]{adan2017queueing},
\begin{equation}\label{e22}
\mathbb{E}\left[\tilde{T}^2\right] \le 2 \mathbb{E}\left[\tilde{T}_1^2\right] + 2 \mathbb{E}\left[\tilde{T}_2^2\right] \le 4\left[\frac{1}{\alpha^2} + \frac{\mu}{(\mu - \alpha)^3}\right].
\end{equation}
Recall from Lemma \ref{lemma:ec18} that
$
\mathbb{P}\br{\Tilde{T}_2\geq t}\leq \sqrt{\frac{\mu}{\alpha}}e^{-(\sqrt{\mu}-\sqrt{\alpha})^2t}, \, t \ge 0.
$
From this, we obtain
\begin{equation}\label{e23}
\mathbb{E}\left[\tilde{T}^4\right] \le 2^4 \mathbb{E}\left[\tilde{T}_1^4\right] + 2^4 \mathbb{E}\left[\tilde{T}_2^4\right] \le C'\left[\frac{1}{\alpha^4} + \sqrt{\frac{\mu}{\alpha}}\frac{1}{(\sqrt{\mu} - \sqrt{\alpha})^8}\right].
\end{equation}
Moreover, using \cite[Lemma EC.17]{GW19}, for any $x \in \mathbb{N}$,
$
\mathbb{P}\left(\overline{Q} \ge x\right) = \frac{1}{\sum_{n=1}^x \left(\frac{\mu}{\alpha}\right)^{n-1}} \le \left(\frac{\alpha}{\mu}\right)^{x-1}.
$
From the above bound, we obtain
\begin{align}\label{e24}
\mathbb{E}\left[\overline{Q}^4\mathds{1}\br{\overline{Q} \ge a}\right] \le \sum_{x=a^4}^{\infty} \left(\frac{\alpha}{\mu}\right)^{\lfloor x^{1/4}\rfloor-1} \le \frac{\mu^2}{\alpha^2}\int_{a^4-1}^{\infty} e^{-\log(\mu/\alpha) z^{1/4}}dz \le \frac{C'}{(\log(\mu/\alpha))^4}\frac{\mu^2}{\alpha^2} \left(\frac{\alpha}{\mu}\right)^{a/4}.
\end{align}
Using \eqref{e22}, \eqref{e23} and \eqref{e24} in \eqref{e21}, we obtain
$$
\mathbb{E}\Big[\left(\int_0^{\Tilde{T}}Q(s)ds\right)^2\Big] \le 8a^2\left[\frac{1}{\alpha^2} + \frac{\mu}{(\mu - \alpha)^3}\right] + 2C'\left[\frac{1}{\alpha^2} + \frac{\left(\mu/\alpha\right)^{1/4}}{(\sqrt{\mu} - \sqrt{\alpha})^4}\right]\frac{\mu/\alpha}{(\log(\mu/\alpha))^2} \left(\frac{\alpha}{\mu}\right)^{a/8},
$$
which proves the lemma.
\end{proof}

\begin{proof}[Proof of Lemma~\ref{M/M/1 behavior}.]
Throughout this proof, we will denote by $C,C'$ universal positive constants, possibly depending on $B,T$ but not $N$, whose values might change between lines.

We will only prove the first case since the second case is similar. 
As before, due to Lemma~\ref{lem:EN-to-1}, we will prove the convergence to zero on the event $\bar{E}^{(N)}$.
Note that the process $I^{(N)}_{u,i}(t)$ is defined only on the time interval $[\xi^{(N)}_{2i+1}, \xi^{(N)}_{2i+3})$.
However, let us assume that $I^{(N)}_{u,i}(t)$ is allowed to continue past time $\xi^{(N)}_{2i+3}$, $i\in \{0,...,i^*_N-1\},$ and let $\xi^{(N)}_{u,i}\coloneqq\inf \{t\geq \xi^{(N)}_{2i+3}: I^{(N)}_{u,i}(t)=0\}$.
Further, let 
$
Z^{(N)}_i=\int_{\xi^{(N)}_{2i-1}}^{\xi^{(N)}_{u,i-1}}\Big(I^{(N)}_{u,i-1}(s)-\frac{\lambda^{(N)}_{u,i-1}}{\mu^{(N)}_{u,i-1}-\lambda^{(N)}_{u,i-1}}\Big)ds, \ i \ge 1.
$
We expand the first integral in the lemma for $j \ge 0$ as follows:
\begin{align}\label{eq:lemmaA7-1}
    &\frac{1}{N^{\frac{1}{2}+\varepsilon}}\sum_{i=0}^j\int_{\xi^{(N)}_{2i+1}}^{\xi^{(N)}_{2i+3}}\Big(I^{(N)}_{u,i}(s)-\frac{\lambda^{(N)}_{u,i}}{\mu^{(N)}_{u,i}-\lambda^{(N)}_{u,i}}\Big)ds\nonumber\\
    =&\frac{1}{N^{\frac{1}{2}+\varepsilon}}\sum_{i=0}^{j}Z^{(N)}_{i+1}-\frac{1}{N^{\frac{1}{2}+\varepsilon}}\sum_{i=0}^{j}\int_{\xi^{(N)}_{2i+3}}^{\xi^{(N)}_{u,i}}\Big(I^{(N)}_{u,i}(s)-\frac{\lambda^{(N)}_{u,i}}{\mu^{(N)}_{u,i}-\lambda^{(N)}_{u,i}}\Big)ds\nonumber\\
    \leq& \frac{1}{N^{\frac{1}{2}+\varepsilon}}\sum_{i=1}^{j+1}Z^{(N)}_{i}+\frac{1}{N^{\frac{1}{2}+\varepsilon}}\sum_{i=0}^{j}\int_{\xi^{(N)}_{2i+3}}^{\xi^{(N)}_{u,i}}\frac{\lambda^{(N)}_{u,i}}{\mu^{(N)}_{u,i}-\lambda^{(N)}_{u,i}}ds.
\end{align}
Define $M^{(N)}_0=0$ and
$
   M^{(N)}_j := \frac{1}{N^{\frac{1}{2}+\varepsilon}}\sum_{i=1}^{j}Z^{(N)}_{i}, \ j \ge 1.
$
Write $\tilde{\mathcal{G}}_i$ for the stopped natural filtration of the original queueing process up till time $\xi^{(N)}_{2i+1}$, $i \ge 0$. Define $\mathcal{G}_0 := \tilde{\mathcal{G}}_0$ and
$
\mathcal{G}_i := \tilde{\mathcal{G}}_i \cup \left(\cup_{l=0}^{i-1}\sigma \{I^{(N)}_{u,l}(t) : t \ge \xi^{(N)}_{2l+1}\}\right), \ i \ge 1.
$
By \eqref{e1}, $\mathbb{E}\left(Z^{(N)}_{i+1} \vert \mathcal{G}_{i}\right) = 0$ for any $i \ge 0$. Hence, $M^{(N)}$ is a martingale. 
We will show that, for any $\eta>0$,
\begin{equation}\label{ptozeromart}
\mathbb{P}\left(\left\lbrace \sup_{j \le i^*_N} \left|M^{(N)}_j\right|  > \eta \right\rbrace \cap \bar{E}^{(N)}\right) \rightarrow 0 \ \text{ as } \ N \rightarrow \infty.
\end{equation}
Define the stopping time $\tilde \tau$ with respect to $\{\mathcal{G}_i\}_{i \ge 0}$ by
$
\tilde \tau := \inf\{i \ge 0: \xi^{(N)}_{2i+1} \ge T(N,B)  \ \text{ or } \ Q^{(N)}_2(\xi^{(N)}_{2i+1}) >  N^{\frac{1}{2}+\varepsilon+\delta}\}.
$
Now, we estimate $\mathbb{E}\left[(Z^{(N)}_{i+1})^2 \vert \mathcal{G}_{i}\right]$ on the event $\{\tilde \tau \ge i+1\}$. Note that for $i \ge 0$,
\begin{equation}\label{qv1}
\mathbb{E}\left[(Z^{(N)}_{i+1})^2 \vert \mathcal{G}_{i}\right] \le 2\mathbb{E}\left[\left(\int_{\xi^{(N)}_{2i+1}}^{\xi^{(N)}_{u,i}}I^{(N)}_{u,i}(s)ds\right)^2 \Big| \mathcal{G}_{i}\right] + 2\mathbb{E}\left[\left(\frac{\lambda^{(N)}_{u,i}}{\mu^{(N)}_{u,i}-\lambda^{(N)}_{u,i}}\right)^2\left(\xi^{(N)}_{u,i} - \xi^{(N)}_{2i+1}\right)^2 \Big| \mathcal{G}_{i}\right].
\end{equation}
Recall $\lambda^{(N)}_{u,i}=2N-S^{(N)}(\xi^{(N)}_{2i+1})+ 13 N^{\frac{1}{2}-\varepsilon+\delta}$ and
 $\mu^{(N)}_{u,i}=N-\beta N^{\frac{1}{2}-\varepsilon}$. By \eqref{e2}, there exists a deterministic universal constant $C>0$ such that for any integer $a \ge 2$,
\begin{align*}
 &\mathbb{E}\bigg[\bigg(\int_{\xi^{(N)}_{2i+1}}^{\xi^{(N)}_{u,i}}I^{(N)}_{u,i}(s)ds\bigg)^2 \Big| \mathcal{G}_{i}\bigg]\le C\bigg[ a^2 \bigg(\frac{1}{\big(\lambda^{(N)}_{u,i}\big)^2} + \frac{\mu^{(N)}_{u,i}}{(\mu^{(N)}_{u,i} - \lambda^{(N)}_{u,i})^3}\bigg)\nonumber\\
 &\qquad + \bigg(\frac{1}{\big(\lambda^{(N)}_{u,i}\big)^2} + \bigg(\frac{\mu^{(N)}_{u,i}}{\lambda^{(N)}_{u,i}}\bigg)^{1/4}\frac{1}{\big(\sqrt{\mu^{(N)}_{u,i}} - \sqrt{\lambda^{(N)}_{u,i}}\big)^4}\bigg)\frac{1}{(\log(\mu^{(N)}_{u,i}/\lambda^{(N)}_{u,i}))^2} \frac{\mu^{(N)}_{u,i}}{\lambda^{(N)}_{u,i}} \bigg(\frac{\lambda^{(N)}_{u,i}}{\mu^{(N)}_{u,i}}\bigg)^{a/8}\bigg].
 \end{align*}
 Note that, on the event $\{\tilde \tau \ge i+1\}$, for sufficiently large $N$,
$
N - N^{\frac{1}{2} +\varepsilon + \delta} \le  \lambda^{(N)}_{u,i} \le N - \frac{B}{2}  N^{\frac{1}{2} +\varepsilon}.
$
Using these bounds, on the event $\{\tilde \tau \ge i+1\}$, for sufficiently large $N$,
\begin{align*}
\frac{1}{\left(\lambda^{(N)}_{u,i}\right)^2} + \frac{\mu^{(N)}_{u,i}}{(\mu^{(N)}_{u,i} - \lambda^{(N)}_{u,i})^3} &\le CN^{-\frac{1}{2} - 3\varepsilon},\\
\bigg(\frac{1}{\big(\lambda^{(N)}_{u,i}\big)^2} + \bigg(\frac{\mu^{(N)}_{u,i}}{\lambda^{(N)}_{u,i}}\bigg)^{1/4}\frac{1}{\big(\sqrt{\mu^{(N)}_{u,i}} - \sqrt{\lambda^{(N)}_{u,i}}\big)^4}\bigg)\frac{1}{(\log(\mu^{(N)}_{u,i}/\lambda^{(N)}_{u,i}))^2} \frac{\mu^{(N)}_{u,i}}{\lambda^{(N)}_{u,i}}  &\le CN^{1-6\varepsilon},\\
\bigg(\frac{\lambda^{(N)}_{u,i}}{\mu^{(N)}_{u,i}}\bigg)^{a/8} &\le e^{-C' a N^{-\frac{1}{2} + \varepsilon}}.
\end{align*}
Therefore, choosing $a =  N^{\frac{1}{2} - \varepsilon + \delta}$, for sufficiently large $N$, on the event $\{\tilde \tau \ge i+1\}$,
\begin{equation}\label{qv11}
\mathbb{E}\bigg[\bigg(\int_{\xi^{(N)}_{2i+1}}^{\xi^{(N)}_{u,i}}I^{(N)}_{u,i}(s)ds\bigg)^2 \Big| \mathcal{G}_{i}\bigg] \le CN^{\frac{1}{2} - 5\varepsilon + 2\delta}.
\end{equation}
Now, we estimate the second term in \eqref{qv1}. Note that, for sufficiently large $N$, 
$
\frac{\lambda^{(N)}_{u,i}}{\mu^{(N)}_{u,i}-\lambda^{(N)}_{u,i}} \le \frac{4}{B}N^{\frac{1}{2} - \varepsilon}.
$
Moreover, using \cite[Section 7.9, Page 75]{adan2017queueing}, on the event $\{\tilde \tau \ge i+1\}$, for sufficiently large $N$,
$$
\mathbb{E}\left[\left(\xi^{(N)}_{u,i} - \xi^{(N)}_{2i+1}\right)^2\right]  \le 4\bigg[\frac{1}{\left(\lambda^{(N)}_{u,i}\right)^2} + \frac{\mu^{(N)}_{u,i}}{(\mu^{(N)}_{u,i} - \lambda^{(N)}_{u,i})^3}\bigg] \le 4C N^{-\frac{1}{2} - 3\varepsilon}.
$$
Hence, on the event $\{\tilde \tau \ge i+1\}$, for sufficiently large $N$,
\begin{equation}\label{qv12}
\mathbb{E}\bigg[\bigg(\frac{\lambda^{(N)}_{u,i}}{\mu^{(N)}_{u,i}-\lambda^{(N)}_{u,i}}\bigg)^2\left(\xi^{(N)}_{u,i} - \xi^{(N)}_{2i+1}\right)^2 \Big| \mathcal{G}_{i}\bigg] \le C N^{\frac{1}{2} -5\varepsilon}.
\end{equation}
Using the bounds \eqref{qv11} and \eqref{qv12} in \eqref{qv1}, we conclude that, on the event $\{\tilde \tau \ge i+1\}$, for sufficiently large $N$,
\begin{equation}\label{qv2}
\mathbb{E}\left[(Z^{(N)}_{i+1})^2 \vert \mathcal{G}_{i}\right] \le CN^{\frac{1}{2} - 5\varepsilon + 2\delta}.
\end{equation}
The quadratic variation of the stopped martingale $\{M^{(N)}_{j \wedge \tilde \tau}\}_{j \ge 0}$ is given by: $\langle M^{(N)}\rangle_{0} = 0$ and
$
\langle M^{(N)}\rangle_{j\wedge \tilde \tau} = N^{-1 - 2\varepsilon}\sum_{i=1}^{j \wedge \tilde \tau}\mathbb{E}\left[(Z^{(N)}_{i})^2 \vert \mathcal{G}_{i-1}\right], \  j \ge 1.
$
Using \eqref{qv2}, we obtain
$
\langle M^{(N)}\rangle_{j\wedge \tilde \tau} \le CN^{-\frac{1}{2} - 7\varepsilon + 2\delta} j, \ j \ge  1.
$
Hence, by Doob's $L^2$ inequality, for any $\eta>0$, for large enough $N$,
\begin{equation}\label{qv3}
\mathbb{P}\bigg(\sup_{j \le \tilde \tau \wedge \lceil 12TN^{\frac{1}{2}+3\varepsilon+\delta}\rceil} \left|M^{(N)}_j\right|  > \eta\bigg) = \mathbb{P}\bigg(\sup_{j \le \lceil 12TN^{\frac{1}{2}+3\varepsilon+\delta}\rceil} \left|M^{(N)}_{j \wedge \tilde \tau}\right|  > \eta\bigg)  \le 12CT\eta^{-2} N^{ - 4\varepsilon + 3\delta}.
\end{equation}
Note that, on the event $\bar{E}^{(N)}$, $\tilde \tau = i^*_N + 1$. Therefore, for any $\eta>0$, for sufficiently large $N$,
\begin{align*}
\mathbb{P}\bigg(\bigg\lbrace \sup_{j \le i^*_N} \left|M^{(N)}_j\right|  > \eta \bigg\rbrace \cap \bar{E}^{(N)}\bigg)  &\le \mathbb{P}\bigg(\sup_{j \le \tilde \tau \wedge \lceil 12TN^{\frac{1}{2}+3\varepsilon+\delta}\rceil} \left|M^{(N)}_{j}\right|  > \eta\bigg) + \mathbb{P}\Big(\big\{i^*_N\geq \lceil 12TN^{\frac{1}{2}+3\varepsilon+\delta}\rceil\big\}\cap \bar{E}^{(N)}\Big)\\
&\le 12CT\eta^{-2} N^{ - 4\varepsilon + 3\delta} + \frac{1}{9T}N^{-4\varepsilon-2\delta},
\end{align*}
where the last bound follows from \eqref{qv3} and Lemma \ref{lem:bbd-i*}. Recalling $\delta < \varepsilon/2$, \eqref{ptozeromart} follows from the above bound. 
Now consider the second term of the right hand side of the inequality \eqref{eq:lemmaA7-1}. 
For $i \le i^*_N-1$, for sufficient large $N$, on $\bar{E}^{(N)}$,
$
    0\leq \frac{\lambda^{(N)}_{u,i}}{\mu^{(N)}_{u,i}-\lambda^{(N)}_{u,i}}\leq \frac{4N^{\frac{1}{2}-\varepsilon}}{B}.
$
Additionally, let $\xi^{(N)}_{l,1,i}\coloneqq\inf \{t\geq \xi^{(N)}_{2i+1}:I^{(N)}_{l,i}(t)>0\}$ and $\xi^{(N)}_{l,i}\coloneqq\inf \{t\geq \xi^{(N)}_{l,1,i}:I^{(N)}_{l,i}(t)=0\}$, $i\in\{0,...,i^*_N-1\}$. We have that for sufficient large $N$,
\begin{align}\label{eq:loacl-A39}
    \sum_{i=1}^{j}\int_{\xi^{(N)}_{2i+3}}^{\xi^{(N)}_{u,i}}\frac{\lambda^{(N)}_{u,i}}{\mu^{(N)}_{u,i}-\lambda^{(N)}_{u,i}}ds&\leq 
    \sum_{i=1}^{j}\int_{\xi^{(N)}_{l,i}}^{\xi^{(N)}_{u,i}}\frac{\lambda^{(N)}_{u,i}}{\mu^{(N)}_{u,i}-\lambda^{(N)}_{u,i}}ds\nonumber\\
    &\leq \frac{4N^{\frac{1}{2}-\varepsilon}}{B}\sum_{i=1}^{i^*_N - 1}\Big(\xi^{(N)}_{u,i}-\xi^{(N)}_{l,i}\Big),
\end{align}
where the first inequality is due to the fact that for $t\in[\xi^{(N)}_{2i+1},\xi^{(N)}_{2i+3})$, $I^{(N)}_{l,i}(t)\leq I^{(N)}(t) \leq I^{(N)}_{u,i}(t)$, which implies that $\xi^{(N)}_{l,i}\leq \xi^{(N)}_{2i+3}\leq \xi^{(N)}_{u,i}$.
For an M/M/1 with arrival rate $\alpha$ and departure rate $\mu$, $\mathbb{E}(\tilde{T})=\frac{1}{\alpha}+\frac{1}{\mu-\alpha}$ where $\tilde{T}$ is same as in Lemma~\ref{EC19}. 
Then, for each $i\in\{0,...,i^*_N-1\}$, we have for all sufficiently large $N$, on $\bar{E}^{(N)}$,
\begin{align*}
   \mathbb{E}\Big(\xi^{(N)}_{u,i}-\xi^{(N)}_{l,i}\Big)= &\mathbb{E}\Big(\xi^{(N)}_{u,i}-\xi^{(N)}_{2i+1}\Big)-\mathbb{E}\Big(\xi^{(N)}_{l,i}-\xi^{(N)}_{2i+1}\Big)\\
   =&\frac{1}{2N-S^{(N)}(\xi^{(N)}_{2i+1})+ 13 N^{\frac{1}{2}-\varepsilon+\delta}}+\frac{1}{S^{(N)}(\xi^{(N)}_{2i+1})-N-\beta N^{\frac{1}{2}-\varepsilon}-13 N^{\frac{1}{2}-\varepsilon+\delta}}\\
   &-\frac{1}{2N-S^{(N)}(\xi^{(N)}_{2i+1})- 23 N^{\frac{1}{2}-\varepsilon+\delta}}-\frac{1}{S^{(N)}(\xi^{(N)}_{2i+1})-N + 23N^{\frac{1}{2}-\varepsilon+\delta}}\\
   \leq&\frac{1}{S^{(N)}(\xi^{(N)}_{2i+1})-N-\beta N^{\frac{1}{2}-\varepsilon}- 13 N^{\frac{1}{2}-\varepsilon+\delta}}-\frac{1}{S^{(N)}(\xi^{(N)}_{2i+1})-N + 23N^{\frac{1}{2}-\varepsilon+\delta}}\\
    \leq& CN^{-\frac{1}{2}-3\varepsilon + \delta}
\end{align*}
and from~\eqref{eq:loacl-A39},
$$\mathbb{E}\left[\left(\sup_{j\in\{0,...,i^*_N-1\}} \sum_{i=0}^{j}\int_{\xi^{(N)}_{2i+3}}^{\xi^{(N)}_{u,i}}\frac{\lambda^{(N)}_{u,i}}{\mu^{(N)}_{u,i}-\lambda^{(N)}_{u,i}}ds\right)\mathds{1}\br{\{i^*_N \le \lceil 12TN^{\frac{1}{2}+3\varepsilon+\delta}\rceil\}\cap \bar{E}^{(N)}}\right]\leq C'N^{\frac{1}{2}-\varepsilon + 2\delta}.$$
By Markov's inequality and Lemma~\ref{lem:bbd-i*}, we get that
\begin{equation}\label{qvmain2}
    \mathbb{P}\bigg(\bigg\lbrace\sup_{j\in\{0,...,i^*_N-1\}} \frac{1}{N^{\frac{1}{2}+\varepsilon}}\sum_{i=1}^{j}\int_{\xi^{(N)}_{2i+3}}^{\xi^{(N)}_{u,i}}\frac{\lambda^{(N)}_{u,i}}{\mu^{(N)}_{u,i}-\lambda^{(N)}_{u,i}}ds\geq C'N^{-2\varepsilon+ 4\delta}\bigg\rbrace \cap \bar{E}^{(N)} \bigg)\leq N^{-\delta}.
\end{equation}
From \eqref{eq:lemmaA7-1}, \eqref{ptozeromart} and \eqref{qvmain2}, we conclude that
$\sup_{j\in\{0,...,i^*_N-1\}}\frac{1}{N^{\frac{1}{2}+\varepsilon}}\sum_{i=0}^j\int_{\xi^{(N)}_{2i+1}}^{\xi^{(N)}_{2i+3}}\Big(I^{(N)}_{u,i}(s)-\frac{\lambda^{(N)}_{u,i}}{\mu^{(N)}_{u,i}-\lambda^{(N)}_{u,i}}\Big)ds\pto0.$ 
\end{proof}

\begin{proof}[Proof of Lemma~\ref{approx. rate}.]
We will prove \eqref{eq:A7-1}. 
The proof of \eqref{eq:A7-2} is similar. 
Note that due to Lemma~\ref{lem:EN-to-1}, it suffices to prove that the left hand side converges to zero under the event $\bar{E}^{(N)}$.
Note that, for any $i \in\{0,...,i^*_N-1\}$,
$\frac{\lambda^{(N)}_{l,i}}{\mu^{(N)}_{l,i}-\lambda^{(N)}_{l,i}}=\frac{2N-S^{(N)}(\xi^{(N)}_{2i+1})- 23N^{\frac{1}{2}-\varepsilon+\delta}}{S^{(N)}(\xi^{(N)}_{2i+1})-N+ 23 N^{\frac{1}{2}-\varepsilon+\delta}}.$
Now, $x\mapsto\frac{2N-x}{x-N}$ is Lipschitz continuous on the interval $[ N+\frac{1}{2}BN^{\frac{1}{2}+\varepsilon},\infty)$ with Lipschitz constant $\frac{4}{B^2}N^{-2\varepsilon}$. Therefore, for any $s \in [\xi^{(N)}_{2i+1},\xi^{(N)}_{2i+3})$, under the event $\bar{E}^{(N)}$,
\begin{align*}
    &\Big|\frac{2N-S^{(N)}(\xi^{(N)}_{2i+1})- 23 N^{\frac{1}{2}-\varepsilon+\delta}}{S^{(N)}(\xi^{(N)}_{2i+1})-N+ 23 N^{\frac{1}{2}-\varepsilon+\delta}}-\frac{2N-S^{(N)}(s)}{S^{(N)}(s)-N}\Big|\\
    \leq & \frac{4}{B^2}N^{-2\varepsilon}\bigg(\sup_{i\in \{0,...,i^*_N-1\}}\sup_{\xi^{(N)}_{2i+1}\leq s\leq \xi^{(N)}_{2i+3}}\left|S^{(N)}(\xi^{(N)}_{2i+1})-S^{(N)}(s)\right|+23 N^{\frac{1}{2}-\varepsilon+\delta}\bigg)\\
    \leq & \frac{144 N^{\frac{1}{2}-3\varepsilon+\delta}}{B^2}.
\end{align*}
Thus,
\begin{align*}
     &\sup_{j\in\{0,...,i^*_N-1\}} \frac{1}{N^{\frac{1}{2}+\varepsilon}}\Big|\sum_{i=0}^j\int_{\xi^{(N)}_{2i+1}}^{\xi^{(N)}_{2i+3}} \Big(\frac{\lambda^{(N)}_{l,i}}{\mu^{(N)}_{l,i}-\lambda^{(N)}_{l,i}}-\frac{2N-S^{(N)}(s)}{S^{(N)}(s)-N}\Big)ds\Big|\\
     \leq & \frac{1}{N^{\frac{1}{2}+\varepsilon}}\int_{0}^{N^{2\varepsilon}T}\frac{144N^{\frac{1}{2}-3\varepsilon+\delta}}{B^2}ds = \frac{144 T}{B^2}N^{-2\varepsilon + \delta},
\end{align*}
which tends to $0$ as $n\to\infty$, as $\delta <  \varepsilon/2$.
\end{proof}

\begin{proof}[Proof of Lemma~\ref{each excursion behavior}.]
Fix any $0<\delta< (2\varepsilon) \wedge \left(\frac{1}{2} - \varepsilon\right)$. 
By Lemma \ref{excursion time}, for all $N\geq N_1$,
\begin{equation}\label{eq:A2-1}
    \sup_{\substack{x\leq K_1N^{\frac{1}{2} - \varepsilon},\\ BN^{\frac{1}{2} + \varepsilon} <y\leq K_2N^{\frac{1}{2} + \varepsilon}}}\mathbb{P}_{(x,y, \underline{0})}\Big( \sup_{i\in\{0,...,i^*_N-1\}}\xi^{(N)}_{2i+3}-\xi^{(N)}_{2i+1}\geq 2 N^{\delta-2\varepsilon}\Big)\leq a_1e^{-b_1N^{\frac{\delta}{5}}}.
\end{equation}
By Proposition~\ref{prop:IDLE-INTEGRAL},
\begin{equation}\label{eq:A2-2}
    \mathbb{P}\big(\sup_{0\leq t\leq T(N,B)}I^{(N)}(t)\geq 5 N^{\frac{1}{2}-\varepsilon+\delta}\big)\leq ae^{-bN^{\frac{\delta}{2}}}.
\end{equation}
and similarly as in~\eqref{eq:prop-A3-3}, for all $N\geq N_0'$,
\begin{equation}\label{eq:lem4-3}
    \sup_{\substack{x\leq K_1N^{\frac{1}{2} - \varepsilon},\\ BN^{\frac{1}{2} + \varepsilon} <y\leq K_2N^{\frac{1}{2} + \varepsilon}}}\mathbb{P}_{(x,y, \underline{0})}\Big(\sup_{0\leq t\leq N^{2\varepsilon}T}Q^{(N)}_3(t)>0\Big)\leq c_1e^{-c_2N^{(\frac{1}{2}-\varepsilon)/5}}.
\end{equation}
Define the event 
\begin{align*}
    E^{(N)}\coloneqq\Big\{\sup_{0\leq t\leq T(N,B)} I^{(N)}(t)< 5 N^{\frac{1}{2}-\varepsilon+\delta}, &\sup_{0\leq t\leq T(N,B)}Q^{(N)}_3(t)=0,\text { and }\\
    &\xi^{(N)}_{2i+3}-\xi^{(N)}_{2i+1}< 2 N^{\delta-2\varepsilon},\forall i\in\{0,...,i^*_N-1\}\Big\}.
\end{align*}
Since $\delta<2\varepsilon$,
there exists $N'_B>0$ such that for $N\geq N'_B$, on the event $E^{(N)}$, for $s\in [0,T(N,B)]$,
\begin{align*}
    S^{(N)}(s)&=Q^{(N)}_2(s)+N-I^{(N)}(s)\\
    &\geq BN^{\frac{1}{2}+\varepsilon}+N- 5 N^{\frac{1}{2}-\varepsilon+\delta}\geq \frac{1}{2}BN^{\frac{1}{2}+\varepsilon}+N,
\end{align*}
implying that for $N\geq N'_B$ and $s\in [0,T(N,B)]$, 
$
    \frac{2N-S^{(N)}(s)}{S^{(N)}(s)-N}\leq \frac{2N^{\frac{1}{2}-\varepsilon}}{B}.
$
Hence, on the event $E^{(N)}$,
\begin{align*}
    \sup_{\xi^{(N)}_{2i+1}\leq t\leq \xi^{(N)}_{2i+3}}\left[\frac{1}{N^{\frac{1}{2}+\varepsilon}}\int_{\xi^{(N)}_{2i+1}}^t I^{(N)}(s)ds+\frac{1}{N^{\frac{1}{2}+\varepsilon}}\int_{\xi^{(N)}_{2i+1}}^t\frac{2N-S^{(N)}(s)}{S^{(N)}(s)-N}ds\right] &\leq 10 N^{2\delta-4\varepsilon}+\frac{4}{B}N^{\delta-4\varepsilon}\\
    &\rightarrow0\text{ as }N\rightarrow\infty,
\end{align*}
since $\delta<(2\varepsilon)\wedge(\frac{1}{2}-\varepsilon)$.
By \eqref{eq:A2-1}--\eqref{eq:lem4-3}, we have $\lim_{N\rightarrow\infty}\mathbb{P}(E^{(N)})=1$.  Thus, the desired result holds.
\end{proof}

\section{Upper bound on $I^{(N)}$}
\begin{prop}\label{prop:upper-bound-I}
For any fixed $K_1, K_2,T>0$ and $0<\delta<\varepsilon$, there exist constants $a, b, N_1>0$,
 such that
for all $N\geq N_1, x\leq K_1N^{\frac{1}{2}-\varepsilon}, y\leq K_2 N^{\frac{1}{2}+\varepsilon}$,
$$ \sup_{\underline{z}} \mathbb{P}_{(x, y, \underline{z})}\Big(\sup_{0\leq t\leq N^{2\varepsilon}T}I^{(N)}(t)\geq 2N^{\frac{1}{2}+\delta}\Big)\leq a e^{-b N^{2\delta}},$$
and consequently, for all $ \Tilde{\varepsilon}>0$,
\begin{equation}\label{eq:idle-integral-2}
    \lim_{N\rightarrow\infty}
    \sup_{\underline{z}} \mathbb{P}_{(x, y, \underline{z})}\Big(\frac{1}{N^{1+2\varepsilon}}\int_0^{N^{2\varepsilon}T} I^{(N)}(s)ds\geq \tilde{\varepsilon}\Big)=0.
\end{equation}
\end{prop}

\begin{proof}
Let $\Tilde{I}^{(N)}$ denote a birth-death process with increase rate $N-\Tilde{I}^{(N)}$ and decrease rate $N-\beta N^{\frac{1}{2}-\varepsilon}$.
By Lemma~\cite[Lemma EC.17]{GW19}, starting from $1$, the probability $h_1^x$ of $\Tilde{I}^{(N)}$ hitting $x$ before returning to $0$ is given by 
\begin{equation}\label{eq:app:B-2}
    \begin{split}
        h_1^x=\frac{1}{\sum_{n=1}^x\prod_{m=1}^{n-1}\frac{N-\beta N^{\frac{1}{2}-\varepsilon}}{N-m}}&\leq \prod_{m\leq x-1}\frac{N-m}{N-\beta N^{\frac{1}{2}-\varepsilon}}\\
        &\leq \exp\Big\{-\sum_{m\leq x-1} \frac{m - \beta N^{\frac{1}{2}-\varepsilon}}{N}\Big\}= \exp \Big\{-\frac{x(x-1)}{2N} + \frac{\beta N^{\frac{1}{2}-\varepsilon}(x-1)}{N}\Big\},
    \end{split}
\end{equation}
where the second inequality is due to $1-z\leq e^{-z}$, for $z \ge 0$. 
Now, recall $\xi_i^{
(N)}$ for $i\geq 1$, from~\eqref{eq:xi-i-def}.
Since rate of increase of $I^{(N)}$ is $Q^{(N)}_1-Q^{(N)}_2=N-I^{(N)}-Q^{(N)}_2\leq N-I^{(N)}$, 
in the interval $[\xi^{(N)}_{2i}, \xi^{(N)}_{2i+1}]$, we can naturally couple the processes $I^{(N)}(t)$ and $\Tilde{I}^{(N)}(t)$ with $I^{(N)}(\xi^{(N)}_{2i})=\tilde{I}^{(N)}(\xi^{(N)}_{2i})=1$, such that sample path-wise,
$I^{(N)}$ is dominated by the process $\tilde{I}^{(N)}$.
Thus, taking  $x=N^{\frac{1}{2}+\delta}$ in~\eqref{eq:app:B-2}, we have that for each $i\geq 1$ and large enough $N$, 
\begin{equation}\label{eq:app-b3}
    \sup_{i \ge 1}\sup_{x, y, \underline{z}}\mathbb{P}_{(x, y, \underline{z})}\Big(\sup_{\xi^{(N)}_{2i}\leq t\leq \xi^{(N)}_{2i+1}}I^{(N)}(t)\geq N^{\frac{1}{2}+\delta}\Big)\leq h_1^{N^{\frac{1}{2}+\delta}}\leq e^{-cN^{2\delta}},
\end{equation}
where $c$ is a positive constant. 
This bounds  $I^{(N)}$ on each of its excursions.
Now we will consider the interval $[0, \xi^{(N)}_1]$.
Take $N_1\in \N$ such that $N^{\frac{1}{2}+\delta}\geq K_1 N^{\frac{1}{2}-\varepsilon}$ and~\eqref{eq:app-b3} holds for all $N \ge N_1$. 
Note that for all $N\geq N_1$ and $x\leq K_1N^{\frac{1}{2}-\varepsilon}$, $y \ge 0$,
\begin{align*}
    &\mathbb{P}_{(x, y, \underline{z})}\Big(\sup_{0\leq t\leq \xi^{(N)}_1}I^{(N)}(t)\geq 2N^{\frac{1}{2}+\delta}\Big) \le \mathbb{P}\Big(\tilde I^{(N)}(t) \text{ hits } 2N^{\frac{1}{2}+\delta} \text{ before hitting } 0 \, \vert \, \tilde I^{(N)}(0) = x\Big)\\
    &\leq \mathbb{P}\Big(\tilde I^{(N)}(t) \text{ hits } 2N^{\frac{1}{2}+\delta} \text{ before hitting } 0  \, \vert \, \tilde I^{(N)}(0) = \lfloor K_1N^{\frac{1}{2}-\varepsilon}\rfloor\Big)\\
    &\leq \mathbb{P}\Big(\tilde I^{(N)}(t) \text{ hits } N^{\frac{1}{2}+\delta} \text{ before hitting } 0  \, \vert \, \tilde I^{(N)}(0) = 1\Big) \leq e^{-cN^{2\delta}},
\end{align*}
where for the third inequality, we consider two copies of the process $\tilde I^{(N)}$: $\big\{\tilde I^{(N)}_1(t),t>0| \tilde I^{(N)}_1(0)=1\big\}$ and $\big\{\tilde I^{(N)}_2(t),t>0 | \tilde I^{(N)}_2(0)=\lfloor K_1N^{\frac{1}{2}-\varepsilon}\big\rfloor\}$ which are coupled as follows: if there is a new arrival in $\tilde I^{(N)}_1$ (that is, $\tilde I^{(N)}_1$ increases by $1$), then there is an arrival in $\tilde I^{(N)}_2$ with probability $\frac{N- \tilde I^{(N)}_2}{N- \tilde I^{(N)}_1}$; if there is a departure in $\tilde I^{(N)}_2$, then there is a departure in $\tilde I^{(N)}_1$ (or if $\tilde I^{(N)}_1$ was zero before the departure, it stays zero after it). 
In this way, $|\tilde I^{(N)}_1(t)- \tilde I^{(N)}_2(t)|\leq K_1N^{\frac{1}{2}-\varepsilon}$ for all $t$ which implies that the third inequality holds. Also, writing $i^{**}_N := \inf\{i \ge 0: \xi^{(N)}_{2i+3} \ge N^{2\varepsilon}T\}$, we have by the same argument used to derive \eqref{eq:bound-i-star}, 
\begin{equation*}
   \sup_{\substack{x\leq K_1N^{\frac{1}{2} - \varepsilon},\\ y\leq K_2N^{\frac{1}{2} + \varepsilon}}} \mathbb{P}\Big(i^{**}_N\geq \frac{3}{2}N^{1+2\varepsilon}T\Big)\leq  e^{-\frac{3}{32}N^{1+2\varepsilon}T}.
\end{equation*}
Hence, uniformly over all $x\leq K_1N^{\frac{1}{2} - \varepsilon}$, $y\leq K_2N^{\frac{1}{2} + \varepsilon}$, and feasible $\underline{z}$,
\begin{equation*}
\begin{split}
       &  \mathbb{P}_{(x,y,\underline{z})}\Big(\sup_{0\leq t\leq N^{2\varepsilon}T}I^{(N)}\geq  2N^{\frac{1}{2}+\delta}\Big)\\
        &\leq  \mathbb{P}_{(x,y,\underline{z})}\Big(i^{**}_N\geq \frac{3}{2}N^{1+2\varepsilon}T\Big)
        +\mathbb{P}_{(x, y, \underline{z})}\Big(\sup_{0\leq t\leq \xi^{(N)}_1}I^{(N)}(t)\geq 2N^{\frac{1}{2}+\delta}\Big)\\
        &\hspace{4.5cm}+\frac{3}{2}N^{1+2\varepsilon}T \sup_{i \ge 1}\mathbb{P}_{(x,y,\underline{z})}\Big(\sup_{\xi^{(N)}_{2i}\leq t\leq \xi^{(N)}_{2i+1}}I^{(N)}\geq N^{\frac{1}{2}+\delta}\Big)
        \leq  a e^{-b N^{2\delta}}.
\end{split}
\end{equation*}
Moreover, for all $N\geq N_1$, $x\leq K_1N^{\frac{1}{2} - \varepsilon}$, $y\leq K_2 N^{\frac{1}{2}+\varepsilon}$,
\begin{align*}
    &\sup_{\underline{z}}\mathbb{P}_{(x, y, \underline{z})}\Big(\frac{1}{N^{1+2\varepsilon}}\int_0^{N^{2\varepsilon}T} I^{(N)}(s)ds\geq 2N^{-\frac{1}{2}+\delta}\Big)\\
    &\hspace{3cm}\quad\leq \sup_{\underline{z}}\mathbb{P}_{(x, y, \underline{z})}\Big(\sup_{0\leq s\leq N^{2\varepsilon}T}I^{(N)}(s)\geq 2N^{\frac{1}{2}+\delta}\Big)\leq ae^{-bN^{2\delta}}.
\end{align*}
This completes the proof. 
\end{proof}

\end{appendix}

%
%
%

\section*{Acknowledgments.}
Banerjee was supported in part by the NSF CAREER award DMS-2141621 and the NSF RTG grant DMS-2134107. Mukherjee and Zhao were partially supported by NSF grant CIF-2113027.


\bibliographystyle{unsrtnat} 


\bibliography{references-mendeley-debankur,reference-manual,mybibliography}

\begin{thebibliography}{42}
\providecommand{\natexlab}[1]{#1}
\providecommand{\url}[1]{\texttt{#1}}
\expandafter\ifx\csname urlstyle\endcsname\relax
  \providecommand{\doi}[1]{doi: #1}\else
  \providecommand{\doi}{doi: \begingroup \urlstyle{rm}\Url}\fi

\bibitem[Ephremides et~al.(1980)Ephremides, Varaiya, and Walrand]{EVW80}
Anthony Ephremides, P~Varaiya, and Jean Walrand.
\newblock {A simple dynamic routing problem}.
\newblock \emph{IEEE Trans. Autom. Control}, 25\penalty0 (4):\penalty0 690--693, 1980.

\bibitem[Winston(1977)]{Winston77}
Wayne Winston.
\newblock {Optimality of the shortest line discipline}.
\newblock \emph{J. Appl. Probab.}, 14\penalty0 (1):\penalty0 181--189, 1977.
\newblock ISSN 00219002.
\newblock URL \url{http://www.jstor.org/stable/3213271}.

\bibitem[Foschini(1977)]{Foschini77}
G~J Foschini.
\newblock {On heavy traffic diffusion analysis and dynamic routing in packet switched networks}.
\newblock \emph{Comp. Perf.}, pages 499--513, 1977.

\bibitem[Foschini and Salz(1978)]{FS78}
G.~Foschini and J.~Salz.
\newblock {A basic dynamic routing problem and diffusion}.
\newblock \emph{IEEE Trans. Commun.}, 26\penalty0 (3):\penalty0 320--327, 1978.
\newblock ISSN 0090-6778.
\newblock \doi{10.1109/TCOM.1978.1094075}.

\bibitem[Reiman(1984)]{Reiman84}
Martin~I. Reiman.
\newblock {Some diffusion approximations with state space collapse}.
\newblock In \emph{Modelling and performance evaluation methodology}, pages 207--240. Springer, Berlin, Heidelberg, 1984.

\bibitem[Zhang et~al.(1995)Zhang, Hsu, and Wang]{ZHW95}
Hanqin Zhang, Guang-Hui Hsu, and Rongxin Wang.
\newblock {Heavy traffic limit theorems for a sequence of shortest queueing systems}.
\newblock \emph{Queueing Syst.}, 21\penalty0 (1):\penalty0 217--238, 1995.

\bibitem[Mukherjee et~al.(2018)Mukherjee, Borst, van Leeuwaarden, and Whiting]{MBLW16-3}
Debankur Mukherjee, Sem~C Borst, Johan S~H van Leeuwaarden, and Philip~A Whiting.
\newblock {Universality of power-of-d load balancing in many-server systems}.
\newblock \emph{Stoch. Syst.}, 8\penalty0 (4):\penalty0 265--292, 2018.
\newblock \doi{10.1287/stsy.2018.0016}.
\newblock URL \url{https://doi.org/10.1287/stsy.2018.0016}.

\bibitem[Eschenfeldt and Gamarnik(2018)]{EG15}
Patrick Eschenfeldt and David Gamarnik.
\newblock {Join the shortest queue with many servers. The heavy traffic-asymptotics}.
\newblock \emph{Math. Oper. Res.}, 43\penalty0 (3):\penalty0 867–886, 2018.
\newblock \doi{moor.2017.0887}.
\newblock URL \url{https://pubsonline.informs.org/doi/pdf/10.1287/moor.2017.0887}.

\bibitem[Liu and Ying(2022)]{LY21}
Xin Liu and Lei Ying.
\newblock {Universal scaling of distributed queues under load balancing in the super-Halfin-Whitt regime}.
\newblock \emph{IEEE/ACM Trans. Netw.}, 30\penalty0 (1):\penalty0 190--201, 2022.
\newblock \doi{10.1109/TNET.2021.3105480}.

\bibitem[Hunt and Kurtz(1994)]{HK94}
P.J. Hunt and T.G. Kurtz.
\newblock {Large loss networks}.
\newblock \emph{Stoch. Proc. Appl.}, 53\penalty0 (2):\penalty0 363--378, 1994.
\newblock ISSN 03044149.
\newblock \doi{10.1016/0304-4149(94)90071-X}.
\newblock URL \url{http://www.sciencedirect.com/science/article/pii/030441499490071X}.

\bibitem[Braverman(2020{\natexlab{a}})]{Braverman18}
Anton Braverman.
\newblock {Steady-state analysis of the join-the-shortest-queue model in the Halfin-Whitt regime}.
\newblock \emph{Math. Oper. Res.}, 45\penalty0 (3):\penalty0 1069--1103, 2020{\natexlab{a}}.
\newblock \doi{10.1287/moor.2019.1023}.
\newblock URL \url{https://arxiv.org/pdf/1801.05121.pdf}.

\bibitem[Banerjee and Mukherjee(2019)]{BM19a}
Sayan Banerjee and Debankur Mukherjee.
\newblock {Join-the-shortest queue diffusion limit in Halfin-Whitt regime: Tail asymptotics and scaling of extrema}.
\newblock \emph{Ann. Appl. Probab.}, 29\penalty0 (2):\penalty0 1262--1309, 2019.
\newblock \doi{10.1214/18-AAP1436}.
\newblock URL \url{http://arxiv.org/abs/1803.03306}.

\bibitem[Banerjee and Mukherjee(2020)]{BM19b}
Sayan Banerjee and Debankur Mukherjee.
\newblock {Join-the-shortest queue diffusion limit in Halfin-Whitt regime: Sensitivity on the heavy traffic parameter}.
\newblock \emph{Ann. Appl. Probab.}, 30\penalty0 (1):\penalty0 80--144, 2020.
\newblock \doi{10.1214/19-AAP1496}.
\newblock URL \url{http://dx.doi.org/10.1214/19-AAP1496}.

\bibitem[Mukherjee et~al.(2016)Mukherjee, Borst, Van~Leeuwaarden, and Whiting]{MBLW16-1}
Debankur Mukherjee, Sem~C. Borst, Johan S.~H. Van~Leeuwaarden, and Philip~A. Whiting.
\newblock {Universality of load balancing schemes on the diffusion scale}.
\newblock \emph{J. Appl. Probab.}, 53\penalty0 (4), 2016.
\newblock \doi{10.1017/jpr.2016.68}.
\newblock URL \url{https://doi.org/10.1017/jpr.2016.68}.

\bibitem[Liu and Ying(2019)]{LY19}
Xin Liu and Lei Ying.
\newblock {A simple steady-state analysis of load balancing algorithms in the sub-Halfin-Whitt regime}.
\newblock \emph{ACM SIGMETRICS Perform. Eval. Rev.}, 46\penalty0 (2):\penalty0 15–17, 2019.
\newblock ISSN 0163-5999.
\newblock \doi{10.1145/3305218.3305225}.
\newblock URL \url{https://doi.org/10.1145/3305218.3305225}.

\bibitem[Atar(2012)]{Atar12}
Rami Atar.
\newblock {A diffusion regime with nondegenerate slowdown}.
\newblock \emph{Oper. Res.}, 60\penalty0 (2):\penalty0 490--500, 2012.
\newblock ISSN 0030-364X.
\newblock \doi{10.1287/opre.1110.1030}.
\newblock URL \url{https://doi.org/10.1287/opre.1110.1030}.

\bibitem[Mandelbaum(2008)]{AM08}
A.~Mandelbaum.
\newblock {QED} {Q}'s: Quality- and efficiency-driven queues.
\newblock \emph{Lecture Notes}, 2008.
\newblock URL \url{https://iew.technion.ac.il/serveng/References/Lecture_QED_Qs_Scientific_TAU.pdf}.

\bibitem[Whitt(2003)]{ww03}
Ward Whitt.
\newblock How multiserver queues scale with growing congestion-dependent demand.
\newblock \emph{Oper. Res.}, 51\penalty0 (4):\penalty0 531--542, 2003.
\newblock \doi{10.1287/opre.51.4.531.16093}.
\newblock URL \url{https://doi.org/10.1287/opre.51.4.531.16093}.

\bibitem[Gurvich(2004)]{IG04}
I.~Gurvich.
\newblock Design and control of the m/m/n queue with multi-class customers and many servers.
\newblock \emph{M.Sc. Thesis}, 2004.
\newblock URL \url{https://iew.technion.ac.il/serveng/References/ThesisItay.pdf}.

\bibitem[Maglaras et~al.(2018)Maglaras, Yao, and Zeevi]{MYZ18}
Costis Maglaras, John Yao, and Assaf Zeevi.
\newblock {Optimal price and delay differentiation in large-scale queueing systems}.
\newblock \emph{Manage. Sci.}, 64\penalty0 (5):\penalty0 2427--2444, 2018.
\newblock \doi{10.1287/mnsc.2016.2713}.
\newblock URL \url{https://doi.org/10.1287/mnsc.2016.2713}.

\bibitem[Gupta and Walton(2019)]{GW19}
Varun Gupta and Neil Walton.
\newblock {Load balancing in the nondegenerate slowdown regime}.
\newblock \emph{Oper. Res.}, 67\penalty0 (1):\penalty0 281--294, 2019.
\newblock \doi{10.1287/opre.2018.1768}.

\bibitem[Hurtado-Lange and Maguluri(0)]{HM20}
Daniela Hurtado-Lange and Siva~Theja Maguluri.
\newblock {Load balancing system under Join the Shortest Queue: Many-server-heavy-traffic asymptotics}.
\newblock \emph{QUESTA}, 0.
\newblock URL \url{http://arxiv.org/abs/2004.04826}.

\bibitem[Hurtado-Lange and Maguluri(2021)]{HM21}
Daniela Hurtado-Lange and Siva~Theja Maguluri.
\newblock {Load balancing system in the many-server heavy-traffic asymptotics}.
\newblock \emph{Preprint}, 2021.

\bibitem[Budhiraja et~al.(2019)Budhiraja, Friedlander, and Wu]{BFW19}
Amarjit Budhiraja, Eric Friedlander, and Ruoyu Wu.
\newblock {Many-server asymptotics for Join-the-Shortest-Queue: Large deviations and rare events}.
\newblock \emph{arXiv:1904.04938}, 2019.
\newblock URL \url{http://arxiv.org/abs/1904.04938}.

\bibitem[van~der Boor et~al.(2022)van~der Boor, Borst, van Leeuwaarden, and Mukherjee]{BBLM21}
M.~van~der Boor, S.C. Borst, J.S.H. van Leeuwaarden, and D.~Mukherjee.
\newblock {Scalable load balancing in networked systems: A survey of recent advances}.
\newblock \emph{SIAM Rev. (forthcoming)}, 2022.

\bibitem[Budhiraja and Lee(2009)]{budhiraja2009}
Amarjit Budhiraja and Chihoon Lee.
\newblock Stationary distribution convergence for generalized jackson networks in heavy traffic.
\newblock \emph{Mathematics of Operations Research}, 34\penalty0 (1):\penalty0 45--56, 2009.

\bibitem[Braverman(2020{\natexlab{b}})]{Anton20}
Anton Braverman.
\newblock Steady-state analysis of the join-the-shortest-queue model in the halfin–whitt regime.
\newblock \emph{Mathematics of Operations Research}, 45\penalty0 (3):\penalty0 1069--1103, 2020{\natexlab{b}}.
\newblock \doi{10.1287/moor.2019.1023}.
\newblock URL \url{https://doi.org/10.1287/moor.2019.1023}.

\bibitem[Meyn and Tweedie(2012)]{meyn2012}
Sean~P Meyn and Richard~L Tweedie.
\newblock \emph{Markov chains and stochastic stability}.
\newblock Springer Science \& Business Media, 2012.

\bibitem[Pang et~al.(2007)Pang, Talreja, and Whitt]{pang07}
Guodong Pang, Rishi Talreja, and Ward Whitt.
\newblock Martingale proofs of many-server heavy-traffic limits for markovian queues.
\newblock \emph{Probab. Surveys}, 4:\penalty0 193--267, 2007.
\newblock \doi{10.1214/06-PS091}.
\newblock URL \url{https://doi.org/10.1214/06-PS091}.

\bibitem[Bramson(2011)]{Bramson11}
Maury Bramson.
\newblock {Stability of join the shortest queue networks}.
\newblock \emph{Ann. Appl. Probab.}, 21\penalty0 (4):\penalty0 1568--1625, 2011.
\newblock \doi{10.1214/10-AAP726}.
\newblock URL \url{https://doi.org/10.1214/10-AAP726}.

\bibitem[Roberts and Tweedie(1996)]{RT96}
Gareth~O. Roberts and Richard~L. Tweedie.
\newblock {Exponential convergence of Langevin distributions and their discrete approximations}.
\newblock \emph{Bernoulli}, 2\penalty0 (4):\penalty0 341 -- 363, 1996.
\newblock \doi{bj/1178291835}.
\newblock URL \url{https://doi.org/}.

\bibitem[Braverman(2022)]{braverman2022join}
Anton Braverman.
\newblock {The join-the-shortest-queue system in the Halfin-Whitt regime: rates of convergence to the diffusion limit}.
\newblock \emph{arXiv:2202.02889}, 2022.
\newblock \doi{10.48550/arxiv.2202.02889}.
\newblock URL \url{https://arxiv.org/abs/2202.02889}.

\bibitem[Lakatos et~al.(2019)Lakatos, Szeidl, and Telek]{LST19}
László Lakatos, Laszlo Szeidl, and Miklos Telek.
\newblock \emph{{Introduction to queueing systems with telecommunication applications}}.
\newblock Springer, 2019.
\newblock ISBN 978-3-030-15141-6.
\newblock \doi{10.1007/978-3-030-15142-3}.

\bibitem[Billingsley(2012)]{pb12}
P.~Billingsley.
\newblock \emph{Probability and Measure}.
\newblock Wiley Series in Probability and Statistics. Wiley, 2012.
\newblock ISBN 9781118341919.
\newblock URL \url{https://books.google.com/books?id=a3gavZbxyJcC}.

\bibitem[Kleinrock(1975)]{KL75}
Leonard Kleinrock.
\newblock \emph{Theory, Volume 1, Queueing Systems}.
\newblock Wiley-Interscience, USA, 1975.
\newblock ISBN 0471491101.

\bibitem[Billingsley(2013)]{billingsley2013convergence}
Patrick Billingsley.
\newblock \emph{Convergence of probability measures}.
\newblock John Wiley \& Sons, 2013.

\bibitem[McDiarmid(1998)]{MC98}
Colin McDiarmid.
\newblock \emph{Concentration}.
\newblock Springer Berlin Heidelberg, Berlin, Heidelberg, 1998.
\newblock ISBN 978-3-662-12788-9.
\newblock \doi{10.1007/978-3-662-12788-9_6}.
\newblock URL \url{https://doi.org/10.1007/978-3-662-12788-9_6}.

\bibitem[Rogers and Williams(2000)]{RW00}
L.~C.~G. Rogers and David Williams.
\newblock \emph{Diffusions, Markov Processes, and Martingales}, volume~1 of \emph{Cambridge Mathematical Library}.
\newblock Cambridge University Press, 2 edition, 2000.
\newblock \doi{10.1017/CBO9781107590120}.

\bibitem[Karatzas and Shreve(2014)]{KS14}
I.~Karatzas and S.~Shreve.
\newblock \emph{Brownian Motion and Stochastic Calculus}.
\newblock Graduate Texts in Mathematics. Springer New York, 2014.
\newblock ISBN 9781461209492.
\newblock URL \url{https://books.google.com/books?id=w0SgBQAAQBAJ}.

\bibitem[Whitt(2007)]{whitt07}
Ward Whitt.
\newblock Proofs of the martingale fclt.
\newblock \emph{Probab. Surveys}, 4:\penalty0 268--302, 2007.
\newblock \doi{10.1214/07-PS122}.
\newblock URL \url{https://doi.org/10.1214/07-PS122}.

\bibitem[Liptser and Shiryayev(2012)]{LS}
Robert Liptser and Albert~Nikolaevich Shiryayev.
\newblock \emph{Theory of martingales}, volume~49.
\newblock Springer Science \& Business Media, 2012.

\bibitem[Adan and Resing(2015)]{adan2017queueing}
Ivo Adan and Jacques Resing.
\newblock \emph{{Queueing Systems}}.
\newblock Eindhoven University of Technology, 2015.
\newblock URL \url{https://www.win.tue.nl/~iadan/queueing.pdf}.

\end{thebibliography}

\end{document}